\documentclass[10pt]{amsart}

\usepackage{amssymb,amsthm,amsfonts,latexsym}
\usepackage{amsmath}
\usepackage{mathrsfs}
\usepackage{stmaryrd}
\usepackage{accents}
\usepackage[utf8]{inputenc}
\usepackage{color}
\usepackage{enumitem}
\usepackage{pbox}
\usepackage{graphicx}
\usepackage{hyperref}

\usepackage{tikz}
\usetikzlibrary{positioning}
\input{xy}
\xyoption{all}
\xyoption{poly}
\usepackage[all]{xy}
\allowdisplaybreaks

\theoremstyle{plain}
\newtheorem{theorem}{Theorem}[section]
\newtheorem*{theorem*}{Theorem}
\newtheorem{lemma}[theorem]{Lemma}
\newtheorem{corollary}[theorem]{Corollary}
\newtheorem*{corollary*}{Corollary}
\newtheorem{proposition}[theorem]{Proposition}

\theoremstyle{definition}
\newtheorem{example}[theorem]{Example}
\newtheorem{definition}[theorem]{Definition}
\newtheorem{hypothesis}[theorem]{Hypothesis}

\theoremstyle{remark}
\newtheorem{remark}[theorem]{Remark}

\theoremstyle{plain}
\newtheorem{innercustomgeneric}{\customgenericname}
\providecommand{\customgenericname}{}
\newcommand{\newcustomtheorem}[2]{%
  \newenvironment{#1}[1]
  {%
   \renewcommand\customgenericname{#2}%
   \renewcommand\theinnercustomgeneric{##1}%
   \innercustomgeneric
  }
  {\endinnercustomgeneric}
}

\newcustomtheorem{theoremNum}{Theorem}
\newcustomtheorem{lemmaNum}{Lemma}
\newcustomtheorem{corollaryNum}{Corollary}

\numberwithin{equation}{section}

\usepackage{environ}

\newcounter{quote}
\renewcommand{\thequote}{(H\arabic{quote})}

\NewEnviron{myquote}{\vspace{1ex}\par
\refstepcounter{quote}%
\hfill\llap{\thequote}\hfill\parbox{\dimexpr \textwidth-2cm}%
{\normalsize{{\BODY}}}%
\hfill\vspace{1ex}\par}

 \renewcommand\AA{{\mathbb{A}}}
 
 \def\QQ{\mathbb{Q}}
 
 \newcommand\ZZ{{\mathbb{Z}}}
 
 \def\A{\mathcal{A}}
 \def\B{{\mathcal B}}

 \def\E{{\mathcal E}}
 \def\F{{\mathcal F}}

 \def\K{{\mathcal K}}

 \def\N{{\mathcal N}}

 \def\R{{\mathcal R}}
 \def\Lk{{\text{Lk}}}
 \def\join{*}
 \def\St{\text{St}}  
 \def\ujoin{\,\underline{\join}\,}
 \def\op{\mathrm{op}}

 \newcommand{\GF}[1]{\mathbb{F}_{#1}}

 \DeclareMathOperator\PSL{PSL}

 \DeclareMathOperator{\Ln}{L}
 
 \DeclareMathOperator{\PSU}{PSU}
 \DeclareMathOperator{\U}{U}
 
 \DeclareMathOperator{\PSp}{PSp}
 
 \DeclareMathOperator{\Symplectic}{Sp}
 \DeclareMathOperator\Sz{Sz}

 \DeclareMathOperator\McL{McL}
 \newcommand{\Fi}[1]{{\mathrm{M}}(#1)}
 \DeclareMathOperator\Ly{{{Ly}}}
 \DeclareMathOperator\ON{{{O'N}}}
 \DeclareMathOperator\HS{{{HS}}}
 \DeclareMathOperator\He{{{He}}}
 \def\HN{F_5}
 \DeclareMathOperator\Ru{{{Ru}}}
 \def\Th{F_3}
 \DeclareMathOperator\Janko{{{J}}}
 \DeclareMathOperator\Mathieu{{{M}}}
 \newcommand{\Conway}[1]{\cdot{#1}}
 \def\Monster{F_1}
 \def\Baby{F_2}
 \DeclareMathOperator\Suz{{{Suz}}}

 \DeclareMathOperator{\Sym}{Sym}
 \DeclareMathOperator{\Alt}{Alt}
 \DeclareMathOperator{\Cyclic}{C}
 \DeclareMathOperator{\Dihedral}{D}
 \DeclareMathOperator\Lie{Lie}
 
 \DeclareMathOperator\Aut{Aut}
 
 \DeclareMathOperator\Out{Out}
 \DeclareMathOperator\Inn{Inn}
 \DeclareMathOperator\Outdiag{Out diag}
 \DeclareMathOperator\Inndiag{Inn diag}
 
 \DeclareMathOperator\Id{Id}

 \def\iret{r_{\mathfrak{i}}}
 
 \newcommand{\tq}{\mathrel{{\ensuremath{\: : \: }}}}

 \DeclareMathOperator\Fix{Fix}

 \def\BrownPoset{\mathcal{S}}
 \def\QuillenPoset{\mathcal{A}}
 \def\BoucPoset{\mathcal{B}}
 \def\Sp{\mathcal{S}_p} 
 \def\Ap{\mathcal{A}_p}
 \def\Bp{\mathcal{B}_p}
 \DeclareMathOperator\Outposet{{O}}

 \def\Imageposet{{\mathcal{A}}}
 \def\Phiposet{\Outposet^{\Phi}}
 \def\QD{(\mathcal{Q}\mathcal{D})}

 \DeclareMathOperator\Syl{Syl}
 
 \def\groupiso{\cong}

 \newcommand\gen[1]{\left\langle#1\right\rangle}
 \newcommand\Span[1]{\left\langle#1\right\rangle}

\newcommand{\donerk}{\hfill $\diamondsuit$ \smallskip \newline}

      \makeatletter
      \def\@setcopyright{}
      \def\serieslogo@{}
      \makeatother
   
\title  [Some results on Quillen's Conjecture \\ via equivalent-poset techniques]
        {Some results on Quillen's Conjecture \\ via equivalent-poset techniques}
\author{ Kevin Iv\'{a}n Piterman* \\
          Departamento de Matem\'{a}tica \\
	  IMAS-CONICET, FCEyN \\
	  Universidad de Buenos Aires \\
	  Buenos Aires  ARGENTINA \\  
	  e-mail: { \tt kpiterman@dm.uba.ar } \\ 
	       \\ 
        Stephen D. Smith \\ 
         Department of Mathematics  \\
         University of Illinois at Chicago \\
         Chicago Illinois USA \\
	 (home: 728 Wisconsin, Oak Park IL 60304 USA) \\
	 e-mail: {\tt smiths@uic.edu}       \\
       }

\thanks{*Supported by an Oberwolfach Leibniz Fellowship, a CONICET postdoctoral fellowship and grants PIP 11220170100357,
        PICT 2017-2997, and UBACYT 20020160100081BA}

\begin{document}

\begin{abstract}
We extend the Main Theorem of Aschbacher and Smith on Quillen's Conjecture
from~$p>5$ to the remaining odd primes~$p = 3,5$.
In the process, we develop further combinatorial and homotopical methods
for studying the poset of nontrivial elementary abelian~$p$-subgroups of a finite group.
The techniques lead to a number of further results on the Conjecture, 
often reducing dependence on the CFSG;
in particular, we also provide some partial results toward the case of~$p=2$.
   \end{abstract}

\subjclass[2010]{20J05, 20D05, 20D25, 20D30, 05E18, 06A11.}

\keywords{$p$-subgroups, Quillen's conjecture, posets, finite groups.}

\maketitle

\tableofcontents

\part{Introduction and relevant literature}

%\footnote{
%  fakenote:
%  I added a "part" subdivision, mainly to clarify the Table of Contents.
%  Do you like it?  (K) Yes!
%          }

\section{Introduction: Background and statement of principal results}

The original version of this paper was designed to prove an extension to~$p=3,5$ of the Main Theorem of~\cite{AS93},
which we state as Theorem~\ref{theoremASExtension} below.
But our methods also led to other useful results;
notably a closely related Theorem~\ref{theoremASExtensionAlternative}---which gives
simplifications both in the hypothesis for the extension,
and in certain aspects of the inductive proof (including a ``milder'' use of the~CFSG).
Other results contribute to the Conjecture for the remaining prime~$p = 2$.

This Introduction-section will provide some general background for all these results;
beginning with a discussion of these two versions of the extension to all odd~$p$,
with their similarities and differences.

\vspace{0.2cm}

Recall for a prime~$p$ that~$\Ap(G)$ is the poset of nontrivial elementary abelian~$p$-subgroups of a finite group~$G$.

Quillen showed (\cite[Prop 2.4]{Qui78}) that if~$O_p(G) > 1$, then~$\Ap(G)$ is contractible;
and the converse is the original statement of Quillen's Conjecture (see~\cite[2.9]{Qui78}):
that if~$O_p(G) = 1$, then~$\Ap(G)$ should be non-contractible. 
As in much of the subsequent literature, we instead work with a stronger ``homology'' form of the conjecture, 
replacing non-contractibility with non-vanishing of rational (reduced) homology: 

\vspace{0.2cm}

\begin{flushleft}
(H-QC) \quad If~$O_p(G) = 1$, then~$\tilde{H}_*(\Ap(G),\QQ) \neq 0$.
\end{flushleft}

\vspace{0.2cm}

\noindent
We will recall some of the general literature on the Conjecture, starting at later Theorem~\ref{psolvableQC}.

\subsection*{Extending the Aschbacher-Smith result and its overall strategy}

\bigskip

But next, we will begin our Introduction
by instead focusing on one particular contribution to that literature, namely the Main Theorem of Aschbacher-Smith in~\cite{AS93}. 
Indeed one of our principal results in this paper is Theorem~\ref{theoremASExtension} below,
in which we extend their original statement, by adding the cases~$p=3,5$ to their original treatment for~$p > 5$.
We mention that the original extension to~$p = 5$ was given in~\cite{KP20}---see later Corollary~\ref{extensionsASQC5}.

For the statement below,
the technical condition~$\QD_p$ is given in later Definition~\ref{defn:QD}---including the term~{\em $p$-extension\/},
for a split extension~$LB$ of some~$L$ by an elementary~$p$-group~$B$ of {\em outer\/} automorphisms
(we will call such a~$B$ a {\em $p$-outer\/} of~$L$).
The term~``(H2u)'' is a name we have invented for the restriction on unitary groups in the hypothesis in~\cite{AS93}.
For {\em components\/} of~$G$, see later Remark~\ref{rk:compsandF*G}---these quasisimple subnormal subgroups
are crucial to building up the structure of~$G$.

\begin{theorem}
\label{theoremASExtension}
We extend~\cite[Main Theorem]{AS93}, from~$p>5$ there, to all odd~$p$ here.

That is, suppose that~$G$ is a finite group,~$p$ is an odd prime, and the following condition holds:

\begin{flushleft}
\begin{tabular}{lcl}
(H2u) & & \pbox{13cm}{Whenever $L \groupiso \U_n(q)$ is a component of~$G$, with~$q$ odd and~$p$ dividing~$q+1$,
          then all~$p$-extensions of~$\U_m(q^{p^e})$ have~$\QD_p$ for all~$m \leq n$ and~$e \in \ZZ$.}
\end{tabular}
\end{flushleft}

Then~$G$ satisfies~(H-QC).
\end{theorem}

\noindent
We will prove this extension Theorem~\ref{theoremASExtension} as later Theorem~\ref{proofQCOdd};
our proof there mainly follows the path of the proof in~\cite{AS93}.

\bigskip

In fact, the possibility of such an extension was already suggested in~\cite{AS93},
just after the statement of the Main Theorem there,
using the language ``by excluding certain further groups as components of~$G$''.
We implement such an approach here---but now using the language of {\em elimination-results\/} in our recent work~\cite{PS}:
that is, results which state that a particular quasi\-simple~$L$ cannot appear as a component in a counterexample~$G$ to~(H-QC).
It will be convenient to describe the role of these eliminations within the following  general context:

\begin{remark}[Adapting the overall strategy of Aschbacher-Smith]
\label{rk:stratelimRob}
Here is our (over-simplified) viewpoint on the basic structure of the proof in~\cite{AS93}.
The first main step can be phrased as:

  \quad (elimQD) \quad Apply~\cite[Prop~1.7]{AS93} to eliminate any~$L$ with~$\QD_p$ as a component of~$G$.

\noindent
The remaining main step, carried out in~\cite[Thm~5.3]{AS93}, amounts to:

  \quad (Rob-nonQD) \quad Show any remaining~$L$ (i.e., without~$\QD_p$) has a ``Robinson subgroup''; 

\noindent
because the final-argument at~\cite[pp~490--492]{AS93} then shows that such Robinson subgroups
lead to~$\tilde{\chi} \bigl( \Ap(G)^g \bigr) \neq 0$ for some $g\in G$, and hence to~(H-QC).
Of course this strategy is idealized: since (only) certain unitary groups might%
\footnote{
  Conjecture~4.1 of~\cite{AS93}, that unitary groups should have~$\QD_p$, remains open.
          }
fail both steps,
and hence they must be excluded via the hypothesis~(H2u).

Now in proving our extension as Theorem~\ref{proofQCOdd},
it turns out that we will need some further elimination-results---which we will establish along the way.
For example: 
For~$p=5$ we can mostly use the original strategy above---but
we will need Theorem~\ref{alternativeParticularComponentes} to eliminate three particular components,
which would otherwise prevent the application of the intermediate result~\cite[Thm~2.3]{AS93}.  
And for~$p=3$, we will need Theorem~\ref{theoremLiep} to eliminate Ree-group components---since 
the usual strategy fails here, because Ree groups do not have Robinson subgroups in the argument for~\cite[Thm~5.3]{AS93}.

As a result, we are in effect implementing the following obvious adaptation of the original Aschbacher-Smith strategy above:

  \quad (elim) \quad Apply {\em various\/} elimination-results, to further reduce the possible components~$L$.
   
  \quad (Rob-nonelim) \quad Show Robinson subgroups occur for any~$L$ in the smaller non-(elim) list. 

\bigskip

\noindent
Furthermore as we will indicate later,
this variant-strategy would seem natural for any future work on extending the approach of~\cite{AS93} to the final prime~$p=2$;
and so we will adopt this strategic viewpoint in the relevant Section~\ref{sectionEvenQC}.
We mention that because of difficulties in extending, to~$p = 2$, the $\QD$-List of~\cite[Thm~3.1]{AS93}
(which we reproduce as later Theorem~\ref{theoremQDList}),
it seems likely that significant further elimination-results will also be needed in treating~$p = 2$.

The emphasis in the present  Remark on overall strategy and elimination-results will continue to be in the background of our later discussions---of many further variants on results and techniques contained within \cite{AS93}.
\donerk
\end{remark}

\subsection*{Variant hypotheses such as~(H1), related to induction-replacement}

Our methods for Theorem~\ref{theoremASExtension} above also lead to another closely related principal result,
namely Theorem~\ref{theoremASExtensionAlternative} below:
this further variant continues the above theme of adding the cases~$p=3,5$;
but it also involves simplifications both in the hypotheses,
and in certain corresponding aspects of the inductive proof (including a ``milder" use of the~CFSG).

In order to cast some light on these simplifications, we will now continue our Introduction
by developing some of the rationale for the variant-hypotheses---which will lead up 
to our statement of Theorem~\ref{theoremASExtensionAlternative}.

\bigskip

In proving~\cite[Main Theorem]{AS93}, Aschbacher and Smith argue by induction.
More specifically, they assume that~$G$ is a counterexample to (H-QC), of minimal order subject to satisfying~(H2u).
Mimicking now the language of~``(H1)'' used in the recent paper~\cite{KP20}, we give this condition the name~(H1u); that is:

\vspace{0.2cm}

\begin{flushleft}
\begin{tabular}{lcl}
(H1u) & & \pbox{13cm}{If~$H$ satisfies~(H2u), and further~$|H|<|G|$, then~$H$ satisfies~(H-QC).}
\end{tabular}
\end{flushleft}

\vspace{0.2cm}

\noindent
Indeed we even adopt the viewpoint that the applications of induction in the proof in~\cite{AS93}
can instead be {\em replaced\/}%
\footnote{
  This roughly avoids the mechanics of re-verifying the various hypotheses, at each application of induction.
          }
by applications of~(H1u)---regarded as a further formal hypothesis in Theorem~\ref{theoremASExtension}.  
Thus in this viewpoint, we regard the original proof in~\cite{AS93} as roughly showing:

\centerline{
If~$p > 5$, and~$G$ satisfies~(H1u) and~(H2u), then~$G$ satisfies~(H-QC).
            }

\noindent
In fact, our variant Theorem~\ref{theoremASExtensionAlternative} takes this viewpoint even further:
Namely, in place of the hypotheses~(H1u,H2u) above, it uses the simplified versions~(H1,H2) indicated below;
and these will lead to the promised simplifications in the proof of Theorem~\ref{theoremASExtensionAlternative}. 

\bigskip

\noindent
We first lead up to the variant-hypothesis~(H1) used in~\cite{KP20}:

Note that~(H1u) above includes mention of~(H2u)---since it is designed
for a minimal-order counterexample to Theorem~\ref{theoremASExtension}.
However, if we instead want to study the context of a minimal-order counterexample just to~(H-QC) itself,
we could at least in principle consider the somewhat simpler induction-replacement hypothesis given by:

\vspace{0.2cm}

\begin{flushleft}
\begin{tabular}{lcl}
(H1-) & & \pbox{13cm}{Any~$H$, with~$|H|<|G|$, must satisfy~(H-QC).}
\end{tabular}
\end{flushleft}

\vspace{0.2cm}

\noindent
And then~(H2u) would be just one possible ``second condition~(H2)'', for our more general counterexample~$G$.  
But we can also further observe that, in~\cite{AS93},
induction is in fact applied to handle just subgroups, and quotients by a central~$p'$-group.
So (as in~\cite{KP20}) we will prefer to work instead with the further-simplified induction-replacement hypothesis:

\vspace{0.2cm}

\begin{flushleft}
\begin{tabular}{lcl}
(H1) & & \pbox{13cm}{Proper subgroups, and proper~$p'$-central quotients, of~$G$ satisfy~(H-QC).}
\end{tabular}
\end{flushleft}

\vspace{0.2cm}

\noindent
And in fact, the work of~\cite{KP20} shows that it can be advantageous in proofs,
to focus our arguments on classes of groups closed under subgroups and central quotients;
here are some further details from that paper:

\begin{remark}[Some context for results under~(H1)---e.g., as seen in some earlier results]
\label{rk:contextofH1results}
First, results proved under~(H1) are not really ``standalone'' results---in the sense that
we usually would not expect to verify~(H1), within some general argument toward~(H-QC);
instead, such results might be regarded as ``segments of logic'',
applicable when we already have~(H1)---typically because we are in the situation of a minimal-order counterexample to~(H-QC). 
So in particular, as we had essentially observed in discussion above:

\begin{flushleft}
\begin{tabular}{lcl}
(MOC) & & \pbox{12.5cm}{(H-QC)-results under~(H1) apply to a minimal-order counterexample to~(H-QC).}
\end{tabular}
\end{flushleft}

\noindent
For example, Corollary~1.2 of~\cite{PS} assumes a counterexample as in~(MOC),%
\footnote{
   By a standard argument (e.g.~\cite[0.11]{AS93}), it suffices to take a minimal (under inclusion) counterexample.
          }
which gives~(H1)---as needed, for quoting the more technical result Theorem~1.6 there.

But we also want to know when proposed results with hypotheses~``(H2x)'' other than~(H1) in fact automatically satisfy~(H1). 
Such results are important for the overall context of the growing literature
of results on the possible structure of a counterexample to~(H-QC).
So we mention:

\begin{flushleft}
\begin{tabular}{lcl}
(Inh) & & \pbox{13cm}{For an~(H-QC)-result,
          if hypotheses~(H2x) other than~(H1) are inherited by subgroups and~$p'$-central quotients,
	  then~(H1) is automatic for a minimal-order counterexample to that result---so that
	  in fact (H-QC) holds, without {\em assuming\/}~(H1).}
\end{tabular}
\end{flushleft}

\noindent
This observation is of course just a version of the effect of induction on~$|G|$ in such a result.
For example, Theorem~4 of~\cite{KP20} has~``(H2x)'' given by~$m_p(G) \leq 4$---which has~(Inh);
this is essentially why its proof can be given under~(H1), so that~(H1) need not be in its statement.

By contrast, Theorems~4.1, 5.1, and~6.1 there are proved under~(H1)---since 
the relevant~(H2x) in those cases does {\em not} satisfy~(Inh),
so that those results do not automatically hold without assuming~(H1).
(This remark, in the case of Theorem~5.1 of~\cite{KP20},
is why we need the adjusted-version given by Theorem~\ref{alternativeParticularComponentes} here,
to treat~$p = 5$---see the discussion before that result.)
\donerk
\end{remark}

Considerations such as the above provide some initial motivation
for stating our particular variant Theorem~\ref{theoremASExtensionAlternative} below under~(H1);
and we will indicate further benefits, as we go on. 

\bigskip

So now we introduce~(H2)---our specific choice of~``(H2x)'',
to use in place of~(H2u) in our statement of Theorem~\ref{theoremASExtensionAlternative} under~(H1).
Here it turns out that a further advantage of the form~(H1),
is that we will be able to avoid the restrictions on~$m \leq n$ and~$e \in \mathbb{Z}$
that appear in~(H2u) of Theorem~\ref{theoremASExtension} above;
namely we can just use:

\vspace{0.2cm}

\begin{flushleft}
\begin{tabular}{lcl}
(H2) & & \pbox{13cm}{If~$L \groupiso \U_n(q)$ is a component of~$G$ with~$q$ odd and~$p \mid q+1$,
         then the~$p$-extensions of~$\U_n(q)$ occurring in~$G$ satisfy the~$\QD_p$ property.}
\end{tabular}
\end{flushleft}

\vspace{0.2cm}

\noindent
This version allows us to focus just on the single component~$\U_n(q)$, and not other unitary extensions for different~$n,q$;
this makes proofs easier---but as noted above,
the application of this version is for situations such as a minimal-order counterexample just to~(H-QC),
where~(H1) does hold.

So we will be proving the following variant-form of Theorem~\ref{theoremASExtension}:

\begin{theorem}
\label{theoremASExtensionAlternative}
Let~$p$ be an odd prime, and let~$G$ be a finite group satisfying~(H1) and~(H2) above.
Then~$G$ satisfies~(H-QC).
\end{theorem}

\noindent
The proof of Theorem~\ref{theoremASExtensionAlternative} appears in the latter part of Section~\ref{sec:2varASforp=35}.

We had mentioned earlier that this proof makes a milder use of the~CFSG
than the proof of Theorem~\ref{theoremASExtension};
we now briefly preview some aspects of that simplification---and some others:

\begin{remark}[Reducing use of the~CFSG in Theorem~\ref{theoremASExtensionAlternative} under~(H1)]
\label{rk:compareCFSGinASExtandAlt}
The proof of our Theorem~\ref{theoremASExtensionAlternative} stated under~(H1)
employs the classification of the finite simple groups to a much lesser extent
than in~\cite{AS93} (and hence in Theorem~\ref{theoremASExtension}); 
this is thanks to our corresponding~(H1)-variants of Propositions~1.6 and~1.7 of~\cite{AS93},
which are given by~\cite[Theorem~4.1]{KP20} and Theorem~\ref{QDpReductionPS} here---as we summarize below:

The original form of those two Propositions in~\cite{AS93} state in effect that under~(H1u),
we can suppose that~$O_{p'}(G) = 1$,
and that any simple component of~$G$ where the indicated~$p$-extensions satisfy~$\QD_p$ leads to~(H-QC).
We contrast that situation with our~(H1)-variants:
In~\cite[Theorem 4.1]{KP20}, it was shown that if~$G$ satisfies~(H1), and~$O_{p'}(G) \neq 1$, then~(H-QC) holds for~$G$.
And in Theorem~\ref{QDpReductionPS} we show that, assuming~(H1),
if a component of~$G$ satisfies~$\QD_p$ for all its~$p$-extensions, then~$G$ satisfies~(H-QC).
Thus our variants proved under~(H1) give us essentially the same reductions as in~\cite{AS93} with Propositions~1.6 and~1.7 there.

Furthermore, the proof of~\cite[Theorem 4.1]{KP20}, which leads to the reduction~$O_{p'}(G) = 1$,
only uses the~CFSG to invoke the~$p$-solvable case of the conjecture;
and the proof of Theorem~\ref{QDpReductionPS} does not use the~CFSG at all.
In contrast, the proofs of Propositions~1.6 and~1.7 of~\cite{AS93} invoke Theorem~2.3 of~\cite{AS93},
which is stated for odd~$p$, and relies deeply on the~CFSG.
\donerk
\end{remark}

\begin{remark}[Equivalent posets and simplification of boundary-calculations]
\label{rk:equivvisualbdry}
In addition, the proofs of both our above extensions of~\cite{AS93} to all odd~$p$
will in fact involve focusing on some more intrinsic combinatorial properties of the~$\Ap$-posets---as
indicated in the title of this paper;
see especially Sections~\ref{sec:overviewequivApG} and~\ref{sec:prejoinandreplposets}.
For example:
Results such as Proposition~\ref{propositionContractibleCentralizers} allow us to ``remove unnecessary points''---so
we can work toward~(H-QC) using smaller and more convenient posets equivalent to~$\Ap(G)$.
Other equivalences such as Theorem~\ref{theoremChangeABSPosets}
(building on the viewpoint of Theorem~\ref{theoremInductiveHomotopyType})
provide a convenient visual-format for poset-chains based on the Cartesian-product and resulting in simplified boundary-calculations
in technical results on homology propagation.
(This last term means, roughly, the process of showing for some~$H < G$
that nonzero reduced homology for~$\Ap(H)$ ``propagates'' to nonzero homology also for~$\Ap(G)$---to give~(H-QC).
We will begin a more detailed discussion of the fundamental technique of propagation
at later Remark~\ref{rk:contextequivandpropag}.)
%In particular, this will also help us to avoid invoking some~CFSG-dependent results (under either~(H1u) or~(H1)).
\donerk
\end{remark}

\subsection*{Some sample results---toward the case of \texorpdfstring{$p=2$}{p=2}}

To continue an~(H1)-theme from Remark~\ref{rk:compareCFSGinASExtandAlt} just above,
we can in fact use \textit{any} prime~$p$, including~$p = 2$, 
in the reduction to~$O_{p'}(G) = 1$ in~\cite[Theorem~4.1]{KP20},
and in the elimination of~$\QD_p$-components in Theorem~\ref{QDpReductionPS}.
Also~\cite[Theorem~4.1]{KP20} does not require the~CFSG.
These features provide strong motivation for working under~(H1), in future contributions toward~(H-QC) for~$p = 2$.

In particular, we could at least in principle follow
essentially the original Aschbacher-Smith strategy of~(elimQD,Rob-nonQD) indicated in earlier Remark~\ref{rk:stratelimRob}. 
However, as we had also indicated in that Remark, 
there seem to be significant difficulties in extending, to~$p = 2$, the $\QD$-List of~\cite[Thm~3.1]{AS93}---namely 
the components which would ideally be eliminated by Theorem~\ref{QDpReductionPS}.
For this reason, it seems likely that many such problem-components
will need to instead be handled via new elimination-results---so 
that we would be using the modified strategy of (elim)/(Rob-nonelim) in Remark~\ref{rk:stratelimRob}. 

So we now also indicate, from later in the paper, some partial elimination-results for~$p = 2$:

\begin{theorem}
\label{mainTheoremEvenCase}
Assume~$p=2$.
Let~$G$ be a group satisfying~(H1), along with one of the following:
\begin{enumerate}
\item $G$ contains a component~$L$, such that~$L/Z(L)$ is isomorphic to one of the following groups:%
\begin{align*}
  \Alt_5, \Alt_6, \Sz(&2^{2n+1}), {}^2F_4(2^{2n+1})', {}^2G_2(3^{2n+1}),
  \Mathieu_{11}, \Mathieu_{12},\Mathieu_{22}, \Mathieu_{23}, \Mathieu_{24},\\
 &\Janko_1, \Janko_4, \Conway{1},\Conway{2},\Conway{3}, \Fi{23}, \HS, \Ly, \Ru, \Th, \Baby,\Monster.
\end{align*}

\item $G$ contains a simple component~$L$, such that every~$p$-extension~$LB \leq G$ of~$L$ satisfies~$\QD_2$.

\item $G$ contains a component~$L$, such that~$L/Z(L)$ is a finite group of Lie type in characteristic $2$ other than the types:~$\PSL_n(2^m) (n \geq 3)\ ; \, D_n(2^m)  (n \geq 4)\ ; \,  E_6(2^m)$;%

 \item If~$G$ contains a component~$L$, such that~$L/Z(L)$ is a finite group of Lie type,
       then either~$L/Z(L)$ is described in~(1--3), or else~$L/Z(L)$ is defined in characteristic~$3$.
\end{enumerate}
Then~$G$ satisfies~(H-QC).
\end{theorem}

We remark that conclusion~(4) is not really a ``pure'' elimination-result---but instead
depends on some of our later results about Robinson-subgroups.%
%\footnote{
%   fakenote:
%   added---if I understand? See the comments added to the proof, in sec10.
%          }
The proof of this Theorem can be found at the end of Section~\ref{sectionEvenQC}.

\bigskip

For~$H \leq G$, let~$C_G(H)$ denote the centralizer of~$H$ in~$G$.
Write:

\centerline{
  $\Outposet_G(H) := \{ E \in \Ap(G) \tq E \cap (H C_G(H)) = 1 \}$.
            }

\noindent
Using the above in conjunction with the results of~\cite{KP20,PS}, we get:

\begin{corollary}
Let~$p = 2$.
If~$G$ satisfies~(H1) but fails~(H-QC), then the following hold:
\begin{enumerate}
\item $O_2(G) = 1 = O_{2'}(G)$;
\item $\A_2(G)$ is simply connected;
\item $G$ has~$2$-rank at least~$5$;
\item for every component~$L$ of~$G$, $\Outposet_G(L) \neq \emptyset$;
\item $G$ contains a component~$L$ of Lie type in characteristic~$\neq 3$, and some~$p$-extension~$LB \leq G$ of~$L$ fails~$\QD_p$;
\item $G$ does not contain components of the following types:
\begin{center}
of Lie type and Lie rank $1$ in characteristic $2$, $ \Alt_6, \Sz(2^{2n+1}), {}^2F_4(2^{2n+1})',$\\
\vspace{0.05cm}
${}^2G_2(3^{2n+1}), \Mathieu_{11}, \Mathieu_{12},\Mathieu_{22}, \Mathieu_{23}, \Mathieu_{24}, 
\Janko_1, \Janko_4, \Conway{1},\Conway{2},\Conway{3}, \Fi{23}, \HS, \Ly, \Ru, \Th, \Baby,\Monster.$
\end{center}
\end{enumerate}
\end{corollary}

\bigskip

\noindent
In view of the length of the paper, we provide the reader with a brief explanatory road-map:

\begin{remark}[Overview: Organization of the paper]
\label{rk:outlineorgpaper}
The preliminary Section~\ref{sec:prelims} establishes our overall context, collecting standard definitions and various quoted results.

\smallskip 

Then Sections~\ref{sec:overviewequivApG}--\ref{sectionLefschetz} lead up to the main-proofs in Section~\ref{sec:2varASforp=35}
of our two variants extending~\cite{AS93} to~$p=3,5$ (namely Theorems~\ref{theoremASExtension} and~\ref{theoremASExtensionAlternative}).
The logical sequence is perhaps best explained in {\em reverse\/} order---regarded roughly as a sequence of successive reductions:

The main-proofs above follow the basic strategy indicated in Remark~\ref{rk:stratelimRob}:
The~(Rob-nonelim) step requires in Section~\ref{sectionLefschetz} some extensions of earlier results on Robinson subgroups.
The (elim)~step requires some new elimination results:  
in Section~\ref{sec:elimQCunderH1},
a variant Theorem~\ref{QDpReductionPS} of the~(elimQD)-step eliminating $\QD_p$-components in~\cite{AS93};
and in Section~\ref{sec:elimLierk1samecharp},
the elimination via Theorem~\ref{theoremLiep} of some Lie-type groups in the same characteristic~$p$.

The required-results above in turn depend on some more technical results, which appear in still-earlier sections:
notably the generalized homology-propagation result Theorem~\ref{theoremNewHomologyPropagation}
% replacing old~\ref{theoremHomologyPropagation} by new~\ref{theoremNewHomologyPropagation}
established in Section~\ref{sec:prejoinposetsinHomologyPropagation}.
This result is in fact stated in the viewpoint of the pre-join (rather than the usual poset-join),
which is developed in Section~\ref{sec:prejoinandreplposets}:
the Cartesian-product structure of the pre-join can be much more convenient for our various calculations than the poset-join.
The propagation uses a replacement-poset equivalent to~$\Ap(G)$; 
that replacement is a suitable adjustment, for the pre-join,
of a more standard replacement-construction indicated
in Definition~\ref{definitionXBHposet} of Section~\ref{sec:overviewequivApG}.

\smallskip 

Finally Appendix section~\ref{sec:appendixAfiber} recalls a generalized version
of the Quillen fiber-Theorem~\ref{variantQuillenFiber};
and further Appendix sections provide various technical details---which we moved to there,
in order to not slow down the more fundamental parts of the logical flow in the main text.
\donerk
\end{remark}

This article is intended to be primarily self-contained:
In particular, to avoid the need for frequent external look-ups, statements will be given for many quoted results. 

Results concerning the Quillen Conjecture will often be expressed with several alternative hypotheses,
with the aim of finding different approaches to the proof of~(H-QC);
notably approaches that avoid much of the use of the~CFSG.

\bigskip

\textbf{Acknowledgements.}
Part of this work was carried out during a research stay of the first author at The Mathematisches Forschungsinstitut Oberwolfach.
He is very grateful for the hospitality and the support of the staff of~MFO,
and also to the other Leibniz Fellows for numerous helpful conversations.

\bigskip
\bigskip

\section{Preliminaries}
\label{sec:prelims}

In this section, we recall various elementary facts that we will need on finite groups, posets, and simplicial complexes.
We also establish the notation that we will use throughout the article,
and review some of the literature on Quillen's Conjecture.

\subsection*{Finite groups}

Our main reference for the properties on finite groups is Aschbacher's book \cite{AscFGT}.
We will follow the conventions and notation on simple groups of~\cite{AS93, GLS98, GL83}.
We use the definitions of~\cite{GLS98} for the different types of automorphisms of simple groups, but also we will sometimes refer to~\cite{GL83} for details. 
The reader should take into account that some notation and definitions on simple groups and their automorphism types
occasionally differ between~\cite{GLS98} and~\cite{GL83}; we will try always to indicate such differences.

We will always work with finite groups.
Denote by $\Cyclic_n$, $\Dihedral_n$, $\Sym_n$ and $\Alt_n$ the cyclic group of order $n$, the dihedral group of order $n$,
the symmetric group on $n$ letters and the alternating groups on $n$ letters respectively.
By a simple group we always mean a {\em non-abelian\/} simple group.

Let $G$ be a finite group.
For subgroups $K,H\leq G$, let $N_K(H)$ be the normalizer of $H$ in $K$, and $C_K(H)$ the centralizer of $H$ in $K$.
Denote by $N_G(H_1,\ \ldots\ ,H_n)$ the intersection of normalizers~$N_G(H_1) \cap \ldots \cap N_G(H_n)$ for subgroups $H_i\leq G$.
Write $[H,K]$ for the subgroup generated by the commutators between elements of $H$ and $K$.
For a fixed prime $p$, $\Omega_1(G)$ is the subgroup of~$G$ generated by the elements of order $p$.
Recall that the $p$-rank of $G$, denoted by $m_p(G)$, is the maximal dimension of an elementary abelian $p$-subgroup of~$G$.

Denote by~$Z(G)$, $F(G)$, $O_p(G)$, $O_{p'}(G)$
the center, the Fitting subgroup, the largest normal~$p$-subgroup, and the largest normal~$p'$-subgroup of~$G$, respectively.
Recall that the Fitting subgroup~$F(G)$ is the direct product of all the subgroups~$O_p(G)$, for~$p$ prime dividing the order of~$G$.
For solvable groups~$G$, we have that~$F(G)$ is self-centralizing: $C_G \bigl( F(G) \bigr) \leq F(G)$.
However, this property does not hold for arbitrary groups~$G$,
and so in general we instead work with the \textit{generalized Fitting subgroup}~$F^*(G)$ of~$G$;
it is fundamental for the structure of finite groups:

\begin{remark}[Components and the generalized Fitting subgroup]
\label{rk:compsandF*G}
Recall that~$F^*(G)$ is the central product of the subgroups~$F(G)$ and~$E(G)$;
here~$E(G)$ is the layer of~$G$, that is, the central product of the {\em components\/} of~$G$---namely
the quasisimple subnormal subgroups of~$G$.
The generalized Fitting subgroup is self-centralizing:

\centerline{
  $C_G \bigl( F^*(G) \bigr) \leq F^*(G)$;
            }

\noindent
and when~$G$ is solvable, $F^*(G) = F(G)$.
Also~$Z \bigl( F^*(G) \bigr) = Z \bigl( F(G) \bigr)$ and $Z \bigl( E(G) \bigr) \leq Z \bigl( F(G) \bigr)$.
\donerk
\end{remark}

We sometimes use the following (non-standard) notation for an orbit of components of~$G$:

\begin{definition}
\label{defOrbitComponentes}
Let $L$ be a component of a finite group~$G$.
We denote by~$\hat{L}$ the (central) product of the $G$-conjugates of $L$, and by $\hat{N}_G(L)$ the kernel of this action.

That is, if $L_1,\ldots,L_r$ is the $G$-orbit of $L$, then:

  \hfill $\hat{L}=L_1\ldots L_r$ and $\hat{N}_G(L) = N_G(L_1,\ldots,L_r) = \bigcap_i N_G(L_i) . $ \donerk
\end{definition}

Note that we always have~$F(G) \leq O_p(G)O_{p'}(G)$.
For the Quillen Conjecture, 
we will often work under the assumption that~$O_p(G) = 1 = O_{p'}(G)$, so that~$F(G) = 1$, and hence~$Z \bigl( E(G) \bigr) = 1$.
Since~$Z \bigl( E(G) \bigr)$ is the product of the centres of the components of~$G$,
we see that the components of~$G$ are simple groups in this case.
We include this in the following frequently-used lemma:

\begin{lemma}
\label{lemmaOpandp}
Let $G$ be a finite group and $p$ a prime number dividing its order.

\begin{enumerate}
\item If~$F(G) = 1$, then~$F^*(G) = E(G) = L_1 \ldots L_t$ is the direct product of the components $L_i$ of~$G$, which are all simple.
      Since~$C_G \bigl( F^*(G) \bigr) = Z \bigl( F(G) \bigr) = 1$, we have a natural inclusion:
  \[ F^*(G) \leq G \leq \Aut \bigl( F^*(G) \bigr) . \]
This holds when~$O_p(G) = 1 = O_{p'}(G)$; and then~$p$ divides the order of the components.

\item $E(\ C_G \bigl( E(G) \bigr)\ ) = 1$ and hence $F^*(\ C_G \bigl( E(G) \bigr)\ ) = F(G)$.

\item Let $L_{1},\ldots,L_{s}$ be distinct components of $G$.
Then the components of $C_G(L_{1}\ldots L_{s})$ are the components of $G$ other than the $L_{i}$ for $1\leq i\leq s$.

\item Under the hypothesis of (3), $F \bigl( C_G(L_1 \ldots L_s) \bigr) = F(G)$. 
In particular, if $O_p(G) = 1$, then~$O_p \bigl( C_G(L_{1} \ldots L_{s}) \bigr) = 1$.
\end{enumerate}
\end{lemma}

\begin{proof}
Part (1) is immediate from the previous discussion.

We prove part (2).
Suppose that~$L$ is a component of~$C_G \bigl( E(G) \bigr)$.
Since~$E(G)$ is normal in~$G$, so is~$C_G \bigl( E(G) \bigr)$, and therefore~$L$ is a component of~$G$ (see~\cite[(31.3)]{AscFGT}).
So~$L \leq E(G)$---but also lies in~$C_G \bigl( E(G) \bigr)$, contrary to quasisimplicity of~$L$.
This proves that~$E(\ C_G \bigl( E(G) \bigr)\ ) = 1$, and so~$F^*(\ C_G \bigl( E(G) \bigr)\ ) = F(\ C_G \bigl( E(G) \bigr)\ )$.
But the latter is a normal nilpotent subgroup of $G$, which contains $F(G)$ since $[F(G),E(G)] = 1$.
Hence $F(\ C_G \bigl( E(G) \bigr)\ ) = F(G)$.

We prove part (3).
Let $\{ M_1, \ldots , M_r \}$ be the set of components of $G$ distinct from the $L_i$, so that $E(G) = L_1\ldots L_s M_1\ldots M_r$.
Write $K = C_G(L_1\ldots L_s)$.
Then $M_j\leq K$ for all $1\leq j\leq r$ and then the $M_j$ are components of $K$.
That is, $M_1\ldots M_r \leq E(K)$.

For the reverse inclusion, let $M$ denote a component of $K$, hence of $E(K)$.
Assume, by way of contradiction, that $M\neq M_j$ for all $j$.
By~\cite[(31.5)]{AscFGT} applied in~$K$, we see that~$M$ centralizes each~$M_j$.
Since $M \leq K = C_G(L_1\ldots L_s)$, we get $M \leq C_G \bigl( E(G) \bigr) \leq K$.
This shows that $M$ is also a component of $C_G \bigl( E(G) \bigr)$.
But this contradicts item (2).
In consequence, $M = M_j$ for some $j$.
This concludes the proof of part (3).

We now prove part~(4).
Let~$M = F \bigl( C_G(L_1 \ldots L_s) \bigr)$.
Part~(3) leads to~$[M,E(G)] = 1$.
Hence~$M \leq C_G \bigl( E(G) \bigr) \leq C_G(L_1 \ldots L_s) \leq G$.
Then~$M$ is a normal nilpotent subgroup of~$C_G \bigl( E(G) \bigr)$ using~(3).
Therefore~$M \leq F(\ C_G \bigl( E(G) \bigr)\ ) = F(G)$.
Since~$F(G) \leq C_G \bigl( E(G) \bigr) \leq C_G(L_1\ldots L_s)$,
we see that~$F(G)$ is a normal nilpotent subgroup of $C_G(L_1 \ldots L_s)$, and so $F(G) \leq M$.
So we have~$M = F(G)$;
hence~$O_p \bigl( C_G(L_1 \ldots L_s) \bigr) = O_p(\ F \bigl( C_G(L_1 \ldots L_s) \bigr)\ ) = O_p \bigl( F(G) \bigr) = O_p(G)$.
\end{proof}

We will also often use (sometimes possibly implicitly) the following elementary lemma:

\begin{lemma}
\label{lemmaNormalizeComponent}
If~$K,L \leq G$, where~$L$ is a component of~$G$ and~$K$ normalizes a non-central subset of $L$, then~$K$ normalizes~$L$.
\end{lemma}

\begin{proof}
Let $S \subseteq L$ be a non-central subset normalized by~$K$.
If~$k \in K$, then~$S \subseteq L \cap L^k$.
Since~$L^g \neq L$ implies~$L\cap L^g\leq Z(L)$ for any~$g \in G$, and~$S$ is not central in $L$, we conclude that~$L^k = L$.
Hence~$K \leq N_G(L)$.
\end{proof}

Lemma~\ref{lemmaTrivialOpPropagationCentralizer} below is for more specialized situations:
it provides a downward-inheritance property for trivial~$p$-core, in centralizers of extensions of components by~$p$-groups---such
as the~$p$-extensions of Definition~\ref{defn:QD}. 
The~$O_p(-)=1$ condition for the centralizers in the lemma will correspond
to the ``nonconical''-centralizer situation in the homology-propagation literature;
see for example the hypothesis~``(someNC)'' in our discussion in later Remark~\ref{remarkClosingSection}.
We will apply the nonconical conclusion of the Lemma to provide the particular version of~(someNC)
which is appropriate for the propagation used in proving later Theorem~\ref{QDpReductionPS};
and also in the proof of Proposition~\ref{propHomotEquivTildePosets}.

\begin{lemma}
\label{lemmaTrivialOpPropagationCentralizer}
Suppose that~$L \leq G$ is a subgroup, and let~$B,E \leq N_G(L)$ be abelian~$p$-subgroups.
Suppose in addition that~$LB \leq LE$ and that~$O_p \bigl( C_G(LE) \bigr) = 1$.
Then~$O_p \bigl( C_G(LB) \bigr) = 1$.

If further~$L$ is a component of~$G$, then also~$O_p(LB) = 1$.
\end{lemma}

\begin{proof}
Let~$N := O_p \bigl( C_G(LB) \bigr)$.
We show that~$N = 1$.

Note that~$LB$ is normal in~$LE$ since~$LB/L \leq LE/L \groupiso E/E\cap L$ is abelian.
Therefore, $E$ acts on~$C_G(LB)$, and hence on the characteristic subgroup $N$.
Since $N$ and $E$ are $p$-groups, from~$N > 1$ we would get~$C_N(E) > 1$.
On the other hand, $C_N(E)$ is a normal subgroup of $C_G(LE)$:
since for~$x \in C_G(LE)$, we see~$C_N(E)^x = C_{N^x}(E^x) = C_N(E)$ using~$C_G(LE) \leq C_G(LB)$.
So we get that~$C_N(E) \leq O_p \bigl( C_G(LE) \bigr) = 1$ using the hypothesis, which forces~$N = 1$, as desired.

Finally assume~$L$ is a component of~$G$; note then that~$O_p(LB)$ is centralized by~$L$ and normalized by~$B$;
so~$C_{O_p(LB)}(B) \leq O_p ( Z(LB) ) \leq O_p \bigl( C_G(LB) \bigr) = N = 1$, which forces~$O_p(LB) = 1$.
\end{proof}

Now we establish some notation on automorphism groups.
Denote by~$\Aut(G)$ the group of automorphisms of~$G$,
by~$\Inn(G)$ the group of inner automorphisms of~$G$ (which is~$G/Z(G)$),
and by $\Out(G) = \Aut(G) / \Inn(G)$ the group of outer automorphisms of~$G$.
If~$x \in \Aut(G) - \Inn(G)$, we say that~$x$ {\em induces an outer automorphism\/} on~$G$.

If $H\leq G$, then $\Aut_G(H) = N_G(H)/C_G(H)$ (resp. $\Out_G(H) = N_G(H)/ (H C_G(H))$)
is the group of automorphisms (resp. outer automorphisms) of $H$ induced by $G$.
The subgroup $HC_G(H)$ can be regarded as the subgroup of $G$ whose elements induce inner automorphisms on $H$.
We say that a subgroup $K\leq G$ induces outer automorphisms on $H$ if $K$ normalizes $H$ and $K$ contains no inner automorphism of $H$.
That is, $K$ induces outer automorphisms on $H$ iff~$K \cap (H C_G(H)) = 1$.

\medskip

We recall now some basic facts on outer automorphisms of {\em simple\/} groups:

\begin{remark}[Summary: Outer automorphisms of the finite simple groups]
\label{rk:outerautsofsimple}
We will usually abbreviate the Classification of Finite Simple Groups by~CFSG.
The~CFSG states in overview that a simple group is one of:

\centerline{
   (CFSG): an alternating group; a simple group of Lie type; or one of the~26 sporadic groups.
            }

\noindent
One celebrated consequence of the~CFSG is the Schreier conjecture;
which states that for a simple group~$L$, $\Out(L)$ is solvable.
But we get fuller details from the study of the cases for $L$
(using the~CFSG to know that our list of the simple~$L$ is complete): 

Namely we have detailed description of the structure of~$\Out(L)$ for the simple groups~$L$.
These specific descriptions are important throughout the literature;
but especially for the study of~$p$-extensions:
for example in establishing the Aschbacher-Smith~$\QD_p$-List (which we quote below as Theorem~\ref{theoremQDList});
and in many other results on the Quillen Conjecture---including our results in Section~\ref{sec:elimLierk1samecharp}.
(Indeed Theorem~\ref{PStheorem} shows that we should {\em expect\/}~$p$-extensions, in pursuing~(H-QC).)

\bigskip

\noindent
Here are some of those fundamental details:

\smallskip

$\bullet$
For~$L$ an alternating group, $\Out(L) = \Cyclic_2$---except
for~$L = \Alt_6$, where~$\Out(\Alt_6) = \Cyclic_2\times \Cyclic_2$.
(See e.g.~[Thm~5.2.1]\cite{GLS98}.)

\smallskip

$\bullet$
For~$L$ a sporadic group, $\Out(L) = 1$ or~$\Cyclic_2$.
(See e.g.~\cite[Sec~5.3]{GLS98} for the specific cases.)

\medskip

$\bullet$
Now consider~$L$ of Lie type---defined over a field of characteristic we will denote by~$r$:

\noindent
Below we give a fairly informal summary of the structure of~$\Out(L)$,%
\footnote{
   In the case of the Lie-type groups~$L^+ \groupiso \Symplectic_4(2)$, ${^2}F_4(2)$, $G_2(2)$, ${^2}G_2(3)$ for $p = 2,2,2,3$, 
   this description applies to~$\Out(L^+)$; whereas the simple group~$L$ is the commutator subgroup of index~$p$---which
   has further automorphisms, that we must treat separately.
   For example, this distinction arises explicitly in our later Definition \ref{defn:sLie-p}(2).
   See also the discussion $Lie_{exc}$ in Definition 2.2.8 of \cite{GLS98}.
          }
%\footnote{
%  "fake footnote":
%  (Decide later.)
%  We need to make sure that we DO take care of these groups---$Lie_{exc}$ Defn~2.2.8 in \cite{GLS98}!
%          }
based on the beginning of Section~7 (Part~I) of~\cite{GL83}.
(For our later results, we will require a more detailed description---based on the treatment in~\cite[Sec~5.2]{GLS98};
see our ``first aspect'' discussion after the statement of Theorem~\ref{theoremLiep}, 
and the underlying structures in Appendix Section~\ref{sec:fldgraphautcharp}.)
In the standard overview:

\centerline{
   Every element of~$\Aut(L)$ can be written as a product~$idfg$,
            }

\noindent
where~$i$ is an inner automorphism, $d$ a diagonal automorphism, $f$ a field automorphism, and~$g$ a graph automorphism.
For present purposes, it will suffice to describe these automorphism types fairly informally:

\smallskip

(d) Diagonal automorphisms correspond to suitable diagonal matrices, 
when~$L$ is given in adjoint form---namely as a group of matrices group over the field~${\mathbb F}_{r^a}$ for suitable~$a$.
In particular, diagonal automorphisms have order dividing~$r^a - 1$ and hence coprime to $r$.

\smallskip

(f) Field automorphisms of that adjoint matrix group 
are induced by elements of the Galois group of~${\mathbb F}_{r^a}$ over ${\mathbb F}_r$. 
In particular, the subgroup of field automorphisms is cyclic.%
\footnote{
  Of order $a$---or at least dividing $a$, depending on naming-conventions
  we describe in later Definition~\ref{defn:fldgraphconvtypePhi}.
          }

\smallskip

(g) Graph automorphisms arise%
\footnote{
  Our wording here is deliberately vague, for the moment;
  we will be more precise about the meaning of ``graph automorphism''
  (and various other points) in the already-mentioned later discussions beginning after Theorem~\ref{theoremLiep}. 
          }
from elements of the symmetry group~$\Delta$ of the underlying Dynkin diagram  for the Lie type of~$L$.
The cases where~$\Delta > 1$, listed by diagram type, are:

\centerline{
  (i) $\Delta = \Cyclic_2$ for~$A_n$ ($n \geq 2$), $B_2$, $C_2$, $D_n$ ($n > 4$), $E_6$, $F_4$, $G_2$;
  and~$\Delta = \Sym_3$ for~$D_4$.
            }

\noindent
In particular, the only possible primes~$s$ dividing the order of~$\Delta$ are~$s = 2,3$. 
Furthermore (see~1.15.4(b) of~\cite{GLS98}) for a diagram-symmetry
to actually lead to a graph automorphism (in the original sense of Steinberg) of an untwisted group~$L$,
at the level of an overlying algebraic group~$\bar L$, 
some of the above Lie-types give a further restriction on the characteristic~$r$:

\centerline{
  (ii) A graph automorphism in~$B_2 \cong C_2$ or~$F_4$ requires~$r=2$; and in~$G_2$ requires~$r=3$;
            }

\noindent
here~$B_2 \cong C_2$ recalls a standard isomorphism of the finite groups of those types.

Finally, for the single-bond diagrams in~(i) above~(Lie-types~$A_n$, $D_n$, $E_6$),
the graph automorphisms from~$\Delta$ always lead to corresponding {\em twisted\/} simple groups;
whereas for the Lie-types in~(ii) above, there is a further restriction:

  (iii) The cases in~(ii) give a twisted group only over a field with order an {\em odd\/} power of~$r$.
\donerk
\end{remark}

\subsection*{Posets and simplicial complexes}

We establish the main notation and definitions on posets and simplicial complexes.

We will always consider finite posets and simplicial complexes.
We typically denote the order-relation of some unspecified poset~$X$ by~$<$;
in posets of subgroups of a finite group~$G$, the relation will usually coincide with group-inclusion~$<$---but
we will also consider some other relations, for which we will normally specify some different notation.

If~$X$ is a finite poset, denote by~$\K(X)$ its {\em order complex\/}:
recall that the simplices of the order complex~$\K(X)$ are the non-empty~$<$-chains of~$X$.
We will study the homotopy properties of the poset~$X$---meaning the homotopy properties of the complex~$\K(X)$.% 
\footnote{
  Thus in topological situations, we in effect identify~$X$ with~$\K(X)$: 
  this should be implicitly clear from the context---but usually we try to make the complex~$\K(X)$ explicit.
          }

If~$f : X \to Y$ is an order-preserving map between finite posets,
then~$f$ induces a simplicial map~$f : \K(X) \to \K(Y)$ between the order complexes.

The basic Homotopy  Property:
Assume~$f,g : X \to Y$ are two order-preserving maps between finite posets.
We say that~$f$ and~$g$ are~{\em $<$-comparable\/}, if either~$f \leq g$ (which means~$f(x) \leq g(x)$ for all~$x \in X$),
or analogously~$f \geq g$.
In that case, we have a homotopy of the maps~$f,g : \K(X) \to \K(Y)$ induced on the order complexes.

We often use the following special case where~$X = Y$---so that~$f$ is a poset {\em endomorphism\/}:

\begin{lemma}[Homotopy from endomorphisms]
\label{lm:increndo}
Let~$X$ be a finite poset.

\noindent
(1) Assume~$f_0, f_1 , \dots, f_t$ are poset endomorphisms of~$X$, with~$f_0 = \Id_X$, such that:

  (i) Successive pairs are~$<$-comparable: That is, for each~$i$ we have~$f_i \leq f_{i+1}$ or $f_i \geq f_{i+1}$;

  (ii) The final map~$f_t$ is constant: That is, there is~$x_t \in X$, such that~$f_t(x) = x_t$ for all~$x \in X$. 

\noindent
Then $\Id_X = f_0$ is homotopic to the the constant map~$f_t$; so~$X$ is contractible to~$\{ x_t \}$. 

\smallskip

\noindent
(2) Assume that~$f$ is a poset endomorphism of~$X$, which is {\em monotone\/}: that is,~$f \leq \Id_X$ or~$f \geq \Id_X$. 

\noindent
Then~$f$ induces a homotopy equivalence~$X \simeq f(X)$ with its image; 

     indeed we get~$X \simeq Y$ for any intermediate~$Y$, namely where~$f(X) \subseteq Y \subseteq X$. 

\noindent
If further~$f$ is the identity on~$f(X)$, then~$f$ induces a poset-strong deformation retraction. 
\end{lemma}

\begin{proof}
Standard; e.g.~\cite[1.5]{Qui78}.
\end{proof}

\noindent
Here are a few comments about our usage of these results:

\medskip

In applications of~(1), the overall setup is often summarized via a ``zigzag'' or ``fence'';
where for~$x \in X$, we record just the following information, roughly in the form:

  (i$^{\prime}$) \quad $x \bigl( = f_0(x) \bigr)\ \leq\ f_1(x)\ \geq\ f_2(x)\  \leq\ \dots\ \geq\ \bigl( f_t(x) = \bigr) x_t$ . 

\noindent
In particular, this includes the definition of the maps~$f_i$,
and their~$<$-comparability relations for~(1)(i)---in this case~$f_0 \leq f_1 \geq f_2 \leq \dots \geq f_t$;
but we emphasize that we must still prove that the set-maps~$f_i$ thus defined really do afford {\em poset\/} endomorphisms.

Moreover, we frequently want to use~(1) to show contractibility of some subposet~$Z \subseteq X$.
And then, in order to have the endomorphism-condition for the restrictions~$f_i \vert_Z$,
we must {\em further\/} show that~$f_i(Z) \subseteq Z$. 
In fact, showing this can often be the trickiest part of the proof.

We will illustrate these zigzag-features, within our standard~$p$-subgroup-poset context, 
in upcoming Example~\ref{ex:zigzagconical}.

\smallskip

In~(2), our strong deformation retraction has a further ``very-strong'' property:
namely the map~$f$ is even poset-homotopic to~$\Id_Z$;
that is, here we have the length-$1$ case of a sequence of comparable maps as in~(1)---where
we replace the constant map~$f_t$ with the retraction~$f$. 
We will normally reserve the term ``deformation retraction'' for this ``poset-strong'' situation,
where~$f$ has such a poset-homotopy sequence.

\bigskip

Assume~$Y \subseteq X$ and~$x \in X$.
Set~$Y_{\geq x} := \{ y\in Y\tq y\geq x\}$.
Define analogously~$Y_{>x}, Y_{\leq x}, Y_{<x}$.
%Denote by~$\Max(X)$ (resp.~$\Min(X)$) the set of maximal (resp. minimal) elements of~$X$.%
%\footnote{
%  "fake footnote":
%  (Decide later)
%  Maybe we want to NOT use this capital-M Max notation in the paper? 
%          }
%For a single element~$x$, We set~$\Max(x) = \Max(X_{\geq x})$ and~$\Min(x) = \Min(X_{\leq x})$.
The {\em link\/}~$\Lk_Y(x)$ in~$Y$ of some~$x \in X$ is given by the subposet~$Y_{<x} \cup Y_{>x}$.
We will usually call~$Y_{<x}$ the {\em lower link\/} in~$Y$ of~$x$, and~$Y_{>x}$ the {\em upper link\/} in~$Y$ of~$x$.

\bigskip

The following standard condition is relevant in later Proposition~\ref{propositionContractibleCentralizers}:

\begin{definition}[upward-closed subposets]
\label{defOpenClosed}
A subposet~$Y \subseteq X$ is said to be {\em upward-closed\/} in~$X$,
if for all~$y \in Y$, we have~$X_{>y} \subseteq Y$.
\donerk
\end{definition}

For a poset~$X$, a {\em chain\/} is an ``inclusion-chain''---a
subset~$a \subseteq X$ such that the elements of~$a$ are pairwise comparable.

\centerline{
   Write~$X'$ for the poset of non-empty chains of~$X$.
            }

\noindent
This is just the {\em face poset\/} of the order complex~$\K(X)$.
An~$n$-chain of~$X$ is a chain~$a \subseteq X$ which has size~$|a|=n+1$.
Recall that~$X$ and~$X'$ are homotopy equivalent.
If we wish to emphasize the order of the elements in the~$n$-chain~$a \in X'$,
we may write~$a = (x_0 < x_1 < \ldots < x_n)$.
We will denote by $\max(a)$ the maximal element of the chain $a$.

\begin{definition}[The join and Cartesian product of posets]
\label{defn:joinCartprodofposets}
We denote by~$X \join Y$ the {\em join\/} of the posets~$X$ and~$Y$.
This is a poset whose underlying set is the disjoint union of~$X$ and~$Y$, and the order is given as follows.
We keep the given order in~$X$ and~$Y$, and we put~$x < y$ for~$x \in X$ and~$y \in Y$.
It can be shown that~$\K(X \join Y) = \K(X) \join \K(Y)$, where the latter join is the join of simplicial complexes.
Moreover, after taking geometric realizations, it coincides with the classical join of topological spaces.
That is, if~$|K|$ denotes the geometric realization of the simplicial complex~$K$,
then we have a homeomorphism~$|K \join L| \cong |K| \join |L|$.
For more details see~\cite{Qui78}.

Recall also that the {\em Cartesian product\/}~$X \times Y$ of two posets 
is a poset with the order given by~$(x,y) \leq (x',y')$ if~$x \leq x'$ and~$y \leq y'$.

We will explore these constructions more carefully in later Section~\ref{sec:prejoinandreplposets}. 
\donerk
\end{definition}

We will often use Quillen's fiber-theorem to prove that certain maps between posets are homotopy equivalences
(see~\cite[Prop~1.6]{Qui78}); and hence that certain posets have the same homotopy type.
We include an equivariant version (see~\cite{TW})
in Theorem~\ref{variantQuillenFiber} below;
and an extended version in the Appendix as Proposition~\ref{prop:QuilFiberConn}.

\begin{theorem}[Quillen fiber-theorem]
\label{variantQuillenFiber}
Suppose $f:X\to Y$ is a map between $G$-posets $X$ and $Y$ such that for all~$y \in Y$, we have that~$f^{-1}(Y_{\leq y}) * Y_{>y}$ (resp.~$f^{-1}(Y_{\geq y}) * Y_{<y}$) is contractible.
Then~$f$ is a homotopy equivalence.
If in addition~$f$ is~$G$-equivariant, and we have~$G_y$-contractibility for each~$y$,
then~$f$ is a~$G$-homotopy equivalence.

In particular if $X \subseteq Y$ (so that $f$ is the inclusion map), 
and for all~$y \in Y - X$, we have~$X_{<y} * Y_{>y}$ (resp.~$X_{>y} * Y_{<y}$) contractible,
then~$f$ is a homotopy equivalence;
and indeed a $G$-homotopy equivalence, when we have $G_y$-contractibility for those links.
\hfill $\Box$
\end{theorem}

Let~$X$ be a finite poset.
We denote by~$\tilde{H}_*(X,R)$ the reduced homology of~$X$ with coefficients in the ring~$R$,
which is the reduced homology of its order complex $\K(X)$.
In general we will work with $R = \QQ$ and we will just write $\tilde{H}_*(X)$.
Finally $\tilde{\chi}(X)$ denotes the reduced Euler characteristic of a topological space or finite poset $X$.

\begin{remark}[Product homology for joins]
\label{rk:prodhomolforjoins}
If $R$ is a field, the homology of a join of spaces is the tensor product of homologies.
That is, we have that:
\[\tilde{H}_*(X\join Y,R) = \tilde{H}_*(X,R)\otimes_R \tilde{H}_*(Y,R),\]
\[\tilde{H}_n(X\join Y, R) = \bigoplus_{i+j=n-1} \tilde{H}_i(X,R) \otimes_R \tilde{H}_j(Y,R).\]
For more details on this isomorphism, see \cite{milnor}.

By the above isomorphism between the homology of a join of spaces and the tensor product of homologies,
we get $\tilde{\chi}(X_1 * \ldots * X_n) = (-1)^{n-1} \prod_{i=1}^n \tilde{\chi}(X_i)$.

Product homology for joins is fundamental for the ``classical'' approach to homology propagation;
see the discussion at later Remark~\ref{rk:homopropviajoinAP}.
\donerk
\end{remark}

\subsection*{The \texorpdfstring{$p$}{p}-subgroup posets and results on Quillen's Conjecture}

We close the preliminaries section with the notation of the different $p$-subgroup posets that we will work with;
along with some background results from the literature on Quillen's conjecture.
The main posets are:

\begin{align*}
\Sp(G) & = \text{ Brown poset }  = \text{ poset of all nontrivial $p$-subgroups of $G$}, \\
\Ap(G) & = \text{ Quillen poset }  = \{\ A\in \BrownPoset_p(G)\tq A  \text{ is elementary abelian}\ \},\\
\Bp(G) & = \text{ Bouc poset } = \{\ P \in \BrownPoset_p(G) \tq P = O_p \bigl( N_G(P) \bigr)\ \} .
\end{align*}

\noindent
We also use:

\begin{definition}[The poset~$\mathfrak{i}(X)$]
\label{defn:i(X)}
Suppose next that~$X$ is finite poset, such that every non-empty and lower-bounded subset~$Y \subseteq X$ has an infimum in~$X$.
Denote by~$\mathfrak{i}(X) \subseteq X$ the subposet whose elements are the infimum of the maximal elements above them.
%That is, $ x \in \mathfrak{i}(X)$ if and only if~$x = \inf \{ y \in \Max(X) \tq y \geq x \}$.%
%\footnote{
%  "fake footnote":
%  (Decide later.)
%  Maybe want to NOT use capital-M Max notation?
%          }
Note that~$\mathfrak{i}(X)$ is a poset-strong deformation retract of~$X$,
via the retraction sending~$x \in X$ to the infimum of the set of maximal elements of $X$ above $x$.
We denote this retraction by~$\iret$.

Since the posets~$\Sp(G)$ and~$\Ap(G)$ satisfy this property
(that is, every non-empty lower bounded subset has an infimum---given by the intersection of that set),
we will also consider the posets~$\mathfrak{i} \bigl( \Sp(G) \bigr)$ and~$\mathfrak{i} \bigl( \Ap(G) \bigr)$.
Note that~$\mathfrak{i} \bigl( \Sp(G) \bigr)$ consists of the nontrivial intersections of Sylow~$p$-subgroups of~$G$.
On the other hand, it can be shown (see~\cite[Remark~4.5]{KP19})
that~$\mathfrak{i} \bigl( \Ap(G) \bigr)$ is the poset of elements~$E \in \Ap(G)$
such that~$E = \Omega_1 \bigl(\ Z(\ \Omega_1 \bigl( C_G(E)\ \bigr)\ )\ \bigr)$.
\donerk
\end{definition}

By the results of~\cite{Qui78,TW}, we have that:

\begin{proposition}
\label{homotopyEquivalenPosets}
Let~$G$ be a finite group and~$p$ a prime.

\noindent
Then the posets~$\Sp(G)$, $\Ap(G)$, $\Bp(G)$, $\mathfrak{i}(\Sp(G))$ and~$\mathfrak{i}(\Ap(G))$ are all~$G$-homotopy equivalent.
\hfill $\Box$
\end{proposition}

\noindent
For example, if~$G$ is a~$p$-group~$P$, then~$\Sp(P)$ is contractible to its maximal element~$P$;
hence the other posets above such as~$\Ap(P)$ are also contractible.

\medskip

In view of the above equivalences, we could equally well use these other posets in the statement of Quillen's Conjecture;
though of course it is customary to use~$\Ap(G)$ there.

\bigskip

At the start of this paper, we had mentioned Quillen's result~\cite[Prop~2.4]{Qui78} on ``conical contractibility'';
here we state it in the form:

\centerline{
  If~$O_p(G) > 1$, then~$\Sp(G)$ is contractible. 
            }

\noindent
In order to illustrate how we will later be  using Lemma~\ref{lm:increndo}(1) in contractibility proofs,
we now use that lemma to prove Quillen's proposition---expanding on
his homotopy-of-maps approach in~\cite[1.5]{Qui78}: 

\begin{example}[Quillen's conical contractibility via the endomorphism-Lemma~\ref{lm:increndo}]
\label{ex:zigzagconical}
Recall the notation of Lemma~\ref{lm:increndo}(1), including~(i$^{\prime}$) in our  discussion thereafter.

Let's begin in a somewhat more general context than the~$p$-subgroup posets above;
namely as our overall-context poset ``$X$'',
we take~$S(G)$---the poset of {\em all\/} subgroups~$H \leq G$ (including the identity~$1 = 1_G$). 
And let's {\em assume\/} that~$N$ denotes some normal subgroup of~$G$ (possibly~$N = 1$). 
Then we have a zigzag of group-inclusions determined by any~$H \in S(G)$: 
  \[  H \leq HN \geq N . \]
In the viewpoint of~(i$^{\prime}$), this defines maps $f_0(H) = H$, $f_1(H) = HN$, and~$f_2(H) = N$;
and we have the $<$-comparability relations $f_0 \leq f_1 \geq f_2$. 
Furthermore the identity map~$f_0 = \Id_{S(G)}$, and the constant map~$f_2$ to~$\{ N \}$, are easily seen to be poset endomorphisms;
but for~$f_1$ we further need to make the following observations:
First, normality of~$N$ has the group-theoretic consequence that~$HN$ is also a subgroup---that is, $f_1(H) = HN \in S(G)$, 
so that~$f_1$ is at least a set-endomorphism. 
Second, properties of group-inclusion show that $H \leq K$ implies $HN \leq KN$---so that~$f_2$ is in fact a poset map.
We have now completed the verification of the hypotheses of Lemma~\ref{lm:increndo}(1);
so we conclude by that result that~$S(G)$ is (conically) contractible to~$\{ N \}$---for any normal subgroup~$N$.

Now we turn to the actual result of Quillen on~$p$-subgroups:
For our subposet~``$Z$'', we take the Brown poset~$\Sp(G)$, of nontrivial~$p$-subgroups;
and we now further assume that our normal subgroup~$N$ is a nontrivial~$p$-group (so that~$N \in \Sp(G)$). 
This time we take~$H \in \Sp(G)$:
We still get the inclusion-zigzag as above;
and also the indicated properties for the restrictions~$f'_i$ of the~$f_i$ to~$\Sp(G)$---{\em except\/}
possibly for the endomorphism-property on~$\Sp(G)$.
That is, here we must further show that each~$f'_i(H)$ really does lie in~$\Sp(G)$.
For the identity map~$\Id_{\Sp(G)}$, we have~$f'_0(H) = H \in \Sp(G)$ by our original choice of~$H$. 
For the constant map to~$\{ N \}$, we saw~$f'_2(H) = N \in \Sp(G)$ above---crucially using
our hypothesis that~$N$ is a {\em nontrivial\/}~$p$-group. 
Finally we see~$f'_2(H) = HN$ is a nontrivial~$p$-group, using the group-theoretic product-order formula. 
This completes the endomorphism-proof, and hence the hypotheses of Lemma~\ref{lm:increndo}(1), applied to~$\Sp(G)$.
So by that result, $\Sp(G)$ is conically contractible to~$\{ N \}$ (and indeed more generally to~$\{ O_p(G) \}$). 

Notice this proof does not work, if for~$Z$ we take the Quillen poset~$\Ap(G)$:
for~$N$ might not be elementary abelian;
and even if so (e.g.~using $\Omega_1 Z(N)$), $f_1(H) = HN$ might not be elementary abelian
(though both~$H$ and~$N$ are, they need not commute). 
Thus to conclude that $\Ap(G)$ is contractible when~$O_p(G) > 1$,
we have to first quote the result for $\Sp(G)$---and then apply the homotopy equivalence of~$\Sp(G)$ with~$\Ap(G)$ for general~$G$.
\donerk
\end{example}

\bigskip

Note that~$m_p(G) - 1$ equals the topological dimension of~$\Ap(G)$; this dimension is the focus of:

\begin{definition}[Quillen Dimension and $p$-extensions]
\label{defn:QD}
We say~$G$ has {\em Quillen dimension\/} at~$p$ if:

  \quad \quad $\QD_p$ \quad If~$O_p(G) = 1$, then~$\tilde{H}_{m_p(G)-1} \bigl( \Ap(G) \bigr) \neq 0$.

\noindent
In particular, $\QD_p$ implies (H-QC). 

The condition~$\QD_p$ is often studied for a~{\em $p$-extension\/}~$LB$ of a component~$L$ of~$G$:
following~\cite[p~474]{AS93}, this means a split extension of~$L$ by an elementary abelian~$p$-group~$B$ of {\em outer\/} automorphisms.
(From now on, we usually reserve the notation~``$LB$'' for this~$p$-extension situation.)
\donerk
\end{definition}

For solvable~$G$ with~$O_p(G) = 1$, Quillen obtained~(H-QC) using the above property (see~\cite[Cor~12.2]{Qui78}).
The term ``Quillen dimension" for the property was introduced later by Aschbacher-Smith in~\cite{AS93}:
they showed in their Proposition~1.7 (roughly) that for a simple component~$L$ of a general~$G$ with~$O_p(G) = 1$, 
having the~$\QD_p$-property for~$p$-extensions~$LB$ guarantees that nonzero reduced homology for~$\Ap(LB)$
(indeed extended to~$LB C_G(LB)$, in the ``nonconical'' case~$O_p \bigl( C_G(LB) ) = 1$)
propagates to~$\Ap(G)$, and hence gives~(H-QC).
The proof of this result uses their original homology propagation Lemma~0.27
(we provide a full statement later, as Lemma~\ref{lemmaHomologyPropagationAS}),
which exploits the interaction of the~$\QD_p$-property with the calculation of the boundary operator on poset-chains.
So in order to extend their ideas,
we will be studying that boundary-calculation in more detail---e.g.~at later Remark~\ref{rk:homopropviajoinAP}.

In order to maximize the applicability of~\cite[Prop~1.7]{AS93} in their main proof, 
Aschbacher and Smith were able to show, using the~CFSG,
that~$p$-extensions of ``most'' simple groups~$L$ do indeed have the~$\QD_p$-property.
That result also will be relevant to our extensions, so for convenience we recall it
in Theorem~\ref{theoremQDList} below.

If~$p \nmid q$, denote by~$|q|$ the multiplicative order of~$q$ mod~$p$.

\begin{theorem}[{$\QD_p$-List, \cite[Theorem 3.1]{AS93}}]
\label{theoremQDList}
Assume that~$p$ is odd, and~$L$ is simple.
Then the~$p$-extensions~$LB$ satisfy~$\QD_p$, except possibly when~$L$ is:
\begin{enumerate}
\item of Lie type in the same characteristic~$p$.
\item $\U_n(q)$ with $q \equiv -1 \pmod{p}$;
      and either $n \geq q(q-1)$ or~$n \geq q^{1/p}(q^{1/p}-1)$, with field automorphisms $B > 1$.
\item For~$q = r^p$, either~$^3D_4(q)$ for~$|q| = 3,6$, or~$E_8(q)$ for~$|q| = 8,12$, with field automorphisms~$B > 1$.
\item $\Monster$ with $p = 7$, and $\Janko_4$ with $p = 11$.

\begin{flushleft}
(Cases with $p = 5$)
\end{flushleft}
\item $E_8(q)$ with~$|q| = 4$.
\item $^2F_4(2)'$; $\Sz(32)$ or~$^2F_4(32)$, with field automorphisms~$B > 1$.
\item $\HN$, $\Baby$, $\Monster$, $\McL$, $\Ly$.

\begin{flushleft}
(Cases with $p = 3$)
\end{flushleft}
\item Alternating~$\Alt_6$ or~$\Alt_n$ ($n \geq 9$).
\item Of nonlinear type over~$\GF{2}$.
\item Of nonlinear type over~$\GF{8}$, with field automorphisms~$B > 1$.
\item $\Ln_3(4)$; $\Ln_2(8)$, $^2F_4(r^3)$ for~$|q| = 2$, with field automorphisms~$B > 1$.
\item $\Mathieu_{11}$, $\Janko_3$, $\Conway{1}$, $\Conway{2}$, $\Conway{3}$, $\Fi{22}$, $\Fi{23}$, $\Fi{24}'$,
      $\HN$, $\Th$, $\Baby$, $\Monster$, $\McL$, $\Suz$, $\Ly$, $\ON$. \hfill $\Box$
\end{enumerate}
\end{theorem}

\noindent
In view of the above list, and the role of~$p$-extensions in many results on Quillen's conjecture,
we will be interested in understanding, for a given subgroup~$L$ of~$G$ (usually a component),
the outer automorphisms of order~$p$ of $L$~which in fact arise in~$N_G(L)$.
Correspondingly we define two related posets (cf.~\cite[4.2 \& 6.1]{PS}):

\begin{definition}
\label{outerPosetImagePoset}
Let $L \leq G$ and $p$ a fixed prime.
Define the~{\em $p$-outer poset\/} of~$L$ to be the poset:
  \[ \Outposet_G(L) := \{ B \in \Ap \bigl( N_G(L) \bigr) \tq B \cap  \bigl( L C_G(L) \bigr) = 1 \} . \]
We usually call an element~$B$ of~$\Outposet_G(L)$ a~{\em $p$-outer\/} of~$L$ in~$G$,
since it induces only outer automorphisms (and hence $LB$ is a~$p$-extension as in Definition~\ref{defn:QD}).
We also write:
  \[ \hat{\Outposet}_G(L) := \Outposet_G(L) \cup \{ 1 \} . \]
The {\em image poset\/} of~$L$ in~$G$ is:

  \hfill $\Imageposet_{G,L} := \{ B C_{G}(L) / C_G(L) \tq B \in \Ap \bigl( N_G(L) \bigr)\ ,\ C_B(L) = 1 \}$ \donerk 
\end{definition}

\begin{remark}
\label{remarkImagePoset}
Note that~$\Ap \bigl( L/Z(L) \bigr) \subseteq \Imageposet_{G,L} \subseteq \Ap \bigl( \Aut_G(L) \bigr)$.
In particular, if~$Z(L) = 1$, then~$\Ap(L) \subseteq \Imageposet_{G,L}$ and:

 \hfill   $\Imageposet_{G,L} - \Ap(L)
        = \{ \bigl( B C_G(L) / C_G(L) \bigr) A \tq A \in \Ap(L) \cup\{ 1 \}\ ,\ B \in \Outposet_G(L) \} $ .  \donerk
\end{remark}

\bigskip

To close this section, we summarize some main results on~(H-QC) in the earlier literature.

\begin{theorem}
\label{psolvableQC}
If~$G$ is a~$p$-solvable group, then it satisfies~$\QD_p$ and in particular~(H-QC).
\hfill $\Box$
\end{theorem}

\noindent
The $p$-solvable case was established by various authors, based on the ideas of Quillen on the solvable case~\cite{Qui78};
see~\cite[Ch~8]{Smi11} for further details.

\begin{theorem}[{Aschbacher-Kleidman~\cite{AK90}}]
\label{almostSimpleQC}
An almost-simple group~$G$ satisfies~(H-QC).
\hfill $\Box$
\end{theorem}

  Here is the original Aschbacher-Smith result for~$p > 5$, motivating our extension in Theorem~\ref{theoremASExtension}:

\begin{theorem}[{Aschbacher-Smith~\cite{AS93}}]
\label{aschbachersmithQC}
Let~$G$ be a finite group and~$p$ a prime.
Assume that:
\begin{enumerate}[label=(\roman*)]
\item $p>5$; and
\item whenever~$G$ has a unitary component~$\U_n(q)$ with $q \equiv -1 \pmod{p}$, \\
      then $\QD_p$ holds for all~$p$-extensions~$\U_m(q^{p^e})$ with~$m \leq n$ and~$e \in \ZZ$.
\end{enumerate}
Then~$G$ satisfies~(H-QC) for~$p$.
\hfill $\Box$
\end{theorem}

Recall that in our discussion in the Introduction leading up through Remark~\ref{rk:contextofH1results},
we indicated in~(MOC) that a minimal-order counterexample to~(H-QC) satisfies the following induction-replacement hypothesis:

\vspace{0.1cm}
\begin{flushleft}
(H1) \quad Proper subgroups and proper~$p'$-central quotients of~$G$ satisfy~(H-QC).
\end{flushleft}
\vspace{0.1cm}

\noindent
Below we summarize various results obtained under~(H1) in the articles~\cite{KP20, PS}:

\begin{theorem}
[{\cite{KP20}}]
\label{generalReduction}
Let~$G$ be a group and~$p$ a prime.
Suppose that~$G$ satisfies~(H1), and that one of the following holds:
\begin{enumerate}
\item $Z(G) \neq 1$ or~$\Omega_1(G) < G$;
\item $O_{p'}(G) \neq 1$;
\item $\Ap(G)$ is not simply connected;
\item $G$ has~$p$-rank at most~$4$;
\item $G$ has a component~$L$ such that~$L/Z(L)$ has~$p$-rank~$1$;
\item $p=3$, and~$G$ has a component~$L$ such that~$L/Z(L) \groupiso \U_3(8)$.
\end{enumerate}
Then~$G$ satisfies~(H-QC).
\hfill $\Box$
\end{theorem}

The result~\cite[Corollary~3]{KP20} first showed that Theorem~\ref{aschbachersmithQC} above extends to~$p = 5$.
(So our extension Theorem~\ref{theoremASExtension}, which we prove as later Theorem~\ref{proofQCOdd},
in effect adds just the ``new'' prime~$p=3$ in the proof.)

\begin{corollary}
\label{extensionsASQC5}
Theorem~\ref{aschbachersmithQC} extends to~$p=5$.
\hfill $\Box$
\end{corollary}

\noindent
However, there is a small gap in the explanation given in~\cite{KP20};
it can be fixed without requiring the results of this paper---and we indicate that adjustment during the proof of Theorem~\ref{proofQCOdd}.

Next, we recall a recent result that establishes~(H-QC) for groups~$G$ containing components~$L$ with~$\Out_G(L)$ a~$p'$-group.
The proof of the following theorem does not depend on the~CFSG.
Moreover, in the language of~\cite{PS} (which we had described in earlier Remark~\ref{rk:stratelimRob}), 
this result is an example of an ``elimination-result":
it allows us to eliminate possible components from a minimal-order counterexample to~(H-QC).

\begin{theorem}[{cf.~\cite[Corollary~5.1]{PS}}]
\label{PStheorem}
Let $p$ be a prime and $L$ a component of a group~$G$.
Let~$H = \hat{N}_G(L)$.
Suppose that:
\begin{enumerate}
\item $p$ divides the order of~$L$, and $C_G(\hat{L})$ satisfies~(H-QC);
\item The induced map~$\Ap(L) \to \A$ is not the zero map in homology, where~$\A$ is one of the following posets:
       \[ \Imageposet_{H,L}, \,\, \Imageposet_{G,L}, \,\, \Ap \bigl( \Aut_H(L) \bigr), \,\,
           \Ap \bigl( \Aut_G(L) \bigr), \,\, \Ap \bigl( \Aut(L)
       \bigr) . \]
\end{enumerate}
Then~$G$ satisfies~(H-QC).

In particular, if~$G$ is a counterexample of minimal order to~(H-QC), then~$\Aut_G(L)$ must contain~$p$-outers;
for odd~$p$, this eliminates alternating and sporadic components.
\hfill $\Box$
\end{theorem}

\noindent
The final statement about odd~$p$ uses details about outer automorphisms in Remark~\ref{rk:outerautsofsimple}.

\bigskip

Using the above results:
If $G$ is a counterexample of minimal order to~(H-QC),
then by~(MOC) in Remark~\ref{rk:stratelimRob}, $G$ satisfies~(H1)---so it fails conditions~(1--6) of Theorem~\ref{generalReduction}.
Therefore, every component~$L$ of~$G$ has order divisible by~$p$ (by Lemma~\ref{lemmaOpandp});
the map~$\Ap(L) \to \Imageposet_{\hat{N}_G(L),L}$ is the zero map in homology (by Theorem~\ref{PStheorem});
$G$ is not~$p$-solvable nor almost-simple (by Theorems~\ref{psolvableQC} and~\ref{almostSimpleQC});
and if~$p \geq 5$, then it contains some unitary component for which some~$p$-extension does not satisfy~$\QD_p$
(by Theorem~\ref{aschbachersmithQC}).

\bigskip
\bigskip

\part{Techniques for homology propagation}

\section{Overview: Using replacement-posets homotopy equivalent to~\texorpdfstring{$\Ap(G)$}{Ap(G)}}
\label{sec:overviewequivApG}

In this section, we implement a theme we had indicated in earlier Remark~\ref{rk:equivvisualbdry}: 
namely we recall from the literature---and extend---some methods
for replacing~$\Ap(G)$ with more convenient (and typically smaller) homotopy-equivalent posets.

First in Definition~\ref{definitionXBHposet},
we will review (and slightly generalize) an earlier construction of an equivalent poset~$X_G(H)$,
obtained with respect to a subgroup~$H \leq G$.
We will indicate a variant in Proposition~\ref{propInductiveIntersectionPoset},
namely a further-reduced equivalent subposet~$X_G \bigl( \mathfrak{i}(H) \bigr)$.

Finally in the latter part of the section, we will review in Remark~\ref{remarkClosingSection}
a different notion of reducing~$\Ap(G)$ to an equivalent poset---by ``removing points with conical centralizer''.
Then in Proposition~\ref{propositionContractibleCentralizers}, 
we present a result for implementing these removals in some relevant situations.

\bigskip

As motivation for our main work on equivalences, we first give some relevant general background.
Recall first that for~(H-QC), we of course need to show under~$O_p(G) = 1$ that~$\tilde{H}_* \bigl( \Ap(G) \bigr) \neq 0$.

Indeed in some situations, we can establish the nonzero homology ``immediately'' for~$\Ap(G)$ itself: 
Sometimes this proceeds by direct construction of reduced homology;
see for example the methods in~\cite[Sec~8.1]{Smi11}
(which in particular are used in obtaining the Aschbacher-Smith~$\QD$-List in Theorem~\ref{theoremQDList}).
But the procedure can also be indirect, namely just establishing {\em existence\/} of nonzero homology; 
this is the case for example in the Aschbacher-Kleidman result Theorem~\ref{almostSimpleQC} for almost-simple~$G$.
(We will examine the underlying Robinson-subgroup method further in our later Section~\ref{sectionLefschetz}.)

But in studying~(H-QC) for more general situations, we more typically proceed via some intermediate proper subgroup~$H < G$:

\begin{remark}[Equivalences in the broad context of homology propagation]
\label{rk:contextequivandpropag}
That is, often we begin just with knowledge of nonzero reduced homology of~$\Ap(H)$ for some proper subgroup~$H < G$. 
(Sometimes via direct construction;
or also possibly via induction---or induction-replacement e.g.~as in~(H1) in earlier Remark~\ref{rk:contextofH1results}.)

And then (as we had briefly mentioned after Remark~\ref{rk:compareCFSGinASExtandAlt}), 
we wish to show that the nonzero homology for~$\Ap(H)$ leads to nonzero homology also for~$\Ap(G)$.
This rough notion of ``homology propagation'' is probably broader than the more conventional meaning in the literature;
but it seems appropriate for preliminary exposition.
Stated more precisely,
we mean here that the natural poset inclusion~$i : \Ap(H) \to \Ap(G)$ should induce a nonzero map~$i_*$ in homology.

Of course the simplest case of such a nonzero map arises when~$i$ is a homotopy equivalence---so that
the induced map~$i_*$ is actually a homology isomorphism. 
This simplest situation arises for example in Claim~4 in the proof of~\cite[Thm~4.1]{KP20},
as well as in Lemma~\ref{lemmaReconstructionFromSubposet}(3) in this paper.
Such results already begin to motivate our study of equivalences in this section.

And this equivalence-theme continues,
when we consider the extension in~\cite[Lm~3.14]{KP20} of the naive propagation above---for convenience,
we have given a full statement of this result as later Lemma~\ref{lemmaHomologyPropagationKP20}.
There the propagation proceeds from~$\Ap(H)$, not to~$\Ap(G)$ but instead to some possibly-proper subset~$X \subseteq \Ap(H)$; 
and then it remains to show independently that this~$X$ is homotopy equivalent to~$\Ap(G)$. 
This type of equivalence arises for example in Claim~5 in the proof of~\cite[Thm~4.1]{KP20};
as well as in various cases in this paper, such as Theorem~\ref{theoremChangeABSPosets}---where~$X$ is given 
by a replacement-poset such as~$W^{\QuillenPoset}_G(H,K)$ defined there.
\donerk
\end{remark}

\subsection*{Equivalences via shrinking a subposet inside a poset union}

We now begin the main work of the section,
with our review of the construction of the ``basic'' replacement-poset~$X_G(H)$.

\bigskip

\noindent
As a rough preliminary overview:

We want to understand how~$\Ap(G)$ is built up from suitable subposets of~$\Ap(H)$.
To that end, we first consider the ``inflation'' ~$\N_G(H)$:
namely a certain larger subposet, which retracts to~$H$;
and then we recover~$\Ap(G)$, by adding the remaining elements~$E \in \Ap(G) - \N_G(H)$---which
``glue'' to the poset~$\N_G(H)$, via connections within the subposets~$\Ap \bigl( C_G(E) \bigr)$.
In particular, we are viewing~$\Ap(G)$ as the poset union~$\N_G(H) \cup \bigl( \Ap(G) - \N_G(H) \bigr)$. 

The  smaller replacement-poset~$X_G(H)$ (equivalent to~$\Ap(G)$) is then obtained
by ``deflating'' the inflation-term~$\N_G(H)$ inside that union:
basically, collapsing it down to~$\Ap(H)$---so that the added-$E$ terms indicated above
now remain glued just via the corresponding smaller centralizer-posets~$\Ap \bigl( C_H(E) \bigr)$.
But this process requires replacing the original inclusion-ordering~$<$ on~$X_G(H)$ by a new ordering~$\prec$.
(As a result, $X_G(H)$ is just a subset, but not a~$<$-subposet, of~$\Ap(G)$.)

\bigskip

Thus we start by recalling from the literature the underlying poset~$\N_G(H)$, which is the inflation of~$\Ap(H)$ 
(indeed slightly generalizing its context from~$\Ap(G)$ to a suitable subposet~$\B$):

\begin{definition}[The~$\N \cup \F$-decomposition]
\label{defInflated}
For~$H$ a subgroup of~$G$, with~$\Ap(H) \subseteq \B \subseteq \Ap(G)$ a subposet, set:
  \[ \N_\B(H) := \{ E \in \B \tq E \cap H \neq 1 \} , \text{ and } \]
  \[ \F_\B(H) := \{ E \in \B \tq E \cap H = 1    \} . \]

\centerline{
   Note that~$\B$ is the {\em disjoint\/} union of~$\N_\B(H)$ and~$\F_B(H)$.
            }

\noindent
We focus on one feature of this union, with respect to the~$<$-ordering:
Note that if~$F \in \F_\B(H)$ and $A \in \N_\B(H)$, then we cannot have~$F > A$---as that would violate~$F \cap H = 1$. 
That is:

\centerline{
  For~$A \in \N_\B(H)$, we have~$\B_{>A} \cap \F_\B(H) \ = \emptyset$; so $\B_{>A} = \N_\B(H)_{>A}$.
            }
\noindent
Thus members of~$\F_\B(H)$ only appear as the {\em left\/} segment of inclusion-chains:

  \quad \quad (i) \quad A $<$-chain begins with any terms from~$\F_\B(H)$, followed by any terms from~$\N_\B(H)$.

\noindent
This type of property will be relevant frequently, as our development continues;
for example in later Remark~\ref{rk:XGHlowerlinkformbdry}, 
we will at least begin to see its significance for our viewpoint on boundary calculations, for homology propagation.

When~$\B = \Ap(K)$ for some~$K \leq G$, write~$\N_K(H) := \N_{\Ap(K)}(H)$ and~$ \F_K(H) := \F_{\Ap(K)}(H) $.
In particular for~$K = G$, we get:

  \hfill $\Ap(G) = \N_G(H) \cup \F_G(H)$. \donerk
\end{definition}

\noindent
We have the following immediate results (cf.~\cite{KP20,PSV,SW}):

\begin{lemma}
\label{lemmaReconstructionFromSubposet}
Assume~$H \leq G$, with~$\Ap(H) \subseteq \B \subseteq \Ap(G)$ as in Definition~\ref{defInflated}.
Then we get:

\smallskip 
(1)   $\N_\B(H)$ poset-strong deformation-retracts onto~$\Ap(H)$---via~$r : \N_\B(H) \to \Ap(H)$,
      which we define by~$r(E) := E \cap H$ (we may call this the ``deflation'' map). 

\smallskip 
(2)   If~$\B = \Ap(K)$ for~$K \leq G$, and~$E \in \F_K(H)$, then~$\N_K(H)_{>E} \simeq \Ap \bigl( C_H(E) \bigr)$.

\smallskip 
(3)   If~$\B = \Ap(K)$ as in (2), and further~$O_p \bigl( C_H(E) \bigr) > 1$ for all~$ E \in \F_K(H)$,
      then~$\Ap(H) \hookrightarrow \B$ is a homotopy equivalence.
      So if this further condition holds for~$K = G$, we get~$\Ap(H) \simeq \Ap(G)$.
\end{lemma}

\begin{proof}
For part~(1),
we apply Lemma~\ref{lm:increndo}(2), with~$r,\N_\B(H)$ in the roles of~``$f,X$'':
Notice that we have~$r(E) \leq E$, with~$r$ the identity when restricted to~$r \bigl( \N_\B(H) \bigr) = \Ap(H)$.
We emphasize that for the endomorphism-hypothesis in that Lemma,
we further need to check that~$r(E) = E \cap H$ lies in~$\N_\B(H)$:
this holds here, since we have~$E \cap H \in \Ap(H) \subseteq \B$, 
while~$(E \cap H) \cap H = E \cap H > 1$---so that~$E \cap H \in \N_\B(H)$. 
For part~(2),
we use the poset map~$f : \N_K(H)_{>E} \to \Ap \bigl( C_H(E) \bigr)$ defined by~$f(B) := B \cap H$,
with homotopy inverse~$g(C) := EC$;
for the homotopies, we observe that~$gf(B) = E(B \cap H) \leq B$, along with~$fg(C) = EC \cap H \geq C$---that is,
$gf \leq \Id_{\N_K(H)_{>E}}$, and~$fg \geq \Id_{\Ap( C_H(E)) }$.
Part~(3) is a consequence of~(2), via the Quillen fiber-Theorem~\ref{variantQuillenFiber}.
\end{proof}

Next, in Definition~\ref{definitionXBHposet}, we will introduce~$X_G(H)$---the first of our replacement-posets for~$\Ap(G)$. 
Similar constructions had been considered earlier in~\cite{SW}, and also at~\cite[p~134]{Thevenaz};
compare also with~\cite[(6.8)]{Asc93}.
The related construction in~\cite{KP20} shrank~$\Ap(G)$ to a smaller equivalent poset, but did not change the order-relation. 
In this paper, we change the ordering---obtaining more general equivalent posets,
which will help to make the later homology-propagation clearer and more explicit.

In fact, we will see in Definition~\ref{definitionXBHposet} how we can use the deflation-retraction~$r$ above, 
to shrink~$\N_\B(H)$ to~$\Ap(H)$---within the decomposition given by the poset union in Definition~\ref{defInflated};
though this comes at the cost of defining a new poset-ordering on the resulting subset.
We will denote the new order-relation by~$\prec$:
since it agrees with inclusion~$<$,
except at certain specified cross-terms between~$\F$ and~$\N$, in the language of the decomposition in Definition~\ref{defInflated}. 
The result is a ``replacement-poset''~$X_\B(H)$:
that is, we will see in Theorem~\ref{theoremInductiveHomotopyType}
that it will be equivalent to~$\B$ in our standard situations:
where it will often be more convenient for us to work with, for our later propagation-applications.

\begin{definition}[The replacement-poset~$X_\B(H)$, with its ordering~$\prec$] 
\label{definitionXBHposet}
Assume a subgroup~$H \leq G$;
and an intermediate poset~$\B$ with~$\Ap(H) \subseteq \B \subseteq \Ap(G)$, as in Definition~\ref{defInflated}.
Define~$X_\B(H)$ to be the poset which---just as a set, rather than a poset---is the disjoint union:

\centerline{
  $X_\B(H) = \Ap(H) \cup \F_\B(H)$;
            }

\noindent
and with a new poset order-relation~$\prec\ :=\ \prec_{X_\B(H)}$ given as follows:

  (1) Inside~$\Ap(H)$, and also inside~$\F_\B(H)$, as $\prec$ we keep the group-inclusion ordering~$<$.

  (2) When~$F \in \F_\B(H)$, and~$A \in \Ap(H)$ with~$C_A(F) > 1$, we define~$F \prec A$.

\noindent
We need to check transitivity for this new relation~$\prec$:
Since~$\prec$ agrees with the old ordering~$<$ inside each of the two factors in the union for~$X_\B(H)$, 
we only need to check transitivity of extensions at either end of cross-terms%
\footnote{
  Notational comment:  We will typically reserve the letter-pairs~$F,A$ and~$F,B$ for such cross-terms.
          }
between the two factors;
and by the~$\F$-left property of~$\prec$ in~(2) above (see also (3) below),
these have form~$F \prec A$, for~$F \in \F_\B(H)$ and~$A \in \Ap(H)$---where we have~$C_A(F) > 1$. 
And such extensions must then have the forms~$E < F$ for~$E \in \F_\B(H)$, or~$A < B$ for~$B \in \Ap(H)$. 
When~$E < F$, we get~$1 < C_A(F) \leq C_A(E)$, so that also~$E \prec A$. 
And when~$A < B$, we get~$1 < C_A(F) \leq C_B(F)$, so that also~$F \prec B$.
This completes the transitivity proof.
Hence~$X_\B(H)$ is indeed a poset.

\medskip

\noindent
Note that~(2) gives the analogue, for~$\prec$, of the $\F$-left property~(i) in Definition~\ref{defInflated}:

 (3)  A~$\prec$-chain begins with any terms from~$\F_\B(H)$, followed by any terms from~$\Ap(H)$. 

\noindent
We typically write out such a~$\prec$-chain~$a$ in the form:

\centerline{
    $ a = ( F_1 < F_2 < \cdots < F_r \prec A_1 < A_2 < \cdots < A_s )$, for~$F_i \in \F_\B(H)$ and~$A_i \in \Ap(H)$ ;
	     }

\noindent
where we have in fact written the new notation~$\prec$ only at the cross-term case (where it is definitely required);
but used the old notation~$<$ at all the other adjacencies, where~$\prec$ in fact agrees with~$<$.
This viewpoint has another alternative form, again analogous to a property in Definition~\ref{defInflated}:

      If~$A \in \Ap(H)$, then~$X_\B(H)_{\succ A} \cap \F_\B(H) = \emptyset$;
      that is, $\Ap(H) \supseteq X_\B(H)_{\succ A} = \Ap(H)_{> A}$.

\bigskip

\noindent
If~$\B = \Ap(K) \supseteq \Ap(H)$, for~$K \leq G$, we write just~$X_K(H)$ for~$X_\B(H)$.
When~$K = G$, we get:

\centerline{
  $X_G(H) = \Ap(H) \cup \F_G(H)$,
             }

\noindent
as a subset of~$\Ap(G)$ (though not a subposet under group-inclusion~$<$).
\donerk
%\donerk
\end{definition}

\begin{remark}[A left-focused format for poset-chains involved in boundary calculations]
\label{rk:XGHlowerlinkformbdry}
A new feature of our use of~$X_G(H)$ in this paper is that
(as we had briefly suggested in earlier Remark~\ref{rk:equivvisualbdry})
we can exploit the~$\F$-left condition in Remark~\ref{definitionXBHposet}(3),
to get a more ``visual'' format for its poset-chains---which makes analysis of their boundaries easier than in~$\Ap(G)$:

Very roughly, for a suitable cycle from~$\Ap(H)$ involving a~$<$-chain~$a$,
the main preliminary calculation in homology propagation shows that chains from~$\Ap(G)$ properly containing~$a$,
which are built just from further members of~$\N_G(H)$, do not contribute to the boundary of~$a$ in~$\Ap(G)$;
so that the only problem-members must come from the other factor~$\F_G(H)$ in the union.

And after applying the homotopy equivalence of~$X_G(H)$ with~$\Ap(G)$ in Theorem~\ref{theoremInductiveHomotopyType} below, 
it will be an advantage of the~$X_G(H)$-formulation for the boundary-calculation
that any such problem-members of~$\F_G(H)$ appearing in~$\prec$-chains containing~$a$
are concentrated together at the {\em left\/} end of those chains.  
We'll flesh out this vague description later; notably in Remark~\ref{rk:QDfullbdrycalc}.

This left-focused format for the chains in~$X_G(H)$ (and in other posets later in the paper)
will help simplify the proofs of some new homology-propagation results such as later Theorem~\ref{theoremHomologyCharP}.
\donerk
\end{remark}

As promised earlier, we next get an explicit equivariant homotopy equivalence between~$\B$ and~$X_\B(H)$---under
suitable extra conditions on~$\B$: which hold, for example, when~$\B$ is of the usual form~$\Ap(K)$:

\begin{theorem}
\label{theoremInductiveHomotopyType}
Assume~$H \leq G$, with~$\Ap(H) \subseteq \B \subseteq \Ap(G)$ as in Definition \ref{defInflated};
and write $N_G(H,\B)$ for the largest subgroup of $N_G(H)$ acting on $\B$.
Recall the poset $X_\B(H) = \N_\B(H) \cup \F_\B(H)$ under $\prec$ from Definition \ref{definitionXBHposet}.
Assume also the further condition:

  (i) When~$F \in \F_\B(H)$ and~$B \in \N_\B(H)$ satisfy~$F \prec (B \cap H)$ (that is, $C_{B \cap H}(F) > 1$),

  \quad $\phantom{0}$ then also~$C_B(F), C_B(F)F \in \B$.

\noindent
Then we have an~$N_G(H,\B)$-homotopy equivalence~$\alpha_{\B,H} : \B \to X_\B(H)$,
defined using the deflation map on the~$\N_\B(H)$-term:
  \[ \alpha_{\B,H}(D) =
       \begin{cases}
         D \cap H & \tq D \in \N_\B(H) \\
         D        & \tq D \in \F_\B(H).
       \end{cases}
  \]
Moreover, the following hold:

\smallskip 
 (1) The homotopy equivalence is the identity when restricted to~$\Ap(H)$.

\smallskip 
 (2)  For~$F \in \F_\B(H)$, $X_\B(H)_{\succ F} - \F_\B(H) = \Ap(H)_{\succ F} = \N_H \bigl( C_H(F) \bigr)$;
      this has a poset-strong

      \quad \quad deformation retraction to~$\Ap \bigl( C_H(F) \bigr)$
      via the $C_H(F)$-deflation map 

      \quad \quad \quad \quad $B \mapsto B \cap C_H(F) = C_{B \cap H}(F) = C_B(F)$. 

\smallskip 
 (3)  If~$\B \subseteq \N_G(H)$, then~$\B = \N_\B(H) \simeq \Ap(H)$;
      and~$\alpha_{\B,H}$ is the restriction to~$\B$ of the deflation, 

      \quad \quad i.e.~the retraction~$B \mapsto B \cap H$.

\smallskip \noindent
In particular, if~$\B = \Ap(K) \supseteq \Ap(H)$ for~$K \leq G$, the above hypotheses (notably~(i)) hold;
and we write~$\alpha_{K,H}$ for the above map, which is an~$N_G \bigl( H , \Omega_1(K) \bigr)$-homotopy equivalence.

Thus for~$K = G$, we get an~$N_G(H)$-homotopy equivalence of~$\Ap(G)$ with~$X_G(H)$.%
\footnote{
  Notice we do not say ``retraction'': because~$X_G(H)$ is just a subset---not a subposet under~$<$---of~$\Ap(G)$.
          }
\end{theorem}

\begin{proof}
First we check that~$\alpha_{\B,H}$ is order-preserving---from~$<$ to~$\prec$:
Since~$<$ agrees with~$\prec$ inside the images of each of the two factors in the union giving the domain~$\B$,
we only need to check images of cross-terms between those two factors;
which by the~$\F$-left property~(i) in Definition~\ref{defInflated} must be of the form~$F < A$,
for~$F \in \F_\B(H)$, and~$A \in \N_\B(H)$ (so that~$A \cap H > 1$).  
Here~$F$ centralizes~$A$ just since~$A$ is abelian, so that~$C_{A \cap H}(F) = A \cap H > 1$---and so
we do indeed get~$F \prec A \cap H$ for the images, as needed.

\bigskip

\noindent
To show~$\alpha_{\B,H}$ is an equivariant homotopy equivalence, we will use the Quillen fiber-Theorem~\ref{variantQuillenFiber}:

\medskip

The first step is the equivalence (with equivariance indicated later):
For~$x \in X_\B(H)$, we will show below
that the preimage~$Z_x := \alpha_{\B,H}^{-1}({X_\B(H)}_{\succeq x})$ is contractible.

\bigskip

Consider first the case~$x = A \in \Ap(H)$:
Then using the~$\F$-left property in Definition~\ref{definitionXBHposet}(3),
and the hypothesis~$\Ap(H)~\subseteq \B$,
we have~$X_\B(H)_{\succeq A} = \Ap(H)_{\geq A}$, so that our preimage is:

\centerline{
  $Z_A = \alpha^{-1}_{\B,H}({X_\B(H)}_{\succeq A}) = \alpha^{-1}_{\B,H} \bigl( \Ap(H)_{\geq A}  \bigr) 
                                             = \{ B \in \N_\B(H) \tq B \cap H \geq A \} = \N_\B(H)_{\geq A}$ ; 
             }

\noindent
which is contractible---since it has minimal element~$A > 1$.

\bigskip

Now consider the remaining case~$x = F \in \F_\B(H)$; this argument will be lengthier.
Note that:

\centerline{
  $  X_\B(H)_{\succeq F} = \Ap(H)_{\succeq F} \cup \F_\B(H)_{\succeq F} = \N_H \bigl( C_H(F) \bigr) \cup \F_\B(H)_{\geq F}$ ,
             }

\noindent
and in particular we get the first part of~(2). 
Furthermore our preimage then has the form:

   (ii) $Z := Z_F = \alpha_{\B,H}^{-1}({X_\B(H)}_{\succeq F}) = \N_\B \bigl( C_H(F) \bigr) \cup \F_\B(H)_{\geq F}$ . 

\noindent
We will show using Lemma~\ref{lm:increndo}(1) that~$Z$ is contractible (to~$\{ F \}$).
As in~(i$^{\prime}$) in our discussion after that lemma
(compare also the worked-out details of the application in earlier Example~\ref{ex:zigzagconical}), 
the arguments are summarized by the following zigzag of group-inclusions for~$D \in Z$:

\centerline{
  $D \geq C_D(F) \leq C_D(F)F \geq F$.
            }

\noindent
This abbreviated form expands to defining maps~$j_0, j_1, j_2, j_3$ on~$Z$, 
with images at~$D$  given by~$D, C_D(F), C_D(F)F, F$. 
In particular, $j_0 = \Id_Z$, and $j_3$ is the constant map to $\{ F \}$. 
We see via elementary group-theoretic properties that each~$j_i$ is a poset map;
with the $<$-comparability conditions~$j_0 \geq j_1 \leq j_2 \geq j_3$. 

\smallskip

So to complete the hypotheses of Lemma~\ref{lm:increndo}(1), it remains to check the endomorphism-condition: 
namely that~$j_i(D) \in Z$ for all~$i$.  
For~$j_0 = \Id_Z$ this is clear (by choice of~$D = j_0(D)$ in~$Z$);
and for~$j_3$ also, since~$j_3(D) = F \in \F_\B(H)_{\geq F} \subseteq Z$ by~(ii).
Thus it remains to consider~$j_1,j_2$---that is, to show that the corresponding images~$C_D(F), C_D(F)F$ lie in~$Z$. 

Now since our~$Z$ is defined in the context of a subposet~$\B$---which is not necessarily all of~$\Ap(G)$---as a preliminary,
we first need to see that these images fall into~$\B$:  

  (a) \quad The main complication is to first show that the images~$C_D(F)$ and~$C_D(F)F$ lie in~$\B$:

\noindent
In the case of~$D$ in the right-factor for~$Z$ in~(ii) above, i.e.~$D = E \in \F_B(H)_{\geq F}$, this is no problem:
for then~$C_E(F) = E = C_E(F)F$ where~$E \in \F_\B(H)_{\geq F} \subseteq \B$---giving~(a) for both images.
Indeed since these images lie in~$\F_\B(H)_{\geq F} \subseteq Z$ in~(ii), this even gives~(b) and~(c) below for such~$D = E$. 
The remaining case, that is, the left-factor for~$Z$ in~(ii), has~$D = B \in \N_\B \bigl( C_H(F) \bigr)$, so that~$B \cap C_H(F) > 1$. 
Observe that we have several ways of viewing this intersection,
so that we can write this condition in the expanded form:
  \[ B \cap C_H(F) = C_{B \cap H}(F) = C_B(F) \cap C_H(F) > 1 . \]
In particular, we have~$C_{B \cap H}(F)  > 1$, so that we are in the situation of hypothesis~(i) of the Theorem;
and we now crucially quote that hypothesis, to get~$C_B(F),C_B(F)F$ in~$\B$, as desired.
This completes~(a) in both cases~$D = E,B$ for a member of~$Z$.

  (b) \quad After~(a), it is routine to check that~$j_1(D) = C_D(F) \in Z$:

\noindent
For we covered the right-factor case~$D = E \in \F_\B(H)_{\geq F}$ in the proof of~(a) above.
The left-factor case, where~$D = B \in \N_\B \bigl( C_H(F) \bigr)$, has~$C_B(F) \cap C_H(F) > 1$ as in that proof;
and so now using the conclusion of~(a), we have that~$C_B(F) \in \N_\B \bigl( C_H(F) \bigr) \subseteq Z$ in~(ii), as desired.

  (c) \quad After~(a), it is routine to check that~$j_2(D) = C_D(F)F \in Z$:

\noindent
For again the right-factor case~$D = E \in \F_\B(H)_{\geq F}$ was covered in the proof of~(a).
In the left-factor case given by~$D = B \in \N_\B \bigl( C_H(F) \bigr)$,
we have~$C_B(F)F \cap C_H(F) \geq  B \cap C_H(F) > 1$ as in that proof;
so using the conclusion of~(a), we get~$C_B(F)F \in \N_\B \bigl( C_H(F) \bigr) \subseteq Z$ in~(ii), as desired. 

\smallskip

This completes the proof that~$j_1(D),j_2(D) \in Z$---and hence of the hypotheses of Lemma~\ref{lm:increndo}(1);
so we conclude by that lemma that $Z = Z_F$ is contractible, in this case for~$x = F$. 

\medskip

In conjunction with the contractibility proof for the preimage in the previous case~$x = A$, 
we have completed the proof that~$\alpha_{\B,H}$ is a homotopy equivalence.

\bigskip

For the equivariance of the equivalence,
we need to note that the contractibility-homotopies we produced above
preserve the action of the stabilizer of~$x$.

\medskip

Furthermore the homotopies above are the identity when we restrict them to elements of $\Ap(H)$---additionally giving (1).
Finally~(3) follows from Lemma~\ref{lemmaReconstructionFromSubposet}(1);
and this also gives the remaining part of~(2), when we take~$H$,$C_H(F)$ in the roles of~``$G$,$H$''. 
\end{proof}

\noindent
In our later equivalence-result Proposition~\ref{propApGsimeqW0},
we will adjust the~$X_G(H)$-construction of Definition~\ref{definitionXBHposet}:
First, the role of~``$H$'' will in fact be played by a central product~$HK = H*K$;
and second, we will replace the customary corresponding poset-join~$\Ap(H) * \Ap(K) \simeq \Ap(HK)$
by its ``pre-join'' variant---which we describe in Section~\ref{sec:prejoinandreplposets}.
Indeed in the subsequent equivalence-result Theorem~\ref{theoremChangeABSPosets}, 
we will even generalize our notion of shrinking to actual replacement of the pre-join factor~$\Ap(H)$
by related posets such as~$\Sp(H)$ and~$\Bp(H)$---resulting in replacement-posets which are more general than just~$\Ap$-posets.

%\begin{remark}[The upward-closed property for~$\Ap(H)$ in~$X_{\B}(H)$]
%\label{remarkClosedSubposet}
%We emphasize here that one of the key properties of the poset~$X_\B(H)$ is that, by the definition of the ordering,
%the members of~$\F_\B(H)$ can appear only {\em below\/} elements of~$\Ap(H)$---that is, not above.  
%So~$\Ap(H)$ is an upward-closed subposet of $X_\B(H)$, in the sense of Definition~\ref{defOpenClosed}:
%if~$A \in X_\B(H)$, and we have~$A > B \in \Ap(H)$, then~$A \in \Ap(H)$ also.
%
%This upward-closure, for more general posets related to~$X_{\B}(H)$,
%will be fundamental in our generalized approach to homology propagation in Theorem~\ref{theoremHomologyPropagation};
%see the discussion in Remark~\ref{rk:upclosedandhomolopropag}.
%\donerk
%\end{remark}

\bigskip

\noindent
For the moment, we'll indicate several variations on~$X_{\B}(H)$, using just naive-shrinking:

First, inside~$X_{\B}(H)$, we can shrink~$\Ap(H)$ to the subposet~$\mathfrak{i} \bigl( \Ap(H) \bigr)$ of Definition~\ref{defn:i(X)}:

\begin{proposition}
\label{propInductiveIntersectionPoset}
Under the hypotheses of Theorem~\ref{theoremInductiveHomotopyType}, set:

\centerline{
  $X_{\B} \bigl( \mathfrak{i}(H) \bigr) := \mathfrak{i} \bigl( \Ap(H) \bigr) \cup \F{_\B}(H) \subseteq X_{\B}(H)$.
            }

\noindent
Then the retraction map~$\iret : \Ap(H) \to \mathfrak{i} \bigl( \Ap(H) \bigr)$
extends to an~$N_G(H,\B)$-equivariant retraction map~$\hat{\iret} : X_{\B}(H) \to X_{\B} \bigl( \mathfrak{i}(H) \bigr) . $
In particular, we have~$N_G(H,\B)$-homotopy equivalences:
  $$ \B \simeq X_{\B}(H) \simeq X_{\B} \bigl( \mathfrak{i}(H) \bigr) = \mathfrak{i} \bigl( \Ap(H) \bigr) \cup \F_{\B}(H) , $$
which are the identity when restricted to~$\mathfrak{i} \bigl( \Ap(H) \bigr)$.

In the subcase for~$\B = \Ap(G)$, we get~$N_G(H)$-homotopy equivalences:

\centerline{
  $\Ap(G) \simeq X_G(H) \simeq X_G \bigl( \mathfrak{i}(H) \bigr) = \mathfrak{i} \bigl( \Ap(H) \bigr) \cup \F_G(H)$.
            }
\end{proposition}

%More generally, suppose that $Y\subseteq \Sp(H)$ is a $\N_G(H,\B)$-invariant subposet
%and $r:\Ap(H)\to Y$ is a $\N_G(H,\B)$-homotopy equivalence such that $A\leq r(A)$.
%Then $r$ extends to a $\N_G(H,\B)$-homotopy equivalence:
%\[\B\simeq X_\B(H)\simeq Y\cup \F_\B(H)=: X_\B(Y),\]
%where the order of $X_\B(H)$ is as follows: for $A\in \F_\B(H)$ and $B\in Y$, put $A<B$ if $C_B(A)\neq 1$.

\begin{proof}
Let~$\hat{\iret} : X_{\B}(H) \to \mathfrak{i} \bigl( \Ap(H) \bigr) \cup \F_{\B}(H)$
be the set-map extending~$\iret$ by the identity on~$\F_{\B}(H)$.
We check that it is a poset-map, for the corresponding extension of the ordering~$<$ to~$\prec$:
Recall that~$\iret(B) \geq B$ for all~$B \in \Ap(H)$; and the identity is a poset map on the~$\F_\B(H)$-part of the union.
So it remains to consider cross-terms, which by the $\F$-left property in Definition~\ref{definitionXBHposet}(3)
must be of form~$F \prec A$, for~$F \in \F_\B(H)$ and~$A \in \mathfrak{i} \bigl( \Ap(H) \bigr)$.
Then we have~$\hat{\iret}(F) = F \prec A \leq \iret(A) = \hat{\iret}(A)$.
Therefore using transitivity, $\hat{\iret}$ is order-preserving. 
Further~$\hat{\iret}(A) \geq A$, for all~$A \in X_\B(H)$: that is, we have~$\hat{\iret} \geq \Id_{X_\B(H)}$; 
the equivariant-retraction follows from this (cf.~Lemma~\ref{lm:increndo}(2)).
\end{proof}

Here is one sample use of the above further-shrinking Proposition:

\begin{example}[Constructing~$X_{\Sym_5} \bigl(\mathfrak{i}(\Alt_5) \bigr)$]
\label{exampleA5S5}
Let~$G = \Sym_5$ and~$H = \Alt_5$, with~$p = 2$.
The choice~$\B = \A_2(G)$ satisfies the ``In particular'' hypotheses of Theorem~\ref{theoremInductiveHomotopyType}, 
and hence also of Proposition~\ref{propInductiveIntersectionPoset}---from which
(using normality of~$H$ in~$G$) we obtain $G$-homotopy equivalences:

\centerline{
   $\A_2(G) \simeq X_G \bigl(\mathfrak{i}(H) \bigr) \simeq \mathfrak{i} \bigl( \A_2(H) \bigr) \cup \F_G(H)$.
            }   

\noindent
We can use this equivalent form to explicitly determine the~$G$-homotopy type of~$\A_2(G)$:
Note that~$\A_2(G)$ is obtained from~$\A_2(H)$ by adding~$\binom{5}{2} = 10$ involutions.
That is, $|\F_G(H)| = 10$, and the elements of this set correspond to the involutions of~$G-H$ (that is, the transpositions).
Now the Sylow~$2$-groups of $\Alt_5$ are elementary, so here we get~$\A_2(H) = \mathcal{S}_2(H)$;
and recall from~Definition~\ref{defn:i(X)} that~$\mathfrak{i} \bigl( \mathcal{S}_2(H) \bigr)$
is the poset of~$2$-Sylow intersections in~$H$.
So since the distinct conjugates of the~$5$ Sylow $2$-subgroups of~$\Alt_5$ intersect trivially,
the poset~$\mathfrak{i} \bigl( \A_2(H) \bigr)$ is here just a discrete set with those~$5$ points
(and indeed it coincides with~$\B_2(H)$).
For each~$A \in \F_G(H)$, there are exactly three maximal elementary abelian groups containing some involution commuting with~$A$,
since~$C_H(A) \groupiso \Sym_3$.
Therefore~$X_G \bigl( \mathfrak{i}(H) \bigr)$ is in fact a connected bipartite graph with~$10+5$ points and~$3 \cdot 10 = 30$ edges.
Hence computing with the~$G$-equivalent~$X_G \bigl( \mathfrak{i}(H) \bigr)$ in place of~$\A_2(G)$: 
we see~$\tilde{H}_1 \bigl( \A_2(G) \bigr)$ has dimension~$30 - 15 + 1 = 16$,
and~$\A_2(G)$ is $G$-homotopy equivalent to a bouquet of~$16$~$1$-spheres.

We mention that this graph~$X_G \bigl( \mathfrak{i}(H) \bigr)$ is in fact the ``triples-geometry" for~$\Sym_5$,
as described in~\cite[Example~2.3.5]{Smi11}, which has been used in various places in the geometric literature.

Finally we note that this~$X_G \bigl( \mathfrak{i}(H) \bigr)$ can also be realized 
as~$W_G^{\BoucPoset}(H,1)$ in our later construction in Theorem~\ref{theoremChangeABSPosets}. 
\donerk
\end{example}

Another interesting variant of~$X_G(H)$ uses the context of the poset introduced at~p.134 of~\cite{Thevenaz}; notice it does {\em not\/} force members of~$\F_G(H)$ to the left-end of~$\prec$-chains, but instead leaves them on the right
(that is, the left end of~$\succ$-chains).
We get the relevant analogue of Theorem~\ref{theoremInductiveHomotopyType}:

\begin{proposition}[Th\'{e}venaz poset]
\label{propHomotEquivPosetReverseOrdering}
Let~$H \leq G$.
Define the poset~$\widehat{X}_G(H)$ as the disjoint union of~$\Ap(H)$ and ~$\F_G(H)$,
and with the following ordering~$\prec\ :=\ \prec_{\widehat{X}_G(H)}$:

Inside~$\Ap(H)$, as~$\prec$, we keep the subgroup-inclusion ordering~$<$; 

inside~$\F_G(H)$, as~$\prec$, we use the {\em opposite\/} of~$<$---that is, we use group-{\em containment\/} ordering~$>$;

and when~$F \in \F_G(H)$, we set~$A \prec F$ for all~$A \in \Ap \bigl( C_H(F) \bigr)$.

\noindent
Then~$\Ap(G) \simeq \widehat{X}_G(H)$ (and indeed this is an~$N_G(H)$-equivalence).%
%\footnote{
%  Details for equivariance will be supplied,
%  once Appendix Theorem~\ref{prop:QuilFiberConn} has been correspondingly updated.
%          }
\end{proposition}

\begin{proof}
There will be some limited analogies with arguments from the proof of Theorem~\ref{theoremInductiveHomotopyType}.
Indeed since here we are in effect using the full~$\Ap(G)$ in the role of~``$\B$'' there, the arguments will be easier.
In particular, we mentioned there that its hypothesis~(i) is automatic for~$\Ap(G)$.

\medskip

The proof of transitivity of~$\prec_{\widehat{X}_G(H)}$ here is roughly dual
to that for~$\prec_{X_G(H)}$ in Definition~\ref{definitionXBHposet}---due to the reversal of the usual ordering in~$\F_G(H)$:
As before, we only need to check cross-terms; so suppose $A \prec F$, so that~$A \leq C_H(F)$. 
Given an extension~$B < A$, we have~$B < A \leq C_H(F)$, so that also~$B \prec F$:
and similarly given~$F > E$, we have~$A \leq C_H(F) \leq C_H(E)$, so that also~$A \prec E$. 
And then by the transitivity, $\widehat{X}_G(H)$ is indeed a poset with this order-relation.

\medskip

Next, in order to construct a homotopy equivalence between~$\Ap(G)$ and~$\widehat{X}_G(H)$,
we are going to work with face-posets, that is, the poset of chains from our original posets;
this should help avoid confusion, since we have reversed some orders in the new poset~$\widehat{X}_G(H)$.
Concretely:
We first abbreviate our original poset by~$Y := \Ap(G)$, with the usual ordering~$<$;
and set~$X := \widehat{X}_G(H)$, with the new ordering~$\prec$ above.
The set~$\F := \F_G(H)$ is a subset of both~$Y$ and~$X$---and indeed a subposet of each,
though with opposite orderings~$<$ and~$\prec_{\F_G(H)}$ (namely $>$).
We now consider chains in these posets:
For a~$<$-chain~$a \in Y'$, we can write~$a = (a \cap H) \cup (a \cap \F)$:
where we define~$a \cap H := \{ A \cap H \tq A \in \N_G(H) \cap a \}$ and~$a \cap \F := \{ A \in a \tq A \in \F \}$;
and we recall that the ordering from~$a$ on the set~$a \cap \F$ is given by group-inclusion~$<$.

Now for our homotopy equivalence, define the~map:
  \[ \beta : Y' \to X' , \quad \beta(a) := (a \cap H)\ \cup\ (a \cap \F)_{\prec} \ , \]
where the subscript~$\prec$ indicates reversing the~$<$-ordering on~$a \cap \F$ to~$>$, to get the~$\prec$-ordering.
Notice that this map is essentially the deflation-map,
as in the equivalence-map~$\alpha_{\B,H}$ in the proof of Theorem~\ref{theoremInductiveHomotopyType}.

\medskip

To see that~$\beta$ is an order-preserving map (from~$<$ to~$\prec$):
Since~$\prec$ is either equal, or dual, to~$\prec$ on the two factors in the union defining the domain~$Y'$, 
we only need to check images from cross-terms~$F < A$ between the two factors;
here~$A$ centralizes~$F$ just by abelian-ness, so~$\beta(A) = A \cap H \leq C_H(F)$, 
and hence we get $\beta(A) \prec F$ for the images.
It follows at the level of face-posets that~$\beta(a)$ is indeed a~$\prec$-chain in~$X'$.

\bigskip

Now we will show that~$\beta$ is a homotopy equivalence.
To this end, we will show that the preimage~$\beta^{-1} \bigl( (X_{\preceq x} \bigr)' )$ is contractible for all~$x \in X$.
%(By contrast, in the proof of Theorem~\ref{theoremInductiveHomotopyType}, we considered preimages of $\succeq x$-posets.)
And then in summary, in the language used by McCord in~\cite{McCord},
the~$(X_{\preceq x})'$ as~$x$ varies give a basis-like open cover,
and by Theorem~6 of~\cite{McCord} we will be able to conclude that~$\beta$ is a homotopy equivalence.

We split the contractibility proof into the two usual cases for~$x$.
Because we use lower links (with~$\preceq x$), 
rather than upper links (with~$\succeq x$) as in the proof of Theorem~\ref{theoremInductiveHomotopyType},
the preimages here are slightly different in form:
\begin{itemize}

\item $x = A \in \Ap(H)$: In this case we have:
  \[ X_{\preceq A} = \Ap(H)_{\leq A} = \Ap(A) , \]
so that for the preimage we get:
  \[ \beta^{-1} \bigl( (X_{\preceq A})' \bigr) = Z'  \text{ , where }
      Z := Z_A = \{ B \in \N_G(H) \tq 1 < B \cap H \leq A \}  .  \]
Then~$Z$ is conically contractible, essentially via Quillen's homotopy:
  \[ B  \geq B \cap H \leq A ; \]
the proof is basically dual to that in Example~\ref{ex:zigzagconical}:
Namely we apply Lemma~\ref{lm:increndo}(1) to the corresponding maps $j_0, j_1, j_2$;
they are visibly poset maps, with $j_0 \geq j_1 \leq j_2$. 
For the endomorphism-requirement: 
Note using the description of~$Z$ displayed above that we get~$1 < B \cap H \leq A$;
and then as~$(B \cap H) \cap H = B \cap H$,
we see again using that description of~$Z$ that we also get~$j_1(B) = B \cap H \in Z$ for the middle term~$B \cap H$ above.
This completes the hypotheses of Lemma~\ref{lm:increndo}(1); so by that result, we get contractibility of~$Z = Z_A$.

\item $x = F \in \F$: Here we have:
  \[ X_{\preceq F} = \Ap \bigl( C_H(F) \bigr)\ \cup\ (\F_{<}^{\op})_{\leq F} 
                   = \Ap \bigl( C_H(F) \bigr)\ \cup\ \F_{\geq F} .  \]

\noindent
Therefore for the preimage we get an analogue of~(ii) in the proof of Theorem~\ref{theoremInductiveHomotopyType};
namely we have~$\beta^{-1} \bigl( (X_{\preceq F})' \bigr) = Z'$, where:

  (ii)$^{\prime}$ \quad $Z := Z_F =\ \{ B \in \N_G(H) \tq 1 < B \cap H \leq C_H(F) \}\ \cup\ \F_{\geq F}$.
			 
\noindent
We claim~$Z$ is contractible, via the homotopy zigzag abbreviated by the group-inclusions:
  \[ D \geq C_D(F) \leq C_D(F) F \geq F , \]
as~$D$ runs over~$Z$.
Recall this is the sequence that arose in the case~``$x = F$'' in the proof of Theorem~\ref{theoremInductiveHomotopyType};
and indeed the proof now parallels that earlier one: 
we show using Lemma~\ref{lm:increndo}(1) that~$Z$ is contractible:

Again~$j_0, j_1, j_2, j_3$ are clearly poset maps, with~$j_0 \geq j_1 \leq j_2 \geq j_3$. 
And again~$j_0 = \Id_Z$, and the constant map~$j_3$ to~$F \in Z$, visibly satisfy the endomorphism-requirement. 

So it remains to establish the endomorphism requirement for~$j_1,j_2$; 
that is, we must show~that~$j_1(D) = C_D(F)$ and~$j_2(D) = C_D(F)F$ lie in~$Z$.
This time, since here~$\Ap(G)$ plays the role of~``$\B$'' there, we automatically get step~(a) of the earlier argument; 
so it remains to obtain steps~(b) and~(c) there.
In the subcase for~$Z$ in~(ii$^{\prime}$) where~$D = E \in \F_{\geq F}$,
we again have~$C_E(F) = E = C_E(F)F$---so again the earlier argument goes through.
In the other subcase of~(ii$^{\prime}$), we have~$D = B \in \N_G(H)$ with~$1 < B \cap H \leq C_H(F)$.
Then:
  \[ 1 < B \cap H \leq \bigl( C_B(F) \cap H \bigr) \cap C_H(F) , \]
which shows using~(ii$^{\prime}$) that also~$j_1(B) = C_B(F) \in Z$---giving step~(b).
Step~(c) follows, on replacing~$C_B(F)$ in the above display by~$C_B(F)F = j_2(B)$.
This completes the proof of contractibility of~$Z$ in the case~$x = F$. 
%Note that all the terms of the above homotopy lie in $\F_{\geq F} \cup \{B\in \N_G(H) \tq B\cap H\leq C_H(F)\}$.
\end{itemize}

\medskip

\noindent
We have now shown that for every set~$U \in \mathcal{U} := \{ (X_{\preceq x})' \tq x \in X \}$, $\beta^{-1}(U)$ is contractible.
Here~$\mathcal{U}$ is a basis-like open cover: this means that it is a basis for a topology.
That is, if~$U_1,U_2 \in \mathcal{U}$ and we have~$x \in U_1 \cap U_2$,
then there exists~$U_3 \in \mathcal{U}$ such that~$x \in U_3 \subseteq U_1 \cap U_2$.
To illustrate this situation in our present case:
If~$U_1 = (X_{\preceq B_1})'$, with~$U_2 = (X_{\preceq B_2})'$ and~$x \in U_1 \cap U_2$,
then for the~$\prec$-maximal element of~$X$, we have~$\max(x) \preceq B_1,B_2$,
so that we get~$x \in (X_{\preceq \max(x)})' \subseteq (X_{\preceq B_1})' \cap (X_{\preceq B_2})'$,
and further that~$U_3 :=  (X_{\preceq \max(x)})' \in \mathcal{U}$.

So since we have contractibility for this basis-like cover,
\cite[Theorem 6]{McCord} applies, and we conclude that~$\beta$ is a homotopy equivalence.

Finally, to show that $\beta$ is an $N_G(H)$-homotopy equivalence, note first that $\beta$ is a $N_G(H)$-equivariant poset map.
Second, for $K\leq N_G(H)$, we have that
\[ \mathcal{U}^K:=\{ (X_{\preceq x}^K)' \tq x\in X^K\},\]
is a basis-like open cover of the $K$-fixed subposet $(X')^K$ (notice that $|X'|^K = |(X')^K|$).
Then, the same proof as above shows indeed that the restriction $\beta^K: (Y')^K\to (X')^K$ verifies that $\beta^K( (X_{\preceq x}^K)' )$ is contractible for all $x\in X$: since if $K$ fixes $B,D,F$ as above, then $K$ fixes $B\cap H$, $C_D(F)$ and $C_D(F)F$.
Again, by Theorem 6 of \cite{McCord}, we see that $\beta^K$ is a homotopy equivalence.
Since $K\leq N_G(H)$ is arbitrary, an application of the Equivariant Whitehead's Theorem allows us to conclude that $\beta$ is an $N_G(H)$-homotopy equivalence.
\end{proof}

\subsection*{Equivalences via removal of points with conical centralizers}

In the final part of the section, we turn to a different kind of replacement-poset for~$\Ap(G)$.
As in the later propagation-analysis (cf.~Hypothesis~\ref{hyp:LB}),
our background aim is to analyze the homotopy type of~$\Ap(G)$ with respect to a fixed component~$L$ of~$G$.
This time, however, our equivalent subposet is obtained by using the Quillen fiber-Theorem~\ref{variantQuillenFiber};
to ``homotopically remove'' only certain members~$\N$ of the inflation~$\N_G(L)$ of~$L$
(rather than applying deflation, which was the method of the previous subsection)---namely 
those which are faithful on~$L$, and are suitably associated with ``conical centralizers''. 
Here is some further background from the literature, regarding this distinction:

\begin{remark}[Removing points via conical centralizers]
\label{remarkClosingSection}
We quickly preview a particular feature in our further discussion of homology propagation
in the early part of Section~\ref{sec:prejoinandreplposets}:
Namely we often have the situation in Hypothesis~\ref{hyp:LB},
of a component~$L$ of~$G$, with a~$p$-outer~$B \in \Outposet_G(L)$; 
and we study the resulting~$p'$-central product~$H * K$, where~$H := LB$ and~$K : = C_G(LB)$ as in Hypothesis~\ref{hyp:HKcentralprod}.
If~$\Ap(K) = \emptyset$, then~$\Ap(HK)$ in fact just reproduces~$\Ap(H)$, so that we need no further constructions for~$HK$;
so our development below is really for the case that~$\Ap(K) \neq \emptyset$.% 
%\footnote{
%    "fake footnote":
%    (Decide later)
%    (S) Here I'm trying to implement your recommendation re conification; but am far from certain that I have what you intend.
%    It might be best for you to just fix it, when the tex files  come back to you?
%          }

Now in view of results like the almost-simple case in Theorem~\ref{aschbachersmithQC}, 
we can usually begin with knowledge of nonzero homology~$\tilde{H}_* \bigl( \Ap(H) \bigr)$ for our~$H = LB$. 
However, we need to deal with the possibility that this homology could become zero in~$\Ap(G)$,
from calculating in cones over those cycles---arising from points of~$\Ap(K)$ for our~$K = C_G(LB)$. 
(Instead, we would like to preserve nonzero homology, by having at least double-cones, i.e.~suspensions, over those cycles.)
One obvious special case of this cone-problem arises when~$O_p(K) > 1$:
so that~$\Ap(K)$ is contractible to a point~$*_K$ by Quillen's theorem, 
and then we cannot produce nonzero homology cycle by joining cycles of $\Ap(LB)$ and $\Ap(K)$.
%and this point then gives a common cone-point over all the cycles of~$\Ap(LB)$.
However when~$O_p(K) = 1$, we can hope instead for nonzero cycles for~$\Ap(K)$---giving
nonzero product-homology cycles for~$\Ap(HK)$, which we can then hope to propagate to nonzero homology for~$\Ap(G)$. 
Correspondingly---adapting the terminology of~\cite{AS93}---we
say~$K$ is {\em conical\/} if~$O_p(K) > 1$; and~$K$ is {\em nonconical\/} otherwise, i.e.~if~$O_p(K) = 1$. 

Thus for classical homology propagation,
we need to start (as just indicated) with~$O_p(K) = 1$ and indeed~$\tilde{H}_* \bigl( \Ap(K) \bigr) \neq 0$: 
see e.g.~hypothesis~(2) of Lemma~\ref{lemmaHomologyPropagationAS} (quoted from~\cite{AS93}),
and hypothesis~(v) of Lemma~\ref{lemmaHomologyPropagationKP20} (quoted from~\cite{KP20});
and compare also hypothesis~(3) in our new pre-join propagation Theorem~\ref{theoremNewHomologyPropagation}.
Ideally we can obtain this nonzero homology by induction, when we have the above nonconical-situation~$O_p(K) = 1$.
So we often have a natural case-division, where the more propagation-friendly branch corresponds to:

\smallskip

\begin{flushleft}
\begin{tabular}{lcl}
(someNC) & & \pbox{13cm}{There is some~$B \in \Outposet_G(L)$ with~$O_p \bigl( C_G(LB) \bigr) = 1$.}
\end{tabular}
\end{flushleft}

\smallskip

\noindent
For example, Theorem~2.4 of~\cite{AS93} (existence of ``nonconical complements'')
is a crucial tool, which in effect establishes~(someNC)---allowing for the homology propagation required in the proof.
Furthermore earlier Lemma~\ref{lemmaTrivialOpPropagationCentralizer} in this paper
helps to extend nonconicality to further subgroups.

The other, less propagation-friendly, branch (which we instead might call ``removal-friendly'') of the case-division is: 

\smallskip

\begin{flushleft}
\begin{tabular}{lcl}
(allC) & & \pbox{13cm}{For all~$B \in \Outposet_G(L)$, we have~$O_p \bigl( C_G(LB) \bigr) > 1$.}
\end{tabular}
\end{flushleft}

\smallskip

\noindent
Some version of~(allC) arises in various results in the literature; 
and we could say roughly that it  is more relevant to homotopy-equivalence methods,
typically based on using the Quillen fiber-Theorem~\ref{variantQuillenFiber};
the equivalence is with a subset of~$\Ap(G)$, obtained by removing points corresponding to the conical centralizers.  
Suitable such equivalences are implemented in a number of ways in the literature---cf.~Theorems~4.1, 5.1, and~6.1 of~\cite{KP20}.
And Proposition~\ref{propositionContractibleCentralizers} below
produces such equivalences for situations satisfying an upper-link version of~(allC).
\donerk
\end{remark}

\noindent
In Proposition~\ref{propositionContractibleCentralizers} below, we provide a somewhat-flexible homotopy equivalence result:
where the set~$\F_1$ corresponds to members of~$\Ap(G)$ with the above condition~(allC);
and the conclusion allows for some leeway in choosing the set~$\N$ of points to be removed. 
This equivalence is used in later Proposition~\ref{propHomotEquivTildePosets}, with~$\N = \N_1$ itself;
and in Theorem~\ref{QDpReductionPS}, with~$\N$ a possibly-proper subset of~$\N_1$.

We mention that the first condition~$E \in \Ap \bigl( N_G(L) \bigr)$ in the definition of~$\N_1$ below is actually redundant:
for in view of Lemma~\ref{lemmaNormalizeComponent}, we get~$E$ normalizing the component~$L$
from the second condition~$E \cap L > 1$---once we use the third (``faithful'') condition~$C_E(L) = 1$,
to guarantee that~$E \cap L \nleq Z(L)$.
However we keep that condition about~$N_G(L)$ in the definition,
as it provides some natural context for those later conditions.

\begin{proposition}
\label{propositionContractibleCentralizers}
Let~$G$ be a group, with a component~$L$ of order divisible by the prime~$p$.
Define~$\F_1$ as the subset of~$p$-outers of~$L$ with~$C_G(LB)$ conical:
  \[ \F_1 := \{\ B \in \Outposet_G(L)  \tq O_p \bigl( C_G(LB) \bigr) > 1\  \} , \]
Note that~$\F_1$ is upward-closed in~$\Outposet_G(L)$,
by the contrapositive of Lemma~\ref{lemmaTrivialOpPropagationCentralizer}. 

Define the subset~$\N_1$ of members of the inflation~$\N_G(L)$ acting faithfully on~$L$, and containing some member of~$\F_1$:

$\begin{array}{rcl}
  \N_1 & :=        & \{\ E \in \Ap \bigl( N_G(L) \bigr)\ \tq\ E \cap L > 1 ,\ C_E(L) = 1 ,\ \exists B \in \F_1 \tq B < E\ \} \\
       & \subseteq & \N_G(L) \cap \Ap \bigl( N_G(L) \bigr) .
\end{array}$

\noindent
Note that the condition~$B < E$ guarantees that~$\N_1 \cap \Ap \bigl( L C_G(L) \bigr) = \emptyset$.

Assume that~$\N$ is an upward-closed subset of~$\N_1$.

\noindent
Then~$\Ap(G) - \N \hookrightarrow \Ap(G)$ is a homotopy equivalence.
Furthermore~$\Ap \bigl( L C_G(L) \bigr) \subseteq \Ap(G) - \N$.

In particular, this holds for~$\N = \N_1$: so~$\Ap(G) - \N_1 \hookrightarrow \Ap(G)$ is a homotopy equivalence.
\end{proposition}

\begin{proof}
We abbreviate~$Y := \Ap(G) - \N$.
Note since~$\N_1 \cap \Ap \bigl( L C_G(L) \bigr) = \emptyset$ that~$\Ap\bigl( L C_G(L) \bigr) \subseteq Y$, as required.

We will obtain the desired equivalence using the Quillen fiber-Theorem~\ref{variantQuillenFiber}:
namely we will show for any~$E \in \Ap(G) - Y = \N$ that~$Y_{>E}$ is contractible.

Recall since~$E \in \N \subseteq \N_1$ that we have:

\centerline{
    $E \leq N_G(L)$, $E \cap L >  1$, $C_E(L) = 1$, and there is some~$B \in \F_1$ such that~$B < E$.
            }

\bigskip

\noindent
We will begin by showing that~$Y_{>E}$ is homotopy equivalent to~$\Ap \bigl( C_G(LE) \bigr)$: 

For this, we will directly construct poset maps in each direction, which are homotopy inverses.
So consider any $A$ in the domain-space~$Y_{>E}$.
Then~$A \cap L \geq E \cap L > 1$.
But~$A \cap L \nleq Z(L)$, since~$E \cap Z(L) \leq C_E(L) = 1$; so~$A \leq N_G(L)$ by Lemma~\ref{lemmaNormalizeComponent}.
Further~$B < E < A$.
Thus we have for~$A$ three of the four defining properties for membership in~$\N_1$.
So if we had the remaining faithful-property~$C_A(L) = 1$, we would then get~$A \in \N_1$;
but then since~$E \in \N \subseteq \N_1$ with~$E < A$, and~$\N$ is upward-closed in~$\N_1$ by hypothesis,
we would get~$A \in \N$---contrary to its choice in~$Y = \Ap(G) - \N$.
This contradiction shows that~$C_A(L) > 1$.
Now~$A = C_A(E)$ since~$A > E$; so~$ C_A(LE) = C_A(L) \cap C_A(E) = C_A(L) > 1 $. 
Thus~$a : A \mapsto C_A(LE)$ is a poset map from~$Y_{>E}$ to~$\Ap \bigl( C_G(LE) \bigr)$.
For the other direction, consider now some~$C \in \Ap \bigl( C_G(LE) \bigr)$.
Since~$C_E(L) = 1$, we have~$CE > E$;
and also~$C_{CE}(L) = C > 1$, so that~$CE \not \in \N_1$---and hence~$CE \not \in \N$;
so we conclude that~$CE \in Y_{>E}$.
Then~$b : C \mapsto CE$ is a poset map from~$\Ap \bigl( C_G(LE) \bigr)$ to~$Y_{>E}$.
Note that~$ba(A) = C_A(LE)E \leq A$ and~$ab(C) = C_{CE}(LE) = C$;
so that the two maps are homotopy inverses, giving the desired homotopy equivalence.

Using this equivalence,
the contractibility of~$Y_{>E}$ will follow once we show~$O_p \bigl( C_G(LE) \bigr) > 1$.

\bigskip

\textbf{Case 1.} Assume~$E$ contains an element~$e \in L C_G(L) - L$.
                 Then~$O_p \bigl( C_G(LE) \bigr) > 1$:

\noindent
Here we write~$e = lc$, where~$l \in L$ and~$c \in C_G(L)$.
Since~$e$ has order~$p$, but~$e \notin L$, we see that~$c$ has nontrivial~$p$-power order.
We will show that~$c \in Z \bigl( C_G(LE) \bigr)$; so consider any~$x \in C_G(LE)$:
  \[ lc = e = e^x = l^x c^x = l c^x , \]
which forces~$c = c^x$.
This shows~$c \in Z \bigl( C_G(LE) \bigr)$.
So~$O_p \bigl( C_G(LE) \bigr) \geq O_p(\ Z \bigl( C_G(LE) \bigr)\ ) > 1$,
since we saw~$c$ is a nontrivial $p$-element.
This finishes the proof of Case~1.

\bigskip

\textbf{Case 2.} Otherwise~$E \cap (L C_G(L) - L) = \emptyset$. And again we get~$O_p \bigl( C_G(LE) \bigr) > 1$:

\noindent
Here the hypothesis implies that when~$e \in E$ induces an inner automorphism of~$L$, then~$e \in L$.
And so if we write~$E = (\ E \cap  \bigl( L C_G(L) \bigr)\ )\ C$ where~$C \in \Outposet_G(L)$,
we see that~$E \cap \bigl( L C_G(L) \bigr) = E \cap L$; hence~$LE = LC$.
Moreover since~$B < E$, we can choose~$C$ so that~$C \geq B$.
Since~$B \in \F_1 \subseteq \Outposet_G(L)$, and~$C \in \Outposet_G(L)_{\geq B}$,
we conclude that~$O_p \bigl( C_G(LE) \bigr) = O_p \bigl( C_G(LC) \bigr)) > 1$,
by the upward-closed property of~$\F_1$.
This finishes the proof of Case 2.

\bigskip

In summary:
We have shown that~$Y_{>E} \simeq \Ap(C_G(LE)) \simeq *$, since~$O_p \bigl( C_G(LE) \bigr) > 1$.
To complete the proof, apply the Quillen fiber-Theorem~\ref{variantQuillenFiber}
to the inclusion~$Y \hookrightarrow \Ap(G)$.
\end{proof}

\bigskip
\bigskip

\section{The pre-join construction and some corresponding replacement-posets}
\label{sec:prejoinandreplposets}

In the section following this one, we will prove Theorem~\ref{theoremNewHomologyPropagation}:
which is essentially a variant of Lemma~\ref{lemmaHomologyPropagationAS},
namely the original homology-propagation result~\cite[Lemma~0.27]{AS93} of Aschbacher-Smith---but now using,
in place of the usual join of posets, 
the context of the ``pre-join'', which we develop in this section.

Furthermore, Theorem~\ref{theoremNewHomologyPropagation} is stated in the language just of posets---which
now do not have to be~$\Ap$-posets, in contrast to ``classical'' propagation of~\cite{AS93}.
In our context of~(H-QC), this allows us to use more general replacement-posets~``$W$''
equivalent to~$\Ap(G)$---in the spirit of
the various equivalent-constructions~``$X$'' in the previous Section~\ref{sec:overviewequivApG}.
So the present section in particular develops some replacement-posets which use the pre-join;
notably in Theorems~\ref{theoremChangeABSPosets} and~\ref{propHomotEquivTildePosets}. 

\subsection*{Background: the classical poset join in homology propagation}

Before introducing the pre-join, 
we will first review the more standard context of the usual poset-join in homology propagation.

We start with a brief continuation of our earlier more general viewpoint just on propagation:

\begin{remark}[Some further broad context on homology propagation]
\label{rk:overviewhomolpropag}
Recall from Remark~\ref{rk:contextequivandpropag}
that we typically have some proper subgroup~$H < G$ with $\tilde{H_*} \bigl( \Ap(H) \bigr) \neq 0$;
and we wish to show that this nonzero reduced homology ``propagates'' to nonzero homology also for~$\Ap(G)$;
which means more precisely that the natural inclusion~$i : \Ap(H) \to \Ap(G)$ should induce a nonzero map~$i_*$ in homology.
In that earlier Remark, we focused on the simplest case, where~$i$ is a homotopy equivalence, and~$i_*$ an isomorphism;
but now we wish to examine more general situations for~$i_*$.

\smallskip

Note first that~$i_*$ can be nonzero, without even needing to be injective;
that is, it suffices to have just the (non-constructive) {\em existence\/} of a nonzero image~$i_*$ in homology.
For example, this is the case in Theorem~\ref{PStheorem} quoted from~\cite{PS}.
(We mention also that the knowledge of nonzero homology for~$\Ap(H)$ in that result is also non-constructive: 
it relies on the existential result Theorem~\ref{almostSimpleQC} quoted from \cite{AK90}---applied
to~$H = \Aut_G(L)$, for~$L$ a component of~$G$.)

\smallskip

However, it is in fact more common in the literature to have the propagation%
\footnote{
   And indeed the nonzero homology of~$\Ap(H)$.
          }
be more constructive:
by showing~$i_*$ is essentially a monomorphism in some particular nonzero degree.

For historical background, we illustrate this via Quillen's Corollary~12.2 in~\cite{Qui78}
(which in fact essentially inspired the study of homology propagation):
There we have~$G$ solvable, with~$H := LA$; where in fact~$L := O_{p'}(G)$,
and~$A$ is of maximal rank~$m$ in~$\Ap(H)$ (which is also the maximal rank for~$\Ap(G)$).
In particular, we get the following description of the upper link of~$A$:

  \quad \quad (i) \quad  $ \Ap(G)_{>A}\ =\ \emptyset $ ,

\noindent
just from the particular maximality-condition that we have assumed for~$A$.

Next, the construction there of nonzero reduced homology for~$\Ap(H)$ (basically the property~$\QD_p$)
gives us a nonzero cycle~$\alpha$ in the topological dimension~$m-1$ of~$\Ap(H)$:
this must involve some poset-chain~$a$ with~$m$ members, so we may take our~$A := \max(a)$.
So suppose, by way of contradiction to~(H-QC), that~$\alpha$ is zero in $\tilde{H}_* \bigl( \Ap(G) \bigr)$
(that is, in the kernel of~$i_*$), 
Then~$\alpha$ would appear in the image under the boundary operator of some longer chain~$b$---with $(m+1)$ members from~$\Ap(G)$,
and with~$b \in \Lk_{\Ap(G)}(a)$.

We examine this situation using the later viewpoint of \cite{AS93}:
Note since~$a$ has $m$~members, with~$\max(a) = A$ of rank~$m$,
those members have consecutive orders~$p, p^2, \dots , p^m$;
so  that naively, there is ``no room'' to insert further members of~$\Ap(G)$ either among, or below, those of~$a$---only above.
We state this condition in the language of ``full chain'' at later Definition~\ref{defn:fullandinitialchains}(2):

  \quad \quad (ii) \quad $\Lk_{\Ap(G)}(a)\ \subseteq\ \Ap(G)_{> A}$, where~$A = \max(a)$.
                         (That is, $a$ is a full chain in~$\Ap(G)$.)

\noindent
And now applying~(i) in~(ii), we see that no such longer chain~$b$ is possible. 
This contradiction shows that~$\alpha$ is nonzero for~$\Ap(G)$, giving~(H-QC);
and in particular shows that~$i_*$ is injective in this degree.
(Even in all degrees, since here $\Ap(LA)$ is a wedge of spheres.)

\smallskip
 
The above interaction of the maximality-condition on~$A$ from~$\QD_p$, giving the upper-link condition in~(i), 
with a relevant restriction in~(ii) on the structure of links of simplices, 
is essentially the starting point for later extensions of the notion of this elementary propagation.
So our continuing exposition below will follow this interaction-theme;
up through the analysis in Remark~\ref{rk:QDfullbdrycalc},
which underlies our later main propagation-result Theorem~\ref{theoremNewHomologyPropagation}.
\donerk
\end{remark}

\bigskip

The use of the usual poset-join in propagation originates with Aschbacher-Smith~\cite{AS93},
where the argument indicated above is adapted
to cover the case where Quillen's~$L = O_{p'}(G)$ is replaced by a component~$L$ of~$G$.
The adjustment in particular replaces the single subgroup~``$H$'' above, with a~$p'$-central product~$H * K$;
and then the corresponding poset~$\Ap(H*K)$ is homotopy equivalent to the usual poset-join~$\Ap(H) * \Ap(K)$.  
So as further background for our later pre-join variants,
we now indicate some relevant details in their development:

Thus we  consider a component~$L$ of a counterexample~$G$ to~(H-QC).
We'll need to consider possible~$p$-outers of~$L$;
so we replace~``$A$'' above by some~$B \in \hat{\Outposet}_G(L)$ as in Definition~\ref{outerPosetImagePoset}
(which covers the~$p$-outers~$B \in \Outposet_G(L)$, as well as the identity case~$B=1$).
Hence the basic context for propagation in~\cite{AS93} is the component-hypothesis:

\begin{hypothesis}
\label{hyp:LB}
$L$ is a component of a counterexample~$G$ to~(H-QC), with~$B \in \hat{\Outposet}_G(L)$.
\donerk
\end{hypothesis}

\noindent
Recall that in Remark~\ref{remarkClosingSection},
we had also indicated the relevance, for propagation, of the product of $LB$ with~$C_G(LB)$. 
In Lemma~\ref{lm:HypLBimpliesHypHKcentralprod} below, we will see (postponing momentarily the easy proof)
that the above component-Hypothesis~\ref{hyp:LB}, via the product of~$LB$ with~$C_G(LB)$, 
gives our standard case of a somewhat more general~$\Ap$-poset Hypothesis~\ref{hyp:HKcentralprod}.
This latter product-hypothesis is what is actually assumed (as~\cite[Hyp~0.15]{AS93})
in the propagation lemma~\cite[0.27]{AS93} (which we have quoted as later Lemma~\ref{lemmaHomologyPropagationAS}).

\begin{lemma}
\label{lm:HypLBimpliesHypHKcentralprod}
Assume the component-Hypothesis~\ref{hyp:LB} for~$L,B$.
As~$H$, choose either~$LB$---or possibly just~$L$ when $B>1$;
with~$K := C_G(LB)$ for either choice. 
Then we also have the following central-product Hypothesis~\ref{hyp:HKcentralprod} for~$H,K$:
\end{lemma}

\begin{hypothesis}[$p'$-central product]
\label{hyp:HKcentralprod}
We have~$H,K \leq G$, with~$[H,K] = 1$ and~$H \cap K$ a~$p'$-group.
In particular, we have a~$p'$-central product~$HK = H * K$.
\donerk
\end{hypothesis}

\begin{proof}[Proof of Lemma~\ref{lm:HypLBimpliesHypHKcentralprod}]
Notice~$[H,K] = 1$ from the hypothesis~$K = C_G(LB)$ of the Lemma.

We work first with the ``classical'' choice of~$H = LB$, as in~\cite{AS93}; we have:

\centerline{
   $H \cap K = LB \cap C_G(LB) = Z(LB) \leq Z(L)$, since~$B$ is faithful (including when~$B=1$);
            }

\noindent
and this intersection is a~$p'$-group since~$O_p(L) \leq O_p(G) = 1$ in our counterexample to~(H-QC).

The case of~$H=L$ (even when~$B>1$) is easier,
since then~$H \cap K = L \cap C_G(LB) \leq Z(L)$ which is a $p'$-group as above.
We will use this new case, in our later extensions of the classical case.
\end{proof}

\noindent
In fact in the proof of~\cite[Prop~1.7]{AS93}, we also have~$O_{p'}(G) = 1$---so as in Lemma~\ref{lemmaOpandp}(1),
in this case the product~$HK = H \times K$ is even direct.

\bigskip

Since we will be using such central products within the Cartesian-product context of the pre-join below,
we will frequently make use of the property:
\begin{equation}
\label{eq:HKprojs}
  \text{Under Hypothesis~\ref{hyp:HKcentralprod}, for $A \in \Ap(HK)$, we have $A \leq p_H(A) \times p_K(A)$,}
\end{equation}
where~$p_H$ and~$p_K$ denote the respective projections.

\bigskip

This central-product situation of~$H * K$ leads us to analyze some features of the usual join:

\begin{remark}[Homology propagation using the classical join of~$\Ap$-posets]
\label{rk:homopropviajoinAP}
We continue to study certain aspects of the proof of~\cite[Prop~1.7]{AS93}---the elimination of $\QD_p$-components;
under Hypothesis~\ref{hyp:LB} (and hence~\ref{hyp:HKcentralprod}, in view of Lemma~\ref{lm:HypLBimpliesHypHKcentralprod}).
Recall in that result we assume also that~$H = LB$ has~$\QD_p$. 
We follow the interaction-viewpoint of Remark~\ref{rk:overviewhomolpropag}:

Here~``Step~iii'' at~\cite[p~489]{AS93} shows that we may choose an~$A \in \Ap(LB)$ with~$LA = LB = H$, 
so that~$A$ exhibits~$\QD_p$ for~$H$.
And then with~$\alpha$ and~$a$ as in the Quillen-case in Remark~\ref{rk:overviewhomolpropag}, 
as before we get condition~(ii) of that Remark:

\quad \quad (ii$^{\prime}$) \quad Again~$\Lk_{\Ap(G)}(a) \subseteq \Ap(G)_{>A}$ (that is, $a$ is a full chain in~$\Ap(G)$).

\noindent
However, we can no longer expect to get condition~(i) of Remark~\ref{rk:overviewhomolpropag}, 
that~$\Ap(G)_{>A} = \emptyset$:
for example, consider larger products of~$A$ with any~$C \in \Ap(K)$, where~$K = C_G(LB)$. 
Instead, the argument at~\cite[p~490]{AS93} (referring back to~\cite[p~488]{AS93})
obtains a suitable replacement for~(i):

  \quad \quad (i$^{\prime}$) \quad $ \Ap(G)_{>A}\ \subseteq\ A \times K . $

\noindent
We mention that in the replacement-poset language of earlier Definition ~\ref{definitionXBHposet}, 
with~$HK$ in the role of~``$H$'', this amounts to showing that:%
\footnote{
  Notice the similarity of this condition, with the displayed condition before Definition~\ref{defInflated}(i).
  This starts to link our earlier equivalence methods with propagation-conditions;
  the theme continues, e.g.~at~(i$^{\prime\prime}$) below.
          }

\centerline{
  $\Ap(G)_{> A} \cap \F(HK) = \emptyset$, so that $\Ap(G)_{>A} \subseteq \Ap(H * K) \simeq \Ap(H) * \Ap(K)$ .
            }

\noindent
And this suffices:  
For the original propagation-lemma~\cite[0.27]{AS93} (which we have quoted as Lemma~\ref{lemmaHomologyPropagationAS})
assumes~(i$^{\prime}$) as its first hypothesis.
That lemma shows roughly that~$A * K$-terms in~(i$^{\prime}$) do not obstruct propagation:
Namely if a shuffle product%
\footnote{
  The classical shuffle product (see~\cite[Defn~0.21]{AS93})
  implements product-homology for the join as in Remark~\ref{rk:prodhomolforjoins};
  which amounts to recovering the K\"{u}nneth formula for the join of spaces---cf.~\cite{eilenbergzilber} and~\cite{milnor}.
  %This was first observed by Aschbacher-Smith in \cite{AS93}.
          }
$\alpha \times \beta$ (with~$\beta$ nonzero in~$\tilde{H}_* \bigl( \Ap(K) \bigr)$\thinspace)
is a boundary for~$\Ap(G)$, then the ``global'' boundary-calculation for~$\alpha \times \beta$ from~$\Ap(G)$
reduces via the poset-join $\Ap(H) * \Ap(K)$ to a parallel calculation from~$\Ap(K)$---showing that~$\beta$ a boundary,
contrary to its nonzero choice.
This contradiction establishes~(H-QC).

\bigskip

We mention that Lemma~\ref{lemmaHomologyPropagationKP20} (the quoted result~\cite[Lm~3.14]{KP20})
generalizes the arguments above,
in several ways relevant to our exposition so far:
First, the propagation from~$\Ap(H) * \Ap(K)$ is not necessarily to~$\Ap(G)$, 
but instead to a suitable subposet~$X$---which 
must then be shown to be homotopy equivalent to~$\Ap(G)$ (as in the~$X$-viewpoint in Remark~\ref{rk:contextequivandpropag}).
Further, the sufficient condition for propagation now takes the related form:

  \quad \quad (i$^{\prime\prime}$) \quad  $ X_{>A}\ \subseteq\ \N_G(K) , $

\noindent
in the language of Definition~\ref{defInflated}.
Finally, the assumption of~$\QD_p$ for~$H$ is weakened
essentially to its consequence of~$a$ being a full chain with~$A = \max(a)$;
that is, we have the relevant analogue of Remark~\ref{rk:overviewhomolpropag}(ii):

  \quad \quad (ii$^{\prime\prime}$) \quad $\Lk_X(a) \subseteq X_{>A}$, where~$A = \max(a)$ ($a$ is a full chain in~$X$). 

\noindent
The proof of propagation in this situation then proceeds
via  a similar reduction of the boundary calculation, via the shuffle product for the link, to~$\Ap(K)$.
\donerk
\end{remark}

\subsection*{The pre-join and some basic properties using the Cartesian product}

With the above classical context in place, we now turn to the extended viewpoint of the present paper.
For applications in the next section (cf.~Hypothesis~\ref{hyp:LBforW} there), 
we will start with the basic situation of Hypothesis~\ref{hyp:LB} for a component~$L$,
along with a nontrivial~$p$-outer~$B \in \Outposet_G(L)$---but in Lemma~\ref{lm:HypLBimpliesHypHKcentralprod},
we will now usually make the non-classical choice of~$H := L$ and~$K := C_G(LB)$.
We emphasize that by that Lemma,
{\em we will still get the central-product Hypothesis~\ref{hyp:HKcentralprod} for this choice of~$H,K$\/}---as needed
for the hypothesis of the results given later in this section. 

\bigskip

In fact our main new propagation-result Theorem~\ref{theoremNewHomologyPropagation} will use the pre-join,
in generalizing the classical context:
% replacing old~\ref{theoremHomologyPropagation} by new~\ref{theoremNewHomologyPropagation}
We will replace~$\Ap(G)$ with non-Quillen equivalent posets 
such as~$W_G^{\BoucPoset}(H,K)$ in Theorem~\ref{theoremChangeABSPosets}
(a further development beyond~$X_G(HK)$ in earlier Definition~\ref{definitionXBHposet}).
This will give us more flexibility for propagating nonzero homology,
in the sense that we can correspondingly relax
certain hypotheses that had been required in the classical Lemma~\ref{lemmaHomologyPropagationAS} (i.e.~\cite[0.27]{AS93}).
We can also pursue further analogues
of the interaction of~(i$^{\prime\prime}$) and~(ii$^{\prime\prime}$) in Remark~\ref{rk:homopropviajoinAP}:
using a more natural visual format for simplices and their links---in which
the Cartesian-product form of the pre-join will be more convenient for our calculations than the usual poset-join.
These extensions to propagation will allow us to establish new elimination results, such as Theorem~\ref{theoremLiep}---in
whose proof we apply Theorem~\ref{theoremNewHomologyPropagation},
% replacing old~\ref{theoremHomologyPropagation} by new~\ref{theoremNewHomologyPropagation}
replacing the Quillen-poset for a~$p$-extension of a Lie-type group (in the same characteristic~$p$)
with the homotopy equivalent non-Quillen poset~$W := W_G^{\BoucPoset}(H,K)$ of typically-smaller topological dimension.
Indeed the climax of the argument occurs in the proof of Lemma~\ref{lemmaCharPandFG-typeNC},
%Lemma~\ref{lemmaCharPAndNCFields}, 
where~$a$ is shown to be a full chain in~$W$---i.e.~the analogue of Remark~\ref{rk:homopropviajoinAP}(ii$^{\prime\prime}$).

%In particular, an  ``upward-closed'' hypothesis in Theorem~\ref{theoremHomologyPropagation}
%is used in place of the hypothesis on~$\Ap(G)_{>A}$ in~\cite[0.27]{AS93}, indicated above in Remark~\ref{rk:homopropviajoinAP}:
%since the proof can focus more directly on the technical requirements for not being a boundary.
%(In effect, any candidates for an offending~$B>A$ are shown to ``not matter''
%and hence are removed during the proof of the homotopy equivalence of~$X_G(H)$ with~$\Ap(G)$;
%this removal then provides the upward-closed condition which can now be used for the pre-join hypothesis.)

\bigskip

\noindent
We now begin the main work of the section, primarily adapting results to the pre-join viewpoint:

Inspired by the constructions of~\cite{Qui78} and \cite{AS93} using the poset-join,
we generalize the context by further defining the {\em pre-join\/}
of two general posets $X,Y$---which are of course inspired by the classical $\Ap(H),\Ap(K)$; 
but for us will not always be~$\Ap$-posets. 
The pre-join turns out to be homotopy equivalent to the classical join of the posets;
but the pre-join will be more convenient for our later purposes.

\begin{definition}[The pre-join of posets]
\label{defPreJoin}
Let~$X$ and~$Y$ be finite posets.
Denote by~$C^-X$ the poset obtained from~$X$ by adjoining a minimum element~$1$.
The \textit{pre-join} of~$X$ and~$Y$ is the subposet:
  \[X \ujoin Y := C^-X \times C^-Y - \{ (1,1) \}   \]
of the Cartesian-product poset (cf.~Definition~\ref{defn:joinCartprodofposets}).
Regard~$X$ and~$Y$ as subposets of~$X \ujoin Y$ via~$x \mapsto (x,1)$ and $y \mapsto (1,y)$ respectively.
\donerk
\end{definition}

\noindent
Note in particular that in the format for poset-chains in the pre-join,
any terms from~$X$ itself will appear to the {\em left\/} of ``mixed'' terms---that is, of $(x,y)$ with~$y \neq 1$.
(This continues the theme of left-focusing that we had mentioned in earlier Remark~\ref{rk:XGHlowerlinkformbdry}.)

\bigskip

%In~\cite{Qui78}, it is shown that the pre-join~$X \ujoin Y$
%is homotopy equivalent to the classical join~$X*Y$ of finite posets~$X$ and~$Y$ (cf.~Definition~\ref{defn:joinCartprodofposets}).
%Moreover, this equivalence is~$G$-equivariant if both~$X$ and~$Y$ are~$G$-posets.
The pre-join of posets~$X \ujoin Y$ was used by Quillen to show that the join of simplicial complexes (or posets)
is homotopy equivalent to the classical join of topological spaces of their geometric realizations (see~\cite[p 104]{Qui78}).
We record the~$G$-homotopy equivalence as:

\begin{proposition}
\label{prejoinEquivalence}
If~$X$, $Y$ are~$G$-posets, then we have a~$G$-homotopy equivalence~$\Gamma: X \ujoin Y \to X\join Y$ defined by:
  \[ \Gamma(x,y) :=
       \begin{cases}
         x \in X & \text{ when } y = 1,\\
         y \in Y & \text{ when } y \neq 1.
       \end{cases}
  \]
\end{proposition}

\begin{proof}
This follows by applying (to the preimages of upper links)
the equivariant version of the Quillen fiber-Theorem~\ref{variantQuillenFiber}.
\end{proof}

Next we examine the pre-join in the context of subgroup posets of products:
When~$X$ is a poset of subgroups, regard the minimum element~$1 \in C^ -X$ as the trivial subgroup.
(Of course~$1$ lies outside our usual posets of nontrivial~$p$-subgroups, such as~$\Ap$-posets.)

Then in Quillen's application of the usual poset-join to the~$\Ap$-poset of a product,
we can now also add the pre-join:

\begin{corollary}
\label{joinApPosets}
Assume the~$p'$-central product Hypothesis~\ref{hyp:HKcentralprod} for~$H,K$.

\noindent
Then we have~$N_G(H,K)$-homotopy equivalences:
  \[ \Ap(HK)\  \simeq\  \Ap(H) \ujoin \Ap(K)\  \simeq\ \Ap(H) * \Ap(K) . \]
Here~$\Ap(HK)$ poset-strongly retracts to its subposet of members which are the product of their projections;
and we may in turn identify that subposet with~$\Ap(H) \ujoin \Ap(K)$. 

In particular for a direct product~$G = HK = H \times K$, normality gives~$G$-homotopy equivalences.
\end{corollary}

\begin{proof}
The first homotopy equivalence is given by the co-ordinate map~$a: A \mapsto \bigl( p_H(A) , p_K(A) \bigr)$,
where the projection~$p_H : \Ap(HK) \to C^-\Ap(H)$ is the map obtained from the quotient by~$K$.
Observe, since~$H \cap K$ is a central~$p'$-subgroup of~$H$, that~$\Ap(H) \cong \Ap(H / H \cap K)$.
Define~$p_K$ analogously.
The map~$b$ in the reverse direction is obtained by multiplying the two coordinates;
observe that~$ba(A) = p_H(A) p_K(A) \geq A$ using~(\ref{eq:HKprojs}), 
and~$ab(C,D) = \bigl( p_H(CD) , p_K(CD) \bigr) = (C,D)$, so that the maps are homotopy inverses.
Note using Lemma~\ref{lm:increndo}(2) that from~$ba \geq \Id_{\Ap(HK)}$ above,
we get the retraction conclusion.

The second equivalence is that of Quillen~\cite[Prop~2.6]{Qui78}, in earlier Proposition~\ref{prejoinEquivalence}.
The assertions about the case~$G = H \times K$ are then straightforward.
\end{proof}

We mention that as our development proceeds,
we will be seeing a number of further variations on the projection/product equivalence above.

\subsection*{Some replacement-posets based on the pre-join}

We just gave in Corollary~\ref{joinApPosets}
the relationship of~$\Ap(HK)$, for the central product~$HK$ of Hypothesis~\ref{hyp:HKcentralprod},
with the pre-join~$\Ap(H) \ujoin \Ap(K)$.

We will wish to apply this to the context of the replacement-poset~$X_G(HK)$, as defined in the previous section;
note that from Theorem~\ref{theoremInductiveHomotopyType}, with~$HK$ in the role of~``$H$'', we get:

\centerline{
  $X_G(HK) = \Ap(HK) \cup \F_G(HK)$ is~$N_G(H,K)$-homotopy equivalent to~$\Ap(G)$;
            }   

\noindent
where we will write~$\prec_X$ for the ordering in~$X_G(HK)$.

Next we indicate how to substitute the corresponding pre-join for the term~$\Ap(HK)$ above;
so as to get a replacement-subposet~$W_G(H,K)$ of~$X_G(HK)$, adapted now to our pre-join viewpoint: 

\begin{proposition}
\label{propApGsimeqW0}
Assume the central-product Hypothesis~\ref{hyp:HKcentralprod} for~$H,K$. 
Regard~$\Ap(H) \ujoin \Ap(K)$ as the direct-product subposet of $\Ap(HK)$,
by the identification in Corollary~\ref{joinApPosets}.

Then we may similarly regard as a subposet of~$X_G(HK)$ the following poset:
  \[ W_G(H,K) := \bigl( \Ap(H) \ujoin \Ap(K) \bigr)\ \cup\ \F_G(HK)\ ; \]
whose order-relation~$\prec_W$ is the obvious corresponding re-interpretation of~$\prec_X$:

  (1) Inside~$\Ap(H) \ujoin \Ap(K)$ and~$\F_G(HK)$, $\prec_W$ agrees with $\prec_X$ (i.e.~with inclusion~$<$).

  (2) For~$F \in \F_G(HK)$ and~$(B,C) \in \Ap(H) \ujoin \Ap(K)$, we set~$F \prec_W (B,C)$ iff~$C_{BC}(F) > 1$.%
\footnote{
  This definition begins to illustrate how we exploit the Cartesian-product format of the pre-join;
  we'll be seeing more such definitions as we proceed.
          }

\noindent
Using the retraction in Corollary~\ref{joinApPosets}, we see that $W_G(H,K)$ is $N_G(H,K)$-homotopy equivalent to $\Ap(G)$. 
\end{proposition}

\begin{proof}
Here we are essentially observing that the poset retraction $\Ap(HK)$ to~$\Ap(H) \ujoin \Ap(K)$
in Corollary~\ref{joinApPosets} extends, via the identity on~$\F_G(HK)$, 
to a poset-strong retraction from~$X_G(HK)$ to~$W_G(H,K)$. 
The retraction again proceeds via the projection/product equivalence. 
\end{proof}

We next further modify~$W_G(H,K)$ from Proposition~\ref{propApGsimeqW0}, 
in the spirit of our earlier modification Proposition~\ref{propInductiveIntersectionPoset}
of Theorem~\ref{theoremInductiveHomotopyType}:
this time replacing the term~$\Ap(H)$ in the pre-join by~$\Sp(H)$ or~$\Bp(H)$.
Note that the resulting equivalent replacement-posets are typically not~$\Ap$-posets:
since they include pre-joins such as~$\Sp(H) \ujoin \Ap(K)$,
which do not have a natural~$HK$-interpretation like Corollary~\ref{joinApPosets}---except possibly
via subposets of~$\Sp(H) \ujoin \Sp(K) \simeq \Sp(HK)$.
In fact the new poset~$W_G^{\BoucPoset}(H,K)$ will be the~``$W$'' used
to fit the hypotheses of our new propagation-result Theorem~\ref{theoremNewHomologyPropagation},
in the proof of later Theorem~\ref{theoremLiep}---more specifically
in the branch of that Theorem given by Lemma~\ref{lemmaCharPandFG-typeNC}.
(We had also mentioned, in earlier Example~\ref{exampleA5S5}, 
that the construction there via Proposition~\ref{propInductiveIntersectionPoset}
could also be realized in the language here of~$W_{\Sym_5}^{\BoucPoset}(\Alt_5,1)$.)

\begin{theorem}
\label{theoremChangeABSPosets}
Assume the~$p'$-central product Hypothesis~\ref{hyp:HKcentralprod} for $H,K$.
Consider the posets:
  \[ W_G^{\QuillenPoset}(H,K) = \bigl(\ \Ap(H) \ujoin \Ap(K)\ \bigr)\ \cup\ \F_G(HK) , \]
  \[ W_G^{\BrownPoset}(H,K)   = \bigl(\ \Sp(H) \ujoin \Ap(K)\ \bigr)\ \cup\ \F_G(HK) , \]
  \[ W_G^{\BoucPoset}(H,K)    = \bigl(\ \Bp(H) \ujoin \Ap(K)\ \bigr)\ \cup\ \F_G(HK) . \]
The order-relation~$\prec$ in each of these posets
is the obvious extension of~$\prec_W$ in Proposition~\ref{propApGsimeqW0}:

  (1) Inside~$\Sp(H) \ujoin \Ap(K)$ and~$\F_G(HK)$, as~$\prec$ we keep the ordering induced by inclusion~$<$.

  (2) For~$F \in \F_G(HK)$ and~$(B,C) \in \Sp(H) \ujoin \Ap(K)$, we set~$F \prec (B,C)$ iff~$C_{BC}(F) > 1$.

\noindent
Thus~$W_G^{\QuillenPoset}(H,K)$ and~$W_G^{\BoucPoset}(H,K)$ are subposets of~$W_G^{\BrownPoset}(H,K)$;
and from Proposition~\ref{propApGsimeqW0} we have:

\centerline{
   $W_G^{\QuillenPoset}(H,K) = W_G(H,K) \simeq \Ap(G)$ .
            }

\noindent
Then the three posets displayed above are homotopy equivalent; thus we have:
  \[ \Ap(G) \simeq W_G^{\QuillenPoset}(H,K) \simeq W_G^{\BrownPoset}(H,K) \simeq W_G^{\BoucPoset}(H,K) . \]
\end{theorem}

\begin{proof}
Note that in our construction, we do not get~$W_G^{\BrownPoset}(H,K)$ as a natural subposet of~$X_G(HK)$---though 
we could in fact construct a suitable over-poset of~$X_G(HK)$, which uses~$\Sp$-posets rather than~$\Ap$-posets.
However, here we really just need a poset-inclusion~$i : W_G^{\QuillenPoset}(H,K) \to W_G^{\BrownPoset}(H,K)$.

So we can begin the proof by checking transitivity of our extension of the order-relation~$\prec_W$
from~$W_G^{\QuillenPoset}(H,K) = W_G(H,K)$, to~$\prec$ for~$W_G^{\BrownPoset}(H,K)$.
And here we can just follow the transitivity proof for~$\prec_X$ in Definition~\ref{definitionXBHposet}---since
that argument depends just on standard properties of group-inclusion, and not on elementary-abelianness in~$\Ap$-posets.

Now let~$i$ above and~$j : W_G^{\BoucPoset}(H,K) \to W_G^{\BrownPoset}(H,K)$ be the natural inclusion maps.
To complete the proof of the Theorem, we will show that~$i,j$ are homotopy equivalences.

\vspace{0.2cm}
\textbf{Case~$i$:} $i$ is a homotopy equivalence.
\begin{proof}
We will apply the Quillen fiber-Theorem~\ref{variantQuillenFiber} to preimages for the map~$i$:
namely we will show that for any~$x \in W_G^{\BrownPoset}(H,K) -  W_G^{\QuillenPoset}(H,K)$,
we get contractibility for the corresponding
preimage~$i^{-1} \bigl( W_G^{\BrownPoset}(H,K)_{\preceq x} \bigl) = W_G^{\QuillenPoset}(H,K)_{\prec x}$.
(Notice that this preimage is the left-hand term of the join of form~``$f^{-1}(Y_{\leq y}) * Y_{>y}$''
in our statement of Theorem~\ref{variantQuillenFiber}.)

Such an~$x$ has the form~$(C,D)$, with~$C \in \Sp(H) - \Ap(H)$ (so that~$C$ is not elementary abelian), 
and~$D \in \Ap(K) \cup \{ 1 \}$.
Then:
 \begin{align*}
%   i^{-1}\left( W_G^{\BrownPoset}(H,K)_{\preceq x} \right)
    \text{(i)} \quad \quad W_G^{\QuillenPoset}(H,K)_{\prec (C,D)} 
      & = \bigl( \Ap(H)_{\leq C} \bigr) \ujoin \bigl( \Ap(K)_{\leq D} \bigr)\ \cup\ \F_G(HK)_{\prec (C,D)} \\
      & = \bigl(\ \Ap(C) \ujoin \Ap(D)\ \bigr) \ \cup\ \{ F \in \F_G(HK) \tq C_{CD}(F) > 1 \}  .
 \end{align*}
Now even though~$C$ is not elementary abelian,
it is a standard property of the~$\Ap$-poset in the case of a~$p$-group that~$\Ap(C) \simeq \Sp(C) \simeq *$;
on the other hand, we could get~$\Ap(D) = \emptyset$, when we have~$D = 1$.
But in any case for~$D$, we get~$\Ap(C) \ujoin \Ap(D) \simeq *$.

So it suffices to show
the inclusion~$k : \Ap(C) \ujoin \Ap(D) \hookrightarrow \bigl(\ \Ap(C) \ujoin \Ap(D)\ \bigr)\ \cup\ \F_G(HK)_{\prec (C,D)}$
is a homotopy equivalence---as this gives~$W_G^{\QuillenPoset}(H,K)_{\prec (C,D)} \simeq \Ap(C) \ujoin \Ap(D) \simeq *$,
as desired.
% i^{-1} \left( W_G^{\BrownPoset}(H,K)_{\preceq x} \right)

To establish the equivalence for~$k$, we will again use the Quillen fiber-Theorem~\ref{variantQuillenFiber},
to remove the~$F \in \F_G(HK)_{\prec (C,D)}$ from the union:
by showing that~$\bigl( \Ap(C) \ujoin \Ap(D) \bigr)_{\succ F}$ is contractible.
Recall that such an~$F$ satisfies~$C_{CD}(F) > 1$;
and~$\Ap \bigl( C_{CD}(F) \bigr) \simeq \Sp \bigl( C_{CD}(F) \bigr)$ is contractible.
So it will even suffice to show for each such~$F$ that:

\centerline{
   (ii) \quad \quad $\left( \Ap(C) \ujoin \Ap(D) \right)_{\succ F}$
                    is homotopy equivalent to~$\Ap \bigl( C_{CD}(F) \bigr) \simeq *$.
            }

\noindent
To obtain~(ii):
For any~$(A,B) \in \left( \Ap(C) \ujoin \Ap(D) \right)_{\succ F}$, we get~$1 < C_{AB}(F) \leq C_{CD}(F)$.
Then our desired homotopy equivalence will be given
by a variation on the projection/product equivalence in Corollary~\ref{joinApPosets}:
Here we replace the earlier product-map $b$ with
the centralizer-adjusted variant given by~$b' : (A,B) \mapsto C_{AB}(F) \in \Ap \bigl( C_{CD}(F) \bigr)$;
while the homotopy inverse for~$b'$ will still be given
by the earlier projections-map~$a : E \mapsto \bigl( p_C(E) , p_D(E) \bigr)$ for~$E \in \Ap \bigl( C_{CD}(F) \bigr)$.
For observe using~(\ref{eq:HKprojs}) that we have~$1 < E \leq p_C(E) p_D(E)$,
so that we get~$C_{p_C(E) p_D(E)}(F) > 1$---and hence we obtain that~$F \prec \bigl( p_C(E) , p_D(E) \bigr)$,
giving the image~$a(E) = \bigl( p_C(E) , p_D(E) \bigr) \in \left( \Ap(C) \ujoin \Ap(D) \right)_{\succ F}$. 
Finally 
note~$ab'(A,B) =    ( p_C \bigl( C_{AB}(F) \bigr) , p_D \bigl( C_{AB}(F) \bigr) )
               =    ( p_A \bigl( C_{AB}(F) \bigr) , p_B \bigl( C_{AB}(F) \bigr)
	       \leq (A,B)$;
and for the inverse correspondingly we have~$b'a(E) = C_{p_C(E) p_D(E)}(F) = p_C(E) p_D(E) \geq E$.
These homotopies complete the proof of the contractibility in~(ii),
and hence of the contractibility of~$W_G^{\QuillenPoset}(H,K)_{\prec (C,D)}$ in~(i). 
This in turn establishes the homotopy equivalence needed for Case~$i$.
\end{proof}

\textbf{Case~$j$:} $j$ is a homotopy equivalence.

\begin{proof}
As in Case~$i$, we will apply the Quillen fiber-Theorem~\ref{variantQuillenFiber}:
but this time not to the left-term,
but instead to the right-term of the join of form~``$f^{-1}(Y_{\leq y}) * Y_{>y}$'' in the Theorem:  
That is,
we will show that if~$x \in W_G^{\BrownPoset}(H,K) -  W_G^{\BoucPoset}(H,K)$,
then~$ W_G^{\BrownPoset}(H,K)_{\succ x}$ is contractible.
Such an~$x$ has the form~$x = (C,D)$: where we have~$D \in \Ap(K) \cup \{ 1 \}$,
and~$C \in \Sp(H) - \Bp(H)$ (so that~$C$ is not a~``$p$-radical'' subgroup of~$H$).
We get:
\begin{align*}
  W_G^{\BrownPoset}(H,K)_{\succ x} & = \{ (A,B)\in \Sp(H)\ujoin \Ap(K) \tq (A,B) > (C,D)\} \\
                   & = \bigl( \Sp(H)_{\geq C} \times \Ap(K)_{>D} \bigr)\
		       \cup\ \bigl( \Sp(H)_{>C} \times \left( \Ap(K) \cup \{ 1 \} \right)_{\geq D}\ \bigr)  \\
                   & = \Sp(H)_{>C} \ujoin \Ap(K)_{>D} \simeq * ;
\end{align*}
where the contractibility holds because we have~$\Sp(H)_{>C}$ contractible:
this is a standard property (cf.~\cite[p~152]{Smi11}) of non-$p$-radical subgroups~$C \in \Sp(H) - \Bp(H)$.
Then~$j$ is a homotopy equivalence by the Quillen fiber-Theorem~\ref{variantQuillenFiber}.
This concludes the proof of Case~$j$.
\end{proof}

\medskip

\noindent
And that in turn completes the proof of the equivalences in Theorem~\ref{theoremChangeABSPosets}.
\end{proof}

It turns out that in proving Theorem~\ref{theoremLiep}---in
the~(allC)-branch covered by Lemma~\ref{lemmaCharPAndConicalFields}---we
need to apply the propagation-Theorem~\ref{theoremNewHomologyPropagation}
to a variant of the replacement-poset~$W = W_G^{\BoucPoset}(H,K)$ indicated in Theorem~\ref{theoremChangeABSPosets} above. 
We denote this variant using a tilde, in Proposition~\ref{propTildePosetsXandW} below.

So to finish the section, we now describe a further modification 
of the posets used in Theorems~\ref{theoremInductiveHomotopyType} and~\ref{theoremChangeABSPosets};
recall these are based on the method of ``shrinking a subposet inside a union''.
in the first subsection of Section~\ref{sec:overviewequivApG}.
The modification roughly allows us to take account of the method of removals 
in the second subsection of Section~\ref{sec:overviewequivApG}:
namely the~(allC)-removals indicated in Proposition~\ref{propositionContractibleCentralizers} there.
We won't in fact make any further removals of elements from our poset union;
we work instead at the level of poset-orderings, where we will adjust the earlier~$\prec$ to a new~$\sqsubset$ ;
by in effect removing certain order-{\em relations\/} (i.e.~pairs) from~$\prec$, rather than removing actual elements from the poset.
\footnote{
  However, when we later consider subposets of the form~$Z_{\sqsupseteq x}$,
  there {\em will\/} be corresponding exclusions of elements.
          }
We will thus obtain suitable new replacement-posets, now denoted using a tilde,  which are homotopy equivalent to~$\Ap(G)$.

We will first {\em construct\/} the posets, in Proposition~\ref{propTildePosetsXandW} below;
with the equivalences given afterwards, in Proposition~\ref{propHomotEquivTildePosets}.
Recall that earlier Proposition~\ref{propositionContractibleCentralizers} is for the analysis of a fixed component~$L$ of~$G$.
But in these two upcoming Propositions, we will not actually require~$G$ to be a counterexample to~(H-QC);
and correspondingly, instead of assuming the component-Hypothesis~\ref{hyp:LB} for~$L$ (with~$B = 1$),

\centerline{
  {\em In the remainder of the section, we assume~$Z(L)$ is a~$p'$-group, and set~$H := L C_G(L)$.}
            }

\noindent
This will in effect still provide the~$p'$-central product Hypothesis~\ref{hyp:HKcentralprod} for~$H$;
in particular we still have the analogue of~(\ref{eq:HKprojs}):
\begin{equation}
\label{eq:ApHindirprodofprojs}
  \text{For~$A \in \Ap(H)$, we have~$A \leq p_L(A) \times p_{C_G(L)}(A)$}, 
\end{equation}
where~$p_L$ and~$p_{C_G(L)}$ denote the respective projections.
And we also get from Lemma~\ref{lemmaNormalizeComponent}:

\begin{equation}
\label{eq:ZLp'thenCLA>1soAnormL}
  \text{As~$Z(L)$ is a~$p'$-group, for any~$A \in \Ap(G)$ with~$C_L(A) > 1$ we get~$A \leq N_G(L)$.}
\end{equation}

We will use such properties in our results in the remainder of the section, beginning with:

\begin{proposition}
\label{propTildePosetsXandW}
Assume~$G$ has a component~$L$, with~$Z(L)$ a~$p'$-group. 
Set~$H := L C_G(L)$; so:
  \[ \F_G(H) \cap \Ap \bigl( N_G(L) \bigr) = \Outposet_G(L) . \]
Recall also from Proposition~\ref{propositionContractibleCentralizers} the subposet:
% \[ \F_1 := \{ B \in \Outposet_G(L)  \tq O_p \bigl( C_G(LC) \bigr) > 1 , \forall C \in \Outposet_G(L)_{\geq B} \} , \]
  \[ \F_1 := \{\ B \in \Outposet_G(L)  \tq O_p \bigl( C_G(LB) \bigr) > 1\  \} , \]
which we saw there is upward-closed in~$\Outposet_G(L) \subseteq \F_G(H)$. 

Then the definitions below give transitive order-relations---so the given set-unions afford posets:
\begin{enumerate}

\item $\tilde{X}_G(H) := \Ap(H)\ \cup\ \F_{G}(H)$,
      with order~$\sqsubset_X$ ``mostly'' given by~$\prec_X$ from~$X = X_G(H)$:

        Inside~$\Ap(H)$, and inside~$\F_G(H)$,  $\sqsubset_X$ agrees with~$\prec_X$ (hence with inclusion~$<$).

        For~$A \in \Ap(H)$ and~$F \in \F_G(H)$, $\sqsubset_X$ agrees with~$\prec_X$---{\em except\/}
	we exclude certain pairs:

        For~$F \in \F_1$, we require~$F  \not \sqsubset_X A$, when~$F \prec_W A$ but~$C_A(F) \leq L$.
	We denote this by:

\centerline{
  $F \sqsubset_X A$ iff~$1 < C_A(F)$ ($\nleq L$ when~$F \in \F_1$). 
            }

\item $\tilde{W}^{\BrownPoset}_G \bigl( L , C_G(L) \bigr) := \Sp(L) \ujoin \Ap \bigl( C_G(L) \bigr)\ \cup\ \F_{G}(H)$,

      with order~$\sqsubset_W$ ``mostly'' given by~$\prec_W$ from~$W = W^{\BrownPoset}_G \bigl( L , C_G(L) \bigr)$:

        Inside~$\Sp(L) \ujoin \Ap \bigl( C_G(L) \bigr)$, and inside~$\F_G(H)$,
	$\sqsubset_W$ agrees with~$\prec_W$ (hence with~$<$).

        For~$(A,C) \in \Sp(L) \ujoin \Ap \bigl( C_G(L) \bigr)$ and~$F \in \F_G(H)$,
	$\sqsubset$ agrees with~$\prec_W$---{\em except\/}

	\quad that we exclude certain pairs:

	For~$F \in \F_1$, we require~$F \not \sqsubset_W (A,C)$, when~$F \prec_W (A,C)$ but~$C_{AC}(F) \leq L$.
	We write:

	\centerline{
           $F \sqsubset_W (A,C)$ iff~$1 < C_{AC}(F)$ ($\nleq L$ when~$F \in \F_1$).  
	            }

\item We also define the following subposets of~$\tilde{W}^{\BrownPoset}_G \bigl( L , C_G(L) \bigr)$:
  \[ \tilde{W}^{\BoucPoset}_G \bigl( L , C_G(L) \bigr) := \Bp(L) \ujoin \Ap \bigl( C_G(L) \bigr)\ \cup\ \F_{G}(H) , \]
  \[ \tilde{W}^{\QuillenPoset}_G \bigl( L , C_G(L) \bigr) := \Ap(L) \ujoin \Ap \bigl( C_G(L) \bigr)\ \cup\ \F_{G}(H) . \]
\end{enumerate}

Using the identification of~$\Ap(L) \ujoin \Ap \bigl( C_G(L) \bigr)$ as a subposet of~$\Ap \bigl( L C_G(L) \bigr)$ in Corollary~\ref{joinApPosets},
we may regard~$\tilde{W}^{\QuillenPoset}_G \bigl( L , C_G(L) \bigr)$
as a subposet of~ $\tilde{X}_G \bigl( L C_G(L) \bigr) = \tilde{X}_G(H)$.

Note that each~$\sqsubset$ has the~$\F$-left property:
that members of~$\F_G(H)$ can appear only {\em below\/} members of the left-factor in the relevant union.
\end{proposition}

\begin{proof}
We must show the restricted-relations above
for the sets~$\tilde{X}_G(H)$ in~(1) and~$\tilde{W}_G^{\BrownPoset} \bigl( L , C_G(L) \bigr)$ in~(2)
are indeed transitive order-relations.
Since the new~$\sqsubset$ agrees with our known ordering~$\prec$ within the factors in the union,
as well as on cross-terms between the factors when the~$\F$-term lies outside~$\F_1$, 
the work amounts to checking transitivity of extensions at either end of a cross-term relation
(of the form~``$F \sqsubset A$'' by the~$\F$-left property), 
where the extension involves at least one term from~$\F_1$.

\begin{enumerate}
\item To see transitivity, so that~$\tilde{X}_G(H)$ is a poset:

\noindent
      Assume~$A,B \in \Ap(H)$, with~$A < B$; and~$F,F' \in \F_G(H)$, with~$F < F'$; and also:

          \quad (i) $F' \sqsubset A$. Hence~$1 < C_A(F')$ ($\nleq L$ when~$F' \in \F_1$).

\noindent
      We must show that~$F \sqsubset A$ and~$F' \sqsubset B$; we may assume at least one of~$F,F'$ lies in~$\F_1$.

\noindent
      We begin by adding some easy subgroup-inclusions to~(i):

          \quad (ii) $1 < C_A(F') \leq C_A(F) \cap C_B(F')$.

\noindent
Then we consider the cases where at least one of~$F,F'$ lies in~$\F_1$:

\begin{itemize}
\item $F' \in \F_1$: Here~(i) includes also~$C_A(F') \nleq L$. 
      So from~(ii) we get~$1 < C_A(F) \nleq L$, and similarly ~$1 < C_B(F') \nleq L$;
      and hence~$F \sqsubset A$ and~$F' \sqsubset B$.

\item $F' \notin \F_1$, $F \in \F_1$:
      Here~(ii) gives $C_B(F') >1$, so we get~$F' \sqsubset B$;
      while for~$F \in \F_1$, we similarly get~$C_A(F) >  1$---but we still need to show~$C_A(F) \nleq L$. 
      So assume by way of contradiction that~$C_A(F) \leq L$.
      Then using~(ii), $F'$ centralizes the nontrivial subgroup~$C_A(F') \leq C_A(F) \leq L$,
      so that~$F'$ normalizes~$L$ by~(\ref{eq:ZLp'thenCLA>1soAnormL}). 
      So we obtain that~$F' \in \F_G(H) \cap \Ap \bigl( N_G(L) \bigr) = \Outposet_G(L)$.  
      Now~$F < F'$, and we have~$F \in \F_1 \subseteq \Outposet_G(L)$;
      and we saw in Proposition~\ref{propositionContractibleCentralizers}
      that~$\F_1$ is by its construction upward-closed in~$\Outposet_G(L)$.
      We conclude that~$F' \in \F_1$---contrary to our choice in this case of $F' \notin \F_1$.
      This contradiction shows that we have~$C_A(F) \nleq L$, and so completes the proof that~$F \sqsubset A$
      in this final case.
\end{itemize}

This establishes transitivity of the order-relation; so that~$\tilde{X}_G(H)$ is indeed a poset.

\item To see transitivity, so that~$\tilde{W}_G^{\BrownPoset} \bigl( L , C_G(L) \bigr)$ is a poset:

\noindent
      The proof is parallel to that above, as follows:

\noindent
      We replace~``$A < B$'' above by~$(A,C) < (B,C')$ from the pre-join~$\Sp(L) \ujoin \Ap \bigl( C_G(L) \bigr)$,
      where of course~$C \leq C'$; with members~$F < F'$ of~$\F_G(H)$. 
      Assume the analogue of~(i):

          \quad (i$^{\prime}$) $F' \sqsubset_W (A,C)$. Hence~$1 < C_{AC}(F')$ ($\nleq L$ when~$F' \in \F_1$).

\noindent
      We must show~$F \sqsubset_W (A,C)$ and~$F' \sqsubset_W (B,C')$; we may assume one of~$F,F'$ lies in~$\F_1$.

\noindent
      We get the analogue of~(ii):

          \quad (ii$^{\prime}$) $1 < C_{AC}(F') \leq C_{AC}(F) \cap C_{BC'}(F')$.

\noindent
      And we have a similar case division according to membership of~$F,F'$ in~$\F_1$: 

      In the first case, where~$F' \in \F_1$, (i$^{\prime}$) includes also~$C_{AC}(F') \nleq L$;
      which is what is further needed to complete the proof, using (ii$^{\prime}$),
      that~$F \sqsubset_W (A,C)$ and~$F' \sqsubset_W (B,C')$.

      So the only substantial argument required is for the second case, where~$F \in \F_1$, $F' \notin \F_1$:
      Again~$F' \sqsubset (B,C')$ follows directly from~$1 < C_{BC'}(F')$ in~(ii$^{\prime}$);
      but for~$(A,C)$ we only get~$C_{AC}(F) > 1$---so we still need to show~$C_{AC}(F) \nleq L$.
      So suppose by way of contradiction that~$C_{AC}(F) \leq L$.
      Then~$F'$~centralizes the nontrivial subgroup~$C_{AC}(F') \leq C_{AC}(F) \leq L$,
      so again~$F'$ normalizes~$L$ by~(\ref{eq:ZLp'thenCLA>1soAnormL}),
      so that we get~$F' \in \F_G(H) \cap \Ap \bigl( N_G(L) \bigr) = \Outposet_G(L)$. 
      And again we have~$F < F'$ with~$F \in \F_1 \subseteq \Outposet_G(L)$,
      so the upward-closed property of~$\F_1$ gives~$F' \in \F_1$---contrary to its choice in this case.
      This contradiction shows that~$C_{AC}(F) \nleq L$---which gives~$F' \sqsubset_W (A,C)$,
      to complete the proof for the order-relation;
      so that~$\tilde{W}_G^{\BrownPoset} \bigl( L , C_G(L) \bigr)$ is indeed a poset, as desired.
\end{enumerate}

\noindent
Of course conclusion~(3) is just a definition of subsets giving subposets, and so does not require further proof.
\end{proof}

Now we prove our equivalence-result analogous to Theorem~\ref{theoremChangeABSPosets}, 
showing that the new tilde-posets can be used as replacement-posets for~$\Ap(G)$---notably 
in applying our new propagation-Theorem~\ref{theoremNewHomologyPropagation}
in situations related to Proposition~\ref{propositionContractibleCentralizers};
as we will do in the proof of the branch of Theorem~\ref{theoremLiep} given by Theorem~\ref{lemmaCharPAndConicalFields}.

\begin{proposition}
\label{propHomotEquivTildePosets}
Continue the hypotheses of Proposition~\ref{propTildePosetsXandW}:
Thus~$G$ has a component~$L$, with~$Z(L)$ a~$p'$-group;
and we set~$H := L C_G(L)$ for the~$p'$-central product of these groups, so that:
  \[ \F_G(H) \cap \Ap \bigl( N_G(L) \bigr) = \Outposet_G(L) . \]
Recall also from Proposition~\ref{propositionContractibleCentralizers} the poset:
%  \[ \F_1 := \{ B \in \Outposet_G(L)  \tq O_p \bigl( C_G(LC) \bigr) > 1 , \forall C \in \Outposet_G(L)_{\geq B} \} , \]
  \[ \F_1 := \{\ B \in \Outposet_G(L)  \tq O_p \bigl( C_G(LB) \bigr) > 1\  \} , \]

\noindent
which we saw there is upward-closed in~$\Outposet_G(L)$.

Then we have homotopy equivalences relating the new posets in Proposition~\ref{propTildePosetsXandW}:
  \[ \Ap(G) \simeq \tilde{X}_G(H) \simeq \tilde{W}^{\BrownPoset}_G \bigl( L , C_G(L) \bigr)
                                  \simeq \tilde{W}^{\BoucPoset}_G \bigl( L , C_G(L) \bigr)
                                  \simeq \tilde{W}^{\QuillenPoset}_G \bigl( L , C_G(L) \bigr) . \]
%If in addition $\F_1 = \{B\in \Outposet_G(L) \tq O_p(C_G(LB))\neq 1\}$ contains at most cyclic $p$-outers,
%then $\Ap(G)\simeq \tilde{X}_G(H)$.
%\footnote{
%  ``fake footnote": I would like to get rid of the ``only cyclic $p$-outers", but so far I couldn't find a proof without using that.
%          }
\end{proposition}

\begin{proof}
Recall the further notation of Proposition~\ref{propositionContractibleCentralizers}, including:
  \[ \N_1 = \{ E \in \Ap \bigl( N_G(L) \bigr) \tq E \cap L > 1, C_E(L) = 1, \exists D \in \F_1 \text{ with }  D < E \}\
            \subseteq\ \N_G(L) \subseteq \Ap \bigl( N_G(L) \bigr) , \]
where the last inclusion holds using~(\ref{eq:ZLp'thenCLA>1soAnormL}).
It will be convenient to begin by replacing~$\Ap(G)$ with:

\centerline{
    $Y := \Ap(G) - \N_1$;
    recall by Proposition~\ref{propositionContractibleCentralizers} that~$\Ap(H) \subseteq Y \simeq \Ap(G)$.
            }

\noindent
In fact we will show first that~$Y \simeq \tilde{X}_G(H)$.
% if $\F_1$ contains at most cyclic $p$-outers.
Then we will prove that~$\tilde{X}_G(H) \simeq \tilde{W}_G^{\QuillenPoset}(L,C_G(L))$;
and finally that the 
inclusions~$\tilde{W}_G^{\QuillenPoset} \bigl( L , C_G(L) \bigr) \hookrightarrow \tilde{W}_G^{\BrownPoset} \bigl( L , C_G(L) \bigr)$
and~$\tilde{W}_G^{\BoucPoset} \bigl( L , C_G(L) \bigr) \hookrightarrow \tilde{W}_G^{\BrownPoset} \bigl( L , C_G(L) \bigr)$
are homotopy equivalences.

\vspace{0.2cm}

\textbf{Claim~1.} $Y \simeq \tilde{X}_G(H)$.
% if $\F_1$ contains at most cyclic $p$-outers.

\begin{proof}
%Recall that $\F_1$ is an upward-closed discrete subposet of $\Outposet_G(L)$.
%Since it contains only cyclic $p$-outers, if $D\in \F_1$ and $D'\in \F_G(H)_{>D}$ then $D'\notin \Outposet_G(L)$.
The proof of Claim~1 is somewhat intricate, and will require several pages.

\medskip

We will mimic certain parts of
the proof of Theorem~\ref{theoremInductiveHomotopyType}:
Let~$\tilde{\alpha}:Y \to \tilde{X}_G(H)$ be the analogue of the map $\alpha$ there;
defined by using the deflation map on the~$\N_Y(H)$-term:
\begin{align*}
 A \in \N_Y(H) & \Rightarrow \tilde{\alpha}(A) = A \cap H , \\
 F \in \F_Y(H) & \Rightarrow \tilde{\alpha}(F) = F.
\end{align*}
Note since~$\N_1 \subseteq \N_G(L) \subseteq \N_G(H)$
that~$\N_Y(H) = \N_G(H) - \N_1 \supseteq \Ap(H)$ and~$\F_Y(H) = \F_G(H)$.

\bigskip

\noindent
We first check that~$\tilde{\alpha}$ is order-preserving---from~$<$ to~$\sqsubset$:

The argument begins by paralleling the corresponding segment of the proof of Theorem~\ref{theoremInductiveHomotopyType}:
On pairs~$E < F$ within~$\F_Y(H)$, for the images we get~$E \sqsubset F$, since~$\sqsubset$ coincides with~$<$ here.
On pairs~$A < B$ within~$\N_Y(H)$, for the images we get $A \cap H \sqsubset B \cap H$:
since again~$\sqsubset$ coincides with~$<$ here---and we checked at Lemma~\ref{lemmaReconstructionFromSubposet}(1)
that the deflation map preserves~$<$.
So it remains to consider cross-terms---which by the $\F$-left property~(i) in Definition~\ref{defInflated}
must be of the form~$F < A$, where~$F \in \F_Y(H)$, and~$A \in \N_Y(H)$ (so~$A \cap H > 1$).  
Here~$[F,A] = 1$, which implies that:
  \[ C_{A \cap H}(F) = A \cap H > 1 . \]
When~$F \notin \F_1$, this already gives~$F \sqsubset A \cap H$ for the images under $\tilde{\alpha}$.
So assume~$F \in \F_1$: where it remains to show that~$C_{A \cap H}(F) \nleq L$.
(That is, now we must go beyond the earlier argument in the proof of Theorem~\ref{theoremInductiveHomotopyType}.)
Assume by way of contradiction that~$C_{A \cap H}(F) \leq L$---that is, that~$A \cap H \leq L$;
we will show~$A \in \N_1$:
We see since~$H = L C_G(L)$ that~$A \cap L = A \cap H > 1$;
thus since~$F$ centralizes~$A$, it centralizes~$A \cap L > 1$,
and so~$A$ normalizes the component~$L$ by~(\ref{eq:ZLp'thenCLA>1soAnormL}).
Furthermore:
  \[ C_A(L) = A \cap C_G(L) = (A \cap H) \cap C_G(L) = (A \cap L) \cap C_G(L) \leq Z(L) , \]
which is a~$p'$-group by hypothesis---so that~$C_A(L) = 1$, and hence~$A$ is faithful on~$L$.
Recall~$F < A$ with~$F \in \F_1$;
so we have completed the requirements for~$A \in \N_1$---contrary to~$A \in Y = \Ap(G) - \N_1$.
This contradiction shows that~$C_{A \cap H}(F) \nleq L$ when~$F \in \F_1$;
so we conclude in this final case also that $\tilde{\alpha}(F) = F \sqsubset A \cap H = \tilde{\alpha}(A)$.
Thus we have shown that~$\tilde{\alpha}$ is an order-preserving map.

\bigskip

We will now prove that~$\tilde{\alpha}$ is a homotopy equivalence, by applying the Quillen fiber-Theorem~\ref{variantQuillenFiber}:
This reduces the rest of the proof of Claim~1
to showing for all~$x \in \tilde{X}_G(H)$ that the preimage:
  \[ Z_x := \tilde{\alpha}^{-1}(\tilde{X}_G(H)_{\sqsupseteq x}) \]
is contractible.
Recall that~$\F_Y(H) = \F_G(H)$.

\medskip

The case~$x = A \in \Ap(H)$ is quickly treated;
we can just parallel the case~``$x = A$'' in proof of Theorem~\ref{theoremInductiveHomotopyType}:
Recall that~$\sqsubset$ agrees with group-inclusion~$<$ inside~$\Ap(H)$.
Hence using the~$\F$-left property of~$\sqsubset$ from the definition in Proposition~\ref{propTildePosetsXandW},
we get~$\tilde{X}_G(H)_{\sqsupseteq A} = \Ap(H)_{\sqsupseteq A} = \Ap(H)_{\geq A}$;
so that:
  \[ Z_A = \tilde{\alpha}^{-1}(\tilde{X}_G(H)_{\sqsupseteq A}) = \tilde{\alpha}^{-1} \bigl( \Ap(H)_{\geq A} \bigr)
                                                               = \N_Y(H)_{\geq A} . \] 
And now, as in the proof of Theorem~\ref{theoremInductiveHomotopyType},
contractibility follows from the minimum~$A$. 

\medskip

This reduces the proof of Claim~1 to contractibility of~$Z := Z_x = Z_F$, for the remaining case~$x = F \in \F_Y(H) = \F_G(H)$.

Recall~$\sqsubset$ agrees with~$<$ also on pairs inside~$\F_G(H)$;
and~$\F_G(H) \cap \Ap \bigl( N_G(L) \bigr) = \Outposet_G(L) \supseteq \F_1$. 
Here the parallel with the case~``$x = F$'' in the proof of Theorem~\ref{theoremInductiveHomotopyType}
will require more substantial adaptations:  

For example, here for~$\sqsubset$ we get the exclusion-modified form:
\begin{align*}
  \tilde{X}_G(H)_{\sqsupseteq F} & = \Ap(H)_{\sqsupseteq F} \cup \F_G(H)_{\sqsupseteq F} \\
   & = \{ A \in \N_H \bigl( C_H(F) \bigr) \tq 1 < C_A(F) \text{ ($\nleq L$ when $F \in \F_1$)} \}\
				     \cup\ \F_G(H)_{\geq F} .
\end{align*}
Hence for the preimage we get a correspondingly exclusion-modified form
(for which we choose the numbering~(ii) parallel to that in the proof of Theorem~\ref{theoremInductiveHomotopyType}):

\smallskip

\begin{align*}
   \text{(ii) } \quad \quad Z  &  \phantom{:}= \tilde{\alpha}^{-1}(\tilde{X}_G(H)_{\sqsupseteq F})
                                    = Z_{\N}\  \cup\ \F_G(H)_{\geq F} , \text{ where }  \\
                         Z_\N  &  := \{ B \in \N_Y \bigl( C_H(F) \bigr) \tq 1 < B \cap C_H(F) \text{ ($\nleq L$ when~$F \in \F_1$)} \} . 
\end{align*}

\smallskip

\noindent
Thus for this~$F$, we have~$Z_\N \subseteq \N_Y \bigl( C_H(F) \bigr)$---with equality except possibly when~$F \in \F_1$. 

\medskip

\noindent
As usual, we begin with an easier subcase---namely~$F \nleq N_G(L)$; that is, $F \notin \Outposet_G(L)$:

Then in particular~$F \notin \F_1$;
so for~$Z$ in~(ii) above, we have~$Z_{\N} = \N_Y \bigl( C_H(F) \bigr)$---that is, there are no exclusions required.
Hence~$Z$ is the analogue, with~$Y$ in the role of~``$\B$'',
of the preimage~$Z$ in the case~``$x = F$'' in the proof of Theorem~\ref{theoremInductiveHomotopyType}.
So to continue our theme of adapting that proof:
For any~$D \in Z$, we have the same group-inclusion relations:
  \[ D \geq C_D(F) \leq C_D(F)F \geq F\ . \]
So here, as there, we will again show via Lemma~\ref{lm:increndo}(1) that~$Z$ is contractible to~$\{ F \}$. 
Along the way, we will need to adapt the proof of~(a) there---now with~$Y$ in the role of~``$\B$'':
Again the right-factor case, where~$D = E \in \F_G(H)_{\geq F}$, is no problem:
since as there, we get that both~$C_E(F)$ and~$C_E(F)F$ are equal to~$E \in \F_G(H)_{\geq F} \subseteq Y$.  
But in the left-factor case, where~$D = B \in \N_Y(H)$ so that~$B \cap C_H(F) > 1$, 
we do {\em not\/} try to prove hypothesis~(i) of Theorem~\ref{theoremInductiveHomotopyType},
with~$Y$ in the role of~``$\B$'', for a general $F \in \N_Y(H)$: 
instead, we argue for this particular~$F$---using~$F \nleq N_G(L)$---that we must have~$C_B(F)$ and~$C_B(F)F$ in~$Y$ for~(a):
Consider~$C_B(F)$ first:
Assume by way of contradiction that~$C_B(F) \in \N_1 \subseteq \N_G(L)$.
Then we have~$C_B(F) \cap L > 1$;
so we conclude by~(\ref{eq:ZLp'thenCLA>1soAnormL}) that~$F \leq N_G(L)$---contrary to the choice of~$F$ in this case.
This contradiction shows~$C_B(F) \in Y$, as desired. 
Similarly, assuming $C_B(F)F \in \N_1$ would give~$1 < C_B(F)F \cap L$ centralized by~$F$---so
that~$F$ normalizes~$L$ by~(\ref{eq:ZLp'thenCLA>1soAnormL}), again contrary to the choice of~$F$.
This contradiction gives~$C_B(F)F \in Y$, as desired---completing the proof of~(a).
With~(a) in hand, steps~(b) and~(c) go through as in the proof of Theorem~\ref{theoremInductiveHomotopyType}.
This completes the proof of contractibility of the preimage $Z$,~for this case where~$F \nleq N_G(L)$. 

\medskip

\noindent
We have now reduced the proof of Claim~1 to contractibility of~$Z$ for~$F \leq N_G(L)$. 

For such~$F$, we will show that the preimage~$Z$
(where now we might have~$F \in \F_1$, so that we could get~$Z_\N \subsetneq \N_Y \bigl( C_H(F) \bigr)$ in~(ii)),
we will obtain contractibility via a somewhat more elaborate argument using the general method above,
namely Lemma~\ref{lm:increndo}(1):
That is, we will use the zigzag of homotopies abbreviated by the following longer sequence of group-inclusions (for~$D \in Z$):
\begin{equation}
\label{homotopyYXtilde}
\begin{array}{ll}
 D & \geq (D \cap F) (D \cap H)                                                      \\
   & \leq (D \cap F)\  p_L(D \cap H)\ p_{C_G(L)}(D \cap H)                           \\
   & \geq (D \cap F)\  p_L \bigl( C_{D \cap H}(F) \bigr)\ p_{C_G(L)} \bigl( C_{D \cap H}(F) \bigr)  \\
   & \leq F\ p_L \bigl( C_{D \cap H}(F) \bigr)\ p_{C_G(L)} \bigl( C_{D \cap H}(F) )  \\
   & \geq F                                                                          ,
\end{array}
\end{equation}
Here we use~$j_1, \cdots ,j_5$ to denote the maps
where the image of~$D$ is given in each successive line of~(\ref{homotopyYXtilde}), with~$j_0 := \Id_Z$.
Note that these~$j_i$ are order-preserving,
with~$<$-comparability relations given by~$j_0 \geq j_1 \leq j_2 \geq j_3 \leq j_4 \geq j_5$.

Therefore we have now reduced the proof of Claim~1
to verifying the endomorphism-requirement for that Lemma: namely that each~$j_i(D) \in Z$. 
Of course this holds already for~$j_0 = \Id_Z$, and the constant map~$j_5$ to~$F \in Z$;
so we only need to work with~$j_i$ for~$1 \leq i \leq 4$. 

Since~$F \in \Outposet_G(L)$, so that~$F$ acts on~$L$ and on~$C_G(L)$,
we see that if further~$[F,D]=1$, then also~$[F,p_L(D)] = 1 = [F, p_{C_G(L)}(D)]$.

\bigskip

As usual in such zigzags, the right-factor for~$Z$ in~(ii), where~$D = E \in \F_G(H)_{\geq F}$, is no problem:
For note that~$E \cap F = F$ and~$E \cap H = 1$; so we see that~$j_1(E) = F$. 
Also~$j_i(F) = F$ for all~$i$;
so that for all~$i \geq 1$, we get~$j_i(E) = F \in Z$, as desired.

\medskip

Since we dealt with~$D = E \in \F_G(H)_{\geq F}$ just above, 
we have in fact made yet another reduction:
namely to~$D$ lying in the left-factor for~$Z$ in~(ii):
that is, to $D = B \in Z_\N \subseteq Z$.
So in particular from~(ii) we have for this~$B$ that:

\medskip

 (iii)  $1 < B \cap C_H(F)$ ($\nleq L$ if~$F \in \F_1$).

\noindent
As in the proof of Theorem~\ref{theoremInductiveHomotopyType},
the first part of~(iii) can be viewed in several ways:

  (iii$^{\prime}$) $1 < B \cap C_H(F) = C_{B \cap H}(F) = C_B(F) \cap C_B(H)$. 

\vspace{0.2cm}

\noindent
We begin our treatment of such~$B$ with:

\medskip

\noindent
\textbf{Case~$i = 1$:} We have~$j_1(B) = (B \cap F) (B \cap H) \in Z$.

\begin{proof}
We will show~$(B \cap F)(B \cap H) \in Z_\N$---where of course~$Z_\N \subseteq Z$ in~(ii).

We begin with the case where~$B \cap F = 1$:
Here~$(B \cap F)(B \cap H) = B \cap H \in \Ap(H) \subseteq Y$;
and we see using~(iii/iii$^{\prime}$) that~$1 < B \cap C_H(F) = (B \cap H) \cap C_H(F)$ ($\nleq L$ when $F \in \F_1$)---so
by~(ii) we have~$B \cap H \in Z_\N$---as desired. 

\medskip

\noindent 
We now turn to the case where~$B \cap F > 1$:

We will show first that~$(B \cap F)(B \cap H) \in Y$:
So we assume by way of contradiction that we have~$(B \cap F)(B \cap H) \in \N_1$;
we will show this leads to~$B \in \N_1$.
For we observe, using the definition of~$(B \cap F)(B \cap H) \in \N_1$,
that:~$1 < C_{(B \cap F)(B \cap H)}(F) \leq C_B(F)$;
that~$1 < \left( (B \cap F)(B \cap H)  \right) \cap L = B \cap L$ (since~$L \leq L C_G(L) = H$);
and finally, that there is~$E \in \F_1$ with~$E \leq (B \cap F)(B \cap H) \leq B$.
Thus we have $B \in \N_1$---contrary to the choice of~$B \in Y = \Ap(G) - \N_1$.
This contradiction shows that~$(B \cap F)(B \cap H) \in Y$.

Using~(iii/iii$^{\prime}$) as before now gives:
  \[ 1 < B \cap C_H(F) = (B \cap H) \cap C_H(F) \leq (B \cap F)(B \cap H) \cap C_H(F) \text{ ($\nleq L$ when~$F \notin \F_1$) } , \]
so that by~(ii) we obtain~$(B \cap F)(B \cap H) \in Z_\N$, as desired---completing Case~$i = 1$.
\end{proof}

\vspace{0.2cm}

Thus to complete the proof of the endomorphism-requirement, it now remains to show:

\medskip

\noindent
\textbf{Cases~$i = 2,3,4$:} For~$2 \leq i \leq 4$ we have~$j_i(B) \in Z$.  

\medskip

\noindent
Using~(iii$^{\prime}$) and~(\ref{eq:ApHindirprodofprojs}),
we get a consequence for the~$H$-projection factors in the~$j_i(B)$:

\medskip

  (iv) $1 < C_{B \cap H}(F) \leq p_L \bigl( C_{B \cap H}(F) \bigr)\ \times\ p_{C_G(L)} \big( C_{B \cap H}(F) \bigr)
                            \leq p_L(B \cap H) \times\ p_{C_G(L)}(B \cap H)$ .

\medskip

\noindent
\textbf{Initial Reduction:}  The endomorphism-result for~$2 \leq i \leq 4$ holds: if~$F \in \F_1$; or if~$B \cap C_H(F) \nleq L$.

\begin{proof}
By~(iii), the first if-condition~$F \in \F_1$ implies the second; so we just assume~$B \cap C_H(F) \nleq L$:

\medskip

\noindent
This assumption shows via~(iii$^{\prime}$)
that~$C_{B \cap H}(F)$ does not lie in the~$p_L \bigl( C_{B \cap H}(F) \bigr)$-factor in~(iv)---so:

   (v) Since~$B \cap C_H(F) \nleq L$ here, we have~$p_{C_G(L)} \bigl( C_{B \cap H}(F) \bigr) > 1$. 

\noindent
In particular, $p$-overgroups of~$p_{C_G(L)} \bigl( C_{B \cap H}(F) \bigr)$ (such as the~$j_i(B)$ for~$i = 2,3,4$)
are not faithful on~$L$, and so do not lie in~$\N_1$---that is, they lie in~$Y$.

\medskip

\noindent
To complete the proof that the~$j_i(B) \in Z$, we consider first the case of~$j_4$.
Note using~(v) that:
  \[ j_4(B) \cap C_H(F) =    F\ p_L \bigl( C_{B \cap H}(F) \bigr)\ p_{C_G(L)} \bigl( C_{B \cap H}(F) \bigr)\ \cap\ C_H(F) 
                        \geq p_{C_G(L)} \bigl( C_{B \cap H}(F) \bigr) > 1 ; \]
so that~$j_4(B) \in \N_Y \bigl( C_H(F) \bigr)$---hence $j_4(B) \in Z_\N \subseteq Z$ by~(ii),
since we assume~$B \cap C_H(F) \nleq L$.

Now note that we in fact have~$1 < p_{C_G(L)} \bigl( C_{B \cap H}(F) \bigr) \leq p_{C_G(L)}(B \cap H)$;
and these two projections give the right-hand factors in~$j_3(B)$ and~$j_2(B)$.
Hence we can use this nontriviality, as we used~(v) in the argument above for~$j_4(B) \cap C_H(F)$,
to similarly get~$j_2(B), j_3(B) \in Z_\N \subseteq Z$.
\end{proof}

\noindent
By the Initial Reduction, in the remainder of the proof of Cases~$i = 2,3,4$ we have:

\centerline{
  $F \in \Outposet_G(L) - \F_1$---so that~$Z_\N = \N_Y(C_H(F))$ in~(ii); and further~$1 < C_{B \cap H}(F) \leq L$.
            }

\noindent
Hence~(iv) simplifies to:
\begin{equation}
\label{equalityCentralizers}
  1 < C_{B \cap H}(F) = p_L \bigl( C_{B \cap H}(F) \bigr) = C_{B \cap L}(F) \leq L ,
  \text{ and }  p_{C_G(L)} \bigl( C_{B \cap H}(F) \bigr) = 1 .
\end{equation}
We also get the correspondingly simplified expressions:
  \[ j_3(B) = (B \cap F) C_{B \cap L}(F) \text{ and } j_4(B) = F C_{B \cap L}(F) . \]
At this point,
we can use the nontriviality in~(\ref{equalityCentralizers}) of~$p_L \bigl( C_{B \cap H}(F) \bigr)$ (and hence of~$p_L(B \cap H)$),
as we used~(v) in the argument on~$j_4(B) \cap C_H(F)$ in the proof of the Initial Reduction,
to conclude that~$j_i(B) \in \N_G \bigl( C_H(F) \bigr)$. 

But this time, it will require some further work below, to get~$j_i(B) \in Y$:
which will complete the proof,
by giving~$j_i(B) \in \N_Y \bigl( C_H(F) \bigr) = Z_\N \subseteq Z$ via~(ii), since we are assuming~$F \notin \F_1$.

\medskip

\noindent
So we treat the successive cases for~$j_i(B) \in Y$:

\begin{itemize}
\item To see~$j_2(B) = (B \cap F)\ p_L(B \cap H)\ p_{C_G(L)}(B \cap H)\in Y$:
      If~$p_{C_G(L)}(B \cap H) > 1$, then~$j_2(B)$ is not faithful on~$L$, and so cannot lie in~$\N_1$---so it lies in~$Y$.
      Otherwise~$p_{C_G(L)}(B \cap H) = 1$, so that~$B \cap H = B \cap L = p_L(B \cap H)$.
      Hence~$j_2(B) = (B \cap F)(B \cap L) = (B \cap F) (B \cap H)$.
      And this equals~$j_1(B) \in Z$ using Case~$i = 1$ (so in particular, $j_2(B) \in Y$). 
\item We saw using~(\ref{equalityCentralizers}) that~$j_3(B)$ simplifies to~$(B \cap F)\ C_{B \cap L}(F)$---which
      is a subgroup of~$LF$.
      Assume by way of contradiction that $(B \cap F)\ C_{B \cap L}(F) \in \N_1$. 
      Then there is some~$E \in \F_1$ with~$E \leq (B \cap F)\ C_{B \cap L}(F)$; so in particular, $LE \leq LF$.
      Recall since~$F \in \Outposet_G(L) - \F_1$ that~$O_p \bigl( C_G(LF) \bigr) = 1$.
      So by Lemma~\ref{lemmaTrivialOpPropagationCentralizer}, we also get~$O_p \bigl( C_G(LE) \bigr) = 1$;
      that is, $E \notin \F_1$---contrary to its choice.
      This contradiction shows that~$j_3(B) = (B \cap F)\ C_{B \cap L}(F) \in Y$.
\item To see~$j_4(B) = F\ p_L \bigl( C_{B \cap H}(F) \bigr) \in Y$:
      The same proof as just above, with~$(B \cap F)$ replaced by~$F$, goes through for this case.
\end{itemize}

\noindent
This completes the proof that~$j_i(B) \in Y$ for~$i = 2,3,4$;
and hence (as we had mentioned) that~$j_i(B) \in \N_Y \bigl( C_H(F) \bigr) = Z_\N \subseteq Z$. 

\medskip

This in turn completes the proof for all~$1 \leq i \leq 5$
of the endomorphism-requirement for our zigzag of equivalences in~(\ref{homotopyYXtilde}); 
establishing contractibility of~$Z$ in this final case where~$F \leq N_G(L)$. 

\medskip

Now working back through our long sequence of reductions in the proof:
This case for~$F$ then completes the proof of contractibility of~$Z$ in all cases for~$F$, and hence in all cases for~$D \in Z$;
and thus also completes the proof that~$\tilde{\alpha}$ is a homotopy equivalence.
So we have finally completed the proof of Claim~1.
\end{proof}

\vspace{0.2cm}

\textbf{Claim 2.}~$\tilde{X}_G(H) \simeq \tilde{W}_G^{\QuillenPoset}(L,C_G(L))$.

\begin{proof}
Our equivalence arises essentially by taking the projection/product maps~$a,b$ of Corollary~\ref{joinApPosets}, 
retracting~$\Ap(H)$ to its product-subposet that we identify with~$\Ap(L) \ujoin \Ap \bigl( C_G(L) \bigr)$, 
and extending them via the identity on~$\F_G(H)$.

We define the forward-map~$\tilde{a} : \tilde{X}_G(H) \to \tilde{W}_G^{\QuillenPoset} \bigl( L,C_G(L) \bigr)$ by:
  \[ \tilde{a}(x)
       = \begin{cases}
            \bigl( p_L(A), p_{C_G(L)}(A) \bigr) & x = A\in \Ap(H),\\
            F                                   & x = F\in \F_G(H) .
         \end{cases}
    \]
To see that~$\tilde{a}$ is order-preserving, as usual it suffices to check cross-terms:
Assume~$F \sqsubset_X A$, so that~$1 < C_A(F)$ ($\nleq L$ when~$F \in \F_1$). 
So using~(\ref{eq:ApHindirprodofprojs}) we get:
  \[ 1 < C_A(F) \leq p_L \bigl( C_A(F) \bigr) \times p_{C_G(L)} \bigl( C_A(F) \bigr)
                \leq C_{p_L(A) p_{C_G(L)}(A)}(F) . \]
When~$F \not \in \F_1$, this completes the proof that~$F \sqsubset_W \bigl( p_L(A) , p_{C_G(L)}(A) \bigr)$
for the images under $\tilde{a}$.
So now assume further that~$F \in \F_1$. 
Then we have~$C_A(F) \nleq L$---so that~$C_A(F) \nleq p_L \bigl( C_A(F) \bigr)$, and hence~$p_{C_G(L)} \bigl( C_A(F) \bigr) > 1$. 
This shows that we cannot have~$C_{p_L(A) p_{C_G(L)}(A)}(F) \leq L$ in the displayed containments above.
That is, $C_{p_L(A) p_{C_G(L)}(A)}(F) \nleq L$,
so that~$F \sqsubset_W \bigl( p_L(A) , p_{C_G(L)}(A) \bigr)$ for the images.
Hence $\tilde{a}$ is indeed order-preserving.

For the reverse direction,
note that~$\tilde{W}_G^{\QuillenPoset}(L,C_G(L))$ embeds into~$\tilde{X}_G(H)$
via the map~$\tilde{b}$ such that~$\tilde{b}(A,C) = AC \in \Ap(H)$ for all~$(A,C) \in \Ap(L) \ujoin \Ap(C_G(L))$,
and~$\tilde{b} |_{\F_G(H)} = \Id_{\F_G(H)}$.
To see that~$\tilde{b}$ is order-preserving, assume~$F \sqsubset_W (A,C)$:
then~$1 < C_{AC}(F)$ ($\nleq L$ when~$F \in \F_1$). 
And this is just the definition of~$\sqsubset_X$, in the relation~$F \sqsubset_X AC$ for the images under~$\tilde{b}$, as desired.

Finally:
Note that~$\tilde{a}\tilde{b}$ is the identity on~$\tilde{W}_G^{\QuillenPoset} \bigl( L , C_G(L) \bigr)$.
Further~$\tilde{b}\tilde{a}(A) = p_L(A) p_{C_G(L)}(A) \geq A$ using~(\ref{eq:ApHindirprodofprojs}), 
and~$\tilde{b}\tilde{a}(F) = F$ for~$F \in \F_G(H)$; so that~$\tilde{b} \tilde{a} \geq \Id_{ \tilde{X}_G(H) }$.
Then~$\tilde{a}$ is a homotopy equivalence with homotopy inverse~$\tilde{b}$.
This completes the proof that~$\tilde{X}_G(H) \simeq \tilde{W}_G^{\QuillenPoset} \bigl( L , C_G(L) \bigr)$,
that is, of Claim~2.

In particular, assuming the identification in Corollary~\ref{joinApPosets} 
of the product-subposet of~$\Ap(H)$ with~$\Ap(L) \ujoin \Ap \bigl( C_G(L) \bigr)$, 
we see that our maps retract~$\tilde{X}_G(H)$ to a subposet we may identify with~$\tilde{W}_G^{\QuillenPoset} \bigl( L,C_G(L) \bigr)$
\end{proof}

\textbf{Claim 3.} The maps given by the natural
inclusions~$\tilde{i}:\tilde{W}_G^{\QuillenPoset} \bigl( L , C_G(L) \bigr)
            \hookrightarrow \tilde{W}_G^{\BrownPoset} \bigl( L , C_G(L) \bigr)$
and~$\tilde{j}:\tilde{W}_G^{\BoucPoset} \bigl( L , C_G(L) \bigr) \hookrightarrow \tilde{W}_G^{\BrownPoset} \bigl( L , C_G(L) \bigr)$
are homotopy equivalences.

\begin{proof}
We proceed by making relevant adjustments to the proof of Theorem~\ref{theoremChangeABSPosets}:

\medskip

We first observe that~$\tilde{j}$ is a homotopy equivalence, by the argument given in Case~$j$ there:
For we want to similarly remove
the elements~$x \in \tilde{W}_G^{\BrownPoset} \bigl( L , C_G(L) \bigr)\ -\ \tilde{W}_G^{\BoucPoset} \bigl( L , C_G(L) \bigr)$;
and the subposets~$\tilde{W}_G^{\BrownPoset} \bigl( L , C_G(L) \bigr)_{\sqsupset x}$
do not contain members of~$\F_G \bigl( L C_G(L) \bigr)$.
Hence~$\sqsupset$ agrees with~$\succ$ (that is, with~$>$) on that subposet, 
so the contractibility argument there using~$\prec$ and~$>$ goes through here also. 

\bigskip

So it remains to show that~$\tilde{i}$ is a homotopy equivalence.
We follow the general strategy
for Case~$i$ in the proof of Theorem~\ref{theoremChangeABSPosets}:
We will again show that for~$x \in \Sp(L) \ujoin \Ap \bigl( C_G(L) \bigr)$,
then~$\tilde{W}_G^{\QuillenPoset}(L,C_G(L))_{\sqsubset x}$ is contractible;
so that~$x$ can be homotopically removed from~$\tilde{W}_G^{\BrownPoset} \bigl( L , C_G(L) \bigr)$.
Again such an~$x$ has the form~$(C,D)$, with~$C \in \Sp(H) - \Ap(H)$ (so that~$C$ is not elementary abelian), 
and~$D \in \Ap(K) \cup \{ 1 \}$.
But this time, we must replace~(i) there with an exclusion-modified version
for~$\sqsubset$:

  (i$^{\prime}$) $\tilde{W}_G^{\QuillenPoset} \bigl( L , C_G(L) \bigr)_{\sqsubset(C,D)}$

  \quad \quad \quad \
     $= \bigl(\ \Ap(C) \ujoin \Ap(D)\ \bigr)\
      \cup\ \{ F \in \F_G \bigl( L C_G(L) \bigr) \tq  1 < C_{CD}(F) \text{ ($\nleq L$ when~$F \in \F_1$)}  \} $ .

\noindent
Again it will suffice for contractibility in~(i$^{\prime}$)
to homotopically remove the indicated elements of~$F \in \F_G \bigl( L C_G(L) \bigr)_{\sqsubset (C,D)}$,
by showing that~$\left( \Ap(C) \ujoin \Ap(D) \right)_{\sqsupset F}$ is contractible for each such~$F$. 

\medskip

So assume first that~$F \notin \F_1$.
Then for this~$F$, the exclusions specified by~$C_{CD}(F) \nleq L$ in~(i$^{\prime}$) are not required; 
so that~$\left( \Ap(C) \ujoin \Ap(D) \right)_{\sqsupset F}$ agrees with~$\left( \Ap(C) \ujoin \Ap(D) \right)_{\succ F}$ as in~(i), 
and of course~$\sqsubset$ agrees with~$\prec$ (that is, with~$<$) on this set.
So the argument for contractibility from the proof of Theorem~\ref{theoremChangeABSPosets},
namely via the equivalence in~(ii) there with $\Ap \bigl( C_{CD}(F) \bigr) \simeq *$, 
goes through here also---completing the case~$F \notin F_1$.

\bigskip

Now assume the remaining case where~$F \in \F_1$. 
Here to describe~$\left( \Ap(C) \ujoin \Ap(D) \right)_{\sqsupset F}$,
we now do require the exclusions specified in~(i$^{\prime}$)---namely~$1 < C_{CD}(F) \nleq L$. 
In this situation we have~$C_{CD}(F) \neq p_L \bigl( C_{CD}(F) \bigr) = p_C \bigl( C_{CD}(F) \bigr)$;
which in turn shows that:
  \[ 1 < p_{C_G(L)} \bigl( C_{CD}(F) \bigr) = p_D \bigl( C_{CD}(F) \bigr) \leq C_D(F) . \] 
These exclusions require a corresponding change to the codomain
for the homotopy equivalence in~(ii) of Theorem~\ref{theoremChangeABSPosets}:
This time, we will get our equivalence with the exclusion-modified poset
given by the difference~$\Ap \bigl( C_{CD}(F) \bigr) - \Ap \bigl( C_C(F) \bigr)$.
(And then contractibility of the latter poset
will be the corresponding variant below of the standard contractibility argument for~$\Ap \bigl( C_{CD}(F) \bigr)$.)

In particular, for the equivalence we use essentially the adjusted projection/product maps~$a,b'$
as in the analogous proof in~(ii) for Theorem~\ref{theoremChangeABSPosets};
the only new feature is that we must check that their images really do lie in our exclusion-modified subposets:
First take~$(A,B) \in \left( \Ap(C) \ujoin \Ap(D) \right)_{\sqsupset F}$,
so that we have the restriction~$1 < C_{AB}(F) \nleq L$;
then~$C_{AB}(F) \neq p_L \bigl( C_{AB}(F) \bigr) = p_C \bigl( C_{AB}(F) \bigr)$, so that~$C_{AB}(F) \nleq C$. 
So we conclude that~$b'(A,B) = C_{AB}(F) \in \Ap \bigl( C_{CD}(F) \bigr) - \Ap \bigl( C_C(F) \bigr)$, as required. 
Next consider some~$E \in \Ap \bigl( C_{CD}(F) \bigr) - \Ap \bigl( C_C(F) \bigr)$---here we see
using~(\ref{eq:ApHindirprodofprojs}) that we get~$1 < E \leq p_C(E) p_D(E) = C_{p_C(E) p_D(E)}(F) \nleq L$ since~$E \nleq C_C(F)$. 
Hence~$F \sqsubset \bigl( p_C(E) , p_D(E) \bigr)$,
so that $a(E) = \bigl( p_C(E) , p_D(E) \bigr) \in \left( \Ap(C) \ujoin \Ap(D) \right)_{\sqsupset F}$, as required.
We had already checked the usual poset-equivalence properties of the pair~$a,b'$, so we get the equivalence
of~$\left( \Ap(C) \ujoin \Ap(D) \right)_{\sqsupset F}$ with~$\Ap \bigl( C_{CD}(F) \bigr) - \Ap \bigl( C_C(F) \bigr)$.

The contractibility of the latter poset~$\Ap \bigl( C_{CD}(F) \bigr) - \Ap \bigl( C_C(F) \bigr)$
will follow using Lemma~\ref{lm:increndo}(1) via the standard Quillen conical-contractibility zigzag:
  \[ E \leq\  E \cdot  \Omega_1 Z \bigl( C_{CD}(F) \bigr)\ \geq\ \Omega_1 Z \bigl( C_{CD}(F) \bigr) , \]
once we show the endomorphism-requirement---namely that the two right-hand image-terms
in fact lie in the poset~$\Ap \bigl( C_{CD}(F) \bigr) - \Ap \bigl( C_C(F) \bigr)$.
But~$\Omega_1 Z \bigl( C_{CD}(F) \bigr) \geq \Omega_1 Z \bigl( C_C(F) \bigr) \times \Omega_1 Z \bigl( C_D(F) \bigr)$;
so since we saw above that~$C_D(F) > 1$, we conclude that the two terms indeed do not fall into~$C_C(F)$, as required.
This completes the proof of contractibility of~$\Ap \bigl( C_{CD}(F) \bigr) - \Ap \bigl( C_C(F) \bigr)$, 
and hence also of~$\left( \Ap(C) \ujoin \Ap(D) \right)_{\sqsupset F}$, for~$F \in \F_1$. 

\medskip

We have now established the contractibility of~$\left( \Ap(C) \ujoin \Ap(D) \right)_{\sqsupset F}$,
for all cases of~$F$.
So we conclude that~$\Ap(C) \ujoin \Ap(D) \hookrightarrow \tilde{W}_G^{\QuillenPoset}(L,C_G(L))_{\sqsubset (C,D)}$
is a homotopy equivalence.

Thus~$\tilde{W}_G^{\QuillenPoset} \bigl( L , C_G(L) \bigr)_{\sqsubset (C,D)} \simeq \Ap(C) \ujoin \Ap(D) \simeq *$;
where the latter contractibility follows from the standard contractibility of~$\Ap(C)$,
as in the argument after~(i) in the proof of Theorem~\ref{theoremChangeABSPosets}.
This contractibility for all~$(C,D)$ in turn completes the proof that~$\tilde{i}$ is a homotopy equivalence;
and hence completes the proof of Claim~3.
\end{proof}

And that in turn completes the proof of Proposition~\ref{propHomotEquivTildePosets}.
%\hfill $\Box$
\end{proof}

\bigskip
\bigskip

\section{Using pre-join replacement-posets for homology propagation}
\label{sec:prejoinposetsinHomologyPropagation}

In this section, we will study how to propagate nonzero homology,
using the more general context of the pre-join of posets (which might not be~$\Ap$-posets):
namely propagating (essentially) from a pre-join poset~$X \ujoin Y$, up to a larger poset~$W$.

In proving our new propagation result Theorem~\ref{theoremNewHomologyPropagation}, 
we will in effect give a pre-join generalization of the arguments
leading up to the original propagation result~\cite[0.27]{AS93} of Aschbacher and Smith
(we give a statement of their result as Lemma~\ref{lemmaHomologyPropagationAS}). 
Note that their overall central-product context at~\cite[0.15]{AS93}
is correspondingly generalized at Hypothesis~\ref{hyp:Zleft} below. 
So after that, we will be providing various associated definitions,
which roughly generalize those starting~0.9 and especially~0.19 of~\cite{AS93}.

\bigskip

First we recall some standard features of chain complexes and homology for general posets.
The following constructions work in any commutative ring with unit~$R$; but for simplicity we will work with~$R = \QQ$ or~$\ZZ$.
We also suppress the notation of the coefficient ring~$R$,
and we implicitly assume that all the computations are always carried out in~$R$.
For a finite poset~$X$ and $n \geq -1$, let~$C_n(X)$ denote the~$n$-th chain group of the augmented chain complex of~$X$.
Recall that~$C_n(X)$ is the free~$R$-module generated by the~$n$-chains of~$X$
(these are inclusion-chains of length~$n$---that is, with~$(n+1)$ members).
And~$Z_n(X)$ is the group of~$n$-cycles, the kernel of the boundary map.

\medskip
 
\noindent
We emphasize an important notational distinction:

\centerline{
  {\em For an abstract~$X$ in this section, we write the order-relation as~$<$---used ``unspecifically''\/}:
            }

\noindent
meaning that in later applications, we will specify~$<$---as either standard group-inclusion~$<$;
or else one of our variants from Sections~\ref{sec:overviewequivApG} and \ref{sec:prejoinandreplposets} such as~$\prec$.

We follow~\cite[Sec~0]{AS93} in using the following general terminology:

\begin{definition}[Full and~$a$-initial chains]
\label{defn:fullandinitialchains}
Let~$X$ be a finite poset, and~$a \subseteq X$ an~$n$-chain.
\begin{enumerate}
\item If $ \alpha \in Z_n(X)$ is a nonzero cycle,
      write~$a \in \alpha$ if~$a$ has nonzero coefficient when~$\alpha$ is written in the canonical basis of~$n$-chains.
      If in addition~$\alpha$ is nonzero in~$\tilde{H}_n(X)$,
      we say that~$a$ (or its right-end member~$\max(a)$) \textit{exhibits} homology for~$X$.

\item We say that~$a$ is a \textit{full} chain in~$X$, if for every~$x \in X$ such that~$a \cup \{ x \}$ is a longer chain,
      we have~$x > \max(a)$; that is, $\Lk_X(a) \subseteq X_{> \max(a)}$. 
      (That is, there is ``no room'' for extending~$a$ either below or within it.)

\item For~$a \neq \emptyset$, a chain~$b \in X'$ containing~$a$ is called an \textit{$a$-initial chain}
      if for every~$x \in b-a$, we have that $x> \max(a)$.  (That is, $b$ ``begins with'' $a$.)

\end{enumerate}
In particular, if~$a$ is a full chain, then every chain~$b$ containing~$a$ is~$a$-initial.

Furthermore~$a$ is a {\em maximal\/} chain in~$X$ if it cannot be lengthened---that is, $\Lk_X(a) = \emptyset$.
In particular, a maximal chain is full.  
\donerk
\end{definition}

Our further development below will continue the theme around the previous Remark \ref{rk:homopropviajoinAP}, 
of propagation-calculations which exploit a full chain~$a$ as in Definition~\ref{defn:fullandinitialchains}(2) above, 
in interaction with restrictions arising from~$\QD$-type conditions.

\bigskip

%formerly, \subsection{New homology propagation lemma}

Next, in the spirit of earlier Hypothesis~\ref{hyp:HKcentralprod}---which provided the more general group-theoretic context
for the original propagation in Lemma~\ref{lemmaHomologyPropagationAS} (i.e.~\cite[0.27]{AS93})---we give
in Hypothesis~\ref{hyp:Zleft} below a more general poset-context
for our new propagation in Theorem~\ref{theoremNewHomologyPropagation}.

Here is a brief, informal preview:
We will propagate up to a poset~$W$ as in Theorem~\ref{theoremChangeABSPosets}---where this~$W$
is the union of a pre-join with a related~$\F$-poset, for~$\F$ in the spirit of Definition~\ref{defInflated}. 
Now the naive analogue of the classical case would begin propagation just at that pre-join.
However, we saw after Remark~\ref{rk:homopropviajoinAP} that we wish to make the non-classical choice of~$H := L$ rather than~$LB$;
so in order to include treatment of the~$p$-outer~$B$ in the~$p$-extension~$LB$,
we will actually begin propagation at the union of the pre-join with~$\Outposet_{LB}(L)$---which is
a subposet of the~$\F$-poset in~$W$.
(Indeed in Hypothesis~\ref{hyp:LBforW} below, we see that~$\Outposet_{LB}(L)$  will be the model
for the more abstract~``$Z$'' in our upcoming Definition~\ref{hyp:Zleft}.)

\bigskip

\noindent
We now flesh out the above overview more formally.

First, here is our augmented version Hypothesis~\ref{hyp:LBforW} of the earlier component-hypothesis~\ref{hyp:LB}:

\begin{hypothesis}
\label{hyp:LBforW}
Assume Hypothesis~\ref{hyp:LB} for~$L,B$.
Choose~$K := C_G(LB)$ as in the classical case.

But now take~$H := L$ only (even though we will sometimes have~$B > 1$).
Recall from Lemma~\ref{lm:HypLBimpliesHypHKcentralprod} that we still get Hypothesis~\ref{hyp:HKcentralprod}---so that
we can apply the later results from Section~\ref{sec:prejoinandreplposets}.
 
In particular we can choose our replacement-poset as in Theorem~\ref{theoremChangeABSPosets}:

\centerline{
   $W := W_G^{\BoucPoset}(H,K)    = \Bp(H) \ujoin \Ap(K)\ \cup\ \F_G(HK) = X \ujoin Y\ \cup\ \F_G(HK)$;  
             }

\noindent
where we have used the abbreviations~$X := \Bp(H) = \Bp(L)$ and~$Y := \Ap(K) = \Ap \bigl( C_G(LB) \bigr)$.
Finally set~$Z := \Outposet_{LB}(L)$. 

Here we specify the abstract order-relation~``$<$'' as the relation~$\prec$ in the definition of~$W$.
\donerk
\end{hypothesis}

Next---mimicking the sequence after earlier Hypothesis~\ref{hyp:LB},  and again postponing briefly the proof---we see
in part~(b) of Lemma~\ref{lm:LBforWgivesZleft} below
that the above hypothesis is a case of a somewhat more general poset-Hypothesis~\ref{hyp:Zleft},
which is what is actually assumed in our new propagation result Theorem~\ref{theoremNewHomologyPropagation}.
We also add a further part~(a)---which we use in later Corollary~\ref{cor:newpropagextextends0.27AS},
in checking that Lemma~\ref{lemmaHomologyPropagationAS} (i.e.~our quoted statement of ``old'' Lemma~0.27 of~\cite{AS93})
is in fact a particular case of our new-Theorem~\ref{theoremNewHomologyPropagation}.

\begin{lemma}
\label{lm:LBforWgivesZleft}
Assume we have either:

  (a) Hypothesis~\ref{hyp:HKcentralprod} for~$H,K$---with~$X := \Ap(H)$, $Y := \Ap(K)$, $Z := \emptyset$,
      and with~$W$ given

      \quad \quad by~$\Ap(G)$ under group-inclusion~$<$;  or:

  (b) Hypothesis~\ref{hyp:LBforW} above for~$L,B$ (including the definitions of~$H,K,X,Y,Z$, and~$W$ under~$\prec$).

\noindent
Then we also get Hypothesis~\ref{hyp:Zleft} below for~$X,Y,Z,W$. 
\end{lemma}

\begin{hypothesis}[The poset-context~(Zleft)]
\label{hyp:Zleft}
\vspace{0.2cm}
For our new propagation results, we will work in the following poset-context;
with the order-relation denoted unspecifically by~$<$:

\begin{tabular}{ll}
(Zleft) &
\begin{tabular}{rl}
(i)   & We have posets~$X,Y,Z$ (disjoint) and~$W$, with~$(X \ujoin Y) \cup Z \subseteq W$ . \\
(ii)  & If~$(x,y) \in X \ujoin Y$ and~$z \in Z$ are $<$-comparable in~$W$, then~$z \in W_{<(x,y)}$ . \\
(iii) & If~$(x,y) \in X \ujoin Y$ with~$y \neq 1$, then~$Z \subseteq W_{<(x,y)}$ .
%(iii) & If~$(x,y) \in X \ujoin Y$, with~$y \neq 1$, and~$z \in Z$, then~$z < (x,y)$ .
\end{tabular}
\end{tabular}

\noindent
We identify~$X$ and~$Y$ as subposets of~$W$ via~$X\times 1$ and~$1 \times Y$, respectively.
\donerk
\end{hypothesis}

\noindent
For propagation in the context of the (Zleft)-Hypothesis~\ref{hyp:Zleft},
we will be showing how to construct homology cycles for~$W$, from homology cycles of~$X \cup Z$ and~$Y$.
But first, we had postponed:

\begin{proof}[Proof of Lemma~\ref{lm:LBforWgivesZleft}]
First consider the~$X,Y,Z,W$ chosen in part~(a) of the Lemma:
Here we have~$X \ujoin Y = \Ap(H) \ujoin \Ap(K)$, which as in Corollary~\ref{joinApPosets} can be identified
with the subposet of~$\Ap(HK)$ whose elements are of form~$C \times D$, where~$(C,D) \in \Ap(H) \ujoin \Ap(K)$;
these~$\Ap$-posets are disjoint, since~$H \cap K$ is a~$p'$-group by Hypothesis~\ref{hyp:HKcentralprod}.
And since here we also have~$Z = \emptyset$, we now easily get the desired containment for~(i) of~(Zleft):

\centerline{
   $X \ujoin Y \cup Z = X \ujoin Y = \Ap(H) \ujoin \Ap(K) \subseteq \Ap(HK) \subseteq \Ap(G) = W$, as chosen in~(a).
            }

\noindent
And again using~$Z = \emptyset$, we get conditions~(ii) and~(iii) vacuously.
This completes~(a).

We turn to the posets~$X,Y,Z,W$ chosen in~(b), i.e.~in Hypothesis~\ref{hyp:LBforW}---where we recall~$H=L$ and~$K=C_G(LB)$. 
Here we observe that:

\centerline{
     $Z = \Outposet_{LB}(L) \subseteq \F_{G} \bigl( L C_G(LB) \bigr) = \F_G(HK)$:
            }

\noindent
since given a member~$C \in \Outposet_{LB}(L)$, we have~$C \cap L C_G(LB) \leq C \cap L C_G(L) = 1$,
for~$C$ a~$p$-outer on~$L$.
Further~$X,Y,Z$ are disjoint---since we have~$H \cap K$ a $p'$-group just as in part~(a);
while members of~$\F_G(HK)$ have trivial intersection with~$HK$.
Hence~$X \ujoin Y \cup Z \subseteq X \ujoin Y \cup \F_G(HK) = W$, as chosen in Hypothesis~\ref{hyp:LBforW}, giving~(i).
Now just as we had emphasized in Remark~\ref{rk:XGHlowerlinkformbdry} on the earlier~$X_G$-construction, 
the order-relation~$\prec$ in the construction of our present choice of:

\centerline{
   $W = W_G^{\BoucPoset}(H,K)  = \Bp(H) \ujoin \Ap(K)\ \cup\ \F_G(HK)$
            }

\noindent
shows elements~$z \in Z \subseteq \F_G(HK)$ 
appear only {\em below\/} elements of~$\Bp(H) \ujoin \Ap(K) = X \ujoin Y$---giving~(ii).
Finally, the~$p$-outer~$B \in Z = \Outposet_{LB}(L) \leq LB$ commutes with all of the members of~$\Ap \bigl( C_G(LB) \bigr) = Y$;
so we see that for any~$(C,D)$ with~$D > 1$ as assumed in~(iii),
we in fact get~$1 < D = C_{D}(B) \leq C_{CD}(B)$---so that~$B \prec (C,D)$ in the ordering on~$W$, as desired.
\end{proof}

We mention that conditions~(ii) and~(iii) for~$Z$, in the poset-context in our (Zleft)-Hypothesis~\ref{hyp:Zleft} above,
will tie in with our various~$\F$-left conditions in Sections~\ref{sec:overviewequivApG} and~\ref{sec:prejoinandreplposets};
and hence with the left-focusing theme of earlier Remark~\ref{rk:XGHlowerlinkformbdry};
this is of course the reason for the nickname~``(Zleft)''.
To continue with that viewpoint:

\begin{remark}[Simplex format under Hypothesis~\ref{hyp:LBforW}, for boundary calculations under~(Zleft)] 
\label{rk:QDfullbdrycalc}
We are now in a position to expand on our earlier preview of left-focused format in Remark~\ref{rk:XGHlowerlinkformbdry};
here is a brief (and still somewhat informal) overview of corresponding aspects in the upcoming arguments:

We will mainly describe applications of propagation based on~$LB$ under Hypothesis~\ref{hyp:LBforW};
and we had mentioned in introducing that Hypothesis that we expect~$B > 1$.
However, we will essentially avoid dealing with the full poset of elementary subgroups~$\Ap(LB)$ of~$LB$---by
working instead with the smaller poset~$\Bp(L) \cup \Outposet_{LB}(L) = X \cup Z$;
here we are using the language of the~(Zleft)-Hypothesis~\ref{hyp:Zleft}, as we may by Lemma~\ref{lm:LBforWgivesZleft}.
For this smaller poset, we will have a suitable maximal-dimension homology condition analogous to~$\QD_p$;
cf.~later Theorem~\ref{theoremHomologyCharP}.%
%\footnote{
%  fakenote:
%  Now using latest version of Bouc-dimension here.
%          }
As usual, we then study a chain~$a$---involved in a nonzero cycle~$\alpha$ in that appropriate maximal dimension.
Recall that by definition of the ordering in~$W_G^{\BoucPoset} \bigl( L , C_G(LB) \bigr) = W$,
we get the format:

  \quad  (i) A chain~$a$ of~$X \cup Z$
             has the format~$a = ( B_1 < B_2 < \cdots < B_r \prec A_1 < A_2 < \cdots < A_s ) $;

\noindent
where~$B_i \in \Outposet_{LB}(L) = Z$ and~$A_i \in \Bp(L) = X$.
In particular, the~$p$-outer terms~$B_i$ all appear at the {\em left\/} end of the sequence
(prior to the occurrence of the order-relation~$\prec$, which may differ from group-inclusion~$<$).

Next, just as in the classical case in earlier Remark~\ref{rk:homopropviajoinAP},
we must in fact consider products: namely of~$\alpha$, with~$\beta$ arising from~$Y = \Ap \bigl( C_G(LB) \bigr)$;
see Lemma~\ref{lemmaPropertiesInitialShuffleProduct} below for the appropriate generalization~$T_*(\alpha \otimes \beta)$
of the classical shuffle product.
We will wish to apply propagation to these products, to see they remain nonzero in the homology of~$W$, hence establishing~(H-QC).

Note that these usual~$\alpha$ and~$\beta$ appear in hypotheses~(1), (2)(a), and~(3) 
of our new propagation result Theorem~\ref{theoremNewHomologyPropagation}.
So in order to apply that propagation,
results assuming Hypothesis~\ref{hyp:LBforW} will mainly focus on establishing the other hypothesis~(2)(b);
note that this is the relevant analogue of~(i$^{\prime}$) and~(ii$^{\prime}$),
in the interaction-viewpoint of Remark~\ref{rk:homopropviajoinAP}.
The condition in~(2)(b) primarily involves eliminating the possibility
of any member of~$\F_G \bigl(L C_G(LB) \bigr)$ lying in~$\Lk_M(a)$ there;
note that this analysis takes place in the left-hand part of~(i) above---before~$A_1$, hence in the vicinity of the~$B_i$.  
After that, it remains to similarly eliminate members of~$(\ \Bp(L) \ujoin \Ap \bigl( C_G(LB) \bigr)\ )_{<A_s}$, 
and establish the precise structure in~(2)(b) for members above~$A_s$.
Note that this latter analysis instead takes place in the right-hand part of~(i) above---first
before or among the~$A_i$, and then after~$A_s$.
These remarks illustrate the ``visual'' convenience (mentioned earlier in the paper)
of the format in~(i), arising from the definition of the ordering.

Now the argument so far has basically reduced~$\Lk_W(a)$ to members of~$X \ujoin Y$ having the special form given in~(2)(b).
Applying the propagation result Theorem~\ref{theoremNewHomologyPropagation} will show in effect 
that these members are {\em not\/} an obstruction to propagation---by reducing to a boundary calculation just in~$Y$.
We had made a similar in-effect comment at Remark~\ref{rk:homopropviajoinAP},
in describing~0.27 of~\cite{AS93}(i.e.~ our quoted result Lemma~\ref{lemmaHomologyPropagationAS}); 
and indeed the new propagation proof below is essentially a pre-join version of that earlier proof.
\donerk
\end{remark}

\noindent
Now we begin the details for our pre-join variant of homology propagation:

The next few pages, through Lemma~\ref{lemmaPropertiesInitialShuffleProduct},
require the (Zleft)-Hypothesis~\ref{hyp:Zleft} only for~$X,Y,Z$;
we will not involve~$W$ until Theorem~\ref{theoremNewHomologyPropagation}.
More precisely, recall that we regard~$X$ and~$Y$ as subposets of~$W$, via~$X \times 1$ and~$1 \times Y$, respectively.
So before propagating to~$W$, we will first in Lemma~\ref{lemmaPropertiesInitialShuffleProduct} below
construct product-homology cycles for~$(X \ujoin Y) \cup Z$, from homology cycles of~$X \cup Z$ and~$Y$.

\smallskip
\smallskip

We begin by introducing some notions which will allow us to generalize the classical shuffle product;
recall that definition was  given at~\cite[p~483]{AS93}, and later extended in Section~3 of~\cite{KP20}.

Let~$a$ be a chain of~$X \cup Z$.
Then~$a = a_Z \cup a_X$, where~$a_Z = a \cap Z$, $a_X = a \cap X$.
Note by condition~(Zleft)(ii) that for all~$z \in a_Z$, and~$x \in a_X$, we get~$z < (x,1)$,
where of course we have identified~$(x,1)$ in the pre-join~$X \ujoin Y$ with~$x \in X$.
That is, (Zleft)(ii) in particular essentially reproduces the $\F$-left property in Hypothesis~\ref{hyp:LBforW};
so we still have the simplex-format of Remark~\ref{rk:QDfullbdrycalc}(i):

  \quad \quad (i$^{\prime}$) \quad  A chain~$a$ of~$X \cup Z$
                                    has the form~$a = a_Z \cup a_X = ( z_1 < \cdots < z_r < x_1 < \cdots < x_s ) $;

\noindent
where~$z_i \in Z$ and~$x_j \in X$. 

\begin{definition}[The notation~$a_X * b$, and classical shuffles]
\label{defn:ax*bandshuffles}
Now take a chain~$b \subseteq Y$.
Following~\cite[0.19]{AS93}, we write:

\centerline{
   $a_X * b := $ the chain~``$a_X < b$'', obtained by placing~$a_X$ to the left end of the chain~$b$.  
            }

\noindent
More formally, $a_X * b$ must be a~$<$-chain in the pre-join~$X \ujoin Y$;
so our discussion here is a variant of that in~\cite{AS93}, which takes place instead in the context of~$\Ap(HK)$. 
Namely we first ``naively'' regard~$a_X * b$ just as a sequence, from the {\em set\/} $X \cup Y$---and then 
we specify the requirements for the needed {\em poset\/}-relations,
in terms of the Cartesian-product coordinates of the pre-join~$X \ujoin Y$.
Note that the chain~$a_X * b$ is~$a_X$-initial, in the sense of Definition~\ref{defn:fullandinitialchains}(3).

As a simple example:
Suppose we have~$a = (z_1 < z_2 < x_1 < x_2)$ and~$b = (y_1 < y_2)$; 
then from the naive-sequence~$(x_1,x_2,y_1,y_2)$, 
we get the~$<$-chain~$a_X * b = \bigl(\ (x_1,1) < (x_2,1) < (x_2,y_1) < (x_2,y_2)\ \bigr)$ in the pre-join. 
In particular, the terms~$y_i \in b$ in the sequence for~$a_X * b$
are appearing in the pre-join element~$a_X * b$---but {\em not\/} in the form~$(1,y_i)$
(from our standard identification of~$Y$ with a subset of the pre-join)---they appear instead
in the two right-hand terms, in the form needed for the poset-relations in the pre-join.

\medskip

\noindent
This sequence-viewpoint also makes it easier to discuss the classical shuffle product:

Recall, essentially as in~\cite[p~483]{AS93}, that a {\em shuffle\/}
is a set-permutation~$\sigma$ of the sequence for~$a_X * b$,
whose image preserves the original order within the two subsets~$a_X$ and~$b$.
Thus it can ``shuffle'' members of~$b$ to the left, among those of~$a_X$
(and in particular, the result need not be~$a_X$-initial). 
For instance, if we apply the shuffle~$\sigma = (1 3)(2 4)$ to the sequence for~$a_X * b$ in the example above, 
we get the shuffled-sequence~$(y_1,y_2,x_1,x_2)$---resulting
in the pre-join element given by~$\bigl(\ (1,y_1) < (1,y_2) < (x_1,y_2)  < (x_2,y_2)\ \bigr)$. 

We will write~$(a_X \times b)_{\sigma}$ for the pre-join element
corresponding to the shuffled-sequence~$\sigma(a_X * b)$; and then the classical {\em shuffle product\/} of~$a_X$ and~$b$ is defined as the alternating sum of those elements:
  \[ a_X \times b := \sum_\sigma\ (-1)^\sigma (a_X \times b)_\sigma , \]
where~$\sigma$ runs through all the shuffle-permutations.
\donerk
\end{definition}

\noindent
We now examine a special feature of the shuffling process above---making crucial use of our~$Z$-left Hypothesis~\ref{hyp:Zleft}:
Consider some member~$(x,y)$ of~$(a_X * b)$ for some~$\sigma$.
First if~$y=1$, then~$x \in a_X$ (with no member of~$b$ shuffled before it);
and we see using~(Zleft)(ii) (compare~(i$^{\prime}$) above)
that each~$z \in a_Z$ has~$z < (x,1) = (x,y)$.
Otherwise~$y>1$: and here by~(Zleft)(iii), we see that each~$z \in a_Z$ has~$z < (x,y)$. 
Thus we obtain the fundamental property:

\centerline{
  Every element of~$a_Z$ is $<$-below every element of~$(a_X \times b)_\sigma$.
            }

\noindent
Therefore we can write~$a_Z \cup (a_X \times b)_\sigma$,
for the unique~$<$-chain obtained by adding the elements of~$a_Z$ to the bottom of the chain~$(a_X \times b)_\sigma$.
(This ability to add-before illustrates another visually-natural feature of the format in~(i$^{\prime}$) above.)
We can extend this notation by linearity, and define the operation:
  \[ (a,b) \mapsto \sum_\sigma\ (-1)^\sigma \left( a_Z \cup (a_X \times b)_\sigma \right) =: a_Z \cup (a_X \times b) . \]
We emphasize a notational point here:
our definition above is on {\em pairs~$(a,b)$ of chains\/} taken from~$(X \cup Z)$ and~$Y$; 
we are {\em not\/} claiming there is any subposet of form~$(X \cup Z) \ujoin Y$---the important point here
is that the {\em image\/} of our map is in fact a chain, in the target-poset~$(X \ujoin Y) \cup Z$.

\begin{definition}
\label{defInitialShuffleProduct}
Let~$a,b$ be chains in~$X \cup Z$ and~$Y$ as above, under the (Zleft)-Hypothesis~\ref{hyp:Zleft} for~$X,Y,Z$.
Then the \textit{initial shuffle product~$T$} of~$a$ and~$b$ (i.e.~of the pair~$(a,b)$ in~$(X \cup Z)' \times Y'$) is:

  \hfill $T(a,b) := a_Z \cup (a_X \times b)$ . \donerk
\end{definition}

\noindent
In Lemma~\ref{lemmaPropertiesInitialShuffleProduct} below, 
we record pre-join versions of the basic facts (cf.~\cite[pp 484--485]{AS93}) about the classical shuffle product---for use
with our generalized version of the product in Definition~\ref{defInitialShuffleProduct}.

To this end, we first introduce the following useful notation:
if~$a$ denotes a chain of a poset~$X$, which we write as a sequence~$a = (x_0 < \ldots < x_r)$, then for~$0 \leq i \leq r$,
we set:

\centerline{
  $a^{\wedge i} := a - \{x_i\} = (x_0 < \ldots < \widehat{x_i}< \ldots < x_r)$.
            }

\noindent
We let~$|a| = r+1$ denote the size of~$a$ as a set;
recall that this~$a$ has length~$r$ as a poset-chain---and defines an~$r$-chain, in the chain complex~$C_*(X)$.

We will also use the following straightforward extension to complexes 
of the notion of full chain in earlier Definition~\ref{defn:fullandinitialchains}(2):

\begin{definition}
\label{defn:fullchaininsubcomplex}
Let~$X$ be a finite poset.
For a subcomplex~$M \subseteq \K(X)$, and a chain~$a \in X'$ such that~$a \in M$,
we say that~$a$ is a {\em full chain in~$M$\/} if~$\Lk_M(a) \subseteq \K \bigl( X_{>\max(a)} \bigr)$.%
%\footnote{
%  "fake footnote":
%  I changed to lower-case "max(a)", as I've been using consistently (AFTER sec2)...
%          }
\end{definition}

We now recollect the basic facts needed for our generalization to the initial shuffle product:

\begin{lemma}
\label{lemmaPropertiesInitialShuffleProduct}
Let~$X,Y,Z$ be as in the~(Zleft)-Hypothesis~\ref{hyp:Zleft},
and let~$a,b$ denote chains of~$X \cup Z$ and~$Y$, respectively.
Write~$a_Z := a \cap Z$ and~$a_X := a \cap X$.
\begin{enumerate}
\item If $a_Z = \emptyset$, then~$T(a,b) = a \times b$ is the classical shuffle product.
\item $T$ extends by linearity to a map of chain complexes:
   \[ T_*: C_m(X \cup Z) \otimes C_n(Y) \to C_{m+n+1}( (X \ujoin Y) \cup Z) . \]
\item In particular, if~$\alpha$ is a cycle of~$X \cup Z$, and~$\beta$ is a cycle of~$Y$,
      then~$T_*(\alpha \otimes \beta)$ is a cycle of~$(X \ujoin Y) \cup Z$.
\item The $a$-initial part of~$T(a,b)$ is~$T(a,b)_a = a_Z \cup (a_X * b)$.  (Recall Definition~\ref{defn:ax*bandshuffles} for~$*$.)
\item If~$V$ is a finite poset, and~$c$ is a full chain of a subcomplex~$M \subseteq \K(V)$,
      then when~$\gamma \in C_t(M)$ for some~$t$, we have~$(\partial \gamma)_{c} = (\partial(\gamma_{c}))_{c}$.
\end{enumerate}
\end{lemma}

\begin{proof}
Item~(1) is clear from the definition of the initial shuffle product~$T$.

\smallskip 
\smallskip

We next show item~(2):
We need to prove that~$T_*$ commutes with the boundary operator; 
that is, that~$T_* \bigl( \partial(a \otimes b) \bigr) = \partial \bigl( T_*(a\otimes b) \bigr)$,
for every~$m$-chain~$a$ of~$X \cup Z$, and~$n$-chain~$b$ of~$Y$.

We first compute~$T_* \bigl( \partial(a \otimes b) \bigr)$.
Using the graded Leibniz product rule, we have: 

  \quad \quad $ T_* \bigl(\ \partial(a \otimes b)\ \bigr) 
                   = T_*\left(\ (\partial a) \otimes b\ +\ (-1)^{m}\ a \otimes (\partial b)\ \right)$. 

\noindent 
Decomposing~$(\partial a) \otimes b$ into its~$Z$- and~$X$-sums, we get:

  $= T_* \left(   \sum_{i=0}^{|a_Z|-1}\ (-1)^i\ (a_Z\cup a_X)^{\wedge_i} \otimes b\
               +\ \sum_{i=|a_Z|}^{m}\ (-1)^i\ (a_Z \cup a_X)^{\wedge_i} \otimes b\ 
               +\ (-1)^m\ a \otimes (\partial b) \right) .  $
%\end{align*}

\noindent
Moving~$T_*$ inside the sums by linearity, and re-indexing the second for the~$a_X$-ordering, we have:
\begin{align*}
   & = \sum_{i=0}^{|a_Z|-1}\  (-1)^i\ T_* \bigl( (a_Z^{\wedge_i} \cup a_X) \otimes b)\
            +\ (-1)^{|a_Z|}\ \sum_{i=0}^{m-|a_Z|}\ (-1)^i\ T_* \bigl( (a_Z \cup a_X^{\wedge_i}) \otimes b)\ \\
   & \quad \quad \quad \quad \quad
            +\ (-1)^{m}\ T_* \bigl( (a_Z \cup a_X) \otimes (\partial b) \bigr).
\end{align*}
Recall from Definition~\ref{defInitialShuffleProduct} that we defined~$T_*$
on  pairs~$(a,b)$ of chains from~$X \cup Z$ and~$Y$;
so identifying~$a \otimes b$ with the pair~$(a,b)$, and applying that definition (which has the effect of shifting relevant parentheses), we have:
\begin{align*}
        & = \sum_{i=0}^{|a_Z|-1}\ (-1)^i\ \bigl( a_Z^{\wedge_i} \cup (a_X\times b) \bigr)\
                         +\ (-1)^{|a_Z|}\ \sum_{i=0}^{m-|a_Z|}\ (-1)^i\ \bigl( a_Z \cup (a_X^{\wedge_i} \times b) \bigr)\ \\
	&   \quad \quad  \quad \quad +\ (-1)^{m}\ \bigl( a_Z \cup (a_X \times (\partial b) \bigr) . \\
\end{align*}
Now combining the last two summands, while extracting the common factor~$a_Z$, yields:
\begin{align*}
        & = \sum_{i=0}^{|a_Z|-1}\ (-1)^i\  \bigl( a_Z^{\wedge_i} \cup (a_X \times b) \bigr) \\
        &   \quad \quad \quad \quad +\ (-1)^{|a_Z|}\ a_Z\ \bigcup\ \left( \sum_{i=0}^{m-|a_Z|}\ (-1)^i\ (a_X^{\wedge_i} \times b)\
   	                +\ (-1)^{m - |a_Z|}\ \bigl( a_X \times (\partial b) \bigr) \right) \\
        & = \sum_{i=0}^{|a_Z|-1}\ (-1)^i\ \bigl( a_Z^{\wedge_i} \cup (a_X \times b) \bigr)\
   	                +\ (-1)^{|a_Z|}\ a_Z\ \bigcup\ \left( \partial(a_X \times b) \right) ;
\end{align*}
where the last equality holds by another application of the product rule
to $\partial$---since~$a_X \otimes b \mapsto a_X \times b$ defines a chain map
from~$C_{m-|a_Z|}(X) \otimes C_n(Y)$ to~$C_{m-|a_Z| + n +1}(X \ujoin Y)$ just using the classical shuffle product,
and we have linearity of the union-operation with the chain~$a_Z$ of~$Z$.

The above expression for~$T_* \bigl( \partial(a \otimes b) \bigr)$
will now turn out to be equal to our computation below of~$\partial \bigl (T_*(a \otimes b) \bigr)$:
\begin{align*}
   & = \partial \left( a_Z \cup (a_X \times b) \right) \\
   & = \sum_{i=0}^{|a_Z|-1} (-1)^i\ \left( a_Z^{\wedge_i}\ \cup\ (a_X \times b) \right)\
       +\ \sum_\sigma (-1)^\sigma\ \sum_{i=|a_Z|}^{m+n+1} (-1)^i\ (a_Z \cup \left( a_X \times b \right)_\sigma)^{\wedge i} \\
   & = \sum_{i=0}^{|a_Z|-1}\ (-1)^i\ \bigl( a_Z^{\wedge_i}\ \cup\ (a_X \times b) \bigr)\
       +\ (-1)^{|a_Z|}\ \sum_\sigma\ (-1)^\sigma\ \sum_{i=0}^{m-|a_Z|+n+1}\ (-1)^i\ a_Z \cup ((a_X\times b)_\sigma)^{\wedge i} \\
   & = \sum_{i=0}^{|a_Z|-1}\ (-1)^i\ \bigl( a_Z^{\wedge_i}\ \cup\ (a_X \times b) \bigr)\
       +\ (-1)^{|a_Z|}\ a_Z \bigcup\ \bigl[\ \sum_\sigma (-1)^\sigma\
                                             \sum_{i=0}^{m-|a_Z|+n+1}\ (-1)^i\ \bigl(( a_X \times b)_\sigma)^{\wedge i}\ \bigr] \\
   & = \sum_{i=0}^{|a_Z|-1}\ (-1)^i\ \bigl( a_Z^{\wedge_i}\ \cup\ (a_X\times b) \bigr)\
       +\ (-1)^{|a_Z|}\ a_Z\ \bigcup\ \left( \partial (a_X \times b) \right) . \\
\end{align*}
This is exactly the expression that we obtained above for~$T_* \bigl( \partial(a \otimes b) \bigr)$;
completing the proof of item~(2).
Note also that item~(3) follows from item~(2).

Item~(4) follows from the same property of the classical shuffle product since:
  \[ T(a,b)_a = \bigl( a_Z \cup (a_X \times b) \bigr)_a = a_Z \cup (a_X \times b)_{a_X} = a_Z \cup (a_X * b) . \]
Item~(5) holds in the context of general posets; cf.~\cite[Lm~0.24]{AS93}.
\end{proof}

\noindent
Now we state and prove our pre-join version of the Homology Propagation Lemma~\cite[0.27]{AS93}
(for comparison, we will state that earlier result afterward, as Lemma~\ref{lemmaHomologyPropagationAS}):

\begin{theorem}
\label{theoremNewHomologyPropagation}
Let~$X,Y,Z,W$ be as in the~(Zleft)-Hypothesis~\ref{hyp:Zleft}.
Assume~$M$ is some intermediate simplicial complex, namely one satisfying:
  \[ \K \bigl( (X \ujoin Y) \cup Z \bigr) \subseteq M \subseteq \K(W) . \]
Suppose that we further have the following:
\begin{enumerate}
\item A nonzero-homology cycle~$\alpha \in Z_m(X \cup Z)$, $m \geq 0$.
\item A chain~$a \in \alpha$ such that:
\begin{enumerate}[label=(\alph*)]
\item The coefficient~$q$ of~$a$ in~$\alpha$ is invertible,
\item $\Lk_M(a) \subseteq \K \bigl( \{ \max(a) \} \times Y)$.
      In particular, $a$ is a maximal chain of~$X \cup Z$.
%\item $a$ is a full chain in $W$,
%\item $W_{>\Max(a)} \subseteq \{\Max(a)\}\times Y$.
\end{enumerate}
\item A nonzero-homology cycle~$\beta \in Z_n(Y)$, $n \geq -1$.
\end{enumerate}
Then~$T_*(\alpha \otimes\beta)$ defines a nonzero-homology cycle in~$Z_{m+n+1}(M)$.
\end{theorem}

\begin{proof}
We mimic the original proof of~\cite[Lemma~0.27]{AS93}---i.e.~our quoted result Lemma~\ref{lemmaHomologyPropagationAS} below. 
(But compare also the proof of~\cite[Lemma~3.14]{KP20}---i.e.~our quoted result Lemma~\ref{lemmaHomologyPropagationKP20}.) 

By Lemma~\ref{lemmaPropertiesInitialShuffleProduct}(3),
$T_*(\alpha \otimes \beta)$ is a cycle of~$(X \ujoin Y) \cup Z$.
So since~$\K \bigl( (X \ujoin Y) \cup Z \bigr) \subseteq M$ by hypothesis
(which implicitly involves the containment~$(X \ujoin Y) \cup Z \subseteq W$ in~(Zleft)(i)), 
it follows that~$T_*(\alpha \otimes \beta)$ is also a cycle of~$M$.
Suppose by way of contradiction that it is in fact a boundary:
\begin{equation}
\label{eq1a}
  T_*(\alpha \otimes \beta) = \partial\gamma ,
\end{equation}
for some~$\gamma \in C_{m+n+2}(M)$.
Note that hypothesis~(2)(b) implies that~$a$ is a full chain of~$M$.
Then we may apply Lemma~\ref{lemmaPropertiesInitialShuffleProduct}(5), to get: 
\begin{equation}
\label{eq2a}
  \bigl( \partial(\gamma_a) \bigr)_a  = (\partial\gamma)_a  =  T_*(\alpha \otimes \beta)_a  = q \bigl( a_Z \cup (a_X * \beta) \bigr) ,
\end{equation}
where the second equality follows using~(\ref{eq1a}), and the third using Lemma~\ref{lemmaPropertiesInitialShuffleProduct}(4).

We describe the chains involved in~$\gamma_a$; these must involve elements of~$\Lk_M(a)$.
We saw via our format-discussion---at~(i$^{\prime}$) after Remark~\ref{rk:QDfullbdrycalc}---that
we may write~$a = a_Z \cup a_X$: where~$a_X$ has the form~$\bigl( (x_{r+1},1) < \ldots < (x_m,1) \bigr)$
(with all members of~$a_Z$ below~$\min(a_X) = (x_{r+1},1)$).
Let~$c \in \gamma_a$ be~$a$-initial, and pick an element~$w \in c-a$.
By the description of~$\Lk_M(a)$ in hypothesis~(2)(b), it follows that~$w > \max(a) = (x_m,1)$;
so~$w = (x_m,y)$ for some element~$y \in Y$---by the definition of the ordering in the pre-join.
Hence we see~$c = a_Z \cup (a_X * \hat{c})$, where~$\hat{c} = \{ y \in Y : (x_m,y) \in c \}$.
Note that~$\hat{c}$ is a non-empty chain in~$Y$ since~$|c| = m+n+2 > m = |a|$.
If we write:
  \[ \gamma_a = \sum_{i \in I}\ q_i c_i = \sum_{i \in I}\ q_i \bigl( a_Z \cup (a_X * \hat{c}_i) \bigr) , \]
then working at the~$i$-th term, we have:
  \[ \partial(a_X * \hat{c}_i)   = \sum_j      (-1)^j (a_X * \hat{c}_i)^{\wedge j} ; \]
and so since the $a$-initial part of the boundary only removes members of the chain~$\hat{c}_i$, we see:
  \[ \partial(a_x * \hat{c}_i)_a = \sum_{j>m}\ (-1)^j (a_X * \hat{c}_i)^{\wedge j}
                                 = \bigl( a_X * (-1)^m\ \partial  (\hat{c}_i) \bigr) . \]
Thus we obtain:
\begin{equation}\label{eq3a}
   \begin{array}{ll}
     (\partial \gamma_a)_a &= \sum_i\ q_i (\ \partial \bigl( a_Z\ \cup\ (a_X *\hat{c}_i) \bigr)\ )_a
                            = \sum_i\ q_i (\ a_Z\ \cup\ \bigl(a_X * (-1)^m\ \partial(\hat{c}_i) \bigr)\ ) \\
                           &= a_Z\ \cup\ \left( a_X * (-1)^m\ \partial\ \bigl( \sum_i q_i \hat{c}_i \bigr)\ \right) .
    \end{array}
\end{equation}
Equating the final values for~$(\partial \gamma_a)_a$ in~(\ref{eq2a}) and~(\ref{eq3a}), we get:
  \[ q (a_Z \cup (a_X*\beta)) = a_Z \cup \left( a_X * (-1)^m\ \partial \left( \sum_i q_i \hat{c}_i\right) \right) , \]
and this holds if and only if:
  \[q \beta =  \partial \left( (-1)^m\ \sum_i\ q_i \hat{c}_i \right) . \]
But then~$\beta = \partial \left( q^{-1} (-1)^m\ \sum_i\ q_i \hat{c}_i \right)$ is zero in the homology~$\tilde{H}_n(Y)$,
contrary to hypothesis~(3).
This contradiction completes the proof.
\end{proof}

As a natural first application,
we deduce the original propagation result~\cite[0.27]{AS93} from Theorem~\ref{theoremNewHomologyPropagation}.
For reference, we first provide a statement of that earlier result:

\begin{lemma}[{original homology propagation, \cite[Lemma~0.27]{AS93}}]
\label{lemmaHomologyPropagationAS}
Assume the~$p'$-central product Hypothesis~\ref{hyp:HKcentralprod} for~$H,K \leq G$;
and that the following further conditions hold:
\begin{enumerate}
\item[(i)] for some~$A$ exhibiting~$\QD_p$ for~$H$, $\Ap(G)_{>A} \subseteq A \times K$,
\item[(ii)] $\tilde{H}_* \bigl( \Ap(K) , \QQ \bigr) \neq 0$.
\end{enumerate}
Then also~$\tilde{H}_* \bigl( \Ap(G),\QQ ) \neq 0$.
\end{lemma}

\begin{corollary}
\label{cor:newpropagextextends0.27AS}
Lemma~0.27 of~\cite{AS93} (i.e.~Lemma~\ref{lemmaHomologyPropagationAS})
is a particular case of Theorem~\ref{theoremNewHomologyPropagation}.
\end{corollary}

\begin{proof}
Let~$A,H,K,G$ be as in the hypotheses of~\cite[0.27]{AS93}---that is, Lemma~\ref{lemmaHomologyPropagationAS} above.

We saw in Lemma~\ref{lm:LBforWgivesZleft}(a) that we get the~(Zleft) Hypothesis~\ref{hyp:Zleft},
when we make the indicated choices~$X = \Ap(H)$, $Y = \Ap(K)$, $Z = \emptyset$, and~$W = \Ap(G)$ under group-inclusion~$<$.

We now check the conditions required for our new propagation Theorem~\ref{theoremNewHomologyPropagation},
for the further choice~$M = W = \Ap(G)$. 
By hypothesis~(i) of Lemma~\ref{lemmaHomologyPropagationAS}, 
$H$ has~$\QD_p$ exhibited by~$A$; giving a nonzero homology cycle~$\alpha \in Z_{m_p(H)-1} \bigl( \Ap(H) \bigr)$,
and a chain~$a \in \alpha$ such that~$\max(a) = A$---yielding condition~(1) for Theorem~\ref{theoremNewHomologyPropagation}.
Moreover, since in~\cite{AS93} the authors work with rational homology,
we immediately get invertibility for the coefficient~$q$ of~$a$ in~$\alpha$---giving condition~(2)(a).
The latter part of hypothesis~(i) also gives~$\Ap(G)_{>\max(a)} = \Ap(G)_{>A} \subseteq \{ A \}\times \Ap(K)$. 
And since~$A$ has maximal rank in~$\Ap(H)$, we have~$\Lk_{\Ap(G)}(a) \subseteq \Ap(G)_{> A}$;
and this with the previous inclusion completes condition~(2)(b).
Finally, condition~(3) holds since~$\tilde{H}_* \bigl( \Ap(K) \bigr) \neq 0$ by hypothesis~(ii).

Thus all the conditions for Theorem~\ref{theoremNewHomologyPropagation} hold;
so we conclude~$\tilde{H}_* \bigl( \Ap(G) \bigr) = \tilde{H}_*(M) \neq 0$.
\end{proof}

\begin{remark}
We can alternatively prove the above corollary by invoking Theorem~\ref{theoremNewHomologyPropagation},
with the same posets~$X=\Ap(H)$, $Y=\Ap(K)$, $Z=\emptyset$,
but by taking~$M = W$ to be the Th\'{e}venaz poset~$\widehat{X}_G(HK)$ of Proposition~\ref{propHomotEquivPosetReverseOrdering}
(with its order-relation~$\prec$)---of course we saw there that this poset is homotopy equivalent to~$\Ap(G)$.
\end{remark}

We mention that for the purposes of this paper, the main applications of Theorem~\ref{theoremNewHomologyPropagation}
will come in the proof of Theorem~\ref{theoremLiep}---more precisely,
in the proofs of Lemmas~\ref{lemmaCharPandFG-typeNC} and~\ref{lemmaCharPAndConicalFields},
covering the two parts of that Theorem. 

\bigskip

We close the section by giving a statement below the extended propagation result Lemma~3.14 of~\cite{KP20};
which we will use several times later---but
which does {\em not\/} follow from Theorem~\ref{theoremNewHomologyPropagation} above.
% replacing old~\ref{theoremHomologyPropagation} by new~\ref{theoremNewHomologyPropagation}
(However, note the similarity of conditions~(iv)+(ii) in the Lemma 
with condition~(2)(b) in the Theorem.)

The lemma holds for homology with coefficients in any commutative ring~$R$ with unit.
For the purposes of this article, we can suppose that~$R = \QQ$ or~$\ZZ$.

\begin{lemma}[Lm~3.14 in~\cite{KP20}]
\label{lemmaHomologyPropagationKP20}
Assume the~$p'$-central product Hypothesis~\ref{hyp:HKcentralprod} for~$H,K \leq G$:
\begin{enumerate}[label=(\roman*)]
\item $[H,K]=1$, and~$H \cap K$ is a~$p'$-group.

\noindent
      Further assume there is~$X \subseteq \Ap(G)$ with:
\item $\N_G(K)\subseteq X$;
\item There exist a chain~$a \in \Ap(H)' \cap X'$ and a cycle~$\alpha \in C_m \bigl( \Ap(H) \bigr) \cap C_m(X)$
      such that the coefficient of~$a$ in~$\alpha$ is invertible,
      and~$\alpha \neq 0$ in~$\tilde{H}_m \bigl( \Ap(H) \bigr)$ (for some~$m \geq -1$);
\item In addition, $a$ is a full chain in~$X$, and also~$X_{>\max(a)} \subseteq \N_G(K)$;
\item $\tilde{H}_* \bigl( \Ap(K) ) \neq 0$.
\end{enumerate}
Then~$\tilde{H}_*(X)\neq 0$.
\hfill $\Box$
\end{lemma}

\bigskip
\bigskip

\part{Applying the techniques for cases of Quillen's Conjecture}

\section{Eliminating many simple components of Lie type in the same characteristic~\texorpdfstring{$p$}{p}}
\label{sec:elimLierk1samecharp}

When we prove later Theorem~\ref{proofQCOdd}
(i.e.,~our extension of the Aschbacher-Smith Main Theorem stated earlier as Theorem~\ref{theoremASExtension}), 
we will see toward the end of that proof that in the case~$p=3$ (cf.~Remark~\ref{remarkReeFailsRobinson}),
the argument for~\cite[Theorem 5.3]{AS93} breaks down for components~$L$ given by the simple Ree groups in characteristic~$3$.  
We will want to avoid this problem---by establishing an elimination-result:
to guarantee (earlier in the overall proof) that such a component could not arise in our counterexample.

In fact, it turns out to be no more difficult,
in our results Theorem~\ref{theoremLiep} and~\ref{thm:Liep(H1)} below,
to eliminate a wide array of simple Lie-type components~$L$
in the ``same'' characteristic~$p$ (i.e.~the~$p$ for which we study~(H-QC))---subject to
a restriction on the~$p$-outers~$\Outposet_G(L)$ arising in~$G$.
This class of possible components is significant:
for example, these same-characteristic groups give the main family~(1) of obstructions
in the Aschbacher-Smith~$\QD_p$-List (which we have quoted as Theorem~\ref{theoremQDList}).

To treat them, we will be able to use propagation techniques to show ``generically''
in Theorem~\ref{theoremLiep} that if further the~$p$-outers in~$G$ for such an~$L$ are of the ``field-like'' type~$\Phi$
(in the sense of Definition~\ref{defConventionFieldAndGraphsCharp} below),
then~(H-QC) holds for~$G$---under several suitable inductive hypotheses (notably~(H1)) on~$C_G(L)$.

\medskip

It will be convenient to abbreviate our overall generic-context by the hypothesis:

\begin{definition}
\label{defn:sLie-p}
We will say that:

\centerline{
  A group~$L$ is of {\em type (sLie-$p$)\/}, if it is of Lie-type in characteristic~$p$, and is simple. 
            }
            
\noindent
We expand on several technical issues, which may not be obvious, related to this definition:

\medskip

  (1) Our choice of~$L$ simple---as opposed to quasisimple (e.g.~for a component)---is primarily for convenience:
In fact our arguments will typically apply to~$L$ quasisimple of Lie-type;%  
\footnote{
   But note that we avoid the further generalization to just quasisimple with~$L/Z(L)$ of Lie-type:
   for then we might get an ``exceptional'' Schur multiplier term (cf.~Table~6.1.3 in \cite{GLS98}) in~$Z(L)$, 
   which would prevent us from applying the Lie theory to~$L$ itself. 
          }
and indeed we hope, in a future version, to adjust this section to proceed under that more general quasisimple hypothesis~(qLie-$p$).
But for now, we assume~$L$ is simple; and this guarantees for example that~$L$ is adjoint
(cf.~Theorems~2.2.6 and~2.2.7 of~\cite{GLS98}).

\medskip

  (2) Of course those adjoint Lie-type groups are almost always automatically simple;
and those simple~$L$ in~(sLie-$p$) correspondingly provide our generic-case.  
But we also need to indicate the status, for our arguments, of the few non-generic simple groups not satisfying~(sLie-$p$):
These arise from the four non-solvable members of the set called~${\mathcal Lie}_\text{exc}$ in Definition~2.2.8 of~\cite{GLS98};
namely the Lie-type groups~$L^+ = B_2(2) , {^2}F_4(2), G_2(2) , {^2}G_2(3) $. 
For these, it is the commutator subgroup~$(L^+)'$ (of index~$p = 2,2,2,3$) which is the actual simple group.
These simple commutators~$(L^+)'$ are considered (more or less honorarily) to be of Lie type---for the purposes of stating the~CFSG.
However, they are not themselves of Lie-type, in the strict sense of satisfying the Lie theory (though the overgroup~$L^+$ does).
As a result, these four commutator subgroups---namely~$\Symplectic_4(2)' \groupiso \Alt_6$, ${^2}F_4(2)'$ (the Tits group), 
$G_2(2)' \groupiso \U_3(3)$, and the smallest Ree group~${^2}G_2(3)' \groupiso \Ln_2(2^3)$---do not satisfy~(sLie-$p$)
for the relevant primes~$p = 2,2,2,3$.
Hence they must be excluded, when we apply results depending on the full Lie theory---notably our generic-analysis
of outer automorphisms in the set~$\Phi$ mentioned earlier, as in Theorem~\ref{theoremLiep}.
However, we are still able to treat these four non-generic cases via other arguments,
particularly under the inductive hypothesis~(H1)---as we will see in Theorem~\ref{thm:Liep(H1)} below.
Of course we will need to make a specific mention of this exceptional situation later, when any of these non-generic cases arises.
\donerk
\end{definition}

\noindent
Thus~(sLie-$p$) gives the main hypothesis for our generic-case Theorem~\ref{theoremLiep} below.

\medskip

Furthermore, now {\em assuming\/} characteristic-$p$ in~(sLie-$p$),
we can also provide some initial motivation for another main hypothesis used in part (2) of Theorem~\ref{theoremLiep}, namely:%
%\footnote{
%  fakenote:
%  Here I am starting to adjust the exposition, for the splitting of Theorem~\ref{theoremLiep} recommended in your 11/12/21 email.
%          }

\centerline{
  ($p$-cyclic): \quad the members of~$\Outposet_G(L)$ should be cyclic;
            }

\noindent
that is, the~$p$-outers arising in~$G$ should be of order exactly~$p$.
This means that~$\Outposet_G(L)$ will be a poset of dimension~$0$, which will be important for our applications.

\medskip

We can begin the motivation for frequently expecting~($p$-cyclic)
by using just the overview-description of~$\Out(L)$ in earlier Remark~\ref{rk:outerautsofsimple}:
Observe first that:

\centerline{
  Under~(sLie-$p$), we have~$p = r$ (namely the characteristic prime~``$r$'' for~$L$).
            }

\noindent
And then since~$p$ cannot divide the order~$p^a - 1$ of the multiplicative group of the field~${\mathbb F}_{p^a}$
in the description of diagonal automorphisms in~(d) there, we further see that:

   (i) $\Outdiag(L)$ is a~$p'$-group.

\noindent
It follows from~(i) that each member of~$\Outposet_G(L)$ maps faithfully into the quotient:

\centerline{
  $\Out(L)^* := \Out(L) / \Outdiag(L)$. 
            }
 
\noindent
We observe next that each~$\Outdiag(L)$-conjugacy class falls into a single coset of~$\Outdiag(L)$---which
might contain several such classes.
So to investigate the behavior of some particular individual~$p$-outer~$B \in \Outposet_G(L)$,
i.e.~defined at the group-element level in~$\Aut(L)$, 
we can view~$B$ up to such conjugacy via its image~$B^*$ in~$\Out(L)^*$,
i.e.~at the level of quotients modulo inner-diagonal elements.
In summary:

  (ii) It is convenient to study~$p$-outers via their faithful images
       in~$\Outposet_G(L)^* \subseteq \Ap \bigl( \Out(L)^* \bigr)$.

\medskip

\noindent
Now we move on to a description of~$\Out(L)$ more detailed than in Remark~\ref{rk:outerautsofsimple}:

{\em Caution\/}:
We now temporarily---up to the end of Remark \ref{rk:casesPhiGamma}---suspend our assumption that~$p = r$ (which came from~(sLie-$p$));
that is, this more-detailed description holds for simple~$L$ of Lie-type of general characteristic~$r$,
independent of our special prime~$p$ for~(H-QC).

The above quotient~$\Out(L)^*$ can be described via a standard normal series for~$\Aut(L)$,
which more generally gives a natural viewpoint on the~$idfg$-representation of automorphisms
that we quoted in Remark~\ref{rk:outerautsofsimple}.
One form of the series appears in~(7-1) of~\cite[Part~I]{GL83};
but we will use the more extensive treatment in~\cite[Sec~2.5]{GLS98}.
By Theorem~2.5.12(b) there, $\Out(L)$ is a split extension of~$\Outdiag(L)$%
\footnote{
   We mention also that~$\Outdiag(L)$ is cyclic---except
   when~$L \groupiso D_n(q)$ with~$n \geq 4$ even and~$q$ odd, where~$\Outdiag(L) = C_2 \times C_2$,
          }
by the product:
\begin{equation}
\label{eq:Out*=PhiGamma}
  \Out(L)^* = \Phi_L \Gamma_L , \text{ with~$\Phi_L$ cyclic and normal;}
\end{equation}
where~$\Phi_L$ and~$\Gamma_L$ are given in parts~(a) and~(b,c) of Definition~2.5.10 there---for
present purposes, we'll just informally single out a few overview-features:

First, elements of~$\Gamma_L$ require elements of the group~$\Delta$ of symmetries of the Dynkin diagram---and
we listed the possible Dynkin diagrams for~$L$ with~$\Delta > 1$, in~(i) of Remark~\ref{rk:outerautsofsimple}.
Second, the normal subgroup~$\Phi_L$ is the image of the natural field automorphisms of an overlying algebraic group~$\bar{L}$:
arising (cf.~1.15.5(a) of~\cite{GLS98}) from powers of the generator~$x \mapsto x^r$ of the Galois group 
of the algebraic closure~$\overline{ {\mathbb F}_r }$
of the characteristic-$r$ field~${\mathbb F}_q$ of definition of~$L$. 
This in particular shows that~$\Phi_L$ is cyclic.

Furthermore, the case of~$\Delta > 1$ for~$\Phi_L \Gamma_L$ splits into cyclic and ``nearly-cyclic'' subcases below:

\begin{remark}[The cases for~$\Phi_L \Gamma_L$]
\label{rk:casesPhiGamma}
From~(f,e) of~\cite[2.5.12]{GLS98} we get the following possibilities
for the product~$\Out(L)^* = \Phi_L \Gamma_L$ in (\ref{eq:Out*=PhiGamma}):

  (0) For twisted~$L$, or untwisted~$L$ with diagram having~$\Delta = 1$:
      $\Gamma_L = 1$; so~$\Phi_L \Gamma_L = \Phi_L$ is cyclic.

  Otherwise~$L$ is untwisted, with diagram having~$\Delta > 1$: and the possibilities are:

  (1) For untwisted~$L$ of type~$B_2,F_4,G_2$ with~$r=2,2,3$:
      $\Phi_L \Gamma_L$ is cyclic, and~$\Phi_L \Gamma_L / \Phi_L \groupiso \Delta \groupiso \Cyclic_2$.

  (2) For untwisted types~$A_{\geq 2},D_{\geq 4},E_6$:
      $\Phi_L \Gamma_L = \Phi_L \times \Gamma_L$, and~$\Gamma_L \groupiso \Delta \groupiso \Cyclic_2$, or~$\Sym_3$ for~$D_4$. 
\donerk
\end{remark}

\noindent
Now as case~(0) above suggests, we will often have the condition~$\Out_G(L)^* \leq \Phi_L$;
indeed under~(sLie-$p$) where~$p = r$, we will usually get the following stronger condition for~$p$-outers:

\quad ($p$-$\Phi$): \quad $\Outposet_G(L)^* \subseteq \Ap(\Phi_L)$. 

\noindent
Note since $\Phi_L$ is cyclic in~(\ref{eq:Out*=PhiGamma})
that~($p$-$\Phi$) implies~($p$-cyclic).

%\begin{hypothesis}[The alternative hypothesis~($p$-$\Phi$)]
%\label{hyp:pPhi}
%Sometimes under~(sLie-$p$) we will assume:
%
%    \quad ($p$-$\Phi$): \quad $\Outposet_G(L)^* \subseteq \Ap(\Phi_L)$. 
%
%\noindent
%Note since $\Phi_L$ is cyclic in~(\ref{eq:Out*=PhiGamma})
%that~($p$-$\Phi$) implies~($p$-cyclic) in~(2) of Theorem~\ref{theoremLiep} below.
%\donerk
%\end{hypothesis}

\bigskip

\noindent
Furthermore the definition of~$\Phi_L$ via powers of~$x \mapsto x^r$ suggests
that its members exhibit ``field-like enough'' behavior for our purposes in this section.
(Of course under~(sLie-$p$) we will have~$p=r$.)
consequently we define:

\begin{definition}
\label{defConventionFieldAndGraphsCharp}
Assume that~$L$ has type~(sLie-$p$), and that~$L \leq G$ for some group~$G$.
\begin{enumerate}
\item We say that a~$p$-outer~$B \in \Outposet_{G}(L)$ of order~$p$ is of \textit{type~$\Phi$} if~$B^* \leq \Phi_L$.

\item Write~$\Phiposet_{G}(L)$ for the subposet of~$\Outposet_G(L)$ whose elements are of type~$\Phi$.

\noindent
      Note by order-$p$ that this is a {\em discrete\/} poset of cyclic~$p$-outers (no proper~$<$-relations).
\end{enumerate}
\end{definition}

%\noindent
%From~(2) we again see the relevance of~($p$-$\Phi$) for hypothesis~(2)=(someNC+$\Phi$) of Theorem~\ref{theoremLiep}.

\bigskip

\noindent
In the following somewhat technical result, the two cases corresponding to Remark~\ref{remarkClosingSection}
will proceed under somewhat different hypotheses, for reasons to be indicated later.

So we state our generic-Theorem below.
We work with rational homology.
We will prove:

\begin{theorem}[Generic Lie-eliminations under~(sLie-$p$)]
\label{theoremLiep}
Assume~$p$ is a prime, and~$G$ a finite group satisfying:

\begin{tabular}{lcl}
(sLie-$p$)    & & \pbox{13cm}{$G$ has a component~$L$ of type (sLie-$p$) (simple~of Lie type in characteristic~$p$);} \\
\end{tabular}

\noindent
Assume further that one of the following conditions holds:

\begingroup
\renewcommand{\arraystretch}{2}
\begin{tabular}{lcl}
(1)=(allC+H)         & & \pbox{13cm}{For all~$B \in \Outposet_G(L)$, $O_p \bigl( C_G(LB) \bigr) > 1$, and~$C_G(L)$ satisfies~(H-QC).} \\
(2)=(someNC+$\Phi$)  & & \pbox{10cm}{($p$-cyclic) holds for $L$, and \hfill \\
there exists $B_0 \in \Phiposet_G(L)$ such that $\tilde{H}_*(\ \Ap \bigl( C_G(LB_0) \bigr)\ ) \neq 0$.}
\end{tabular}
\endgroup

\noindent
Then~$G$ satisfies~(H-QC).

\medskip

In particular, if~$G$ satisfies~(H1), and also contains a component~$L$ of type~(sLie-$p$)
with~($p$-$\Phi$) (that is, $\Outposet_G(L)^* \subseteq \Ap(\Phi_L)$), then~$G$ satisfies~(H-QC).
Notice that any~$L$ of Lie-rank~$1$ has~($p$-$\Phi$) by Remark~\ref{rk:casesPhiGamma}(0).%

Furthermore the above statement holds if~``(H1)'' is replaced by~``(H1u)+($p$ odd)''.
\end{theorem}

\bigskip

\noindent
We will postpone the proof of Theorem~\ref{theoremLiep}---while we first discuss two aspects of its hypotheses:

\medskip

\noindent
First aspect:

The further hypotheses~(1)=(allC+H) and~(2)=(someNC+$\Phi$) of Theorem~\ref{theoremLiep} are of course
extensions of the case-division we had described earlier in Remark~\ref{remarkClosingSection}, namely:

  \quad \quad (allC)---where for all~$B \in \Outposet_G(L)$ we have~$O_p \bigl( C_G(LB)\bigr) > 1$; and 

  \quad \quad (someNC)---where we have some~$B$ with~$O_p \bigl( C_G(LB) \bigr) = 1$.

\noindent
These extensions are meant (roughly) to provide purely propagation-theoretic hypotheses.
Note that~(someNC+$\Phi$) adds, to the propagation-friendly situation of~(someNC),
a~$p$-outer~$B_0$ which must specifically be of the type~$\Phi$, which behaves well in homology calculations;
here we will need the assumption~($p$-cyclic) to guarantee propagation.
By contrast, hypothesis~(allC+H) adds to~(allC) a propagation-friendly nonzero-homology condition on~$C_G(L)$;
in applications, we will typically obtain this condition essentially automatically,
via a suitable inductive hypothesis such as~(H1)---as in the ``In particular'' statement in the Theorem.

\bigskip

\noindent
Second aspect:

Here we argue roughly that under the main hypothesis~(sLie-$p$) of Theorem~\ref{theoremLiep},
the later hypothesis~($p$-cyclic) in part~(2), and indeed the stronger hypothesis~($p$-$\Phi$),
should hold ``most of the time''. 
Furthermore, the discussion below of when those cyclic~hypotheses might fail, under the generic hypothesis~(sLie-$p$),
will also lead essentially to the list of exceptional situations, 
that we must further treat in the non-generic Theorem~\ref{thm:Liep(H1)}---thus basically explaining
the list of excluded Lie-types (i.e., that we do not treat) in the statement of that Theorem below.

So during this second-aspect discussion, we continue to assume~(sLie-$p$); 
and argue to isolate those situations where~($p$-$\Phi$)---that is,
the condition that~$\Outposet_G(L)^* \subseteq \Ap(\Phi)$---might fail: 

We saw in earlier~(ii) that we may study individual members of~$\Outposet_G(L)$,
up to~$\Outdiag(L)$-conjugacy, via the cases for~$\Phi_L \Gamma_L$ in Remark~\ref{rk:casesPhiGamma}.
In case~(0) (i.e.~for twisted groups, and untwisted groups where~$\Delta = 1$)
we have~$\Out(L)^* = \Phi$, giving~($p$-$\Phi$);
so we are reduced to the untwisted groups~$L$ in cases~(1) and~(2) there (namely with~$\Delta > 1$).

First assume case~(1), with multiple-bonds, and characteristic~$r = p$ by~(sLie-$p$):
If~$L$ has type~$G_2$, we have~$p = r = 3$, with~$| \Phi_L \Gamma_L : \Phi_L | = 2$;
so in particular, $3$-outers lie in~$\Phi_L$.
Here~$\Outposet_G(L)^* \subseteq \A_3(\Phi)$, and again we have~($p$-$\Phi$).
So consider instead types~$B_2$ and~$F_4$, where we have~$p=r=2$.
We first single out the odd-power field cases: 
When~$O_2(\Phi_L) = 1$ (i.e.~$L$ is defined over some~${\mathbb F}_{2^{2a-1}}$),
then we have~$\Out(L)^* = \Phi_L \times \Gamma_L$ with~$\Phi_L$ of~$2'$-order, so~($p$-$\Phi$) fails; but~($p$-cyclic) does hold.
Here~$\Symplectic_4(2)$ is not simple and so does not arise under~(sLie-$p$); 
but we treat~$\Symplectic_4(2)'$ instead via independent arguments in the non-generic Theorem~\ref{thm:Liep(H1)}.
In fact, we are also able to treat~$F_4(2)$, and indeed the higher odd-powers for types~$B_2$ and~$F_4$, similarly there;
so these groups~$B_2$ and~$F_4$ with odd-power fields do not appear in the excluded-list in that Theorem  below.
Now we turn to groups of those types with even-power fields:
We have~$O_2(\Phi_L) > 1$; so since~$\Phi_L \Gamma_L$ is cyclic,
the~$2$-elements of~$\Phi_L \Gamma_L - \Phi_L$ must have order at least~$4$---hence~($p$-$\Phi$) holds here,
and the groups are treated via our generic-methods.

Finally assume case~(2), with single-bonds, and characteristic~$r = p$ by~(sLie-$p$):
Here we always have $\Out(L)^* = \Phi_L \times \Gamma_L$ with~$\Gamma_L \groupiso \Delta \groupiso \Cyclic_2$, or~$\Sym_3$ for~$D_4$.
Now it might happen that~($p$-$\Phi$) holds, when~$\Outposet_G(L)^*$ induces only field automorphisms on~$L$;
and then we can use our generic-methods.
But it may also happen that~$p$ divides the order of both~$\Phi_L$ and~$\Gamma_L$,
and then we see that~($p$-$\Phi$) and~($p$-cyclic) can definitely fail.
If~$\Outposet_G(L)^*$ induces only a graph- or graph-field automorphism, we could still have~($p$-cyclic).
However, we do not attempt to treat such automorphisms by our methods;
so these single-bond groups, in the case away from~($p$-$\Phi$), appear in the excluded-list in the non-generic Theorem below.  

Notice that we have now reduced to exactly the cases in that excluded-list.

\bigskip

In fact, it is now not difficult to establish our non-generic Theorem~\ref{thm:Liep(H1)}:
this time not constructively via propagation as in the generic-Theorem~\ref{theoremLiep},
but instead via the less-constructive elimination-methods of~\cite{PS}---which turn out to be suited
to the few, and suitably-``small'', cases that we in practice must treat.

\begin{theorem}[Lie-eliminations under~(H1)---including non-generic cases]
\label{thm:Liep(H1)}
Now assume that~$G$ satisfies~(H1),
and contains a component~$L$ which is simple of Lie-type in characteristic~$p$, in the wider CFSG-sense:
that is, either~$L$ satisfies~(sLie-$p$),
or~$L$ is one of the four~${\mathcal Lie}_\text{exc}$-commutator groups in Definition~\ref{defn:sLie-p}(2).

Then as long as~$L$ is not of one of the following untwisted Lie-types with the indicated characteristic~$r = p$:
  \[ \PSL_n(2^m) (n \geq 3)\ ; \, D_n(p^m)  (n \geq 4; p=2,3)\ ; \,  E_6(2^m) ; \]
with some~$B \in \Outposet_G(L)$ inducing a graph-automorphism or graph-field automorphism, 
we have that~$G$ satisfies~(H-QC).
%Furthermore the above statement holds if~``(H1)'' is replaced by~``(H1u)+($p$ odd)''---at least
%for the generic cases in Theorem~\ref{theoremLiep}.%
\end{theorem}

We mention that there are further technical conditions related to~$B$ which could reduce the excluded-list above;
but we have preferred to keep this simpler form of the statement.
We provide in Remark \ref{rk:fullCases} some further discussion of the excluded cases.

\begin{proof}[Proof of Theorem \ref{thm:Liep(H1)}]
Note that under~(H1), we may assume by Theorem~1.6 of~\cite{PS} that~$\Outposet_G(L)$ is non-empty.

Here most of the work is done by the ``In particular'' statement for~(H1) in the generic-Theorem~\ref{theoremLiep}:
that is, we are done in the vast majority of cases, where~$L$ has both~(sLie-$p$) and~($p$-cyclic) (and indeed ($p$-$\Phi$).)
Thus we only need to treat the cases where one of the hypotheses fails. 

Now the simple groups failing the strict Lie-theory condition of~(sLie-$p$)
are the four commutator subgroups in Definition~\ref{defn:sLie-p}(2).
And our analysis above isolated the cases where~($p$-$\Phi$) and ($p$-cyclic) can fail:
our excluded-list records the cases that we will not attempt to treat---but we will in fact
cover~$B_2(q)$ and~$F_4(q)$ for~$q$ an odd power of~$p = 2$.

We note that the condition~($p$-cyclic) arises in each of the separate cases just listed in the paragraph above
(it is needed as hypothesis for the results in~\cite{PS}).

\medskip

\noindent
We first consider the four non-generic commutator groups:

First~$L \groupiso {^2}F_4(2)'$ has~$\Outposet_G(L) = \emptyset$,
since outer $2$-automorphisms~$B$ have order~$4$ as group elements in~$LB$;
so this case is already handled by our initial remark. 

Next, the smallest case~$L \groupiso \Symplectic_4(2)'$ is handled specifically for $p = 2$ in Corollary~8.1 of~\cite{PS}.

Furthermore Corollary~8.1 of~\cite{PS} invokes Proposition~6.9 there;
and indeed via an easier application of that Proposition, we can in fact treat in parallel the remaining two cases:
i.e.~$L$=($G_2(2)' ,  {^2}G_2(3)'$), for~$p$=($2,3$).
Since we may assume~$\Outposet_G(L)^*$ is non-empty,
we see that it is given by~$B^*$ inducing an automorphism of type~(graph,graph-field) respectively
(though the first case is instead also called graph-field in~\cite{GLS98}---see Definition~2.5.13(b)(2) there).
We check then that~$C_L(B)$ has structure (order~$24$, $\Sym_3$) with~$O_p \bigl( C_L(B) \bigr) > 1$,
and it follows that the reduced homology of~$\Ap \bigl( C_L(B) \bigr)$ for any term vanishes (i.e., in all degrees). 
However, $\Ap(L)$ is a wedge of spheres of topological dimension~($1$,$0$).  
So taking these values for~``$k$'' in Proposition~6.9 of~\cite{PS}, we conclude that that~$G$ satisfies~(H-QC).

The elimination of the groups of type ($B_2(q),F_4(q)$) (for~$q$ an odd power of~$p = 2$) is only marginally more complicated:
Here $\Outposet_G(L)^*$ is either empty, or given by $B^*$ inducing an automorphism of (graph,graph) type
(again with~\cite{GLS98} using the graph-field terminology).
Now we get~$C_L(B)$ of structure~(${^2}B_2(q),{^2}F_4(q)$):
so that~$\A_2 \bigl( C_L(B) \bigr)$ is a wedge of spheres in topological dimension~($0,1$),
with reduced homology vanishing in degrees above those values.
However, $\A_2(L)$ is a wedge of spheres in topological dimension~($1,3$).
So taking these values for~``$k$'' in Proposition~6.9 of~\cite{PS}, we conclude that that~$G$ satisfies~(H-QC).
\end{proof}

\bigskip

\noindent
We now start the details for our earlier preview, and culminating in the proof of our generic-Theorem~\ref{theoremLiep}.

Recall that the more technical aspects of our study of~$p$-outer automorphisms 
are presented in Appendix Section~\ref{sec:fldgraphautcharp}---notably
the conventions in Definition~\ref{defn:fldgraphconvtypePhi} for the terminology of  ``field'' and ``graph'' automorphisms.

\bigskip

First, we obtain an analogue of the Quillen-dimension property of~\cite{AS93} (recall Definition~\ref{defn:QD}), in the sense 
that, for our almost-simple extensions, we obtain nonzero homology in a particular natural degree:
Given~$L$ of type~(sLie-$p$) with Lie rank~$m$, and a~$p$-outer~$B$ of type~$\Phi$, this degree for~$\Ap(LB)$ is in fact~$m$. 
(Since the behavior arises from an underlying use of the Bouc poset~$\Bp(L)$ of~$L$,
we have nicknamed the property as ``Bouc dimension'' below.)
And this homology will be the starting-point for our propagation up to~$\Ap(G)$.
We establish:

\begin{theorem}[``Bouc dimension'']
\label{theoremHomologyCharP}
Assume that~$L$ is of type~(sLie-$p$) with Lie rank~$m$, and we have some~$B \in \Phiposet_{\Aut(L)}(L)$.
Then~$\tilde{H}_m \bigl( \Ap(LB) \bigr) \neq 0$.%
%\footnote{
%  fakenote:
%  Dear Ron/Richard:  We suspect this easy result \ref{theoremHomologyCharP} may already be known?  Ie, either published, or at 
%   least known to you? More likely, the combinatorial bound \ref{lemmaLiepBoundCentralizer} on number of tori, on which it depends.  (Namely We'd be happy to replace our numerical proof with a reference to anything easier.)
%}
\end{theorem}

\begin{proof}
Fix~$B \in \Phiposet_{\Aut(L)}(L)$.
By part~(3) of Proposition~\ref{propertiesPhiType},
we may choose~$D \in \Outposet^{\Phi}_{LB}(L)$ such that~$O_p \bigl( C_L(D) \bigr) = 1$.
(We would only need to make this choice, if our original~$B$ happened to be in case~(2)(a) of the Proposition.)
Since~$LB = LD$, we may take~$D$ in the role of~``$B$'';
thus we have~$O_p \bigl(C_L(B) \bigr) = 1$---that is, we have our chosen~$B$ in case~(1) or~(2)(b) of the Proposition.

We are going to work with the Bouc poset~$\Bp(L)$.
Recall that~$\Bp(L)$ is homotopy equivalent to the Tits building of~$L$:
it is a poset of topological dimension~$m-1$, where~$m$ is the Lie-rank of~$L$,
and it is homotopy equivalent to a wedge of spheres of dimension~$m-1$.
The number of such spheres is the vector-space dimension of~$\tilde{H}_{m-1} \bigl( \Ap(L) \bigr) = \tilde{H}_{m-1}(\Bp(L))$,
and that dimension is in turn the order of a Sylow~$p$-subgroup of~$L$, namely~$|L|_p$.
In particular we have~$\tilde{H}_* \bigl( \Ap(L) \bigr) = \tilde{H}_{m-1} \bigl( \Bp(L) \bigr)$.

We will use the abbreviation~$\Outposet := \Outposet_{LB}(L)$.
By Theorem~\ref{theoremChangeABSPosets} applied to~$LB$---with $L$,$1$ in the roles of~``$H$,$K$''---we get: 
  \[ \Ap(LB) \simeq W_{LB}^\BoucPoset(L,1)\     (\ = \Bp(L) \cup \Outposet\ )\ . \]
  %\simeq  W_{LB}^\QuillenPoset(L,1)\ (\ = \Ap(L) \cup \Outposet\ )
Therefore~$\Ap(LB)$ is homotopy equivalent to the poset~$\Bp(L) \cup \Outposet$ of topological dimension~$m$.

We will now let~$D$ vary over~$\Outposet$.  
By~(3) of Proposition~\ref{propertiesPhiType}, we get~$D^* \subseteq \Phi_L$;
and since we also have~$B^* = \Omega_1 \bigl( O_p(\Phi_L) \bigr)$ there, we see~$D^* = B^*$.  
Next apply the Proposition to this general~$D$: 
then either case~(2)(a) holds, with~$O_p \bigl( C_L(D) \bigr) > 1$---so that~$\Ap \bigl( C_L(D) \bigr)$ is contractible;
or else case~(1) or case~(2)(b) holds, with~$C_L(D)$ again a group of Lie-type in characteristic~$p$, of the same Lie-rank~$m$
(possibly extended by $p'$-diagonal automorphisms),
so that~$\Ap \bigl( C_L(D) \bigr)$ is homotopy equivalent to a wedge of~$|C_L(D)|_p$ spheres of topological dimension~$(m-1)$.
Thus the reduced homology of~$\Ap \bigl( C_L(D) \bigr)$ either is zero in the contractible case; 
or in the Lie-type centralizer case, vanishes in degrees below~$(m-1)$ (just as for~$L$ itself).

At this point, we invoke the long exact sequence of~\cite[Main Theorem]{SW}:
We take~$L,LB$ in the roles of~``$N,G$'' there;
note that~``$\mathcal{M}$'' there is~$\Outposet$ here, which consists of cyclic elements of order~$p$ of type~$\Phi$.
With the notation of that article (where~$\Sigma$ denotes suspension),
we have that~$\Ap(LB)_L \simeq \bigvee_{D \in \Outposet} \Sigma \Ap \bigl( C_L(D) \bigr)$.
Now we saw that homology from~$\Ap(L)$ and~$\Ap \bigl( C_L(D) \bigr)$ vanishes
(in any case, contractible or otherwise) in degrees below~$(m-1)$;
and similarly homology from~$\Ap(LB)$ vanishes in degrees above the topological dimension~$m$. 
So we get:
  \[ 0 \to \tilde{H}_m \bigl( \Ap(LB) \bigr) \to \bigoplus_{D \in \Outposet}\ \tilde{H}_{m-1} \left(\ \Ap \bigl( C_L(D) \bigr)\ \right)
    \overset{ }{\to} \tilde{H}_{m-1}\bigl( \Ap(L) \bigr) \to \tilde{H}_{m-1} \bigl( \Ap(LB) \bigr) \to 0 . \]
Using the above exact sequence of vector spaces, we now calculate:
with the aim of lower-bounding the vector-space dimension of~$\tilde{H}_m \bigl( \Ap(LB) \bigr)$.
We get zero-contribution from any cases where~$\Ap \bigl( C_L(D) \bigr)$ is contractible; 
and we will use the fact, from~(1) or~(2)(b) of Proposition~\ref{propertiesPhiType},
that there is just one~$\Inndiag(L)$-conjugacy class of~$D$ with the indicated centralizer structure,
represented by our original choice~$B$:
\begin{align*}
  \dim ( \tilde{H}_m &  \bigl(\ \Ap(LB) \bigr)\ ) = \\
    & = \dim \left(\ \tilde{H}_{m-1} \bigl( \Ap(LB) \bigr)\ \right) - \dim \left(\ \tilde{H}_{m-1} \bigl( \Ap(L) \bigr)\ \right)
                          + \sum_{D \in \Outposet} \dim \left(\ \tilde{H}_{m-1} (\ \Ap \bigl( C_L(D) \bigr)\ ) \right) \\
    & \geq  - \dim \left( \tilde{H}_{m-1} \bigl( \Ap(L) \bigr) \right)
                          + \sum_{D \in \Outposet}\ \dim \left( \tilde{H}_{m-1}(\ \Ap \bigl( C_L(D) \bigr)\ ) \right) \\
    & \geq  - \dim \left( \tilde{H}_{m-1} \bigl( \Ap(L) \bigr) \right)
                          + \sum_{gN_L(B) \in L/N_L(B)}\ \dim \left( \tilde{H}_{m-1} (\ \Ap \bigl( C_L(B^g) \bigr)\ ) \right)\\
    & = - |L|_p + |L:N_L(B)| |C_L(B)|_p\\
    & = - |L|_p + \frac{|L|}{|C_L(B)|} \frac{|C_L(B)|}{|N_L(B)|} |C_L(B)|_p\\
    & \geq - |L|_p + \frac{|L|}{|C_L(B)|} \frac{1}{(p-1)} |C_L(B)|_p\\
    & = |L|_p \left( - 1 + \frac{|L|_{p'}}{(p-1)|C_L(B)|_{p'}}\right).
\end{align*}
Now our chosen~$B$ has~$O_p \bigl( C_L(B) \bigr) = 1$
(that is,  case~(1) or~(2)(b) of Proposition~\ref{propertiesPhiType});
so Lemma~\ref{lemmaLiepBoundCentralizer} gives the inequality~$\frac{|L|_{p'}}{(p-1)|C_L(B)|_{p'}} > 1$,
and hence the above dimension is positive.
We conclude that~$\tilde{H}_m \bigl( \Ap(LB) \bigr) \neq 0$, as desired.
\end{proof}

\bigskip

We now begin the details for our main proof, of the generic Theorem~\ref{theoremLiep}, 
splitting it into the two cases of its conicality-hypotheses.
We first deal essentially with the case~(2) there,
corresponding to condition~(someNC) in Remark~\ref{remarkClosingSection}:
%we can in fact use the hypothesis~($p$-cyclic), which is slightly weaker than~($p$-$\Phi$) assumed in~(someNC+$\Phi$):

% And the nonzero homology for~$L B_0$, which is the input to the propagation,
% will in fact supplied by Theorem~\ref{theoremHomologyCharP}.
% Then in the proof of Lemma~\ref{lemmaCharPandFG-typeNC}, which covers this~(someNC+$\Phi$) case, 
% the propagation itself is provided by Theorem~\ref{theoremNewHomologyPropagation}---using 
% the replacement-poset $W = W^{\BoucPoset}_G \bigl( L , C_G(L B_0) \bigr)$ of Proposition~\ref{theoremChangeABSPosets}.
%By contrast, hypothesis~(allC+H) adds to~(allC) a propagation-friendly nonzero-homology condition on~$C_G(L)$.
% and then in the proof of Lemma~\ref{lemmaCharPAndConicalFields}, which covers this~(allC+H) case, 
% we again apply propagation via Theorem~\ref{theoremNewHomologyPropagation}---but now using 
% the tilde-variant replacement-poset given by~$W = \tilde{W}^{\BoucPoset}_G \bigl( L , C_G(L) \bigr)$
% of Proposition~\ref{propHomotEquivTildePosets}.

\begin{lemma}
\label{lemmaCharPandFG-typeNC}
Let $G$ be a finite group such that:

\begin{tabular}{lcl}
  (sLie-$p$)            & & \pbox{13cm}{$L$ is a component of~$G$ of type~(sLie-$p$) ; } \\
  ($p$-cyclic)          & & \pbox{13cm}{$\Outposet_G(L)$ contains only cyclic~$p$-outers; } \\
  (someNC+$\Phi$)$^{-}$ & & \pbox{12cm}{There exists~$B_0 \in \Phiposet_G(L)$
                                  such that $\tilde{H}_*(\ \Ap \bigl( C_G(LB_0) \bigr)\ ) \neq 0$ . } 
\end{tabular}

\noindent
Then~$G$ satisfies~(H-QC).
\end{lemma}

\begin{proof}
As usual to show~(H-QC) for~$G$, we assume~$O_p(G)=1$, and show that~$\tilde{H}_* \bigl( \Ap(G) \bigr) \neq 0$.

Let~$L$ and~$B_0$ be as in the hypotheses~(sLie-$p$) and~(someNC+$\Phi$)$^{-}$.
Then the choices~$H := L$ and~$K := C_G(LB_0)$ satisfy the component-Hypothesis~\ref{hyp:LB}; 
and we showed in Lemma~\ref{lm:HypLBimpliesHypHKcentralprod} that the less-usual choice of~$H = L$ rather than~$LB_0$
still gives the central-product Hypothesis~\ref{hyp:HKcentralprod} for~$HK$.
So by Theorem~\ref{theoremChangeABSPosets}:
  \[ \Ap(G) \simeq W := W^\BoucPoset_G \bigl( L,C_G(LB_0) \bigr)
                      = \Bp(L) \ujoin \Ap \bigl( C_G(LB_0) \bigr)\ \cup\ \F_{G} \bigl( L C_G(L B_0)\bigr) , \]
where we recall that~$\prec$ denotes the order-relation in the poset~$W$.
%By Proposition \ref{propertiesFGType}(4), we may suppose that $O_p(C_L(B_0)) = 1$.
Now note that in the above we have~$\Outposet:=\Outposet_{LB_0}(L) \subseteq \F_{G}(LC_G(LB_0))$; and also we have an inclusion of posets:
  \[ \Ap(LB_0) \simeq     W_{LB_0}^{\BoucPoset}(L,1) = \bigl(\ \Bp(L) \times \{ 1 \}\ \bigr) \cup \Outposet
               \subseteq  W^\BoucPoset_G \bigl( L,C_G(LB_0) \bigr) . \]
We are going to invoke Theorem~\ref{theoremNewHomologyPropagation}:
with the subcomplex~$M = \K(W)$, and making the choices that~$X = \Bp(L)$, $Y =\Ap(K)$, and~$Z = \Outposet$.
Recall that these are the choices of~$X,Y,Z,W$ made in our earlier augmented-Hypothesis~\ref{hyp:LBforW}
(extending the component-Hypothesis~\ref{hyp:LB}); 
and that in Lemma~\ref{lm:LBforWgivesZleft}(b), we showed that these choices satisfy the (Zleft)-Hypothesis~\ref{hyp:Zleft}, 
as needed for that Theorem.

Now we check the other conditions for Theorem~\ref{theoremNewHomologyPropagation}:
By Theorem~\ref{theoremHomologyCharP}, we see that~$W_{LB_0}^{\BoucPoset}(L,1)$ has dimension~$m$,
and that it has nonzero homology in degree~$m$.
Since~$\Bp(L)$ has dimension~$m-1$,
there is a nonzero homology cycle~$\alpha$ of~$\tilde{H}_m(W_{LB_0}^{\BoucPoset}(L,1)) \groupiso \tilde{H}_m(\Ap(LB_0))$ with:
  \[ \alpha = \sum_{i \in I}\ q_i \bigl( B^i \prec (R_1^i,1) < \ldots < (R^i_m,1) \bigr) , \]
where~$R^i_j \in \Bp(L)$ for all~$j = 1 , \ldots , m$, and~$B^i \in \Outposet$.
Moreover, since~$\Bp(L)$ has dimension~$m-1$,
note that:

\centerline{
   $(R_1^i,1) < \ldots < (R^i_m,1)$ is a maximal chain in~$\Bp(L)$;
            }

\noindent
so~e.g.~$R^i_m \in \Syl_p(L)$ for all~$i$.
In particular, if~$x \in W$ and~$x \succ (R^i_m,1)$,
then just from the ordering in the Cartesian-product we have~$x = (R^i_m,E) \in  \K \bigl( \{ \max(a) \} \times Y)$,
for some~$E \in \Ap(K)$.

Now let~$a \in \alpha$ be one of the above chains involved in~$\alpha$, with invertible coefficient~$q = q_i \neq 0$.
We immediately get conditions~(1) and~(2)(a) of Theorem~\ref{theoremNewHomologyPropagation},
and the above-$a$ part of condition~(2)(b) there.
By~(someNC+$\Phi$)$^-$, we can take a nonzero cycle~$\beta$ of~$\Ap(K)$,
giving condition~(3) of Theorem~\ref{theoremNewHomologyPropagation}.
It remains to establish the rest of condition~(2)(b); at this point, we will strongly use our hypothesis~($p$-cyclic).
Namely for~$a$ as just chosen, we will establish:

\vspace{0.2cm}

\textbf{Claim.} $\Lk_M(a) \subseteq \K \bigl( \{ \max(a) \} \times Y)$.
      In particular, $a$ is a maximal chain of~$X \cup Z$.

\begin{proof}
Suppose that~$a = \bigl( B \prec (R_1,1) < \ldots < (R_m,1) \bigr)$,
and let $w \in M - a$ be~$\prec$-comparable with every element of~$a$.
Note that~$|B| = p$, since~$B \in \Outposet_{LB_0}(L) = \Outposet$ and~$|B_0|=p$.
We also have~$B \in \F_{G}(HK)$; so~$B \prec w$ by the~$\F$-left property in the definition of the ordering.
In fact we will show that~$w > \max(a) = R_m$, where we saw earlier that such a~$w$ satisfies the Claim.

First, we also saw earlier that~$a - \{ B \}$ is a maximal chain in~$\Bp(L)$;
so we cannot have~$1 \leq i < m$ with~$(R_i,1) < w < (R_{i+1},1)$.
Thus we only need to eliminate the case
that~$B \prec w \prec (R_1,1)$.

So suppose that $B \prec w \prec (R_1,1)$.
Now~$w \notin \Bp(L)$, again  since we saw~$a - \{ B \}$ is a maximal chain in~$\Bp(L)$.
Hence~$w = D \in \F_{G}(HK)$; that is, $D \cap  \bigl( L C_G(LB_0) \bigr) = 1$.
Since $D \prec (R_1,1)$ we see that~$C_{R_1}(D) > 1$, so~$D$ normalizes~$L$.
Then~$D > B$ implies that~$D$ also normalizes~$LB$.
Now since $D~\leq C_G(B)$ we get:
  \[D \cap \bigl( L C_G(L) \bigr) = D \cap \bigl( L C_G(L) \bigr) \cap C_G(B)
                                  = D \cap \bigl( C_L(B) C_G(LB)) \leq D \cap \bigl( L C_G(LB_0) \bigr) = 1 . \]
Here we have used the fact that~$LB = LB_0$ since~$B \in \Outposet_{LB_0}(L)$.
Thus~$D \in \Outposet_{N_G(LB_0)}(L) = \Outposet$, which consists only of cyclic~$p$-outers by our hypothesis~($p$-cyclic).
This is contrary to~$D$ of~$p$-rank at least~$2$, since~$D > B$.
This contradiction shows that no such~$D$ exists.

We have reduced to the case that~$w > (R_m,1) = \max(a)$; which we saw earlier satisfies the Claim. 
In particular, $a$ is a maximal $M$-chain~in $X \cup Z$.
\end{proof}

Therefore Theorem \ref{theoremNewHomologyPropagation} applies,
with these choices of~$X,Y,Z,W,M=\K(W)$, and of~$a,\alpha,\beta$;
so we get~$\tilde{H}_* \bigl( \Ap(G) \bigr) \groupiso \tilde{H}_*(W) \neq 0$, as desired.
\end{proof}

Next we prove case~(1) of generic-Theorem~\ref{theoremLiep},
corresponding to condition~(allC) in Remark~\ref{remarkClosingSection}:

\begin{lemma}
\label{lemmaCharPAndConicalFields}
Let $G$ be a finite group such that:

\begin{tabular}{lcl}
(sLie-$p$) & & \pbox{13cm}{$L$ is a component of~$G$ of type~(sLie-$p$);} \\
(allC+H) & & \pbox{13cm}{For all~$B \in \Outposet_G(L)$, $O_p \bigl( C_G(LB) \bigr) > 1$; and~$C_G(L)$ satisfies~(H-QC).}
\end{tabular}

\noindent
Then~$G$ satisfies~(H-QC).
\end{lemma}

\begin{proof}
As usual to show~(H-QC) for~$G$, we assume~$O_p(G)=1$, and show that~$\tilde{H}_* \bigl( \Ap(G) \bigr) \neq 0$.

\medskip

\noindent
We begin with a slightly easier version of the setup as in the previous lemma:

Let~$H := L$ and~$K := C_G(L)$: again these choices satisfy the component-Hypothesis~\ref{hyp:LB}; 
and again using Lemma~\ref{lm:HypLBimpliesHypHKcentralprod} we see that these usual choices
give the central-product Hypothesis~\ref{hyp:HKcentralprod} for~$HK$.
This time, in order to avoid centralizers in~$X \ujoin Y$ (of members of~$\Outposet_G(L)$) which lie entirely in~$L$,
we instead apply Proposition~\ref{propHomotEquivTildePosets} (rather than Theorem~\ref{theoremChangeABSPosets})
in order to be able to make use of the variant-poset:
  \[ \Ap(G) \simeq W := \tilde{W}_G^{\BoucPoset}(H,K) = \Bp(H) \ujoin \Ap(K) \cup \F_G(HK) , \] 
namely, the poset of Proposition~\ref{propTildePosetsXandW};
where we recall that~$\sqsubset$ denotes the order-relation in the poset~$W$.

Note now by the initial~(allC)-part of hypothesis~(allC+H), we have:
  \[ \F_1 = \{ B \in \Outposet_G(L) \tq O_p \bigl( C_G(LB) \bigr) > 1 \} = \Outposet_G(L) ; \]
so that the subset~$\F_1$ in fact covers all of~$\Outposet_G(L)$. 

In view of this, we will again be applying Theorem~\ref{theoremNewHomologyPropagation} for propagation:
again with~$M = \K(W)$ but now for the new variant-$W$ as above;
and with~$X := \Bp(L)$, $Y := \Ap(K)$ and now we can use the ``trivial'' choice~$Z := \emptyset$.
These choices of~$X,Y,Z,W$ extend the component-Hypothesis~\ref{hyp:LB}; 
and we can also directly mimic the proof in the simpler case~(a) of Lemma~\ref{lm:LBforWgivesZleft}:
namely we see that these choices satisfy the (Zleft)-Hypothesis~\ref{hyp:Zleft}, 
as needed for Theorem~\ref{theoremNewHomologyPropagation}:
that is, the inclusion of~$X \ujoin Y$ in our present choice of~$W$ for~(Zleft)(i)
is automatic---by construction; and the choice of~$Z = \emptyset$ makes~(ii) and~(iii) there vacuous.

\medskip

\noindent
So we turn to verifying the other conditions needed for Theorem~\ref{theoremNewHomologyPropagation}:

Since~$O_p(G) = 1$, we have~$O_p \bigl( C_G(L) \bigr) = 1$ by Lemma~\ref{lemmaOpandp}(4).
Indeed by the latter part of hypothesis~(allC+H), we see that~$\tilde{H}_*(\Ap(K)) =  \tilde{H}_*(\ \Ap \bigl( C_G(L) \bigr)\ ) \neq 0$;
and thus we can a take a nonzero homology cycle~$\beta \in Z_n \bigl( \Ap(K) \bigr)$ for some~$n \geq -1$.
This establishes condition~(3) of Theorem~\ref{theoremNewHomologyPropagation}.

By hypothesis~(sLie-$p$), there exists a nonzero homology cycle~$\alpha \in Z_{m-1} \bigl( \Bp(L)\bigr)$, $m-1 \geq 0$, 
where~$m$ is the Lie rank of~$L$;
giving condition (1) of Theorem \ref{theoremNewHomologyPropagation}.
To finish, we need to show condition~(2) of Theorem~\ref{theoremNewHomologyPropagation} for such an~$\alpha$.
Pick~$a \in \alpha$; so~$a = \bigl( (R_1,1)< \ldots <(R_m,1) \bigr) \subseteq W$.
First, since we work with homology with rational coefficients,
we automatically get condition~(2)(a) of Theorem~\ref{theoremNewHomologyPropagation}.
For condition~(2)(b), we need to show that~$\Lk_M(a) \subseteq \K( \{ \max(a) \} \times Y)$.
Since~$M = \K(W)$, this is equivalent to showing that~$a$ is a full chain in~$W$,
and that if~$w > \max(a)$ then~$w \in \{ \max(a) \} \times Y$.

So let~$w \in W$ be an element comparable with every element of~ $a$.
Since~$a$ is a maximal chain of~$\Bp(L)$, we see that either~$w > \max(a)$, or else~$w \sqsubset \min(a)$ with~$w \notin X$.
In the former case, by the definition of the order in~$W$ we see that~$w = (\max(a) , y) \in \{ \max(a) \} \times Y$, for some~$y \in Y$.
In the latter case, $w =: D \in \F_G \bigl( L C_G(L) \bigr)$.
Since $D \sqsubset (R_1,1)$, we have~$C_{R_1 \cdot 1}(D) > 1$, and this forces~$D \in \Ap \bigl( N_G(L) \bigr)$.
Now the condition~$D \cap \bigl( L C_G(L) \bigr) = 1$ implies that~$D \in \Outposet_G(L) = \F_1$,
an equality we observed earlier using the (allC)-part of our hypothesis.
By definition of the order-relation in the poset~$W = \tilde{W}_G^{\BoucPoset}(L,C_G(L))$ in Proposition~\ref{propTildePosetsXandW},
we see that $D \sqsubset (E,F) \in X \ujoin Y$ implies~$C_{EF}(D) \nleq L$, so~$F > 1$, contrary to~$D \sqsubset (R_1,1)$ above. 
This contradiction shows that no such~$w$ exists, so this case cannot hold.
In conclusion, $\Lk_M(a) \subseteq \{ \max(a) \} \times Y$, giving condition~(2)(b) of Theorem~\ref{theoremNewHomologyPropagation}.
 
We have now fulfilled the conditions of Theorem~\ref{theoremNewHomologyPropagation},
so~$0 \neq \tilde{H}_*(M) = \tilde{H}_*(W) \groupiso \tilde{H}_* \bigl( \Ap(G) \bigr)$, completing the proof.
\end{proof}

Now we can prove our generic-Theorem~\ref{theoremLiep}:

\begin{proof}[Completing the proof of Theorem~\ref{theoremLiep}]
We saw that Case~(2)=(someNC+$\Phi$) holds by Lemma~\ref{lemmaCharPandFG-typeNC}; 
while the other Case~(1)=(allC+H) follows from Lemma~\ref{lemmaCharPAndConicalFields}:
since both Lemmas assume the initial hypothesis~(sLie-$p$) of the Theorem.

Thus we had already essentially completed the proof of the main assertion of the Theorem; 
so we turn to the final ``In particular'' statements:

\bigskip

Suppose first that~$G$ satisfies~(H1).
Then that assumption gives~(H-QC) for~$C_G(L)$; and for every centralizer~$C_G(LB)$, where~$B \in \Outposet_G(L)$.
Suppose further that~$L$~has type~(sLie-$p$), and that~($p$-$\Phi$) holds.
We see now that either~$O_p(C_G(LB)) > 1$ for all~$B \in \Outposet_G(L)$
(so that Case~(1)=(allC+H) holds, using~(H1) for~$C_G(L)$ as above),
or else that some~$B \in \Outposet_G(L) = \Outposet_G^\Phi(L)$ has~$O_p \bigl( C_G(LB) \bigr) = 1$---so
since~$C_G(LB)$ satisfies~(H-QC), we get Case~(2)=(someNC+$\Phi$). 
Then we have indeed reduced to the hypotheses of the main statement of the Theorem:
Namely~$G$ satisfies~(sLie-$p$); and one of~(allC+H) or~(someNC+$\Phi$). 
Hence~$G$ satisfies~(H-QC) by that main result.

\bigskip

Now suppose finally that we replace~``(H1)'' above with~``(H1u)+($p$ odd)''; 
in particular, we are still assuming that~$G$ also satisfies~(sLie-$p$) and~($p$-$\Phi$) (and hence~($p$-cyclic)).
As usual for~(H-QC), we assume~$O_p(G) = 1$,  and show that~$\tilde{H}_*(\Ap(G)) \neq 0$.

As we are assuming the inductive hypothesis~(H1u) of~\cite{AS93}, we will use the results of that article.
Since~$C_G(L) < G$, and the components of~$C_G(L)$ are components of~$G$ (by Lemma~\ref{lemmaOpandp}(3)),
we see that~$C_G(L)$ satisfies~(H-QC) by~(H1u).
Hence, if either~$\Outposet_G(L) = \emptyset$, or~$O_p \bigl( C_G(LB) \bigr) > 1$ for all~$B \in \Outposet_G(L)$,
then~(allC) holds,
and in fact we have reduced to the hypothesis~(1)=(allC+H) for the main statement of the Theorem,
and we get~(H-QC) for~$G$ just as before. 

Thus to finish, we may assume~(allC) fails,
so there is~$B \in \Outposet_G(L)$ such that~$O_p \bigl( C_G(LB) \bigr) = 1$.
And we need to show that~$\tilde{H}_* (\ \Ap \bigl( C_G(LB) \bigr)\ ) \neq 0$.
In view of our still-assumed conditions above, this will complete the additional hypothesis~(2)=(someNC+$\Phi$) as earlier, 
again reducing us to the main statement of the Theorem.

For this nonzero-condition, we will adapt arguments from~\cite{AS93} (notably Theorem~2.4 there),
to show that there is a replacement~$B' \in \Outposet_G(L)$ for~$B$ satisfying the nonzero-homology condition
for the desired centralizer. 
At some points in the argument,
(H1u) will directly imply~(H-QC) for~$G$, and in those cases we will not need to get nonzero homology for the centralizer.

Following Propositions~1.4 and~1.5 of~\cite{AS93}, if we have nontrivial~$N := Z(G)$ or~$Z(E(G))$,
then the components of~$G/N$ are covered by components of~$G$.
Since~$\tilde{H}_* \bigl( \Ap(G/N) \bigr) \subseteq \tilde{H}_* \bigl(\Ap(G) \bigr)$ by Lemma~0.12 of~\cite{AS93},
without loss of generality we can suppose that~$Z(G) = 1 = Z(E(G))$, since otherwise we get~(H-QC) for~$G$ using~(H1u).
%We show that $G$ satisfies (H-QC) either by the results of \cite{AS93}, or else by showing that condition (1) holds for some $B$.
%Note that by (H1u) and Propositions 1.4 and 1.5 of \cite{AS93}, we get that $Z(E(G)) = 1 = Z(G)$.
By our quoted elimination-result Theorem~\ref{alternativeParticularComponentes},
we can also assume that~$G$ has no component of type~$\Ln_2(2^3)$, $\U_3(2^3)$, or~$\Sz(2^5)$, for~$p=3,3,5$ respectively.

Now Proposition~1.6 of~\cite{AS93} allows us to reduce to the case~$O_{p'}(G) = 1$;
and in particular~$F(G) = 1$ using our Lemma~\ref{lemmaOpandp}.
Note that the argument in~\cite{AS93} that we are quoting here
depends on Theorem~2.4 there---and hence on Theorem~2.3 there,
which requires $p$~odd; so that we are now using that part of our current ``In particular'' hypothesis.
We are also using the elimination-result indicated above, again in order to apply~\cite[Thm~2.3]{AS93}.

Now we pick~$B\in \Outposet_G(L)$ as above.
At this point, we again invoke Theorem~2.4 of~\cite{AS93}, this time with~$C_G(L)$ in the role of~``$I$'' there;
so in (ii) below, as in the previous paragraph, we are once again using our hypothesis that~$p$ is odd, and our eliminations,  
in order to apply~\cite[Thm~2.3]{AS93}.
We check conditions~(i) and~(ii) of~\cite[Theorem~2.4]{AS93} for this~$I$, namely:

\bigskip

(i) Using our Lemma~\ref{lemmaOpandp}(4), $Z(I) \leq F(I) = F \bigl( C_G(L) \bigr) = F(G) = 1$.

\smallskip

(ii) The components of~$I$, which are components of~$G$ by Lemma~\ref{lemmaOpandp},
     have nonconical complements as defined in~\cite[Theorem~2.3]{AS93}:
     since we have~$p$ odd,
     and we eliminated components of~$G$ of type~$\Ln_2(2^3)$, $\U_3(2^3)$ or $\Sz(2^5)$ for $p=3,3,5$ respectively.

\bigskip

\noindent
This establishes conditions~(i) and~(ii) of~\cite[Theorem~2.4]{AS93} for~$I = C_G(L)$.
By that result, there exists a complement~$B' \in\Ap(IB)$ to~$I$ in~$IB$
such that~$1 = O_p \bigl( C_I(B')  \bigr) = O_p \bigl( C_G(LB') \bigr)$.
In particular, $B'\in \Outposet_G(L)$ and $O_p \bigl( C_G(LB') \bigr) = 1$.
Also unitary components of $C_I(B') = C_G(LB')$ satisfy the conditions of~(H1u) by the conclusions of~\cite[Theorem~2.4]{AS93}.
In consequence, by~(H1u), we see that~$C_G(LB')$ satisfies~(H-QC).
So~$B'$ fulfills the requirements of Case~(2)=(someNC+$\Phi$) of Theorem~\ref{theoremLiep};
reducing us, as desired, to the hypotheses of the main statement there---so that~$G$ satisfies~(H-QC).
\end{proof}

\begin{remark}
\label{rk:fullCases}
From Theorem \ref{theoremLiep}, if $G$ is a counterexample of minimal order to (H-QC) and contains a simple component $L$ of Lie type in characteristic $r = p$, then one of the following three (mutually exclusive) conditions should hold:
\begin{enumerate}
\item ($p$-cyclic) fails and there exists $B\in \Outposet_G(L)$ with $O_p(C_G(LB)) = 1$.

\item ($p$-cyclic) holds, $\Outposet_G(L)\neq\emptyset$ and $\Outposet^{\Phi}_G(L) = \emptyset$.

\item ($p$-cyclic) holds, $\Outposet^{\Phi}_G(L)\neq \emptyset $, every $B \in \Outposet^{\Phi}_G(L)$ has $O_p(C_G(LB))>1$, and there exists $D\in \Outposet_G(L) - \Outposet^{\Phi}_G(L)$ with $O_p(C_G(LD)) = 1$.
\end{enumerate}

Case (2) could be in fact ruled out by an application of Proposition 6.9 of \cite{PS}.
Namely, under the conditions of (2), if $B\in \Outposet_G(L)$, then $\Ap(C_L(B))$ has homological dimension less than that of $\Ap(L)$.
This can be proved by using a more detailed description of the structure of centralizers of graph and graph-field automorphisms: since $C_L(B)$ is a group of Lie type in characteristic $p$ but of Lie rank less than that of $L$.
See also Propositions 4.9.1 and 4.9.2 of \cite{GLS98}. 
Then the inclusion map $\Ap(C_L(B))\hookrightarrow \Ap(L)$ is zero in homology.
This fulfills the requirement for the application of Proposition 6.9 of \cite{PS}, where the value ``$k$" there is given by the Lie rank of $L$ minus $1$.
\end{remark}

\bigskip
\bigskip

\section{Eliminating QD-components under~(H1)}
\label{sec:elimQCunderH1}

In this section, we provide in Theorem~\ref{QDpReductionPS} an alternative version of Proposition~1.7 of~\cite{AS93}
(namely elimination of $\QD$-components)---that works for any prime $p$, under a somewhat different inductive hypothesis.
Our proof by contrast does not invoke the~CFSG.
We then use Theorem~\ref{QDpReductionPS}
in our proof, at the end of Section~\ref{sec:2varASforp=35}, 
of the variant Theorem~\ref{theoremASExtensionAlternative}
of our more direct extension Theorem~\ref{theoremASExtension} to $p = 3,5$ of the Aschbacher-Smith Main Theorem.
 
\bigskip

But first, we will recall the original result~\cite[Prop~1.7]{AS93}: which roughly establishes that under~(H1u) for odd~$p$,
either~$G$ satisfies~(H-QC), or every component of~$G$ has a~$p$-extension failing~$\QD_p$.
Namely, we give in Proposition~\ref{QDpReduction} below
a slightly different alternative formulation of Proposition~1.7 of~\cite{AS93}---which is closer to the original argument
than our Theorem~\ref{QDpReductionPS}, and which we expect could be applied in wider contexts, beyond the present paper.

The proof of Proposition~1.7 of~\cite{AS93} uses Theorems~2.3 and~2.4 there,
to obtain the nonconical-complement~$A$ needed for homology propagation.
In Hypothesis~(MaxNC) for our version in the Proposition below,
we in effect isolate that nonconicality-argument (and in particular, its dependence on having odd~$p$)---by
roughly encoding the maximal choices made for~$A$ in the original Aschbacher-Smith proof.

\begin{proposition}
\label{QDpReduction}
Suppose that $G$ satisfies the following conditions:
%\begin{enumerate}
%\item[($p$-ext)] $G$ contains a simple component $L$ of order divisible by $p$ for which all its $p$-extensions satisfy $\QD_p$,
%\item[(Maxnonconical)] there exists $A\in\Ap(N_G(L))$ faithful on $L$ and of maximal $p$-rank, such that $|A\cap L|$ is maximal and $\tilde{H}_*(\Ap(C_G(LA))) \neq 0$.
%\end{enumerate}

\vspace{0.2cm}
\begingroup
\renewcommand{\arraystretch}{2}
\begin{tabular}{ccl}
($p$-extQD) & & \pbox{10cm}{$G$ contains a simple component~$L$ of order divisible by~$p$
                            for which all the~$p$-extensions satisfy~$\QD_p$;} \\
(MaxNC)     & & \pbox{10cm}{There exists~$A \in \Ap(N_G(L))$ faithful on~$L$ and of maximal~$p$-rank,
                            such that~$|A \cap L|$ is maximal and~$\tilde{H}_* (\ \Ap \bigl( C_G(LA) \bigr)\ ) \neq 0$.}
\end{tabular}
\endgroup

\vspace{0.2cm}

Then~$G$ satisfies~(H-QC).

In particular, if $p$ is odd, ($p$-extQD) holds, and~(H1) or~(H1u) holds, then~$G$ satisfies~(H-QC).
\end{proposition}

\begin{proof}
The proof of this Proposition essentially follows the original proof given by Aschbacher-Smith
for their Proposition~1.7 on pages~489--490 of~\cite{AS93}.
But note that Step~iv of that original proof is now replaced by the hypothesis~(MaxNC);
that is, here we in effect assume the conclusion
of Theorem~2.4 of~\cite{AS93}; so we are not using their Theorem~2.3, which would require~$p$ odd.
Also we do not need to invoke the~CFSG since we have~(MaxNC). 

For the ``In particular'' part under~(H1):
We can use the earlier reductions under~(H1) that we quoted in our Theorem~\ref{generalReduction}---first
parts~(1,2), to establish~1.4--1.6 of~\cite{AS93}; 
and then parts~(5,6) to eliminate the problematic components~$\PSL_2(2^3)$, $\PSU_3(2^3)$ and~$\Sz(2^5)$, $p=3,3,5$ respectively.
Then we may use the remainder of the original proofs of Theorems~2.3 and~2.4 of~\cite{AS93},
to establish (MaxNC)---completing the reduction to the hypotheses of the main statement, and hence giving~(H-QC).

Under~(H1u):
The proof is essentially given using the adjustment to~\cite{AS93}
that we used for the corresponding ``In particular'' statement in Theorem~\ref{theoremLiep} above; 
which we now summarize only briefly:
Here we again get~(MaxNC) from Theorem~2.4 of~\cite{AS93}---recalling that we may apply Theorem~2.3 there,
by first eliminating the above components~$\PSL_2(2^3)$, $\PSU_3(2^3)$ and~$\Sz(2^5)$, for~$p=3,3,5$ respectively,
now via~\cite[1.5]{AS93} and our quoted result Theorem~\ref{alternativeParticularComponentes} under~(H1u),
and then reducing to~$O_{p'}(G) = 1$ via~\cite[1.6]{AS93}. 
\end{proof}

We now provide, as Theorem~\ref{QDpReductionPS} below,
a further alternative version of Proposition~1.7 of~\cite{AS93}, that works for \textit{any} prime~$p$.
In particular, we do not invoke Theorems~2.3 and~2.4 of~\cite{AS93}:
which depend on the CFSG; and assume~$p$ odd;
and require the exclusion of components of type~$\PSL_2(2^3)$, $\PSU_3(2^3)$, and~$\Sz(2^5)$, $p=3,3,5$ respectively.
The proof of our theorem instead relies on the combinatorial properties of the~$\Ap$-posets, and does not invoke the~CFSG.

In contrast with Proposition~\ref{QDpReduction} just above,
below we do not assume~$\tilde{H}_*(\Ap \bigl( C_G(LA) \bigr)) \neq 0$,
or some~$A \in \Ap \bigl( N_G(L) \bigr)$ satisfying subtle conditions of maximality.
Instead, we will develop fairly elementary conditions of maximality for~$A$, within the proof---and then
we will see roughly that we can homotopically remove any such~$A$ which would have zero-homology for~$\Ap \bigl( C_G(LA) \bigr)$
(that is, guaranteeing that our~$A$ produces a nonconical complement). 

As a brief summary of the proof: because of the hypothesis~($p$-extHQC) of Theorem~\ref{QDpReductionPS} below,
zero-homology for~$\Ap \bigl( C_G(LA) \bigr)$ would force~$O_p \bigl( C_G(LA) \bigr) > 1$.
This guarantees the contractibility of links in~$\Ap(G)$,
so that by Proposition~\ref{propositionContractibleCentralizers} we can remove such points~$A$ from our poset---getting
a homotopy equivalent smaller poset~$X$, to which we can apply the earlier propagation-result
which we quoted as Lemma~\ref{lemmaHomologyPropagationKP20}.
Any such remaining~$A$ will have the nonconical requirement for hypothesis~(v) of Lemma~\ref{lemmaHomologyPropagationKP20};
and the removals also guarantee the condition on~$X_{>A}$ needed in hypothesis~(iv) of Lemma~\ref{lemmaHomologyPropagationKP20}.

\begin{theorem}
\label{QDpReductionPS}
Suppose that $G$ satisfies the following conditions:

\vspace{0.2cm}
\begingroup
\renewcommand{\arraystretch}{2}
\begin{tabular}{ccl}
($p$-extQD) & & \pbox{11.5cm}{$G$ contains a simple component~$L$ of order divisible by~$p$
                              for which all of the~$p$-extensions~$LB \leq G$, $B\in\hat{\Outposet}_G(L)$, satisfy~$\QD_p$;} \\
($p$-extHQC) & & \pbox{11.5cm}{If~$LB \leq G$ is a~$p$-extension of~$L$, then~$C_G(LB)$ satisfies~(H-QC).}
\end{tabular}
\endgroup

\vspace{0.2cm}

Then~$G$ satisfies~(H-QC).

Notice in particular that assuming~(H1) guarantees~($p$-extHQC).%
%\footnote{
%  fakenote:
%  I added this "in particular" statement, right?  
%  (If I understand, we use this under~(H1) in sec10?)
%          }
\end{theorem}

\begin{proof}
As usual for~(H-QC), we suppose that~$O_p(G) = 1$,
and show that~$\tilde{H}_* \bigl( \Ap(G) \bigr) \neq 0$ under the hypotheses of the Theorem.
We use the notation of Proposition~\ref{propositionContractibleCentralizers},
in particular recalling the undesirable conical-subset of~$p$-outers defined by:
  \[ \F_1 := \{\ B \in \Outposet_G(L)  \tq O_p \bigl( C_G(LB) \bigr) > 1\  \} . \]
Eventually we will choose an~$A$ exhibiting~$\QD_p$ for a suitable~$LB$,
so this~$LB$ will need to instead be a member of the corresponding nonconical-subset:
  \[\mathcal{E} := \{ LE  \tq E \in \hat{\Outposet}_G(L) - \F_1\} . \]
Note that~$L \in \mathcal{E}$ (the case where~$E = 1$)
since~$O_p \bigl( C_G(L) \bigr) = 1$ by Lemma~\ref{lemmaOpandp}(4); in particular, we have~$\mathcal{E} \neq \emptyset$.

Now in pursuing a suitable~$LB$, it is natural to pick a configuration~$LE \in \mathcal{E}$,
with~$m_p(LE)$ maximal among the~$p$-ranks of the groups in~$\mathcal{E}$.
Then we take~$LE$ of maximal order subject to this property;
in particular then~$E$ is a maximal element of~$\hat{\Outposet}_G(L) - \F_1$.
Just from the definition of~$E$, as a~$p$-outer of the component~$L$ not lying in~$\F_1$, we have:%
%\footnote{
%  fakenote:
%  I adjusted the wording here, in view of your email of~17nov21;
%  and added a reference for the second statement, having added the general statement to the end of earlier~2.5...
%  Please check me!
%          }

\vspace{0.2cm}

\textbf{Claim 1.} $O_p \bigl( C_G(LE) \bigr) = 1$,
and hence~$O_p(LE) = 1$ using Lemma~\ref{lemmaTrivialOpPropagationCentralizer}.

\vspace{0.2cm}

\noindent
Next we develop some conditions for suitable~$A$ and~$B$ in $\Ap(LE)$.
Notice that any~$A \in \Ap(LE)$ of maximal~$p$-rank intersects~$L$ non-trivially:
since~$p$ divides the order of~$L$ by~($p$-extQD).
We see:

\vspace{0.2cm}

\textbf{Claim 2.} We now take~$A \in \Ap(LE)$ of~$p$-rank~$m_p(LE)$ such that~$|A \cap L|$ is maximal.
Then~$| A \cap L | > 1$; and in particular, using Claim~1 we have~$O_p(LA) \leq O_p(LE) = 1$.

\vspace{0.2cm}

\noindent
Now decompose~$A = (A \cap L) \times B$.
Since~$O_p(LE) = 1$ by Claim~1, $LE$ embeds into~$\Aut_G(L)$; so we see~$LA = LB \leq LE$.
Now by Claim~1 and Lemma~\ref{lemmaTrivialOpPropagationCentralizer},
we get~$O_p(LB) = 1 = O_p \bigl( C_G(LB) \bigr)$; thus~$LB \in \E$.
In particular, $B \in \hat{\Outposet}_G(L)$ and we get:
  \[ |B| = |A / A \cap L| = |LA| / |L| \leq |LE| / |L| = |E| , \]
so that using our choice in Claim~2 with~$m_p(A) = m_p(LE)$ we have:
  \[ m_p(LE) = m_p(A) \leq m_p(LB) \leq m_p(LE), \text{ and hence equality holds for these ranks.} \]

Finally we show that we can indeed take this~$A$ to exhibit~$\QD_p$ for~$LB$:
Since we saw above that~$O_p(LB) = 1$, 
we know by our hypothesis~($p$-extQD) that~$LB$ has~$\QD_p$---exhibited by some~$A' \in \Ap(LB)$ of~$p$-rank~$m_p(LB) = m_p(A)$.
If we write~$A' = (A'\cap L)\times B'$, then:
  \[ |B'| = |A' / A' \cap L| = |LA'| / |L| \leq |LA| / |L| = |B|,\]
and by maximality of~$|A \cap L|$ we also have:
  \[ |A' \cap L| \leq  |A \cap L| . \]
Now~$m_p(A') = m_p(A)$ implies that the two inequalities above must be equalities;
that is, we must have~$|B'| = |B|$ and~$|A' \cap L| = |A \cap L|$.
Therefore, on replacing our original~$A$ by this~$A'$ which satisfies the same maximality properties,
we can suppose without loss of generality that~$A$ exhibits~$\QD_p$ for~$LB$.
This occurs via some nonzero homology cycle, say~$\alpha$ involving a full chain~$a$ ending in~$A$, of~$\Ap(LB)$.

\bigskip

\noindent
We now set up to verify the hypotheses for Lemma~\ref{lemmaHomologyPropagationKP20}.

Recall from Proposition~\ref{propositionContractibleCentralizers} the definition of
the subset~$\N_1$ of members of the inflation~$\N_G(L)$ acting faithfully on~$L$, and containing some member of~$\F_1$:
 \[ \N_1  :=  \{\ D \in \Ap \bigl( N_G(L) \bigr)\ \tq\ D \cap L > 1 ,\ C_D(L) = 1 ,\ \exists B_0 \in \F_1 \tq B_0 < D\ \}. \] 
Set~$X := \Ap(G) - \N$, where:
  \[ \N := \{ D \in \N_1 \tq m_p(LD) > m_p(LE) \} . \]
Note that~$\N$ is an upward-closed subposet of~$\N_1$.
Therefore~$X \hookrightarrow \Ap(G)$ is a homotopy equivalence by Proposition~\ref{propositionContractibleCentralizers}.

Set~$H := LB$ and $K := C_G(LB)$.
Since elements of~$\N$ are faithful on~$L$,
We immediately get conditions~(i) and~(ii) of Lemma~\ref{lemmaHomologyPropagationKP20}.

Furthermore our cycle~$\alpha$ above exhibiting~$\QD_p$ for~$LB=H$ now extends to a cycle for~$X$:
For note that if~$D \in \Ap(LB)$ is involved in~$\alpha$, then we get~$m_p(LD) \leq m_p(LB) = m_p(LE)$;
so by definition, $D$ cannot be contained in~$\N$.
This proves the following claim:

\vspace{0.2cm}

\textbf{Claim 3.} If~$D \in \Ap(LB)$ is involved in~$\alpha$, then~$D \in X$.
In particular, $\alpha$ is also a cycle of~$X$ containing a (full) chain~$a$ with~$A = \max(a)$.
Hence we get condition~(iii) of Lemma~\ref{lemmaHomologyPropagationKP20}.

\vspace{0.2cm}

\textbf{Claim 4.} If~$D \in X_{>A}$, then~$C_D(LB) > 1$.
Hence we get condition~(iv) of Lemma~\ref{lemmaHomologyPropagationKP20}.

\begin{proof}
Suppose by way of contradiction that~$D \in X_{>A}$, but $1 = C_D(LB) = C_D(LA) = C_D(L)$.
Since~$D \cap L \geq A \cap L > 1$, we get~$D \leq N_G(L)$.

Now suppose by way of further contradiction that we have some~$d \in (D \cap LC_G(L)) - L$;
write~$d = lc$, with~$l \in L$ and $c \in C_G(L)$, $|c| = p$.
Since~$[d,A]= 1$, we get~$[l,A] = 1 = [c,A]$.
But~$l \in L \leq LA$, and~$A$ is of maximal~$p$-rank in~$LA$, so~$l \in A \leq D$.
Therefore $c\in D$, that is, $c \in C_D(L) = 1$, contrary to~$c$ of order~$p$ above.
This intermediate contradiction shows that~$D \cap  LC_G(L) = D \cap L$.

So we may write~$D = (D \cap L) \times D'$, where now~$D' \in \hat{\Outposet}_G(L)$.
Since~$D \notin \N$, while we have~$m_p(LD) \geq m_p(D) > m_p(LE) = m_p(A)$, we must have~$D \notin \N_1$.
In particular, $D' \notin  \F_1$; and hence~$LD = LD' \in \mathcal{E}$ with~$m_p(LD') > m_p(LE)$.
But this contradicts the earlier maximality of~$m_p(LE)$ in~$\E$.
This final contradiction shows that~$1 < C_D(L) = C_D(LB)$, as we wanted.
\end{proof}

Finally we saw earlier that~$O_p(K) = O_p \bigl( C_G(LB) \bigr) = 1$.
Hence we get condition~(v) of Lemma~\ref{lemmaHomologyPropagationKP20}, using our hypothesis~($p$-extHQC).

\vspace{0.2cm}

In conclusion, we have shown that~$X \simeq \Ap(G)$ with~$H = LB$ and~$K = C_G(LB)$
satisfy the hypotheses of Lemma~\ref{lemmaHomologyPropagationKP20}.
Thus~$\tilde{H}_* \bigl( \Ap(G) \bigr) \groupiso \tilde{H}_*(X) \neq 0$, as required.
\end{proof}

\bigskip
\bigskip

\section{Robinson subgroups and the Lefschetz-module version of Quillen's Conjecture}
\label{sectionLefschetz}

In this Section, we recall how to establish the Lefschetz-module version~(L-QC) of Quillen's Conjecture
via the construction of \textit{Robinson subgroups};
cf.~the conditions~(Rob-xx) described in our discussion of strategy in Remark~\ref{rk:stratelimRob}.

In our context, the idea is to show that for a minimal counterexample to~(H-QC),
once we have eliminated (e.g.~via propagation-methods) as many simple components as possible,
we can in fact establish this stronger version~(L-QC) of the conjecture, and so obtain a final contradiction.
The key point is that we can reduce the study of this conjecture to analyzing particular aspects of the components of the group.
To that end, we take advantage of the very restrictive structures of any remaining components in a minimal counterexample to~(H-QC).
The idea goes back to Robinson~\cite{Rob};
and was later exploited by Aschbacher-Kleidman to establish the almost-simple case of the conjecture~\cite{AK90}
(which we have quoted as Theorem~\ref{almostSimpleQC}),
and then by Aschbacher-Smith~\cite{AS93} to conclude the proof of their Main Theorem
(compare~(Rob-nonQD) in Remark~\ref{rk:stratelimRob}).

\bigskip

We denote by~(L-QC) the following version of Quillen's conjecture in terms of the (reduced) Lefschetz module:

\vspace{0.3cm}
\begin{flushleft}
(L-QC) \quad If~$O_p(G) = 1$, then $\tilde{L}_G \bigl( \Ap(G) \bigr) := \sum_{n \geq -1}\ (-1)^n [ C_n \bigl( \Ap(G) \bigr) ] \neq 0$.
\end{flushleft}
\vspace{0.3cm}

\noindent
Recall that~$C_*(X)$ denotes the augmented chain complex of a poset~$X$ with coefficients in the rational numbers.
The above alternating sum is taken in the {\em Grothendieck ring\/} of representations of~$G$ over~$\QQ$,
with  brackets denoting the character of a representation,
and~$\tilde{L}_G(\Ap(G))$ is the {\em reduced Lefschetz module\/} of~$\Ap(G)$.

By the Hopf trace formula, we get a homology-version of the Lefschetz module:
  \[ \tilde{L}_G(\Ap(G)) = \sum_{n\geq -1} (-1)^n [\tilde{H}_n(\Ap(G))] . \]
In particular, we see that~(L-QC) implies~(H-QC).
Moreover, the elements of the Grothendieck ring of~$G$ are completely determined by their (virtual) characters,
so $\tilde{L}_G \bigl( \Ap(G) \bigr) = 0$ if and only if its virtual character is~$0$.
By the Lefschetz fixed~point formula, the value of this character in a given element~$g \in G$ is~$\tilde{\chi}(\Ap(G)^g)$,
the reduced Euler characteristic of the fixed point subposet~$\Ap(G)^g$.
Therefore, (L-QC) is equivalent to the following version of the conjecture.

\vspace{0.3cm}

\begin{flushleft}
(L$^{\prime}$-QC) \quad If~$O_p(G) = 1$, then~$\tilde{\chi} \bigl( \Ap(G)^g  \bigr) \neq 0$ for some~$g \in G$.
\end{flushleft}

\vspace{0.3cm}

\noindent
Following~\cite{AK90,AS93,Rob}, in some particular situations we will establish~(L$^{\prime}$-QC), and hence~(H-QC),
by exhibiting an element~$g \in G$ such that~$\tilde{\chi} \bigl (\Ap(G)^g \bigr) \neq 0$.
In fact, we will show that there exists a suitable~$p'$-subgroup~$Q \leq G$ of the form~$Q = \gen{g} \times O_q(Q)$, with~$q$ a prime,
and:
  \[ \tilde{\chi} \bigl( \Ap(G)^g \bigr) \equiv \tilde{\chi} \bigl( \Ap(G)^Q \bigr) \not\equiv 0 \pmod{q} . \]
Note that~$Q$ is a {\em $q$-elementary\/} group, that is, a direct product of a cyclic group with a~$q$-group.
The point of considering $q$-elementary~$p'$-subgroups~$Q$ is
that the fixed point subposet~$\Ap(G)^Q$ can be smaller and easier to control.
For example, we typically look for~$\Ap(G)^Q$ empty.

With the aim of proving~(H-QC), we will try to establish~(L-QC)
for the groups~$G$ not satisfying the hypotheses of some of the theorems of the previous sections,
such as Theorems~\ref{generalReduction}, \ref{PStheorem}, and~\ref{QDpReductionPS}, and Theorem~\ref{theoremLiep};
(compare~(Rob-nonelim) in Remark~\ref{rk:stratelimRob}).
For these groups~$G$, we have several restrictions on the components;
and possibly a concrete list such as the $\QD$-List Theorem~\ref{theoremQDList} when~$p$ is odd.
The goal is to construct~$q$-elementary~$p'$-subgroups~$Q_L$ for each component~$L$ of~$G$,
in such a way that we can produce a~$q$-elementary~$p'$-subgroup~$Q$ of~$G$ satisfying the extra property:%
\footnote{
  Below we are using product homology as in Remark~\ref{rk:prodhomolforjoins};
  notice this is a crucial use of the basic group-theory fact that~$E(G)$ is a central product---even
  a direct product, in our usual situation of Lemma~\ref{lemmaOpandp}(1).
          }
\begin{equation}
\label{eqHomotopEquivFixedPoints}
  \Ap(G)^Q \simeq \Ap(L_1)^{Q_{L_1}} \join \ldots\join \Ap(L_t)^{Q_{L_t}},
\end{equation}
where~$L_1,\ldots,L_t$ are the components of~$G$.
Then using product homology:
\begin{equation}
\label{eqLQCViaReducedEuler}
  \tilde{\chi}(\Ap(G)^Q) \equiv (-1)^{t-1} \prod_{i=1}^t \tilde{\chi}(\Ap(L_i)^{Q_{L_i}}) \pmod{q}.
\end{equation}
Hence if~$\tilde{\chi}(\Ap(L_i)^{Q_{L_i}}) \not \equiv 0 \pmod{q}$ for all~$i$, then~$\tilde{\chi}(\Ap(G)^Q) \not \equiv 0 \pmod{q}$.
This reduces the problem of establishing~(L$^{\prime}$-QC) for~$G$ to understanding the structural properties on the components~$L_i$,
that allow us to construct these subgroups $Q_{L_i}$.
These~$q$-elementary~$p'$-subgroups~$Q$ are also known as~\textit{Robinson subgroups},
the terminology employed in~\cite{AS93} to indicate the origin of the method in~\cite{Rob}.

Now we will study properties that these $Q_{L_i}$ have to satisfy,
in order to get such subgroup $Q$ satisfying conditions~(\ref{eqHomotopEquivFixedPoints}) and~(\ref{eqLQCViaReducedEuler}).
Note that each~$Q_i$ is~$q$-elementary, for a fixed prime~$q$ common to all the~$i$:
that is, we fix the same prime~$q \neq p$ for all the components.
For example, when~$p$ is odd, the natural choice is~$q = 2$.
When~$p = 2$, one fairly natural choice could be~$q = 3$, since most of the simple groups have order divisible by~$3$.

Thus for our extension of~\cite{AS93},
we want to extend the context of Theorem~5.3 at the end of Section~1 there by considering the following properties:

\vspace{0.2cm}

\begin{flushleft}
\textbf{Property~$\R_0(p)$}:
        $L$  is a simple group, having a subgroup~$Q_L$ such that for some prime~$q \neq p$:
\end{flushleft}
\begin{enumerate}
\item $Q_L$ is a~$q$-elementary~$p'$-subgroup.
\item $O_q(Q_L) > 1$.
\item $\tilde{\chi} \bigl( \Ap(L)^{Q_L} \bigr) \not\equiv 0 \pmod{q}$.
\item $\Ap \bigl( C_{\Aut(L)}(Q_L) \bigr) \subseteq \Ap(L)$.
\end{enumerate}
\textbf{Property $\R(p)$}: $Q_L$ satisfies~(1--4) as in~$\R_0(p)$ above, and further:
\begin{enumerate}
\item[(5)] $q = 2$ when~$p \neq 2$; while~$q = 3$ when~$p = 2$.
\end{enumerate}

\vspace{0.3cm}

\noindent
We pause to state some consequences of~(4) above:

  \quad (4+)  If~$A \in \Ap \bigl( \Aut(L) \bigr)^{Q_L}$, then $A \cap L > 1$. 
    In particular, $\Ap \bigl( \Aut(L) \bigr)^{Q_L} \simeq \Ap(L)^{Q_L}$.
      
\noindent
For suppose by way of contradiction that $A \cap L = 1$: 
Then~$[A , Q_L] \leq A \cap L = 1$, so~$A = C_A(Q_L)$, and then~(4) forces~$A \in \Ap(L)$, contrary to our assumption $A \cap L = 1$;
so this contradiction shows that~$A \cap L > 1$.
In particular, $A \mapsto A \cap L$ is a poset endomorphism of $\Ap \bigl( \Aut(L) \bigr)^{Q_L}$ with image~$\Ap(L)^{Q_L}$, 
giving the final statement via Lemma~\ref{lm:increndo}(2).

Furthermore we will often be able to establish the stronger condition:%
%\footnote{
%  fakenote:
%  Since we seem to use this a lot, I thought it would be better to state it more formally...
%          }
 
 \quad (3+) $\Ap \bigl( \Aut(L) \bigr)^{Q_L} = \emptyset$; which implies~(3) and~(4): 

\noindent
For certainly~(3+) gives~$\Ap(L)^{Q_L} = \emptyset$, with the value of~$-1$ for~(3); 
and gives~$\Ap(C_{\Aut(L)}(Q_L) ) = \emptyset$, so that~(4) holds vacuously.

\bigskip

In the following extension of \cite{AS93},
we show that we can establish~(L-QC) if the components of~$G$ satisfy~$\R(p)$.
It roughly states that, under those very particular Robinson-type conditions,
we can propagate~(L-QC): from the components of~$G$, to~$G$ itself.
An important observation here is that this propagation depends exclusively on the individual structure of each simple component~$L$.
The proof is based on the original argument given in~\cite{AS93}---at the end of Section~1 there.

\begin{proposition}
\label{propRobinson}
Suppose that~$O_p(G) = 1 = O_{p'}(G)$, with~$F^*(G) = L_1 \ldots L_t$, and let~$q$ be a fixed prime distinct from~$p$.
Assume that every component~$L_i$ of~$G$ satisfies Property~$\R_0(p)$ for some~$q$-elementary subgroup~$Q_i \leq L_i$.
Then~$G$ satisfies~(L-QC).
\end{proposition}

\begin{proof}
In overview: we construct a single group~$Q$, to satisfy~(\ref{eqHomotopEquivFixedPoints}) and~(\ref{eqLQCViaReducedEuler});
roughly, we need to establish for~$Q$ properties like~$\R_0(p)$---but simultaneously, that is, with respect to
{\em all\/}  the components~$L$ in the product~$E(G) = F^*(G)$.

Write~$Q_i := \gen{g_i} \times O_q(Q_i)$, and let~$g := g_1 \ldots g_t$, $C := \gen{g}$, $R := O_q(Q_1) \times \ldots \times O_q(Q_t)$,
and~$Q := C \times R$.
Then~$Q$ is a~$q$-elementary~$p'$-subgroup of~$G$ by conditions~(1,2) of Property~$\R_0(p)$ of each~$Q_i \leq L_i$.
We show now that~$\tilde{\chi}(\Ap(G)^g) \neq 0$ by showing that this integer is not zero mod~$q$.
Recall that~$\tilde{\chi} \bigl( \Ap(G)^g \bigr) = \tilde{\chi} \bigl( \Ap(G)^C \bigr) \equiv \tilde{\chi}(\Ap(G)^Q) \pmod{q}$
as~$Q/C = R$ is a~$q$-group.
Therefore, it is enough to establish that~$\tilde{\chi} \bigl( \Ap(G)^Q \bigr) \not\equiv 0 \pmod{q}$.

Consider any~$E \in \Ap(G)^Q$.
By coprime action, $E = C_E(Q)[E,Q]$.
Moreover, we note that~$[E,Q] \leq E \cap F^*(G)$.
We first eliminate a potentially troublesome case for~$E$:

\vspace{0.2cm}

\textbf{Case 1.} We cannot have~$E \cap F^*(G) = 1$.

\noindent
For in this case, we have~$E = C_E(Q)$,
and in particular~$E \leq C_G \bigl( O_q(Q_i) \bigr)$ for all~$i$.
Since~$O_q(Q_i)$ is a nontrivial subgroup of~$L_i$ by~(2) of condition~$\R_0(p)$, we see that~$E \leq N_G(L_i)$.
Fix some~$i$, and decompose~$E$ as the direct product of~$E \cap \bigl( L_i C_G(L_i) \bigr)$ with some complement~$E_0$.
Then~$E_0$ induces outer automorphisms on~$L_i$, and~$Q_i$ centralizes~$E_0$.
Thus we can embed $E_0$ into $\Aut(L_i)$ and we get~$E_0 \leq C_{\Aut(L_i)}(Q_i)$.
However, (4) of condition $\R_0(p)$ along with our case hypothesis that~$E \cap F^*(G) = 1$ forces~$E_0 = 1$.
Therefore $E \leq L_i C_G(L_i)$.
Varying $i$, we see that
  \[ E \leq \bigcap_i L_i C_G(L_i) = F^*(G) C_G \bigl( F^*(G) \bigr) = F^*(G) . \]
This is contrary to our case hypothesis that~$E \cap F^*(G) = 1$,
and this contradiction completes the proof that this case cannot arise. 

Thus all~$E$ must satisfy the remaining case:

\vspace{0.2cm}

\textbf{Case 2.} Every~$E \in \Ap(G)^Q$ intersects~$F^*(G)$ non-trivially.

\noindent
Thus via Lemma~\ref{lm:increndo}(2), we have a homotopy equivalence given by the poset endomorphism:
  \[\Ap(G)^Q \to \Ap \bigl( F^*(G) \bigr)^Q, \quad E \mapsto E\cap F^*(G) , \]
with the homotopy inverse given by the inclusion.

Finally, note that the direct product of components gives us a homotopy equivalence:
 \[ \Ap \bigl( F^*(G) \bigr)^Q = \Ap(L_1 \times\ldots\times L_t)^Q \simeq (\Ap(L_1)* \ldots * \Ap(L_t))^{Q}
                               = \Ap(L_1)^{Q_1}* \ldots * \Ap(L_t)^{Q_t} . \]
That is, we have established~(\ref{eqHomotopEquivFixedPoints}).
For~(\ref{eqLQCViaReducedEuler}), 
we compute the reduced Euler characteristics of these fixed point subposets:
Recall using product homology as in Remark~\ref{rk:prodhomolforjoins}
that the Euler characteristic of a join of spaces is the product of the Euler characteristics of the spaces involved,
up to a dimension-shift sign:
  \[ \tilde{\chi} \bigl( \Ap(G)^Q \bigr) = \tilde{\chi} \bigl(\ \Ap(L_1)^{Q_1}* \ldots*\Ap(L_t)^{Q_t}\ \bigr)
                                         = (-1)^{t+1}\ \prod_i\ \tilde{\chi} \bigl( \Ap(L_i)^{Q_i} \bigr) . \]
The rightmost term is nonzero mod $q$ by~(3) of condition $\R_0(p$) on the~$Q_i$-subgroups.
In consequence, $\tilde{\chi}(\Ap(G)^g)\neq 0$; and so~$G$ satisfies~(L-QC).
\end{proof}

In~\cite{AS93}, Robinson-methods are used to treat components~$L$ which fail the~$\QD_p$-property for odd primes $p$,
as given in the~$\QD$-list, which we have quoted as Theorem~\ref{theoremQDList}.
In this context, simple groups of Lie type are a major exception tor~ $\QD$:
For example, in that list we see that there are basically three sub-families of groups of Lie type which might fail~$\QD_p$ for odd~$p$:
Lie-type groups in the same characteristic~$p$;
Lie-type groups in characteristic~$2$; and unitary groups~$\U_n(q)$ with~$p \mid q+1$.
There are also some assorted particular cases (arising from exceptional or twisted groups).

The following Proposition shows in particular
that groups of Lie type in characteristic~$r = 2$ satisfy Property~$\R(p)$ for odd~$p$.
Consequently we will then have to focus our attention
on studying the remaining components of Lie type described in the previous paragraph.
Most of the work for the Proposition was already  done in~\cite[Theorem~5.3]{AS93} for odd~$p$.

The proof of the Proposition is based on~\cite[Lemma~4]{AK90}.
However, the proof of that Lemma seems to have a small gap, which we show below how to fix.
We are grateful to R.~Solomon, who provided us some references on situations related to the Borel-Tits Theorem.

\begin{proposition}\label{propRobinsonEvenChar}
Suppose that~$L$ is a simple group of Lie type in characteristic~$r \neq p$.
Suppose in addition that~$L\neq B_2(2)' (\groupiso \AA_6)$.
Then a Sylow~$r$-subgroup~$X$ of~$L$ satisfies conditions~(1--4) of Property~$\R_0(p)$, where we take~$r$ as~$q$ in that definition.
Moreover, $\Ap \bigl( \Aut(L) \bigr)^X = \emptyset$.
In particular, if ($p$ is odd, and~$r = 2$) or ($p=2$, and~$r=3$),
we even have Property $\R(p)$ with $Q_L$ a Sylow $r$-subgroup of~$L$.
\end{proposition}

\begin{proof}
Let $X$ be a Sylow $r$-subgroup of $L$.
We show that $X$ satisfies conditions~(1--4) of Property~$\R_0(p)$ with~$q = r$, by using our hypothesis that~$r \neq p$:

Clearly~$X$ satisfies conditions~(1,2) of Property~$\R_0(p)$.
On the other hand, since $P : = N_L(X)$ is a Borel subgroup of~$L$, we have:
  \[ C_{\Aut(L)}(X) \leq X, \]
by \cite[(13-2)]{GL83}, except when~$L = B_2(2)' \groupiso \AA_6$.%
\footnote{
  In the latter case,
  the centralizer of a Sylow~$2$-subgroup of~$B_2(2)'$ in~$\Aut \bigl( B_2(2)' \bigr)$ is isomorphic to~$C_2 \times C_2$,
  and contains nontrivial elementary abelian~$2$-subgroups inducing outer automorphisms on~$L = B_2(2)'$.
          }
This shows that~$X$ satisfies condition~(4) under our hypothesis that~$L \neq B_2(2)'$.

It remains to establish condition~(3) of Property~$\R_0(p)$.
Indeed we will show the stronger condition~(3+) mentioned earlier,
namely~$\Ap \bigl( \Aut(L) \bigr)^X = \emptyset$,
via the ideas of~\cite[Lemma~4]{AK90}:

We first show~$\Ap(L)^X  =\emptyset$ (which in fact is enough for~(3)).
Assume by way of contradiction that we have some~$E \in \Ap(L)^X$.
Let~$M := EX \leq L$.
At this point, \cite{AK90} claims that~$M$ lies in a parabolic by the Borel-Tits theorem.
However, Borel-Tits does not apply here, since~$r \neq p$.
So instead, we need to quote some further development, such as~\cite{SST}:
Since~$|L:M|_r = 1$, we have~$O_r(M) > 1$ by Proposition~C of~\cite{SST} (see also~\cite{Seitz}).
We may now apply the Borel-Tits theorem:
to see that there exists a parabolic subgroup~$P$ of~$L$ such that~$O_r(M) \leq O_r(P)$ and~$N_L \bigl( O_r(M) \bigr) \leq P$.
In particular, we see that~$EX = M \leq P$ and~$O_r(P)$ lies in the~$r$-Sylow~$X$ of~$L$.
Hence, $E$ and~$O_r(P)$ normalize each other, so~$[E, O_r(P)] \leq E\cap O_r(P) = 1$ since~$p \neq r$.
This implies that~$E \leq C_L \bigl( O_r(P) \bigr)$.
Since~$O_r(L) = 1$ and~$O_r(P) > 1$, $P$ is proper in~$L$;
and then using~\cite[(13-2)]{GL83}, $C_{\Aut(L)} \bigl( O_r(P) \bigr) \leq O_r(P)$.
Again~$E \leq O_r(P)$ forces~$E = 1$ since~$p \neq r$, contrary to~$E > 1$ in our assumption that~$E \in \Ap(L)^X$.
This contradiction shows that~$\Ap(L)^X = \emptyset$, and~(3) follows.
(So the proof is actually now complete.)

Finally, if we had some~$E \in \Ap \bigl( \Aut(L) \bigr)^X$, then~$E \cap L$ is invariant under~$X$,
and by so by the previous paragraph, $E \cap L = 1$.
Therefore $[E,X] \leq E \cap L = 1$, so $E \leq C_{\Aut(L)}(X) \leq X$, again contrary to~$p \neq r$.
This contradiction shows~$\Ap \bigl( \Aut(L) \bigr)^X = \emptyset$, so we even get~(3+).
As we had mentioned earlier, this gives another proof that~$X$ satisfies condition~(4) of Property~$\R_0(p)$.
\end{proof}

\bigskip
\bigskip

\section{Proving our two variants of Aschbacher-Smith for the primes~3 and~5}
\label{sec:2varASforp=35}

This section is devoted to proving our two extensions of~\cite[Main Theorem]{AS93} to~$p=3,5$: 
first under~(H1u), and then instead under~(H1). 

\bigskip

So we begin by considering the variant under~(H1u): 
that is, we extend~\cite[Main~Theorem]{AS93} to every odd prime~$p$.
Recall from the Introduction that our extension Theorem~\ref{theoremASExtension}
will be established below as Theorem~\ref{proofQCOdd}, where the proof will involve~(H1u).

\medskip

In fact in~\cite{KP20}, we already extended to~$p=5$ the Aschbacher-Smith theorem.
We recall that the only obstruction there was the possible presence of a component of Suzuki type~$\Sz(2^5)$.
So in~\cite{KP20}, we showed that, if~$G$ contains such a component, then---under~(H1)---$G$ satisfies~(H-QC).
This is of course an example of an elimination-result.

We note that the proof in~\cite{KP20} for the extension to~$p=5$ has a small gap:
For in the proof of~\cite[Main Theorem]{AS93}, the authors do not work under~(H1),
but instead under the~(H1u) hypothesis;
whereas in~\cite{KP20}, to extend~\cite[Main~Theorem]{AS93} to $p=5$, Theorem~5.1 of~\cite{KP20} is invoked;
and the statement of that Theorem requires the~(H1) inductive assumption.
This discrepancy can be easily filled as follows:
We provide below a correspondingly adjusted statement of~\cite[Theorem~5.1]{KP20} instead using~(H1u). 
We will then check that the proof of the original Theorem adapts just as well to this alternative version:

\begin{theorem}
[{cf.~\cite[Theorem~5.1]{KP20}}]
\label{alternativeParticularComponentes}
Suppose that~$G$ satisfies~(H1u) and that:
\begin{quote}
  $G$ has a component~$L \groupiso \PSL_2(2^3)$ ($p=3$), $\PSU_3(2^3)$ ($p=3$), or~$\Sz(2^5)$ ($p=5$).
\end{quote}
Then~$G$ satisfies~(H-QC).
\end{theorem}

\begin{proof}
As usual for~(H-QC) We may suppose that~$O_p(G) = 1$.
The original proof in \cite{KP20} begins by reducing to~$O_{p'}(G) = 1$, just to claim that~$Z(L) = 1$.
However, here we have the property~$Z(L) = 1$ directly by hypothesis; so we do not need to reduce to~$O_{p'}(G) = 1$.

And the rest of the original proof in~\cite{KP20} invokes~(H1) to claim that~$\tilde{H}_*(\ \Ap \bigl( C_G(L) \bigr)\ ) \neq 0$.
For that purpose, we will invoke instead~(H1u):
Hence we need to show that~$O_p \bigl( C_G(L) \bigr) = 1$, and that~$C_G(L)$ satisfies~(H-QC).
By Lemma~\ref{lemmaOpandp}(3,4), the components of~$C_G(L)$ are the components of~$G$ distinct from~$L$,
and $O_p \bigl( C_G(L) \bigr) = O_p(G) = 1$.
Therefore, since $C_G(L)$~is proper in~$G$, while any unitary component of~$C_G(L)$ is a unitary component of~$G$,
$C_G(L)$ satisfies the conditions of~(H1u).
Hence~$\tilde{H}_*(\ \Ap \bigl( C_G(L) \bigr)\ )\neq 0$, as desired.
And now the original proof in~\cite{KP20} goes through.
\end{proof}

With the above re-statement of~\cite[Theorem~5.1]{KP20} in hand,
we can now complete the extension of~\cite[Main~Theorem]{AS93} to every odd prime, by also using the results of this article;
in summary:
While the extension to~$p = 5$ can be done here by using Theorem~\ref{alternativeParticularComponentes}
and the existing argument of~\cite{AS93},
the extension of~\cite[Main~Theorem]{AS93} to~$p=3$ now further requires us to deal with some more exceptional components,
and also to extend certain results of~\cite{AS93} to~$p = 3$.
Concretely: although most of the hard work has already been done in~\cite{AS93},
it remains to eliminate some further components on the~$\QD$-List (recall we have quoted this list as Theorem~\ref{theoremQDList}),
before constructing Robinson subgroups---to finish as in the cases for~$p \geq 5$.
Indeed, the only component that presents a real obstruction to constructing Robinson subgroups is the Ree group~$^2G_2(3^a)'$;
and our earlier result Theorem~\ref{theoremLiep} shows that such a component is not a problem in our setting.

We proceed now with the details of our extension of~\cite[Main Theorem]{AS93} to every odd prime~$p$:

\begin{theorem}
\label{proofQCOdd}
Let~$p$ be an odd prime.
Suppose that if~$G$ has a unitary component~$L \groupiso \PSU_n(q)$ such that~$q$ is odd and~$p \mid q+1$,
then~$\QD_p$ holds for the~$p$-extensions of~$\PSU_m(q^{p^e})$ for~$m \leq n$ and~$e \in \ZZ$.

Then~$G$ satisfies~(H-QC).
\end{theorem}

\begin{proof}
The proof largely follows the path of the original proof of~\cite[Main~Theorem]{AS93}:
indeed we quote those original arguments when~$p>5$; and also, when they apply equally well to the ``new" primes~$p=3,5$---but
when not, we must indicate alternative arguments.

We take~$G$ to be a counterexample of minimal order subject to the conditions of the theorem.
Hence~$G$ satisfies~(H1u).
Then we have~$O_p(G) = 1$ and~$\tilde{H}_*(\Ap(G)) = 0$.
The proofs of the reductions in Propositions~1.3, 1.4, and~1.5 of~\cite{AS93} apply for any~$p$ (including the primes~$p=3,5$);
so we also get~$Z(E(G)) = 1 = Z(G)$.
In particular~$E(G)$ is the direct product of the components of~$G$, which are simple.
 
\medskip

At this point, we insert into the original argument of~\cite{AS93}
an application of the elimination result Theorem \ref{alternativeParticularComponentes},
to exclude $\PSL_2(2^3)$ ($p=3$), $\PSU_3(2^3)$ ($p=3$), and~$\Sz(2^5)$ ($p=5$) as possible components~$L$ of~$G$.
We could optionally also apply further elimination-results here,
which would have the effect of some simplification, over the original path of~\cite{AS93}, of later details;
however, here we instead prefer to closely parallel that original path.

The three eliminations above now remove the obstruction, for the new primes $p=3,5$,
to the proof of the nonconical-complement result~\cite[Theorem~2.3]{AS93}.
Hence also the further result~\cite[Theorem~2.4]{AS93} can be applied in the remainder of our proof.

So we now obtain: first, the reduction~\cite[Proposition~1.6]{AS93} to $O_{p'}(G)=1$;
and then also~\cite[Proposition~1.7]{AS93}, namely the elimination of~$\QD$-components:
This is because their proofs use~\cite[Theorem~2.4]{AS93}, as allowed by the previous paragraph;
we note that this part of the proof does not depend on the particular prime $p$ (and so in particular goes through for~$p=3,5$).

\medskip

In the next few paragraphs, let us describe more precisely the effect of the eliminations above:
Now every component~$L$ of~$G$ has some~$p$-extension~$LB \leq G$ such that~$LB$ fails~$\QD_p$.
In particular, if $L$~is a component of~$G$,
then by the hypothesis~(H1u), $L$ is not~$\PSU_n(q)$ with~$q$ odd and~$p \mid q+1$.
This gives the hypothesis for~\cite[Theorem~3.1]{AS93}---which is stated for all odd~$p$ (hence including~$p=3,5$).
That result shows that~$L$ must appear in the~$\QD$-List---which we have reproduced as Theorem~\ref{theoremQDList} in the present paper.

Now for~$p>5$, Aschbacher and Smith complete the proof by invoking~\cite[Theorem 5.3]{AS93}---which
establishes~$\R(p)$ for any~$L$ in the above~$\QD$-List.
Indeed since that result applies also to~$p=5$, we obtain our extension of their Main Result to~$p = 5$ as well.

\medskip

Hence in the remainder of our proof, we may assume~$p=3$.
In brief overview:
we will first apply some further elimination-results, to avoid certain classes of groups in the~$\QD$-List; and after that,
finish with results on~$\R(3)$ much as in~\cite[Theorem~5.3]{AS93}---cf.~Proposition~\ref{propRobinson}.

So let's briefly assess the~$\QD$-sublist for~$p = 3$ that we must start with:

\begin{enumerate}
\item We must deal with the~$\Lie(3)$ cases in~(1) of the~$\QD$-List.
\item The unitary cases in~(2) of that list are ruled out by hypothesis~(H1u).
\item Case~(3) of the $\QD$-List is eliminated:
        For there, we would have~$q=r^3$, with the orders-mod-$3$ given by~$|q|=|r|=3,6,8,12$;
        but here, $3$ cannot divide~$r,q$ (otherwise these orders are not defined)---and
        so for us, those orders-mod-$3$ can only be~$1,2$.
\item We do not get case~(4), since there~$p=7$ or~$11$.
\item But we must still contend with the specific cases for~$p = 3$ in~(8-12) of the $\QD$-List.
        (We mention that the three special eliminations applied at the start of our proof
        have already removed~$\Ln_2(2^3)$ from~(11), and~$\U_3(2^3)$ from~(10).)
\end{enumerate}

\noindent
Thus it remains to deal with the appropriate cases in~(1) and~(8--12) of the~$\QD$-sublist for~$p = 3$.

We next set up to apply the elimination-result Theorem~\ref{PStheorem}.
Note by our early reduction to~$O_{p'}(G)=1$ that every component~$L$ has order divisible by~$p = 3$.
Furthermore by Lemma~\ref{lemmaOpandp}(3), the components of~$C_G(\hat{L})$ are components of~$G$ and hence satisfy~(H1u);
so since~$G$ is a counterexample of minimal order to our Theorem, we have that~$C_G(\hat{L})$ satisfies~(H-QC).
This gives the hypotheses of Theorem~\ref{PStheorem}, up to~(1) there.

Since~$G$ is in particular a counterexample to~(H-QC),
we see that the remaining hypothesis~(2) of Theorem~\ref{PStheorem} must fail;
so that~$p$ must divide the order of~$\Out_G(L)$.
Thus for~$p$ odd (and hence for our present~$p=3$), this eliminates~$L$ of alternating or sporadic type.
So we eliminate cases~(8) and~(12) in the~$\QD$-List from our list of possible components of~$G$.

We will here apply one further elimination-result, namely Theorem~\ref{theoremLiep}:
Note that we have the ``in particular'' hypotheses~(H1u) and~$p = 3$ odd there; 
furthermore the previous paragraph shows that we have~$B>1$.
And we saw that the remaining in-particular hypothesis of~($p$-$\Phi$) holds for the Lie-type groups of Lie-rank~$1$,
by Remark~\ref{rk:casesPhiGamma}(0).
So we eliminate the subcase of Lie-rank~$1$ from the~$\Lie(3)$ part of case~(1) of the~$\QD$-list.
(Note that the smallest Ree group~$L \cong {^2}G_2(3)' \cong \Ln_2(2^3)$ was already eliminated 
by our earlier application of Theorem~\ref{alternativeParticularComponentes};
we cannot apply Theorem~\ref{theoremLiep} to it, 
since it is one of the exceptional commutators which are not of type~(sLie-$p$) in Definition~\ref{defn:sLie-p}(2).)

To summarize the effect of these eliminations:
we have reduced our~$p = 3$ sublist~(1,8-12) of the~$\QD$-List to:
Case~(1) for Lie-rank~$>1$; and cases~(9-11) (with~$\Ln_2(2^3)$ already eliminated from~(11)).

\medskip

And now we are ready to finish, by turning to results which will give~$\R(3)$ for this latest sublist.
Proposition~\ref{propRobinsonEvenChar} gives~$\R(3)$ for the~$\Lie(2)$ cases---except~$\Symplectic_4(2)'=\Alt_6$,
which was already eliminated above via Theorem~\ref{PStheorem} as the alternating group~$\Alt_6$.
This completes the treatment of~ the cases in~(9--11); reducing us to Case~(1) with Lie-rank~$>1$.

We now observe that the proof of \cite[Theorem~5.3]{AS93} for the~$\Lie(p)$ case actually works for~$p=3$ and~$\R(3)$---except
when~$L \cong {}^2G_2(3)^a$ for~$a > 1$ odd,
which we already eliminated above via Theorem~\ref{theoremLiep}.

In conclusion:
We have shown that every remaining possible component~$L$ of~$G$ (that is, not ruled out by our elimination-results above)
satisfies~$\R(3)$.
By Proposition~\ref{propRobinson}, $G$ satisfies~(L-QC)---and hence also~(H-QC).
This completes our proof of Theorem~\ref{proofQCOdd}.
\end{proof}

\begin{remark}
\label{remarkReeFailsRobinson}
Property~$\R(3)$ cannot be established for the Ree groups $^2G_2(3^a)$ (for~$a > 1$ odd),
since there every field automorphism of order~$3$ of the Ree group commutes with some Sylow $2$-subgroup.
However, $\R(3)$ does hold for the smallest Ree group~${^2}G_2(3)'$, which has no field automorphisms.
\donerk
\end{remark}

\bigskip

Finally, we close this section by proving Theorem~\ref{theoremASExtensionAlternative},
the alternative version under~(H1) of Theorem~\ref{theoremASExtension} (that is, of Theorem~\ref{proofQCOdd} above).
We will particularly emphasize the points of the proof that use the~CFSG.

\begin{proof}
[Proof of Theorem \ref{theoremASExtensionAlternative}]
We assume that~$G$ satisfies both~(H1) and~(H2).
As usual for~(H-QC),  we may suppose further that~$O_p(G) = 1$,
and we must show then that~$\tilde{H}_* \bigl( \Ap(G) \bigr) \neq 0$---now for all odd~$p$, including~$p = 3,5$.

By Theorem~\ref{generalReduction}(2), we can assume that~$O_{p'}(G) = 1$
(this reduction makes use of the~CFSG to invoke the~$p$-solvable case of~(H-QC)).
Note that this implies that every component of~$G$ is simple and of order divisible by~$p$ (see Lemma~\ref{lemmaOpandp}).

We get the following constraints on the components of~$G$, which do not invoke the CFSG:

\begin{enumerate}
\item By~(5) and~(6) of Theorem~\ref{generalReduction}, we can suppose that~$G$ does not contain
      components isomorphic to $\PSL_2(2^3)$, $\PSU_3(2^3)$, or~$\Sz(2^5)$ for~$p=3,3,5$ respectively.

\item By Theorem~\ref{PStheorem}, we can suppose that~$G$ does not contain components of alternating or sporadic type.

\item By Theorem~\ref{QDpReductionPS}, we can suppose that for every component~$L$ of~$G$ there is a~$p$-extension of~$L$ in~$G$
      failing~$\QD_p$.

\item By Theorem~\ref{theoremLiep},
      we can suppose that~$G$ does not contain a component of Lie-type and Lie-rank~$1$ in characteristic~$p$.
\end{enumerate}

\noindent
In view of~(3) above,
we can invoke at this point the~$\QD$-List, which we have quoted as Theorem~\ref{theoremQDList},
to eliminate possible components of $G$.
Here we do use the~CFSG.

Note that~(H2) guarantees that the unitary groups in item~(2) of Theorem~\ref{theoremQDList} are not a possibility in our context.
Hence using the further eliminations noted in (1,2,4) just above,
we see that items~(1), (3), (5), (6), (9), (10), and~(11) of the $\QD$-List Theorem~\ref{theoremQDList}
contain the possible components of~$G$.
Now all of these simple groups satisfy the Robinson property~$\R(p)$ in view of~\cite[Theorem~5.3]{AS93}
and its proof (compare with the end of the proof of Theorem~\ref{proofQCOdd}).
Recall here that~$^2G_2(3^a)'$ does not arise as a component of~$G$ when~$p=3$,
in view of the elimination in item~(4) above, along with~(1) when~$a = 1$ since~$\PSL_2(2^3) \cong {^2}G_2(3)'$.
Finally, note that this step does not require the~CFSG, since Property~$\R(p)$ is shown for specific families of simple groups.
\end{proof}

\begin{remark}
In summary:
In the proof of Theorem~\ref{theoremASExtensionAlternative} above, the~CFSG is used just in two steps:
first, to appeal to the~$p$-solvable case of the Quillen conjecture in reducing to~$O_{p'}(G) = 1$;
and then, to invoke the~$\QD$-List of~\cite{AS93}.
\donerk
\end{remark}

\bigskip
\bigskip

\section{Some results toward the case~\texorpdfstring{$p=2$}{p=2} of Quillen's Conjecture}
\label{sectionEvenQC}

In view of the previous results on odd~$p$, in this section we give some insights for the study of~(H-QC) for~$p=2$.

\medskip

\noindent
First for overall context, recall our general discussions of strategy in Remark~\ref{rk:stratelimRob}.

In general, the theorems of~\cite{AS93} only work for~$p$ odd.
We have been able to extend some of them to every prime~$p$, by using more general combinatorial methods in~$\Ap$-posets.
For example, the reduction~$O_{p'}(G) =1$ was originally stated in~\cite[1.6]{AS93} for~$p>5$;
and an extension to all primes~$p$ was given in~\cite{KP20} by using a more general combinatorial argument.
In addition, the main reductions in Theorems~\ref{generalReduction} and~\ref{PStheorem} here also have no restrictions on~$p$,
and so we can apply them to~$p = 2$.
In a similar vein, Theorem~\ref{QDpReductionPS} extends~\cite[1.7]{AS93} to every prime~$p$,
at the cost of assuming an alternative inductive hypothesis,  namely~(H1);
so in this section, we will make use of the~(H1)-versions of various results.
Finally, the Robinson-path to~(L-QC) described in the previous two sections
(notably Proposition~\ref{propRobinson}) has no a-priori restrictions on~$p$,
so it also admits the possibility of treating~$p = 2$.

Thus (in view of~(MOC) in Remark~\ref{rk:contextofH1results}) we see using~(H1)
that a minimal counterexample~$G$ to~(H-QC) when~$p = 2$ must satisfy~$O_2(G) = 1 = O_{2'}(G)$,
using Theorem~4.1 of~\cite{KP20}.
Further for every component~$L$ of~$G$ we have that:
some~$p$-extensions in~$G$ fails~$\QD_p$ (by Theorem~\ref{QDpReductionPS});
it is not of Lie~type in characteristic~$p = 2$, except possibly for the cases listed in Theorem~\ref{thm:Liep(H1)};
and~$\Out_G(L)$ is not a~$2'$-group (by Theorem~\ref{PStheorem}).

In consequence, if we adapt the strategy of~\cite{AS93} in the spirit of~(elim)/(Rob-nonelim) in Remark~\ref{rk:stratelimRob},
we can try to exhibit a~$\QD$-List for~$p = 2$,
and then show that the possible components in such a list satisfy the Robinson property~$\R(2)$.

So in this section, we give some further elimination-results for~$p = 2$, as consequences of our methods.
We also establish the Robinson property~$\R(2)$ for some of the simple groups.
These results will allow us to conclude (for example)
that a minimal counterexample~$G$ to~(H-QC) for~$p = 2$ has some component~$L$ 
which is of Lie type in some characteristic~$\neq 3$.
(Cf.~Theorem~\ref{mainTheoremEvenCase}, and its proof at the end of this section.)

\bigskip

\noindent
So we now embark on the various results just described:

The following initial elimination-result is an easy consequence of Theorem~\ref{PStheorem}:

\begin{proposition}\label{propNoSuzuki}
Let~$p=2$, and suppose that~$G$ satisfies~(H1).
If a component~$L$ of~$G$ is of Suzuki type~$\Sz(2^{2n+1})$, or of Ree type~$^2F_4(2^{2n+1})'$ or~$^2G_2(3^{2n+1})'$,
then~$G$ satisfies~(H-QC).

In particular, if~$G$ satisfies~(H1) but fails~(H-QC), then~$G$ contains no component of Suzuki type nor of Ree type.
\end{proposition}

\begin{proof}
We saw in the discussion preceding Definition~\ref{defn:fldgraphconvtypePhi} that~$\Out(L)$ is a~$2'$-group in each case.
Hence~$\A_2(L)=\A_2 \bigl( \Aut_G(L) \bigr)$;
and~$\A_2(L) \to \A_2 \bigl( \Aut_G(L) \bigr)$ is the identity map, and so is nonzero in homology
using the almost-simple case of the conjecture, which we quoted earlier as Theorem~\ref{almostSimpleQC}.
The result follows then by Theorem~\ref{PStheorem} (cf.~\cite[Corollary~1.5]{PS}).
\end{proof}

\begin{corollary}
For~$p=2$, if~$G$ satisfies~(H1), then either~$G$ satisfies~(H-QC), or else each component of~$G$ has order divisible by~$3$.
\end{corollary}

\begin{proof}
The only%
\footnote{
   This is a well-known pre-CFSG result, going back to Thompson and Glauberman;
   cf.~Theorem~4.174 and the discussion beforehand (using also Theorem~4.126) in Gorenstein (MR 0698782).
          }%
%\footnote{
%  fakenote:
%  Ron/Richard: Do you have a better reference for this $3'$-fact?
%          }
simple groups of order not divisible by~$3$ are the Suzuki type groups,
which as components lead to~(H-QC) by the previous Proposition.
\end{proof}

\noindent
So it is natural in this section to use~$q=3$ in our~$q$-elementary Robinson-subgroup analysis.

\medskip

In the remainder of the section, we in particular show
that alternating or sporadic components of~$G$ are either eliminated from a minimal counterexample to~(H-QC) with~$p=2$
(by using one of our homology propagation methods), or else they satisfy the Robinson property~$\R(2)$.
We summarize the results for sporadic components, in Proposition~\ref{propSporadicComponents} below;
and for alternating components, in later Corollary~\ref{coroAlternatinComponents}.

The following lemma will give us a quick condition for establishing Property~$\R(2)$ on a simple group.
For primes~$p,q$, recall that the~$p$-local~$q$-rank of a finite group~$L$ is the largest~$q$-rank of a~$p$-local subgroup of~$L$.
More concretely, it is the following number:
  \[ m_{p,q}(L) := \max\{ m_q \bigl( N_L(P) \bigr) \tq P \in \BrownPoset_p(L) \} . \]

\begin{lemma}
\label{lemmaRobinsonCriterion}
Let~$L$ be a simple group such that one of the following conditions holds:
\begin{enumerate}
  \item Whenever~$E \in \A_2 \bigl( \Aut(L) \bigr)$, then~$N_L(E)$ does not contain a Sylow~$3$-subgroup of~$L$.
  \item $m_{2,3}(\Aut(L)) < m_3(L)$.
\end{enumerate}
Then~$L$ has property~$\R(2)$.
\end{lemma}

\begin{proof}
We show first that~(2) implies~(1):
Assuming~(2), take a Sylow~$3$-subgroup~$P$ of~$L$.
Since~$m_3(L) = m_3(P)$, and~$m_{2,3}(\Aut(L)) < m_3(L)$ by hypothesis,
for every~$E \in \A_2 \bigl( \Aut(L) \bigr)$ we have:
  \[ m_3 \bigl( N_L(E) \bigr) \leq m_3 \bigl( N_{\Aut(L)}(E) \bigr) \leq m_{2,3} \bigl( \Aut(L) \bigr) < m_3(L) = m_3(P) . \]
Therefore~$N_L(E)$ cannot contain~$P$, giving~(1) as claimed. 

It remains to show that condition~(1) implies Property~$\R(2)$ (with~$P$ in the role of~``$Q_L$'').
By~(1), $P$ does not normalize an element~$E \in \A_2 \bigl( \Aut(L) \bigr)$.
Therefore we have established~(3+),
namely that~$\A_2 \bigl( \Aut(L) \bigr)^P = \emptyset$;
and we had seen earlier that~(3+) suffices to guarantee~(3) and~(4) of~$\R_0(2)$, 
so we see~$L$ satisfies Property~$\R(2)$ for~$P$.
\end{proof}

We show next how to deal with sporadic components when studying~(H-QC) for~$p= 2$:

\begin{proposition}
\label{propSporadicComponents}
Let $L$ be a sporadic simple group and let $p = 2$.
Then $\A_2(L) \to \A_2(\Aut(L))$ is not the zero-map in homology if $L$ is one of the following~$17$ sporadic groups:
\begin{equation}
\label{casesHP}
     \Mathieu_{11}, \Mathieu_{12},\Mathieu_{22}, \Mathieu_{23}, \Mathieu_{24}, \Janko_1, \Janko_4,
    \Conway{1},\Conway{2},\Conway{3}, \Fi{23}, \HS, \Ly, \Ru, \Th, \Baby,\Monster.
\end{equation}
In the remaining~$9$ sporadic cases, $L$ satisfies Property~$\R(2)$:
\begin{equation}
\label{casesRobinsonNC}
  \Janko_2,\Janko_3, \Fi{22}, \Fi{24}', \McL, \He, \Suz, \ON, \HN.
\end{equation}
\end{proposition}

\begin{proof}
For the proof of this proposition, we use the notation and results of Table~5.3 of~\cite{GLS98}.

We first consider~$L$ in the list~(\ref{casesHP}).
Then either~$\Out(L)=1$ or else~$L = \Mathieu_{12},\Mathieu_{22}, \HS$.
In the cases where~$\Out(L)=1$, we get that~$\A_2(L) \to \A_2 \bigl( \Aut(L) \bigr) = \A_2(L)$ is the identity map,
and hence nonzero in homology as usual by the almost-simple case of~(H-QC),
which we had quoted as Theorem~\ref{almostSimpleQC}.
So we study the remaining three cases of this list:

\begin{itemize}
\item If~$L = \Mathieu_{12}$, then for all~$B \in \Outposet_{\Aut(L)}(L)$,
      $B$ is of type~$2C$ and hence~$O_2 \bigl( C_L(B) \bigr) > 1$ (see Table~5.3b of~\cite{GLS98}).
      So~$\A_2(L) \to \A_2 \bigl( \Aut(L) \bigr)$ is a homotopy equivalence
      using part~(3) of Lemma~\ref{lemmaReconstructionFromSubposet},
      and hence is nonzero in homology using the almost-simple case as above.

\item If~$L = \Mathieu_{22}$, then for all~$B \in \Outposet_{\Aut(L)}(L)$, $B$ is of type~$2B$ or~$2C$,
      and in either case we have $O_2 \bigl( C_L(B) \bigr) > 1$ (see Table~5.3c of~\cite{GLS98}).
      Then~$\A_2(L) \to \A_2 \bigl( \Aut(L) \bigr)$ is again a homotopy equivalence
      using part~(3) of Lemma~\ref{lemmaReconstructionFromSubposet},
      and is nonzero in homology as above.

\item If $L = \HS$, then using~\cite[Theorem~7.1]{PS}, we see that~$\A_2(L) \to \A_2 \bigl( \Aut(L) \bigr)$ is nonzero in homology.
\end{itemize}
This completes the treatment of the list~(\ref{casesHP}).

It remains to consider~$L$ in the list~(\ref{casesRobinsonNC}).
Then~$|\Out(L)| = 2$.
Note that, in each case, if~$B \in \Outposet_{\Aut(L)}(L)$, then~$|B| = 2$, so every element there is maximal, and also $\Aut(L) = LB$.
%Moreover, for each possibility of $L$ we can pick some $B\in \Outposet_{\Aut(L)}(L)$ with $O_2(C_L(B)) = 1$ (see the list of cases below).
We study Property~$\R(2)$ in each case.
From Tables~5.6.1 and~5.6.2 of~\cite{GLS98}, we see that we in fact have~$m_3 \bigl( L \bigr) > m_{2,3}(\Aut(L))$
if~$L$ is one of the following~seven sporadic groups:
  \[ \Janko_3, \Fi{22}, \Fi{24}', \McL, \Suz, \ON, \HN ; \]
and then we may apply Lemma~\ref{lemmaRobinsonCriterion}(2) to obtain~$\R(2)$. 
We complete the details of the proof by analyzing the~two remaining cases~$\Janko_2$ and~$\He$:

\vspace{0.2cm}
\textbf{Case} $L = \Janko_2$.
Let~$S \in \Syl_3(L)$.
Note that~$S$ does not centralize outer involutions.
Moreover, we get~$C_{\Aut(L)}(S) \leq S$ and~$\Ap(L)^S =\emptyset$---either by looking into the maximal subgroups,
or by direct computation.
This gives~(4) and~(3) of Property~$\R_0(2)$, so we see that~$L$ satisfies Property~$\R(2)$.

\vspace{0.2cm}
\textbf{Case} $L=\He$.
We get Property~$\R(2)$ by Lemma~\ref{lemmaRobinsonHE} below.
\end{proof}

The assertions in the following lemma were checked with GAP \cite{GAP}.
We follow the conventions of Table 5.3p of \cite{GLS98}.

\begin{lemma}
\label{lemmaRobinsonHE}
Let~$L = \He$, and let~$Q$ be a cyclic subgroup of order~$21$ generated by elements of type~$3B$ and~$7A$.
Then $Q$ satisfies the requirements of Property $\R(2)$.
%Then~$Q$ is~$q$-elementary for $q = 3$, giving~(1) and~(2) of Property~$\R_0(2)$.
%And~$C_{\Aut(L)}(Q) = Q$, giving~$\A_2 \bigl( C_{\Aut(L)}(Q) \bigr) = \emptyset$, for~(4) there.
%Further~$\B_2(L)^Q$ consists of exactly two non-conjugate elements~$E_1$ and~$E_2$.%
%\footnote{
%  Both are isomorphic to~$C_2 \times C_2$, and are generated by involutions of type~$2A$, with~$E_1 \cap E_2 = 1$;
%  and~$E_1$ is~$L$-conjugate to~$O_2 \bigl( C_L(2A) \bigr)$, with~$\gen{E_1,E_2} \groupiso \PSL_3(2)$.
%          }
%Since the equivalence of~$\A_2(L)$ with~$\B_2(L)$ in Proposition~\ref{homotopyEquivalenPosets}
%is~$L$-equivariant, we get the value of~$+1$
%for the reduced Euler characteristic in~(3) of~$\R_0(2)$. 
%So as $q = 3$, $Q$ satisfies the requirements of property~$\R(2)$.
%\hfill $\Box$
\end{lemma}

\begin{proof}
Clearly $Q$ is~$q$-elementary for $q = 3$, giving~(1) and~(2) of Property~$\R_0(2)$.
And~$C_{\Aut(L)}(Q) = Q$, giving~$\A_2 \bigl( C_{\Aut(L)}(Q) \bigr) = \emptyset$, for~(4) there.
Further~$\B_2(L)^Q$ consists of exactly two non-conjugate elements~$E_1$ and~$E_2$.%
\footnote{
  Both are isomorphic to~$C_2 \times C_2$, and are generated by involutions of type~$2A$, with~$E_1 \cap E_2 = 1$;
  and~$E_1$ is~$L$-conjugate to~$O_2 \bigl( C_L(2A) \bigr)$, with~$\gen{E_1,E_2} \groupiso \PSL_3(2)$.
          }
Since the equivalence of~$\A_2(L)$ with~$\B_2(L)$ in Proposition~\ref{homotopyEquivalenPosets}
is~$L$-equivariant, we get the value of~$+1$
for the reduced Euler characteristic in~(3) of~$\R_0(2)$. 
So as $q = 3$, we conclude that $Q$ satisfies the requirements of Property~$\R(2)$.
\end{proof}

\bigskip

We now make an analysis for alternating components,
similar in spirit to Proposition~\ref{propSporadicComponents} above for sporadic components.

Note that the alternating components~$\Alt_5 \cong \PSL_2(2^2)$ and~$\Alt_6 \cong \Symplectic_4(2)'$
are eliminated from a minimal counterexample in view of Theorem~\ref{thm:Liep(H1)}.
So we show that~$\Alt_n$, with~$n \geq 7$, satisfies the Robinson property~$\R(2)$:

%For the component~$\Alt_6$, we have the following result from~\cite{PS}:
%
%\begin{proposition}
%[{Cf. \cite[Corollary 8.1]{PS}}]
%\label{alt6component}
%Let~$p = 2$.
%Then $\A_2(\Alt_6) \to \A_2(\Aut \bigl( \Alt_6) \bigr)$ is not the zero map in homology.
%Hence, if~$L \groupiso \Alt_6$ is a component of~$G$, and $C_G(\hat{L})$ satisfies~(H-QC), then~$G$ satisfies (H-QC).
%
%In particular, if~$G$ is a minimal counterexample to~(H-QC), $\Alt_6$ is not a component of~$G$.
%\end{proposition}

\begin{proposition}
\label{prof:R(2)forAltngeq7}
Let~$L$ be an alternating group~$\Alt_n$ with~$n \geq 7$.
Then $L$ satisfies Property~$\R(2)$.
\end{proposition}

\begin{proof}
Recall that we have~$\Aut(\Alt_n) = \Sym_n$,  since~$n \geq 7$.
In particular, $\Outposet_{\Sym_n}(\Alt_n)$ consists only of elements of order $2$, which in particular are minimal.
So we see that~$\Sym_n = \Alt_n \cdot B$, for all~$B \in \Outposet_{\Sym_n}(\Alt_n)$.

%Let $\tau$ be the involution $(1\, 2)\in \Sym_n$ and put $E = \langle \tau\rangle$.
%Then $E\in\Outposet_{\Sym_n}(\Alt_n)$ and
%\[C_{\Alt_n}(E)\groupiso \Sym_{n-2}\]
%by \cite[Proposition 5.2.8(a)]{GLS98}.
%Therefore $\Sym_n = \Alt_n E$ and $O_2(C_{\Alt_n}(E)) = 1$.
%Moreover, by Lemma \ref{lemmaPropertyNC}, we get condition (3) of Property $\NC$.
%This shows that $\Alt_n$ satisfies Property $\NC$ if $n\geq 7$.

Now we establish Property~$\R(2)$ for~$\Alt_n$ for our value of~$n \geq 7$.
We let~$Q$ denote the following subgroup of~$\Alt_n$:
  \[Q =\begin{cases}
   \langle (1\, 2\, 3) ,(4\, 5\, \ldots n) \rangle        & 2\mid n\\
   \langle (1\, 2\, \ldots n) \rangle                     & 2\nmid n, 3\mid n\\
   \langle (1\, 2\, 3) ,(4\, 5\,6),(7\, \ldots n) \rangle & 2\nmid n, 3\nmid n .
       \end{cases} \]
Note that in each case $Q$ is a~$3$-elementary $2'$-subgroup of~$\Alt_n$ with $O_3(Q) >   1$,
so we get conditions~(1), (2),  and~(5) of Property~$\R(2)$ for this~$Q \leq \Alt_n$.
So it remains to prove conditions~(3) and~(4) for each case of~$Q$.

\vspace{0.2cm}
\textbf{Case 1.} $2 \mid n$.

\noindent
This case is done in~\cite[p 212]{AK90}, below their Lemma~3, where it is shown that~$\A_2(\Sym_n)^Q = \emptyset$
(including for $n = 6$).
Therefore we have~(3+), which we saw suffices to give conditions~(3) and~(4) of Property~$\R(2)$ for this~$Q$.
This finishes the proof of Case~1.

%Let $x = (1\, 2\, 3)$ and $y = (4\, 5 \, \ldots n)$.
%Note that $Q$ acts on $\Fix(E)\neq \{1,\ldots,n\}$, so it is a union of $Q$-orbits, which are $\{1,2,3\}$ and $\{4,5,\ldots,n\}$.
%Now, if $E$ fixes $\{1,2,3\}$, then, since $E\neq 1$, we see that $E$ acts non-trivially on the $Q$-orbit $\{4,5,\ldots,n\}$ with $n-3$ elements (odd).
%This contradicts Lemma \ref{lemmaAK}.
%Hence $\Fix(E) = \{4,5,\ldots,n\}$ or $\emptyset$.
%
%If $\Fix(E) =\{4,5,\ldots,n\}$ then $E\leq \Sym_4$ and $x\in \Sym_4$ normalizes $E$.
%This implies that $E = O_2(\Sym_4)$, which is not normalized by $y$ (note that $y\neq 1$).
%Therefore $\Fix(E) = \emptyset$.

\vspace{0.2cm}
\textbf{Case 2.} $2 \nmid n$, $3 \mid n$.

\noindent
Note that~$Q$ is transitive on~$X := \{ 1 , \ldots , n \}$, which has an odd number of elements by hypothesis.
If there is any~$E \in \A_2(\Sym_n)$, then~$E$ acts non-trivially on~$X$;
hence Lemma~\ref{lemmaAK} below would show~$n$ is even, contrary to the hypothesis of this case. 
This contradiction shows that~$\A_2(\Sym_n)^Q = \emptyset$.
So again we have~(3+), finishing the proof of Case~2.

\vspace{0.2cm}
\textbf{Case 3a.} $2 \nmid n$, $3 \nmid n$, $n > 7$ (so~$n \geq 11$).

\noindent
Suppose that there is some~$E \in \A_2(\Sym_n)^Q$.
Write~$x := (1 \, 2 \, 3)$, $y = (4 \, 5 \, 6)$, and~$z := (7 \, 8 \, \ldots \, n)$, with~$Q := \Span{x,y,z}$.
Note that both~$C_E(x)$ and~$C_E(y)$ are~$Q$-invariant.

\vspace{0.2cm}
\textbf{Claim.} $C_E(x) = 1 = C_E(y)$; in particular, $E$ is not trivial on~$\{ 1,2,3 \}$ or~$\{ 4,5,6 \}$:

\noindent
Note that~$C_E(x) \leq C_{\Sym_n}(x) = \Span{x} \times \Sym_{ \{ 4,5,\ldots,n \} }$ by~\cite[Proposition~5.2.6]{GLS98}.
Now since~$E$ is a~$2$-group, we conclude that~$C_E(x) \leq \Sym_{ \{ 4,5,\ldots,n \} }$.
But also note that~$y,z \in \Sym_{ \{ 4,5,\ldots,n \} }$, so~$C_E(x) \in \A_2(\Sym_{ \{ 4,5,\ldots,n \} })^{\Span{y,z}} \cup \{ 1 \}$.
By Case~1 (with~$n-3 \geq 8$ in the role of~``$n$''),
we see that we have~$\A_2(\Sym_{ \{ 4,5,\ldots,n \} })^{\Span{y,z}} = \emptyset$:
so that we must have~$C_E(x) = 1$.

Since~$y$ is also a~$3$-cycle, an analogous argument shows that~$C_E(y) = 1$.
This finishes the proof of the Claim.

\smallskip

%The above claim shows that~$E = [E,x] = [E,y]$ using coprime action.
Now since~$n$ is odd with~$E$ a~$2$-group, we see that~$\Fix(E) \neq \emptyset$; and~$Q$ acts on~$\Fix(E)$.
Recall that the~$Q$-orbits of the action of~$Q$ on~$\{ 1,\ldots,n \}$ are~$\{ 1,2,3 \}$, $\{ 4,5,6 \}$, and~$\{ 7,8,\ldots,n \}$.
Further~$\Fix(E)$ is a non-empty union of~$Q$-orbits.
On the other hand, the above Claim shows that~$E$ cannot be trivial on~$\{ 1,2,3 \}$ or~$\{ 4,5,6 \}$.
This forces~$\Fix(E) = \{ 7,8,\ldots,n \}$; and so~$E \leq \Sym_6$, that is, $E \in \A_2(\Sym_6)^{\Span{x,y}}$.
But this is Case~1 again, and we saw there that also for~$n = 6$, we get~$\A_2(\Sym_6)^{\Span{x,y}} = \emptyset$;
contrary to our choice of~$E$ just before the Claim.
This contradiction shows that~$\A_2(\Sym_n)^Q = \emptyset$.
In particular we again get condition~(3+),  and hence Property~$\R(2)$ for this~$Q$.
This concludes the proof of Case~3a.

\vspace{0.2cm}
\textbf{Case 3b.} $n = 7$.

\noindent
It follows by direct computation that~$\A_2(\Sym_7)^Q = \A_2(\Alt_7)^Q$ is discrete with two points.
In particular, $\A_2(C_{\Sym_n}(Q)) \subseteq \A_2(\Alt_7)$ and~$\tilde{\chi} \bigl( \A_2(\Alt_7)^Q \bigr) = 1 \not\equiv 0 \pmod{3}$.
Hence $Q$ satisfies conditions~(4) and~(3) of Property $\R(2)$ for~$L = \Alt_7$.
This finishes the proof of Case~3b, and hence of Case~3.

We have shown in all cases that~$\Alt_n$ for~$n \geq 7$ satisfies Property~$\R(2)$.
\end{proof}

During the proof,
we used the following easy consequence of nontrivial action under a~$p$-group:

\begin{lemma}[Cf.~Lemma~3 of~\cite{AK90}]
\label{lemmaAK}
Assume that~$Q \leq G$ and that~$E \in \Ap(G)^Q$ acts non-trivially on some~$Q$-orbit~$\mathcal{O}$.
Then~$p \mid |\mathcal{O}|$.
\hfill $\Box$
\end{lemma}

We combine the results above on the alternating-component case in the following corollary.

\begin{corollary}
\label{coroAlternatinComponents}
Let~$p = 2$, and suppose that~$G$ satisfies~(H1).
Suppose in addition that~$L$ is a component of~$G$, such that~$L/Z(L) \groupiso \Alt_n$.

\begin{enumerate}
\item If~$n = 5$ or~$6$, then~$G$ satisfies~(H-QC).
\item If~$n \geq 7$, then~$L/Z(L)$ satisfies Property $\R(2)$.
\end{enumerate}
\hfill $\Box$
\end{corollary}

We can now prove Theorem~\ref{mainTheoremEvenCase}; the elimination-parts~(1)--(3) are established as follows: 
Note that from Theorem~\ref{thm:Liep(H1)}
and the list~(\ref{casesHP}) in Proposition~\ref{propSporadicComponents},
we get conclusion~(1) there.
Then Theorem~\ref{QDpReductionPS} gives conclusion~(2); and Theorem~\ref{thm:Liep(H1)} gives conclusion~(3).

In proving conclusion~(4),
we may as well assume we have already eliminated components~$L$ described in (1)--(3). 
Then all the remaining components~$L$ of~$G$ satisfy~$\R(2)$:
the alternating components by Proposition~\ref{prof:R(2)forAltngeq7};
the sporadic components via the ``remaining-$9$''-list~(\ref{casesRobinsonNC}) in Proposition~\ref{propSporadicComponents};
and the Lie-type components, which must be defined in characteristic~$3$ by the hypothesis of~(4),
by Proposition~\ref{propRobinsonEvenChar}.
Then the standard Robinson-type argument on~$\R(2)$
in Proposition~\ref{propRobinson} gives~(L-QC), and hence~(H-QC) as required.

\begin{remark}
The simple group $L = \PSL_3(2^2)$ is a potential candidate to fail most arguments that we describe here, for the case $p = 2$.

\begin{itemize}
\item This simple group satisfies (sLie-$p$), but corresponds to one of groups in the exclusion list of Theorem \ref{thm:Liep(H1)}: There exist $2$-outers $B\in \Outposet_{\Aut(L)}(L)$ of $2$-rank $2$, generated by graph and field automorphisms.
Hence $\Outposet_G(L)$ may contain non-cyclic $2$-outers if $L$ is a component of $G$ (see also the discussion in Remark \ref{rk:fullCases}.)

\item We have $m_2(L) = 4$ and $\A_2(L)\simeq \B_2(L)$ has the homotopy type of a bouquet of $1$-spheres.
This shows that $\QD_2$ fails for $L$.
Thus Theorem \ref{QDpReductionPS} cannot be applied for $L$.

\item It can also be proved that $\A_2(\Aut(L))$ has the homotopy type of a bouquet of $2$-spheres, so the map $\A_2(L) \to \A_2(\Aut(L))$ is the zero map in homology.
This shows that we cannot apply Theorem \ref{PStheorem} in general.

\item Finally, we see that we cannot produce Robinson subgroups for $q=3$.
That is, $\PSL_3(2^2)$ fails $\R(2)$.
Indeed, it fails $\R_0(2)$ for any possible choice of $q$: Since $|\PSL_3(4)| = 20160 = 2^6\cdot 3^2\cdot 5 \cdot 7$, here $q$ could be $3,5,7$.
In any case, for a nontrivial $q$-subgroup $Q$ of $L$, we have that $C_L(Q)$ is a Sylow $q$-subgroup of $L$, and $C_{\Aut(L)}(Q)$ contains $2$-outers of $L$.
So $\A_2(C_{\Aut(L)}(Q))$ is not included in $\A_2(L)$.
\end{itemize}

The pathological behavior of this simple group shows that, as a component of a finite group, this cannot be treated by using some of the theorems described in this article.
In conclusion, further elimination results will be necessary to study the case $p=2$ of the conjecture, if we attach to the usual strategy (elim)/(Rob-nonelim).

This behavior of $\PSL_3(2^2)$ is closely related to the fact that $\Aut(\PSL_3(2^2))$ satisfies the Robinson Property (R2), as defined in \cite{AK90} (Cf. Theorem 2 there.)
\end{remark}

\bigskip
\bigskip

\appendix

\bigskip
\bigskip

\part{Appendix: some more technical background results}

\section{An extended form of Quillen's fiber-Theorem}
\label{sec:appendixAfiber}

We had stated Quillen's fiber-Theorem in earlier Theorem~\ref{variantQuillenFiber}.
Below we recall generalized version, and outline a combinatorial proof.

Recall that an~$n$-equivalence is a continuous function~$f : X \to Y$
such that~$f$ induces isomorphisms in the homotopy groups~$\pi_i$ with~$i<n$, and an epimorphism in~$\pi_n$.
By the Hurewicz theorem, an~$n$-equivalence also induces isomorphisms in the homology groups of degree~$\leq n-1$,
and an epimorphism in degree~$n$.

\begin{proposition}[Extended Quillen Fiber Theorem]
\label{prop:QuilFiberConn}
Let~$f : X \to Y$ be a map between finite posets, and let~$n \geq 0$.

\begin{enumerate}
\item Suppose that for all~$y \in Y$, $f^{-1}(Y_{\leq y}) * Y_{>y}$ (resp.~$f^{-1}(Y_{\geq y}) * Y_{<y}$) is in fact~$(n-1)$-connected.
Then~$f$ is an~$n$-equivalence.
\item In addition, assume that $X,Y$ are $G$-posets for some group $G$, and that $f$ is a $G$-equivariant map of posets such that for all~$y \in Y$, $f^{-1}(Y_{\leq y}) * Y_{>y}$ (resp.~$f^{-1}(Y_{\geq y}) * Y_{<y}$) is $G_y$-contractible.
Then $f$ is a $G$-homotopy equivalence.
\end{enumerate}

\end{proposition}

\begin{proof}
(Sketch)
We begin by showing item (1).
To that end, we consider the non-Hausdorff cylinder of the map $f$.
This is the finite poset~$B(f)$ whose underlying set is the disjoint union of~$X$ and~$Y$.
It keeps the given order in~$X$ and~$Y$; and if~$x \in X$ and~$y \in Y$, set~$x < y$ if~$f(x) \leq y$.
It can be shown that~$Y$ is a strong deformation retract of~$B(f)$,
via the map~$r : B(f) \to Y$ defined by~$r(x) = f(x)$ if~$x \in X$ (and the identity in~$Y$).
Moreover, if~$i : X\hookrightarrow B(f)$ is the inclusion map, then~$(r \circ i)(x) = f(x)$, so $ri = f$.
Since $r$ is a homotopy equivalence, we see that $f$ is an~$n$-equivalence if and only if the inclusion~$i : X \hookrightarrow B(f)$ is an~$n$-equivalence.

Let~$Y := \{ y_1 , \ldots , y_s \}$ be a linear extension of~$Y$, such that~$y_i \leq y_j$ implies~$i \leq j$.
Further set~$X_i := X \cup \{ y_i , \ldots , y_s \}$.
Hence~$X_0 = B(f)$ and~$X_{s+1} = X$.
Note that~$X \hookrightarrow B(f)$ is the composition of all the inclusions~$X_{i+1} \hookrightarrow X_i$.
We show that each inclusion is an~$n$-equivalence.
Note that:
  \[ X_i = X \cup \{ y_i , \ldots , y_s \} = X_{i+1} \cup \{ y_i \}, \]
  \[ {X_i}_{>y_i} = Y_{>y_i} \quad \text{ and } \quad {X_i}_{<y_i} = X_{<y_i} = f^{-1}(Y_{\leq y_i}). \]
Therefore the conclusion of item (1) follows from Lemma \ref{lm:pointRemovalNequivalence} below.

Next, we prove item (2).
Here we will use the fact that if $f,g:X\to Y$ are two equivariant maps between $G$-posets such that $f\leq g$, then $f$ and $g$ are equivariantly homotopic (see also the discussion in \cite[p 285]{KP19}).
Observe that the non-Hausdorff cylinder~$B(f)$ is naturally a~$G$-poset;
we have~$X \subseteq B(f)$ and~$r : B(f)\to Y$ is a~$G$-equivariant strong deformation retract (and hence a $G$-homotopy equivalence) since $r(z) \geq z$ for all $z\in B(f)$.
And by the equality $ri = f$, we see that it remains to show that the inclusion $i:X\hookrightarrow B(f)$ is a $G$-homotopy equivalence.

Now, instead of taking an arbitrary linear extension of~$Y$,
we decompose~$Y$ into a disjoint union of~$G$-orbits~$Y_1 , \dots , Y_m$,
such that if~$y \in Y_i$ and~$y' \in Y_j$ are such that~$y < y'$, then~$i<j$.
(Note that we can never have~$y<y'$ with~$i=j$.)
Let~$X_i = X \cup Y_i \cup \dots \cup Y_m$;
it remains to see that if we remove a whole orbit~$Y_i$,
then we get a~$G$-homotopy equivalence~$X_{i+1} \hookrightarrow X_i$.
Hence it is enough to establish the case~$B(f) = Z = X \cup Y_0$,
where~$Y_0$ is a~$G$-orbit, and for all~$y \in Y_0$ we have that~$\Lk_Z(y) = X_{<y} * X_{>y}$ is~$G_y$-contractible.

At this point, we pass through the posets of chains $X'$ and $Z'$.
Since the natural maps $X'\to X$ and $Z'\to Z$, which send a chain to its maximal element, are $G$-homotopy equivalences, we see that $i:X\hookrightarrow Z$ is a $G$-homotopy equivalence if and only if the inclusion $j:X' \hookrightarrow Z'$ is.
We show that the latter map is a $G$-homotopy equivalence.
To this aim, we prove that the maps $j_1:Z'-Y_0' \hookrightarrow Z'$ and $j_2:X' \hookrightarrow Z' - Y_0'$ are $G$-homotopy equivalences.
Since $j = j_1\circ j_2$, this will conclude the proof.

First, note that an element $s\in Y_0'$ is just a chain of the form~$s = s_y := ({y})$ for some $y \in Y_0$.
Then we see that $(Z'-Y_0')_{>s_y} = Z'_{>s_y} = \bigl( \Lk_Z(y) \bigr)'$ is~$G_y$-contractible by hypothesis.
Therefore $j_1$ is a $G$-homotopy equivalence by the Thevenaz-Webb result \cite[Thm 1]{TW}.

Second, we show that $j_2$ is a $G$-homotopy equivalence.
If $\tau \in Z'-Y_0'$, then $r(\tau) := \tau - Y_0$ is a non-empty chain lying in~$(Z-Y_0)'=X'$.
Thus, the map~$r : Z'-Y_0' \to (Z-Y_0)'$ is an equivariant map that satisfies~$r(\tau) \leq \tau$.
To finish the argument, note that $r\circ j_2 = \Id_{X'}$ and $j_2\circ r \leq \Id_{Z'-Y_0'}$.
Since all these maps are $G$-equivariant, we conclude that $j_2$ is a $G$-homotopy equivalence with inverse $r$.
\end{proof}

\begin{lemma}
\label{lm:pointRemovalNequivalence}
Let~$X$ be a finite poset, and~$x \in X$ such that~$X_{<x}*X_{>x}$ is~$(n-1)$-connected.
Then the inclusion~$X - \{ x \} \hookrightarrow X$ is an~$n$-equivalence.
\end{lemma}

\begin{proof}
(Sketch)
For simplicity, we only show the homology version of the result.
The homotopy version follows from the homotopy excision theorem.

Consider the covering~$X = (X-x) \cup \St_X(x)$, where~$\St_X(x) = \{ y \in X \tq y \leq x \text{ or } y \geq x \}$.
Then~$(X-x) \cap \St_X(x) = X_{<x}*X_{>x}$.
Note that~$\St_X(x)$ is contractible since every element is comparable with~$x$.
By the Mayer-Vietoris sequence applied to this decomposition, and since the intersection is~$(n-1)$-connected,
if~$m \leq n$ then we have:
  \[ \underset{=0\text{ when }m<n}{\tilde{H}_m(X_{<x}*X_{>x})}
           \to \tilde{H}_m(X-x) \to \tilde{H}_m(X) \to \tilde{H}_{m-1}(X_{<x}*X_{>x}) = 0. \]
Hence the inclusion~$X-x \hookrightarrow X$ is an isomorphism in the homology groups of degree at most~$n-1$,
and it is an epimorphism in degree~$n$.
\end{proof}

%\bigskip
%
%Here are some notes (KIP email 9/13/21)%
%\footnote{
%  "fake footnote":
%  (K) Then in the proof of A.1,~maybe we should consider including this equivariant part?
%      Anyway, we can decide this later, but since I didn't find it in the literature, maybe we can add the proof in the Appendix.
%  (S) Here I'm just editing in, for the moment, the draft-form from your email.
%      I'll leave it to you to flesh it out---when the tex files are back with you...
%          }
%towards a future equivariant form of the above extension of the fiber theorem:
%
%\bigskip
%
%The proof could be as follows (this also contains the non-equivariant proof by taking~$G=1$, or the trivial action):
%
%\medskip

\section{Field and graph automorphisms in the same characteristic~\texorpdfstring{$p$}{p}}
\label{sec:fldgraphautcharp}

In this Section, we collect some more-technical aspects of the automorphisms in the set~$\Phi$---which
we had indicated in Definition~\ref{defConventionFieldAndGraphsCharp}
as our preferred ``field-like'' outer automorphisms of~$L$ under~(sLie-$p$).

\bigskip

\textbf{Introduction: Motivation for the terminology.\/}
Our first goal in the Section will be to lead up to the general naming-conventions
for ``field'' and ``graph'' automorphisms in~$\Phi$---we will follow the conventions used in~\cite[Sec~2.5]{GLS98}.%
%\footnote{
%  fakenote:
%  In response to your email of 7nov21, I'm adding some more exposition---ideally to get a better "feel" for the GLS-conventions.
%  Then when you have the tex files, please adjust to your satisfaction...
%          }

{\em Caution\/}:
Much as we did before in the area of~(\ref{eq:Out*=PhiGamma}),
we temporarily---up to Definition~\ref{defn:fldgraphconvtypePhi}---suspend our assumption of (sLie-$p$) (and hence~$p = r$);
so that this initial discussion will hold for any prime~$r$, indepdendent of our prime~$p$ used for~(H-QC).

\medskip

Recall that the group automorphisms of~$L$ in~$\Phi_L$ are defined, at the level of the overlying algebraic group~$\bar L$,
using powers of the generator~$x \mapsto x^r$ of the Galois group
of the algebraic closure~$\overline{ {\mathbb F}_r }$ of the prime field~${\mathbb F}_r$.
So from that algebraic-group viewpoint, it would be natural to call the members of~$\Phi_L$ by the name ``field automorphisms'' of~$L$.

And indeed this is ``mostly'' the convention in~\cite[Sec~2.5]{GLS98}:
Namely we'll see in Proposition~\ref{propertiesPhiType}(1) below that usually the members~$B$ of~$\Phiposet_G(L)$ 
have the intuitively-natural property of centralizing in~$L$ a subgroup of the same Lie-type---but defined instead
over the fixed subfield of~$B$.
However: we'll also see that in some of the~$d$-twisted groups~$L$, 
certain members of~$\Phi_L$ (a subset that we will call~$\Phi_d$) exhibit a further diagram-symmetry
aspect---causing them to centralize in~$L$ a subgroup which is defined instead by a ``twisted'' Dynkin diagram;
as we will record in Proposition~\ref{propertiesPhiType}(2).
These two cases for~$\Phi$ are correspondingly called ``field'' and ``graph'' automorphisms in~\cite{GLS98};
%these can be easily checked to coincide with the conventions of~\cite{GL83}.)%
this post-classical terminology, determined by the above centralizer-behavior, appears in Definition~2.5.13 there.%
%\footnote{
%  fakenote:
%  Ron/Richard:  have we correctly understood your notation, in our discussion in the first few paragraphs of Appendix SecB up to here?}
We will then indicate our particular special subcase, under~(sLie-$p$), in Definition~\ref{defn:fldgraphconvtypePhi} below.

\bigskip

\noindent
We turn to providing some details implementing the above overview: 

We'll now make use of the following expanded-notation of~\cite{GLS98}:
For~$q$ denoting the order of the characteristic-$r$ field of definition, we can write~$L$ in the form:
  \[ L \groupiso {^d\Sigma(q)} : \]
where the case~$d = 1$ indicates that~$L$ is the untwisted group with Dynkin diagram of type~$\Sigma$;
but~$d = 2$ or~$3$ indicates that~$L$ is the twisted group corresponding to~$\Sigma$, 
constructed using a diagram-symmetry of order~$d$.

In particular, this value of~$d > 1$ is one of the possibilities~$s$ in Remark~\ref{rk:outerautsofsimple}(i) for elements
of the group of diagram-symmetries~$\Delta \groupiso \Cyclic_2$, or~$\Sym_3$ for~$D_4$.
We emphasize these distinct notations, as follows:
Here~$d > 1$ indicates the order of a symmetry from~$\Delta$ that was {\em previously fixed\/}---and used:
in the construction of twisted~$L$, and indeed identifying it by name.
By contrast, as we go on to further consider the possible automorphisms~$x \in \Aut(L)$,
the various order-$s$ symmetries in~$\Delta$ are still available---as potential ingredients in the action of~$x$ on~$L$. 
In particlar, the~$d$-notation does {\em not\/} mean that~$L$ itself has a graph automorphism of order~$d$.
Indeed we saw in Remark~\ref{rk:casesPhiGamma}(0) that for twisted~$L$ (where~$d > 1$),
we have~$\Gamma_L = 1$---so that even though~$\Delta$ has order at least~$d$,
it makes no graph-contribution to~$x \in \Aut(L)$.
%The {\em ground\/} field of~$L$ is~$\GF{q^d}$;
%and we saw in discussing~(\ref{eq:Out*=PhiGamma})
%that there is an embedding~$\Aut( {\mathbb F}_{q^d} ) \groupiso \Phi_L \leq \Aut(L)$,

\medskip

We now examine the automorphisms in~$\Phi$, in the above~$d$-context. 

\medskip

\noindent
Consider first the case~$d = 1$: so~$L$ is untwisted, and~$\Sigma$ can be any Dynkin diagram:  

Here~$\Phi_L$ consists just of the usual field automorphisms. 
And since we will be working mainly under~($p$-$\Phi$) later, this is basically what we need to know for untwisted~$L$.

But for completeness, and especially as background to the twisted groups below,
we now examine the conventions related to classical graph autmorphisms for the untwisted groups~$L = {^1}\Sigma(q)$.
So we consider the subcases for~$\Sigma$, namely with~$\Delta > 1$, where such automorphisms might arise from diagram symmetries:

We saw for the single-bond diagrams~$\Sigma = A_n,D_n,E_6$ in Remark~\ref{rk:casesPhiGamma}(2) 
that~$\Out(L)^* = \Phi_L \times \Gamma_L$, with~$\Gamma_L \cong \Delta \cong \Cyclic_2$, or~$\Sym_3$ for~$D_4$. 
Here the nontrivial members of~$\Gamma_L \cong \Delta$ give graph automorphisms in the classical sense of Steinberg;
and their products with nontrivial members of~$\Phi_L$ give graph-field automorphisms.  

And we also saw in Remark~\ref{rk:casesPhiGamma}(1)
for the multiple-bond diagrams~$\Sigma = B_2,F_4,G_2$ with the further restriction~$r=2,2,3$
that we somewhat-similarly have~$\Out(L)^* = \Phi_L \Gamma_L$ cyclic, with~$\Phi_L$ of index~$2$. 
Now in the special case where~$\Phi$ is an~$r'$-group, we get $\Phi_L \Gamma_L = \Phi_L \times \Gamma_L$, 
so that the order-$r$ subgroup~$\Gamma_L$ as above gives a classical graph automorphism in the Steinberg-sense;
which also determine graph-field automorphisms via products as in the previous paragraph. 
But more generally, a generator of~$\Gamma_L$ has~$r$-th power given by a nontrivial member of~$\Phi_L$, 
and so should be considered a graph-field automorphism.
Indeed, the later terminology-convention in Definition~2.5.13(b) of~\cite{GLS98}
is to just call {\em all\/} the members of~$\Phi_L\Gamma_L - \Phi_L$ by the name of graph-field automorphisms 
(that is, even the classical Steinberg-sense graph automorphisms in the above special case);
so there are {\em never\/} any graph automorphisms in this multiple-bond case, in the new-convention.

As already mentioned,
we will be avoiding, via the~$\Phi$-part of our assumption~($p$-cyclic)---indeed typically the more general ($p$-$\Phi$)--- the above graph- and graph-field cases for such untwisted~$L$,
in our generic-Theorem~\ref{theoremLiep}; although we will be able to treat the special case for double-bonds above, where~$\Phi$ is an~$r'$-group for~$p=r$,
via different methods in the non-generic Theorem~\ref{thm:Liep(H1)}. 
(Indeed some of our results might be generalized to include also some of the other possibilities with no~$p$-outers in~$\Phi$.)

\bigskip

\noindent
Now consider the remaining case~$d > 1$, so that~$L = {^d}\Sigma(q)$ is twisted (for the relevant~$\Sigma$ above):

Here we saw in Remark~\ref{rk:casesPhiGamma}(0) that~$\Gamma_L = 1$:
Thus there are no Steinberg-sense graph automorphisms;
and~$\Out(L)^* = \Phi_L$, which classically would be called field automorphisms.
So we now examine the revised-convention in Definition~2.5.13(c) of~\cite{GLS98},
in which the name ``graph automorphism'' is used for certain members of~$\Phi_L$:

In the multiple-bond cases~$\Sigma=B_2,F_4,G_2$ with~$r=2,2,3$, we have~$d=2$;
and here there is no change of terminology:
that is, all members of~$\Phi_L$ are still called ``field automorphisms''.
Furthermore, we recall (e.g.~Theorem~2.2.3(b) in~\cite{GLS98})
that our twisted~$L$ in these cases is defined over an {\em odd\/} power of~$r$;
so we see that for $d=2$ here, we have~$\Phi$ a~$d'$-group.

So now consider the single-bond cases~$\Sigma=A_n,D_n,E_6$, with~$d=2$ (or possibly~$3$ in case~$D_4$).
Again there is no change of terminology,
for elements~$a \in \Phi_L$ of order coprime to~$d$---they are still called ``field automorphisms''.
However, if~$d$ does divide the order of~$a$, the revised-terminology of~\cite{GLS98} is to call~$a$ a ``graph automorphism''.  
And correspondingly for this~$d > 1$ we set:

\centerline{
  $\Phi_d := \{\ a \in \Phi \tq d \text{ divides } |a|\ \}$---which is non-empty iff~$L = {^d}\Sigma(q)$ with single-bond~$\Sigma$;
            }

\noindent
where for the final statement, recall we saw above that~$\Phi$ is a~$d'$-group for the twisted multiple-bond cases.
The motivation for this choice of terminology is indicated in the Remark after Definition~2.5.13 of~\cite{GLS98};
we now expand on one aspect which suggests ``graph-like'' behavior of such~$a$:
Namely our~$d$-twisted group~$L$ arises as the fixed-point subgroup~$C_{\bar L}(\bar{f} \bar{g})$,
where~$\bar L$ is an overlying algebraic group of (untwisted) type~$\Sigma$, 
for a suitable order-$d$ field automorphism~$\bar f \in \Phi_{\bar L}$,
and a commuting order-$d$ graph automorphism~$\bar g \in \Gamma_{\bar L}$.
In particular, we can regard~$L$ as the subgroup on which~$\bar f$ acts as~${\bar g}^{-1}$.
Since~$\Phi_{\bar L}$ is cyclic, we can expect~$a$ to be influenced by this graph-behavior of~$\bar f$
(notably when~$\bar a = \bar f$, which leads to the ``if'' part of the ``iff'' statement above).
And indeed we'll see in Proposition~\ref{propertiesPhiType}(2)(b) below
that we can get~$C_L(a)$ of Lie-type---for a diagram {\em different\/} from~$\Sigma$ for~$L = {^d}\Sigma(q)$.

\bigskip

With the above discussion of general naming-conventions in place,
in the remainder of the Section we will obtain some results
which are related to our further hypothesis of~$L$ satisfying~(sLie-$p$):
Recall this means that~$L$ is simple, and of Lie-type in characteristic~$r$; 
where now~$r = p$ for the prime~$p$ we use in studying~(H-QC).
Note since~$L$ is simple that~$L$ is adjoint (cf.~Theorems~2.2.6 and~2.2.7 of~\cite{GLS98}).

So for this context,
we will first specialize our naming-conventions to outer automorphisms of order~$p$.
Recall that~$\Phi$ is cyclic and normal in~(\ref{eq:Out*=PhiGamma}).
Thus for~$x \in \Aut(L)$, we see that we have~$x$ conjugate to~$y$ with~$y^* \in \Phi$
if and only if we already have~$x^* \in \Phi$. 
In particular, for such an~$x$ of order~$p$, its location in~$\Phi$ is determined, since:

\centerline{
  Cyclic~$\Phi_L$ has a unique subgroup~$\Omega_1 \bigl( O_p(\Phi_L) \bigr)$ of order~$p$ (which is also normal).
             }

\noindent
We can now check, using the discussion above,
that such nontrivial elements~$x \in \Omega_1 \bigl( O_p(\Phi_L) \bigr)$ lie in~$\Phi_d$ 
if and only if~$p = d$; recall we already had~$p = r$.
(Here we are again using the fact that for the twisted groups with multiple-bond diagrams, $\Phi$ is a~$d'$-group for~$d = 2$.)

Below, we indicate the~\cite{GLS98}-naming for such~$x$ of order~$p$ in~$\Phi$;
we will use the case division on whether~$\Phi_d$ is non-empty;
since the alternative form of the division on whether~$p = d$ is not so suggestive of the structures that are really involved.

\begin{definition}[``Field'' and ``graph'' naming-conventions for order-$p$ automorphisms in~$\Phi$]
\label{defn:fldgraphconvtypePhi}
Assume that~$L = {^d}\Sigma(q)$ satisfies~(sLie-$p$), so that~$p = r$;
and we have some~$x \in \Aut(L)$ of order~$p$, with~$x$ conjugate under~$\Aut(L)$ to an element of~$\Phi$. 
Then in fact~$x^*$ itself lies in~$\Phi$, and so generates~$\Omega_1 \bigl( O_p(\Phi) \bigr)$. 
And the naming-convention of~\cite{GLS98} gives us the following two cases:

(1) ($\Phi_L - \Phi_d$=field):
When~$x^* \in \Phi_L - \Phi_d$, $x$ is called a \textit{field automorphism}.
Notice, when we now consider the specific diagram-types $\Sigma$ for $L = {^d}\Sigma(q)$,
that this field-type is the ``usual'' case:
since we saw~$\Phi_d = \emptyset$ for untwisted~$L$ (i.e.~$d = 1$); 
and also for many twisted cases~$L$---namely with ~$d >1$ but not indicated in~(2) below:
Indeed, we saw in our preliminary discussion above 
that for twisted~$L$ where~$\Sigma$ has multiple-bonds (namely~$B_2,F_4,G_2$ with~$d = 2$),
we have~$\Phi$ is a~$2'$-group in each case.
Hence since we are now further assuming~(sLie-$p$) so that~$p = r = 2,2,3$,
such an~$x$ exists if and only if~$p =3$ and~$L$ has type~${^2}G_2$.
By contrast, when~$\Sigma$ has single-bonds, our twisted~$L$ here is in the proper-subcase where~$p \neq d$:
namely~$p \neq 2$ in all types~$A_n$, $D_n$, $E_6$; or~$p \neq 3$ in type~${^3}D_4$. 

(2) ($\Phi_d$=graph): Otherwise~$x^*$ is in~$\Phi_d$, and~$x$ is called a \textit{graph automorphism}.
Recall that~$\Phi_d \neq \emptyset$ requires~$L$ twisted,
with single-bond diagram~$\Sigma$ (namely of type~$A_n,D_n,E_6$), and~$d$ dividing the order of~$\Phi_L$;
while~$p =d $ by our assumption that~$x$ has order~$p$.
Thus the possibilities for the group~$L$ are:

\centerline{
  $L\ =\ {^2A_{n}(q)}=\PSU_{n+1}(q)$ ($n \geq 3$);\ $^2D_n(q)$ ($n \geq 4$);\ $^2E_6(q)$;\ and~$^3D_4(q)$;
            }

\noindent
with~$p=d=2,2,2,3$ respectively.
\donerk
\end{definition}

Next we summarize the structure of the centralizers of order-$p$ elements of type~$\Phi$:
%The following assertions are deduced from Propositions 4.9.1 and 4.9.2 of \cite{GLS98}, and also from \cite[(9-1)]{GL83}.
%See also Theorems 2.2.6 and 2.2.7 of \cite{GLS98}.

\begin{proposition}
\label{propertiesPhiType}
Let~$L \groupiso {^d\Sigma(q)}$ be of type~(sLie-$p$).
Assume~$B \in \Outposet^{\Phi}_{\Aut(L)}(L)$;
and note since~$\Phi$ is cyclic in~(\ref{eq:Out*=PhiGamma}) that~$B^* = \Omega_1 \bigl( O_p(\Phi) \bigr)$.

\begin{enumerate}
\item Assume~$B$ is in case~(1) = $\Phi - \Phi_d$=field of Definition~\ref{defn:fldgraphconvtypePhi}.
      Then the elements of~$\Phiposet_{\Aut(L)}(L)$ are all~$\Inndiag(L)$-conjugate to~$B$.
      Further~$O^{p'} \bigl( C_L(B) \bigr) \groupiso { {^d}\Sigma(q^{1/p}) } $:
      i.e.~the (adjoint) group of the same Lie-type (and hence the same Lie-rank), but over the~$B$-fixed subfield;
      and we have~$C_{\Inndiag(L)}(B) = \Inndiag\left({^d \Sigma(q^{1/p})}\right)$.
      In particular, $O_p \bigl( C_L(B) ) = 1$.

\item Assume instead~$B$ is in case~(2)=$\Phi_d$=graph of Definition~\ref{defn:fldgraphconvtypePhi}.
      Then~$B$ satisfies one of the following two mutually exclusive conditions:
\begin{enumerate}
\item $F^* \bigl( C_L(B) \bigr) = O_p \bigl( C_L(B) \bigr) > 1$ is a nontrivial~$p$-group; or else:
\item $B$ is~$\Inndiag(L)$-conjugate to~$\gen{\gamma}$, where~$\gamma\in \Phi_d$ has order~$p$,
      and the centralizer~$C_L(B)$ is a subgroup of the same Lie-rank~$m$ below as~$L$,
      but of the indicated different, untwisted Lie-type over~${\mathbb F}_q$:%
\footnote{
      The Lie-type of~$C_L(B)$ can be regarded as that for a suitable quotient-diagram of the Dynkin diagram for~$L$,
      under the order-$d$ diagram symmetry.
          }
%\footnote{
%  fakenote:
%  Ron/Richard:  Are we correct in understanding that these centralizers are also adjoint?
%Though seemingly this isn't needed, since larger centralizers would just improve our bound?}
  \[ \begin{array}{|c|r|c|l|}
\hline
        p & L &m & C_L(B)\\
\hline
        2 & \PSU_n(q) \groupiso {^2}A_{n-1}(q) & \left\lfloor \frac{n}{2} \right\rfloor
 	             & C_{\left \lfloor \frac{n}{2} \right\rfloor}(q) \groupiso \PSp_{2 \left\lfloor \frac{n}{2} \right\rfloor }(q) \\
        2 & ^2D_n(q) & n-1 & C_{n-1}(q) \groupiso \PSp_{2(n-1)}(q) \\
        2 & ^2E_6(q) & 4   & F_4(q)\\
        3 & ^3D_4(q) & 2   & G_2(q).\\
\hline
     \end{array}
  \]

      In particular, $O_p \bigl( C_L(B) ) = 1$.
\end{enumerate}

\item For any~$C \in \Phiposet_{\Aut(L)}(L)$, we have~$\Outposet_{LC}(L) = \Phiposet_{LC}(L)$,
      so that we can take~$C$ in the role of~``$B$'' above.
      Further there is~$D \in \Outposet_{LC}(L)$ (hence of type~$\Phi$) with~$O_p \bigl( C_L(D) \bigr) = 1$.
      In particular: if~(2)(a) holds for~$C$, then~(2)(b) holds for~$D$.

\end{enumerate}

\end{proposition}

\begin{proof}
Item~(1) follows directly from Proposition~4.9.1 of~\cite{GLS98};
here~$O_p \bigl( C_L(B) \bigr) = 1$ follows from the Lie-type structure of~$O^{p'} \bigl( C_L(B) \bigr)$.
(Indeed that Lie-type group is usually even simple.)
Item~(2) follows from Proposition~4.9.2 of~\cite{GLS98}---using similar Lie-type remarks
for~$O_p \bigl( C_L(B) \bigr) = 1$ in case~(2)(b).
%Item (3) follows from the structure of the centralizers.
The first part of~(3) then follows
from the observation early in the statement that~$B^* = \Omega_1 \bigl( O_p(\Phi_L) \bigr)$.
And then by that first part, we can---and do---now take~$B$ in the role of the original~``$C$''.

If~(1) holds for~$B$, then the second part of~(3) already holds, with~$B$ itself in the role of~``$D$''.
So we may  assume~(2) holds for~$B$; and in particular,
$L$ is twisted, for~$\Sigma$  with single-bonds (hence all roots are ``long''). 
If in fact case~(2)(b) already holds for~$B$, then again the second part of~(3) holds with~$B$ in the role of~``$D$''. 
So we may assume that case~(2)(a) holds for~$B$.
With the notation of Proposition~4.9.2 of~\cite{GLS98}, let~$Z$ be a long-root subgroup of~$L$.
Then~$Z \leq L$; and as in the proof  of that Proposition,
any member of~$\Phiposet_{\Aut(L)}(L)$ is~$\Inndiag(L)$-conjugate
to~$\gen{\gamma}$ or~$\gen{\gamma z}$ (for~$z \in Z^{\sharp})$, where~$\gamma \in \Phi_L$ of order~$p$ is as in item~(2)(b).
That is, $LB$ contains an element of the form~$\gamma^g$ or~$(\gamma z)^g$, with~$g \in \Inndiag(L)$.
%if~$O_p \bigl( C_L(B) \bigr) > 1$,
If we have the latter, that is~$(\gamma z)^g \in LB$, 
then since~$z^g \in Z^g \leq L^g = L$, we also get~$\gamma^g \in LB$;
so we have~$O_p \bigl( C_L(\gamma^g) \bigr) \groupiso O_p \bigl( C_L(\gamma) \bigr) = 1$, as we had noted for case~(2)(b).
Hence we can take~``$D$'' to be~$\gen{\gamma^g} \in \Outposet_{LB}(L)$. 
This final case completes the proof of the second part of~(3).
(See also the Remark right after the statement of Proposition~4.9.2 in~\cite{GLS98}.)
\end{proof}

Note that item~(3) of the previous Proposition provides a ``nonconical complement",
as in Theorem~2.3 of~\cite{AS93}---we had mentioned these via~(someNC) in Remark~\ref{remarkClosingSection}, 
and we will need them for the corresponding hypothesis~(someNC+$\Phi$) in Theorem~\ref{theoremLiep}.

\bigskip

The starting-point for our homology propagation in Theorem~\ref{theoremLiep}
will be the nonzero homology of~$\Ap(LB)$ provided by Theorem~\ref{theoremHomologyCharP}.
And to bound the vector-space dimension of that homology group above zero, 
we will want the following easy but somewhat-technical numerical observation:
Roughly, we check below, based on our determination of centralizer-structures in Proposition~\ref{propertiesPhiType},
and using just the standard group-order formulas for the individual simple groups~$L$ of Lie-type,
that the index~$|L : C_L(B)|_{p'}$ (which basically records the number of tori) grows very strongly with increasing powers of~$p$;
and in particular, exceeds the very modest bound of~$(p-1)$ that we require.
This result may be well-known, or clear to those familiar with the polynomials in the orders:%
%\footnote{
%   fakenote: Ron/Richard: It seems to us this easy bound may already be known?  Do you have a reference? Maybe even a simpler proof?}

\begin{lemma}
\label{lemmaLiepBoundCentralizer}
Let~$L$ be of type~(sLie-$p$).
If~$B \leq \Phiposet_{\Aut(L)}(L)$ is such that~$O_p \bigl( C_L(B) ) = 1$
(that is, case~(1) or~(2)(b) of Proposition~\ref{propertiesPhiType} holds), 
then $|L:C_L(B)|_{p'} > (p-1)$.
\end{lemma}

\begin{proof}
For the order-formulas discussed below, we will use Table~2.2 of~\cite{GLS98} as our basic source.
And rather than checking each detailed formula, we will try to exploit common features of the formulas
as products of cyclotomic polynomials in~$q$, which are easy to read from that Table.

Recall we write~$L \groupiso {^d\Sigma(q)}$.
We split the verification of the bound into the two cases in the hypothesis:
the field-automorphism case~$B^* \subseteq \Phi_L - \Phi_d$ of~(1) of Proposition~\ref{propertiesPhiType};
and the graph-automorphism case~$B^* \subseteq \Phi_d$ of~(2)(b) there. 

\bigskip

\noindent
\textbf{The field-case~$B^* \subseteq \Phi_L - \Phi_d$:}

We now deviate from our earlier notational convention of using~$r$ for the characteristic of~$L$---of
course~$p$ supplies that characteristic for us under~(sLie-$p$)---to instead write~$r$ for~$q^{1/p}$, 
the order of the fixed-subfield under~$B$. 
Thus~$q = r^p$, where $r$ itself is a power of $p$ (so in particular, $r \geq p$). 

\bigskip

\noindent
We first treat the case of untwisted~$L$~($d = 1$):

The formula for~$|L|_{p'}$ for the adjoint simple group~$L$ has the form:
  \[ \frac{1}{f}\ \Pi_e\ (q^e - 1) , \]
where the exponents~$e$ and the divisor~$f$ are specified in Table~2.2 of~\cite{GLS98}.
In type~$A_n$, we have~$f = \gcd(n+1,q-1) \leq q-1$, while~$f \leq 4 \leq q$ for all other Lie types;
so for crude bounding, it will usually suffice to just say in summary  that: 

\centerline{
   $f \leq q$.
           }

\noindent
Now we will need to compare the above product for $|L|_{p'}$ with the similar product for~$|C_L(B)|_{p'}$:
By case~(1) of Proposition \ref{propertiesPhiType},
$C_L(B)$ lies between also-adjoint~${^d}\Sigma(r)$ and~$\Inndiag \left( {^d\Sigma(r)} \right)$.
Here the possible extra order in the~$\Inndiag$-form is controlled by~$f$; 
that is, for our bounding-purposes, the $\Inndiag$-order formula we need to use
is the analogue for~$r$ of the polynomial-product for~$|L|_{p'}$ above. 
Note that:
  \[        \Pi_e\ \frac{(r^p)^e - 1}{r^e - 1} 
          = \Pi_e\ \frac{(r^e)^p - 1}{r^e - 1} 
          = \Pi_e\ \bigl(\ (r^e)^{p-1} + (r^e)^{p-2} + \cdots + 1\ \bigr) ; \]
so it will suffice to verify the following variant-form of our original bound:
\begin{equation}
\label{eq:prodqtwofcyclpolys}
     \Pi := \Pi_e\ \bigl(\ r^{e(p-1)} + r^{e(p-2)} + \cdots + 1\ \bigr) \text{ should be } > f (p-1) . 
\end{equation}

%\begin{equation}
%\frac{ |L|_{p'} }{I_0 } > (p-1).
%\end{equation}

\noindent
We now observe that the number of exponents~$e$ is equal to the Lie-rank $n$---giving the number of factors in~$\Pi$. 
Furthermore the smallest exponent always takes the value~$e = 2$. 
In particular, note then that each factor in~$\Pi$ exceeds~$r^2 + 1 > 1$.
So it will suffice to find a sub-product satisfying our bound---and it turns out that the first factor (for~$e=2$) already works:
Namely since we saw~$f \leq q$, it suffices for the purposes of~(\ref{eq:prodqtwofcyclpolys}) to set~$e = 2$ there,
and get a value~$> q(p-1)$.

To that end, consider first the case that~$p=2$:
Then~(\ref{eq:prodqtwofcyclpolys}) gives $r^2 + 1 > r^2 = q(p-1)$---as required.
So now we assume the remaining case, that is, $p \geq 3$:
Then~$2(p-1) \geq p+1$, so the leading term $r^{2(p-1)}$ in  the~$e=2$ factor in~(\ref{eq:prodqtwofcyclpolys})
is at least $r^{p+1} = qr > q(p-1)$, since~$r$ is a power of~$p$.
This completes the proof of the bound for field-automorphisms of untwisted~$L$. 

\bigskip

\noindent
So we turn to field-automorphisms for the case of twisted~$L$, where~$d > 1$; we can use similar ideas:

Let's first recall the cases here:
We saw in the corresponding case~(1) of Definition~\ref{defn:fldgraphconvtypePhi}
that the only multiple-bond case arising here is type~${^2}G_2$ for an odd power of~$p = 3$;
while we can have the single-bond twisted types when~$p \neq d$: 
namely~${^2}A_n$, ${^2}D_n$, ${^2}E_6$ when~$p \neq 2$; and type~${^3}D_4$ when~$p \neq 3$. 

Again we have a product~$\Pi$ of quotients of polynomials, in an analogue of~(\ref{eq:prodqtwofcyclpolys}); 
but this time, there are some differences from the untwisted-$L$ case above:
The number of exponents is now the Lie-rank of the overlying Lie type group~$\Sigma(q)$,%
%\footnote{
%  fakenote:
%  I adjusted notation to standard GLS, away from my sloppy $L(q)$; I hope this addresses your comment in email of 5nov21?
%          }
which exceeds the Lie-rank of the twisted group~$L = {^d}\Sigma(q)$; and this number is always~$\geq 2$---so
we have no case where~$\Pi$ has just one single factor (even when~$L$ has Lie-rank~$1$).
And now the exponents~$e$ come with signs~$\epsilon = \pm 1$ (or cube roots of unity for type ${^3}D_4$),
giving the polynomial~$q^e - \epsilon$.  
We still have minimum exponent-value~$e=2$, now always with sign~$\epsilon = -1$. 
One effect is to give, in type~${^2}A_n$, the adjusted value~$f = \gcd(n+1,q+1) \leq q+1$---and correspondingly now:

\centerline{
  $f \leq q+1$.
            }

\noindent
Another effect is that for~$\epsilon \neq  1$, we don't get the even-division of~$q^e + 1$ by ~$r^e + 1$ (nor analogously for cube roots);
however, we still get the property that each factor in~$\Pi$ exceeds~$1$, 
just using crude estimates, such as:
  \[ \frac{r^{pe} + 1} {r^e + 1} > \frac{r^{pe}} {2r^e} \geq \frac{r^{p(e-1)}} {2} \geq 2 > 1 . \]
So again, we may establish our bound on a sub-product (and again it suffices that we have two factors). 

Namely consider any of the above types other than~${^3}D_4$: 
Then in each case, there is a choice of a non-minimal exponent~$e \geq 3$ with sign~$\epsilon = 1$.
So we recall the calculation in the untwisted case, showing this second-factor contributes at least~$qr$ to the product. 
This time using~$r < q$ since~$L$ is twisted, we get:
  \[ qr > qr + (r - q) + (-1) = (q+1)(r-1) \geq f(p-1) , \]
again giving the bound desired in the analogue of~(\ref{eq:prodqtwofcyclpolys}). 

So we turn to the remaining case of type~${^3}D_4$:
Here~$f=1$; so we can just quote our earlier crude bound~$\frac{r^{p(e-1)}} {2}$ with~$e=2$,
giving at least~$\frac{r^p}{2} \geq r^{p-1} \geq r > p-1$, as desired. 

This completes the proof for field automorphisms of twisted~$L$,
and hence for field automorphisms of all~$L$---the case $B^* \in \Phi_L - \Phi_d$.

\bigskip

\noindent
\vspace{0.2cm}
\textbf{The graph-case~$B^* \subseteq \Phi_d$:}

Notice here we do not consider the fixed-subfield under~$B$, 
and so our quotient-calculations in~$\Pi$ will involve only the power~$q$ of~$p$ (that is, no~``$r$'').

The four types for~$L$ in Proposition~\ref{propertiesPhiType}(2)(b) are~${^2}A_{n-1}(q),{^2}D_n(q),{^2}E_6(q),{^3}D_4(q)$, 
where we have~$p=2,2,2,3$; and~$C_L(B) = C_{\lfloor \frac{n}{2} \rfloor}(q),C_{n-1}(q),F_4(q),G_2(q)$. 
Here as before in type~${^2}A_{n-1}(q)$ we have~$f \leq \gcd(n,q+1) \leq q+1$, and~$f \leq 4 \leq q$ in other types,
so that again: 
  \[ f \leq q+1 ; \]
and though this estimate might be sharpened by dividing out~$f$-type terms in~$C_L(B)$, it will suffice to use the above rough bound.

Again we get a product~$\Pi$, in an analogue of~(\ref{eq:prodqtwofcyclpolys})--now with still more differences:
First, the number of exponents for twisted~$L={^d}\Sigma(q)$ is the Lie-rank for untwisted~$\Sigma(q)$, namely~$n,n,6,4$;
which strictly exceeds the number of exponents for untwisted~$C_L(B)$, i.e.~its Lie-rank~$\lfloor \frac{n}{2} \rfloor, n-1,4,2$. 
Now we check in each~$L$ (again excepting type~${^3}D_4(q)$, which varies slightly)
that we may pair each exponent~$e$ for~$C_L(B)$ (which has positive sign $\epsilon = +1$)
with an exponent~$e$ of the same value, but negative sign~$\epsilon = - 1$, starting at the minimum value~$e=2$.
As a result, the factor at~$e$ in the product~$\Pi$ has the value~$\frac{q^e+1}{q^e-1}$, only slightly larger than~$1$.
Thus to get our bound, we focus instead on the exponents~$e$ of~$L$ which are {\em not\/} paired with those of~$C_L(B)$;
these give ``numerator-only'' terms in the product~$\Pi$.

And just one is enough:  note we have~$e \geq 3$ for such an exponent, giving a term at least:
  \[  q^3 + 1 > q^2 - 1 = (q+1)(q-1) \geq f (p-1) , \]
which is the bound that we needed.

To adapt the above proof to cover the remaining type~$L = {^3}D_4(q)$:
Just ``pair''the exponent~$e = 6$ for~$C_L(B) = G_2(q)$ with the two exponents~$e=4$
(and their ``signs''~$\omega$ and~$\bar{\omega}$)---the quotient of~$(q^4-\omega)(q^4-\bar{\omega})=q^8+q^4+1$
by~$q^6-1$ certainly exceeds~ $1$.
This leaves an unpaired numerator-term of $q^8+q^4+1$, which exceeds the desired bound $p - 1$, since $f = 1$ in this case.

We have now verified the bound of Lemma~\ref{lemmaLiepBoundCentralizer} in all cases.
\end{proof}

\section{Additional lemma(s)}

The following lemma provides an alternative way of eliminating certain components $L$ of a minimal counterexample $G$ to (H-QC):
Namely, those simple components $L$ for which $O_p(C_L(B)) > 1$ for every $p$-outer $B$.
The proof of this lemma invokes Theorem \ref{theoremInductiveHomotopyType} with a non-standard $p$-subgroup poset $\B$.
Indeed, we take $\B$ to be the image-poset $\Imageposet_{G,L}$, for a simple component $L$ of a group $G$.

For example, if $L = \PSL_2(2^3),\PSU_3(2^3),\Sz(2^5)$ with $p=3,3,5$ respectively, then every $p$-outer $B$ of $L$ satisfies that $O_p(C_L(B)) > 1$.
Therefore, Lemma \ref{lemmaAllConical} provides an alternative argument, via Theorems \ref{PStheorem} and \ref{theoremInductiveHomotopyType}, to the elimination of these components (cf. \cite[Thm 5.1]{KP20}.)

\begin{lemma}
\label{lemmaAllConical}
Let~$G$ be a finite group, and let~$L \leq G$ be a simple subgroup.
If~$O_p \bigl( C_L(B) \bigr) > 1$ for all~$B \in \Outposet_G(L)$,
then the inclusion~$\Ap(L)\hookrightarrow \Imageposet_{G,L}$ is a homotopy equivalence.

In particular,
if~$L$ is a component of~$G$ of order divisible by~$p$ and~$C_G(\hat{L})$ satisfies~(H-QC), then~$G$ satisfies~(H-QC).%
\footnote{
  Recall from Definition~\ref{defOrbitComponentes} that~$\hat{L}$ denotes the product of the~$G$-orbit of~$L$.
          }
\end{lemma}

\begin{proof}
Since $L$ is simple, we can regard~$\Ap(L) \subseteq \Imageposet_{G,L} \subseteq \Ap(\Aut(L))$.
Let~$\B:= \Imageposet_{G,L}$.
Then:
  \[ \F_\B(L) = \{ D C_G(L)/C_G(L) \tq D \in \Outposet_G(L)\} = \Outposet_{\Aut(L)}(L)\cap \Imageposet_{G,L} , \]
using Remark~\ref{remarkImagePoset}, again since~$Z(L)=1$ by the hypothesis of simplicity of~$L$.
Fix~$D \in \Outposet_G(L)$, and let~$B$ be the image of~$D$ in~$\B$.
Suppose that~$E\in \N_\B(L)$ is such that~$C_{E \cap L}(B) > 1$.
We may pick some~$F\in \Ap(N_G(L))$ with $FC_G(L)/C_G(L) = E$.
In particular, $F \cap L > 1$.
Then it is straightforward to verify that:
 \[ C_E(B)  = C_F(D) C_G(L) / C_G(L) \in \B , \]
 \[ C_E(B)B = (C_F(D)D) C_G(L)  /C_G(L) \in \B . \]
Therefore $\B$ satisfies the requirements of Theorem~\ref{theoremInductiveHomotopyType},
so we conclude that we have a homotopy equivalence~$\B \simeq X_\B(L) = \Ap(L) \cup \F_\B(L)$.

Now if~$B \in \F_\B(L)$, then~$O_p \bigl( C_L(B) \bigr) > 1$ by our hypothesis on such~$B$.
Then we have:
  \[ X_\B(L)_{>B} \cap \Ap(L) = \{ E \in \Ap(L) \tq C_E(B) > 1 \}  \simeq \Ap \bigl( C_L(B) \bigr) \simeq * . \]
Therefore~$X_\B(L)$ collapses to the subposet~$\Ap(L)$; that is, the inclusion~$\Ap(L) \hookrightarrow X_\B(L)$
is a homotopy equivalence by the Quillen fiber-Theorem~\ref{variantQuillenFiber}.
Since the equivalence~$\alpha_{\B,L}:\B\to X_\B(L)$ is the identity when restricted to $\Ap(L)$
by Theorem \ref{theoremInductiveHomotopyType}(1), we conclude that $\Ap(L)\hookrightarrow \B$ is also a homotopy equivalence.
This finishes the proof of the first part.

The ``In particular" part follows from Theorem~\ref{PStheorem},
since~$\Ap(L) \to \Imageposet_{G,L}$ is an isomorphism in homology,
and $\tilde{H}_* \bigl( \Ap(L) \bigr) \neq 0$ by the almost-simple case  of the Conjecture
(recall the result of~\cite{AK90} quoted as Theorem~\ref{almostSimpleQC}.)
\end{proof}

\end{document}